\documentclass[11pt]{article}
\usepackage{amsmath}
\usepackage{amsthm}
\usepackage{amsfonts}
\usepackage{amssymb}
\usepackage{graphicx}
\usepackage{color}
\usepackage{rotating}
\usepackage{authblk}
\usepackage{url}
\usepackage{tcolorbox}
\usepackage{fullpage}
\usepackage{algorithm}
\usepackage[noend]{algpseudocode}
\usepackage{tikz}

\usepackage{subcaption}
\usepackage{mwe}

\usepackage{multirow}
\usepackage{booktabs}
\usepackage{longtable}
\usepackage{array}

\usepackage{parskip}

\makeatletter
\def\BState{\State\hskip-\ALG@thistlm}
\makeatother

\DeclareMathOperator*{\argmin}{arg\,min}

\newtheorem{theorem}{Theorem}
\newtheorem{proposition}{Proposition}
\newtheorem{lemma}{Lemma}
\newtheorem{definition}{Definition}
\newtheorem{corollary}{Corollary}

\newtheorem{assumption}{Assumption}

\begin{document}
\title{An Approximation Algorithm for \\ training One-Node ReLU Neural Network}
\author[1]{Santanu S. Dey\thanks{santanu.dey@isye.gatech.edu}}  
\author[2]{Guanyi Wang\thanks{gwang93@gatech.edu}}
\author[3]{Yao Xie\thanks{yao.xie@isye.gatech.edu}}

\affil[1,2,3]{School of Industrial and Systems Engineering, Georgia Institute of Technology, Atlanta, USA} 
\maketitle

\begin{abstract}
	Training a one-node neural network with ReLU activation function (\ref{eq:One-Node-ReLU}) is a fundamental optimization problem in deep learning. In this paper, we begin with proving the NP-hardness of training \ref{eq:One-Node-ReLU}. We then present an approximation algorithm to solve \ref{eq:One-Node-ReLU} whose running time is $\mathcal{O}(n^k)$ where $n$ is the number of samples, $k$ is a predefined integral constant. Except $k$, this algorithm does not require pre-processing or tuning of parameters. We analyze the performance of this algorithm under various regimes. First, given any arbitrary set of training sample data set, we show that the algorithm guarantees a $\frac{n}{k}$-approximation for training \ref{eq:One-Node-ReLU} problem. As a consequence, in the realizable case (i.e. when the training error is zero), this approximation algorithm achieves the global optimal solution for the \ref{eq:One-Node-ReLU} problem.  Second, we assume that the training sample data is obtained from an underlying one-node neural network with ReLU activation function, where the output is perturbed by a Gaussian noise. In this regime, we show that the same approximation algorithm guarantees a much better asymptotic approximation ratio which is independent of the number of samples $n$. Finally, we conduct extensive empirical studies and arrive at two conclusions. One, the approximation algorithm together with some heuristic performs better than gradient descent algorithm. Two, the solution of the approximation algorithm can be used as starting point for gradient descent -- a combination that works significantly better than gradient descent.
\end{abstract}

\section{Introduction}
Training neural networks is a fundamental problem in machine learning. 

As a first step of understanding the theoretical properties of training neural networks, we study training the most basic neural network with the following structure: a single node with rectified linear unit function (ReLU) as its activation function (See Figure \ref{fig:Onn-Node-NN}).
\def\layersep{2.5cm}
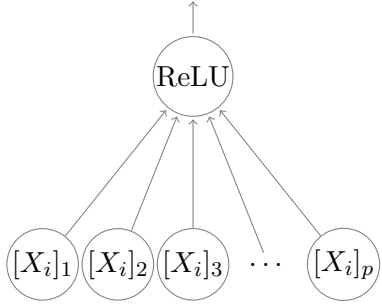
\begin{figure}[h]
\begin{center}
\begin{tikzpicture}[shorten >=1pt,->,draw=black!50, node distance=\layersep,transform shape,rotate=90]  
    \tikzstyle{every pin edge}=[<-,shorten <=1pt]
    \tikzstyle{neuron}=[circle,fill=black!25,minimum size=17pt,inner sep=0.5pt]
    \tikzstyle{input neuron}=[neuron, fill=white!, draw];
    \tikzstyle{output neuron}=[neuron, fill=white!, draw];
    \tikzstyle{annot} = [text width=4em, text centered]
    \tikzset{hoz/.style={rotate=-90}}   
    
	\foreach \name / \y in {1/1,2/2,3/3,5/p}
		\node[input neuron, hoz] (I-\name) at (0,-\name) {$[X_i]_\y$};
		\node[hoz] (I-4) at (0,-4) {$\dots$};
    

    \node[output neuron,pin={[pin edge={->}]right:\rotatebox{-90}{}}, right of=I-3] (O) {\rotatebox{-90}{ReLU}};

    \foreach \source in {1,...,4}
        \draw (I-\source) -- (O);  
    \draw (I-5) -- (O) node [midway, below, sloped] (TextNode) {};

    \node[annot,above of=I-1, node distance=1cm,hoz] (hl) {};
    \node[annot,right of=hl,hoz] {};
\end{tikzpicture}
\end{center}
\caption{Single node neural network with ReLU activate function}
\label{fig:Onn-Node-NN}
\end{figure}
Formally, in this paper, training a single-node neural network with ReLU activation function is the following:
Given a set of $n$ sample points $\{ (X_i, Y_i) \}_{i = 1}^n \in \mathbb{R}^p \times \mathbb{R}$ where $X_i$ is the $i^{\text{th}}$ \textit{input sample (observation sample)}, and $Y_i$ is the $i^{\text{th}}$ \textit{output sample (response sample)}, the task is to minimize the empirical average of sum of square loss 
\begin{align*}
	\min_{(\beta, \beta_0) \in \mathbb{R}^p \times \mathbb{R}} \frac{1}{n} \|\max\{\mathbf{0}, X^{\top}\beta + \beta_0 \mathbf{1}\} - Y\|_2^2 \tag{One-Node-ReLU} \label{eq:One-Node-ReLU}
\end{align*}
where $X = (X_1 | X_2 |\ldots | X_n) \in \mathbb{R}^{p \times n}$ and $Y = (Y_i)_{i = 1}^n \in \mathbb{R}^{n \times 1}$. 

Note that \ref{eq:One-Node-ReLU} can be viewed as a non-linear regression problem.  More general versions of this problem from the perspective of nonlinear regression have been studied in~\cite{toriello2012fitting, magnani2009convex}, which approach this problem from a perspective of Integer programming model and heuristics. 

We caution the reader that in the field of machine learning, the phrase \emph{complexity of training a neural network} has been used to refer to various different problems with corresponding goals. For example, in \cite{goel2016reliably}, their target is to find a feasible hypothesis $h$ that satisfies the false-positive rate condition and the expected loss condition.  In ~\cite{SongVempala2018}, complexity is measured by the number of samples that is needed to learn a certain class of function. Thus, any other hardness/lower bounds depends on how the training problem is defined. \emph{We reiterate here: this paper solely deals with the computational complexity (i.e. the amount of computational effort needed and the related questions of design of algorithm) questions related to solving the optimization problem~\ref{eq:One-Node-ReLU}.}

Note that in \ref{eq:One-Node-ReLU} problem, we do not assume $Y_i > 0$ holds for every $i \in [n]$. Under this formulation, if there exists $i$ such that $Y_i < 0$, then the optimal objective function cannot be $0$. Without loss of generality, we may assume $I^+ := \{i \in [n]: Y_i > 0 \} = \{1, \ldots, m\}$ and $I^- := \{i \in [n]: Y_i \leq 0\}= \{m + 1, \ldots, n\}$.

The rest of the paper is organized as follows: Section~\ref{sec:theory} presents our theoretical results and highlights comparison with related results in the literature. Section~\ref{section:numerical-exp} presents out computational results. In Section~\ref{sec:conclusion}, we provide concluding results. Section~\ref{sec:proofs} contains all the proofs of results presented in Section~\ref{sec:theory}.

\section{Main Theoretical Results \label{section:Main-Result}} \label{sec:theory}
Training a one-node neural network as defined in \ref{eq:One-Node-ReLU} is a non-convex optimization problem which we expect it to be challenging to solve. However, not all non-convex problems are difficult (i.e. NP-hard). For example the classical principal component analysis problem which is non-convex but can be solved in polynomial-time. 

In this paper, we analyze the optimization problem \ref{eq:One-Node-ReLU} in two scenarios:
\begin{enumerate}
\item Arbitrary data: In this case, we do not assume anything about the data, i.e., the training data is arbitrary. We would like to find optimal values of $\beta, \beta_0$ to fit the function $\textup{max}\{0, x^{\top}\beta + \beta_0\}$ to the given data. We would like to answer the following questions: Is this problem NP-hard? Can we come up with an approximation algorithm? How well does this approximation algorithm perform in the worst case?
\item Underlying statistical model: In this case, we assume that the training data is of the form: (1) $X_i$'s are iid sampled from a ``reasonable'' distributions, (2) $Y_i  = \textup{max}\{ 0, X_i^{\top} \beta^{\ast} + \beta_0^{\ast}\} + \epsilon$ where $\epsilon$ is a Gaussian noise and $(\beta^*, \beta^*_0)$ is the ground truth. We show that the same approximation algorithm described for arbitrary data case above, performs significantly better. 
\end{enumerate}


\subsection{Training \ref{eq:One-Node-ReLU} With Arbitrary Data}
For the first scenario, suppose the set of sample points $\{(X_i, Y_i)\}_{i = 1}^n$ are fixed and arbitrary, we study \ref{eq:One-Node-ReLU} in the perspective of computational complexity. In this section, for convenience we drop the $\frac{1}{n}$ term from the objective function from \ref{eq:One-Node-ReLU} since it does not change the optimal solution.  

Our first result formalizes the fact that we expect \ref{eq:One-Node-ReLU} to be NP-hard.

\begin{theorem}[NP-hardness] \label{thm:NP-hard}
	The \ref{eq:One-Node-ReLU} problem is NP-hard. 
\end{theorem}

See Section~\ref{section:NP-hard} for a proof. Our proof of Theorem \ref{thm:NP-hard} is by showing that subset sum problem can be reduced to \ref{eq:One-Node-ReLU} problem. 
\paragraph{Comparison with related results from literature:} We study training ReLU neural networks in the perspective of NP-hardness when the input data are fixed and given. The two-layer $(k + 1)$ nodes neural network problem with ReLU as activation function has been studied in \cite{digvijay2018}, which shows that the training problem is NP-hard. Comparing to our main results, we show that even a more simplified structure, a neural network with one node is NP-hard. 
In \cite{manurangsi2018computational}, P. Manurangsi and D. Reichman independently gave another NP-hardness reduction.
%

Based on NP-hardness result in Theorem~\ref{thm:NP-hard}, it is natural to consider an efficient approximation algorithm with multiplicative bound. We first introduce some basic notions that explain the design of the algorithm. First, we can represent the \ref{eq:One-Node-ReLU} problem by dividing the summation into two parts:
\begin{align*}
	& \min_{(\beta, \beta_0) \in \mathbb{R}^p \times \mathbb{R}} \|\max\{\mathbf{0}, X^{\top}\beta + \beta_0 \mathbf{1}\} - Y\|_2^2 \\
	= & \min_{(\beta, \beta_0) \in \mathbb{R}^p \times \mathbb{R}} \sum_{i \in \{1, \ldots, m\}} (\max\{0, X_i^{\top} \beta + \beta_0\} - Y_i)^2 +  \phi(\beta, \beta_0) \tag{$\ast$} \label{eq:ast}
\end{align*}
where $\phi(\beta, \beta_0) \triangleq \sum_{i \in \{m + 1, \ldots, n\}} (\max\{0, X_i^{\top} \beta + \beta_0\} - Y_i)^2$. It is easy to observe that: 
\vspace{10pt}
\begin{proposition} \label{prop:convex}
		The second term of (\ref{eq:ast}), that is the  function $\phi(\beta, \beta_0)$, is convex. 
\end{proposition} 

The first term of (\ref{eq:ast}) can be further represented as a two-phase optimization problem as follows:
\begin{align*}
	& \min_{(\beta, \beta_0) \in \mathbb{R}^p \times \mathbb{R}} \|\max\{\mathbf{0}, X^{\top}\beta + \beta_0 \mathbf{1}\} - Y\|_2^2 = \left(\min_{I \subseteq [m]} \left( \min_{(\beta, \beta_0) \in P(I)} f_I(\beta, \beta_0) \right) \right) + \phi(\beta, \beta_0)
\end{align*}
where, for a given index set $I$, define set $P(I)$ and function $f_I(\beta, \beta_0)$ be:
\begin{align*}
	P(I) := & ~ \left\{ (\beta, \beta_0): 
	\begin{array}{lll}
		X_i^{\top} \beta + \beta_0 > 0, & i \in I \\
		X_i^{\top} \beta + \beta_0 \leq 0, & i \in [m]\backslash I
	\end{array}
	\right\}, & \text{set of feasible region of $\beta$,} \\
	f_I(\beta, \beta_0) := & ~ \sum_{i \in I} (X_i^{\top} \beta + \beta_0 - Y_i)^2 + \sum_{i \in [m]\backslash I} Y_i^2, & \text{summation of $i \in [m]$ via $P(I)$.}
\end{align*} 
Henceforth, we denote the index set $I$ be the \textit{active set} and denote the index set $I^C = [m] \backslash I$ be the \textit{inactive set}. Hence the original \ref{eq:One-Node-ReLU} problem can be interpreted as a two-phase optimization problem: For any given $I \subseteq [m]$, the inner-phase optimization problem 
\begin{align*}
	z^{\ast}(I) := \min_{(\beta, \beta_0) \in P(I)} f_I(\beta, \beta_0) + \phi(\beta, \beta_0)
\end{align*}
is convex over $(\beta, \beta_0)$. the  Our approximation algorithm will be build based on the fact that we will examine only a polynomial number of distinct $I$'s.

In order to obtain the approximation guarantees it is convenient to work with an `unconstrained version' of the optimization problem corresponding to $z^*(I)$. Let $\sigma: \mathbb{R} \times \mathbb{R} \mapsto \mathbb{R}$ be the convex function (See Figure \ref{fig:sigma}) defined as:
\begin{align*}
	\sigma(x, y) = \left\{
	\begin{array}{lll}
		(x - y)^2 & \text{ if } x > 2y \\
		y^2  & \text{ if } x \leq 2y, 
	\end{array}
	\right.
\end{align*}
\begin{figure}[h]
	\begin{center}
		\includegraphics[width= 0.6 \textwidth]{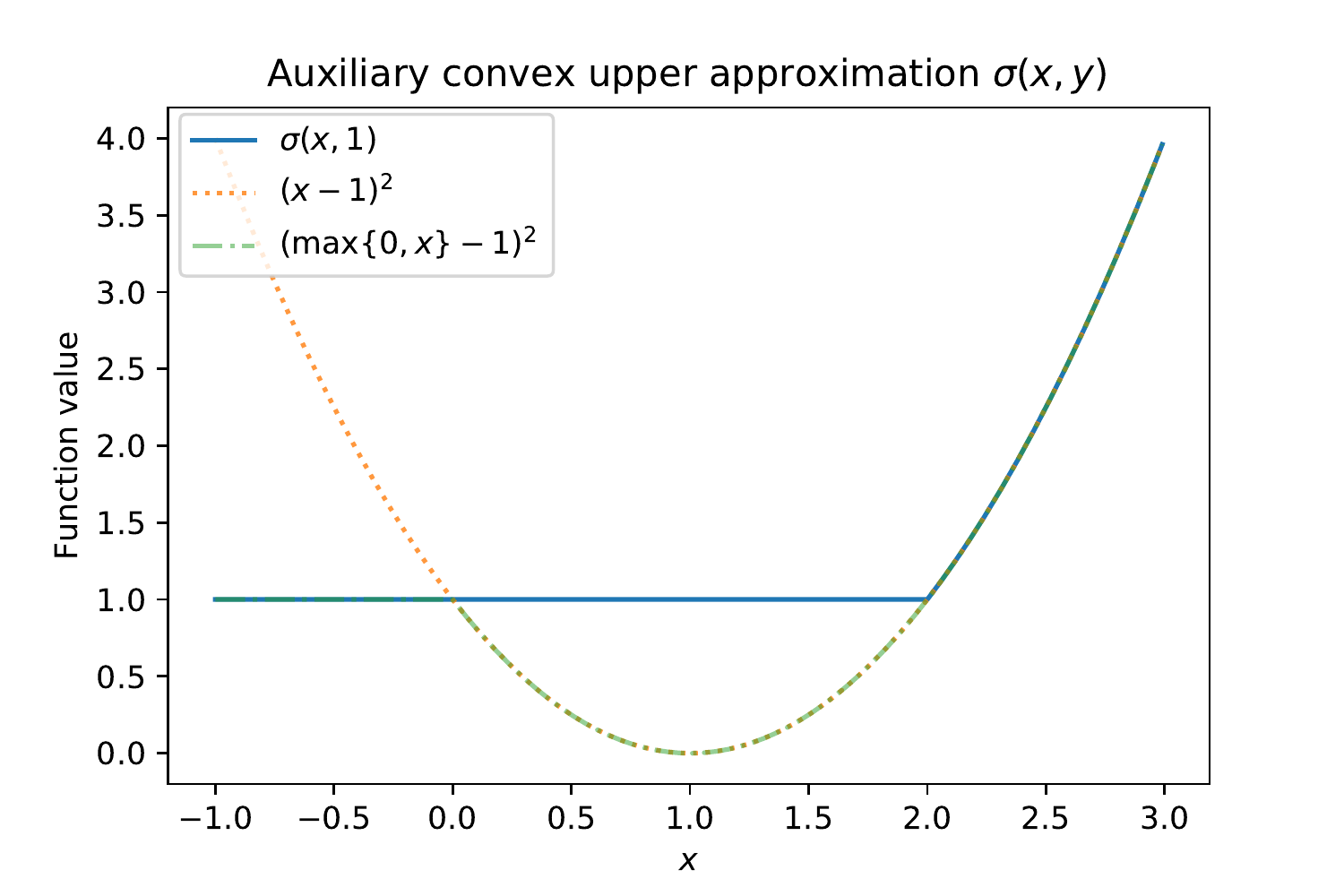}
	\end{center}
	\caption{Function $\sigma(x, y)$ with $y = 1$}
	\label{fig:sigma}
\end{figure}
where $\sigma(X_i^{\top} \beta + \beta_0, Y_i) \geq (\max\{0, X_i^{\top} \beta + \beta_0\} - Y_i)^2$ holds for all $(\beta, \beta_0) \in \mathbb{R}^p \times \mathbb{R}$. Let 
\begin{align*}
	z^{\sigma}(I) := \min_{(\beta, \beta_0) \in \mathbb{R}^p \times \mathbb{R}} { \sum_{i \in I} (X_i^{\top} \beta + \beta_0 - Y_i)^2 + \sum_{i \in [m] \backslash I} \sigma(X_i^{\top} \beta + \beta_0, Y_i) }+ \phi(\beta, \beta_0)
\end{align*}

be an convex upper approximation of $z^{\ast}(I)$. Let
\begin{align*}
	z^{\text{OPT}} & := \min_{(\beta, \beta_0) \in \mathbb{R}^p \times \mathbb{R}} \|\max\{\mathbf{0}, X^{\top}\beta + \beta_0 \mathbf{1}\} - Y\|_2^2, & \text{ be the global optimal value,} \\
	(\beta^{\text{OPT}}, \beta^{\text{OPT}}_0 ) & := \argmin_{(\beta, \beta_0) \in \mathbb{R}^p \times \mathbb{R}} \|\max\{\mathbf{0}, X^{\top}\beta + \beta_0 \mathbf{1}\} - Y\|_2^2, & \text{ be a global optimal solution.} 
\end{align*}
Thus $I^{\text{OPT}} := \left\{i \in [m]: X_i^{\top} \beta^{\text{OPT}} + \beta^{\text{OPT}}_0 > 0 \right\}$ and $[m] \backslash I^{\text{OPT}} := \left\{ i \in [m]: X_i^{\top} \beta^{\text{OPT}} + \beta^{\text{OPT}}_0 \leq 0 \right\}$ are the corresponding active, inactive set of $(\beta^{\text{OPT}}, \beta^{\text{OPT}}_0)$ respectively. Then $z^{\sigma}(I)$ satisfies:

\vspace{10pt}
\begin{proposition} \label{prop:LB}
	For any $I\subseteq [m]$, $z^{\text{OPT}} \leq z^{\sigma}(I)$. Moreover, there exists a $I \subseteq [m]$ such that $z^{\text{OPT}} = z^{\sigma}(I)$. 
\end{proposition}
Proof of Proposition \ref{prop:LB} can be found in Section \ref{section:Prop-LB}. Thus, we can use the  $z^{\sigma}(I)$ functions instead of $z^{*}(I)$ to design the algorithm, and being an unconstrained problem is easier to work with. 

Given any feasible solution $(\beta, \beta_0)$, by definition, the $i^{\text{th}}$ row will contribute 
\begin{align*}
	\left( \max\{0, X_i^{\top} \beta + \beta_0 \} - Y_i \right)^2 
\end{align*}
to the objective value. Suppose the given $(\beta, \beta_0)$ satisfies 
\begin{align*}
	X_i^{\top} \beta + \beta_0 \leq 0, 
\end{align*}
then for some $i \in [m]$ such that $0 \ll Y_i$, the $i^{\text{th}}$ row contributes a large value to objective.
Therefore, we expect the greater $Y_i$, the more likely that the index $i$ is in the active set. 
This is the key intuition behind Algorithm~\ref{algo:gen-approximation}, which explores a polynomial number of active sets with the property that larger the value of $Y_i$ the more likely $i$ is in the active set.
\begin{algorithm}[h]
\caption{Generalized Approximation Algorithm} \label{algo:gen-approximation}
\begin{algorithmic}[1]
\BState \emph{Input}: A set of $n$ sample points $(X, Y) \in \mathbb{R}^{p \times n} \times \mathbb{R}^n$, a positive-label index set $I^+ = \{1, \ldots, m\}$ such that $0 < Y_1 \leq Y_2 \leq \dots \leq Y_m$, a negative-label index set $I^- =\{m + 1, \ldots, n\}$, a fixed integer $k \geq 1$.
\BState \emph{Output}: A feasible $n/k$-approximation solution $(\beta, \beta_0)$ for the \ref{eq:One-Node-ReLU} problem. 
\Function{Generalized Approximation Algorithm}{$X, Y, I^+$} \label{function:gen-approximation}
\For{$j = 1, \ldots, k$ as the size of inactive set $[m] \backslash I$} 
\State Pick $j$ distinct indices $i_1, \ldots, i_j$ such that $0 \leq i_1 < \ldots < i_j \leq m$. 

\State Set inactive set be 
	\begin{align*}
		\left\{
		\begin{array}{lll}
			\left\{1, \ldots, i_1 \right\} \cup \left( \bigcup_{\ell = 2}^{j}\{i_{\ell}\} \right) & \text{ for } j \geq 2, \\
			\left\{1, \ldots, i_1 \right\} & \text{ for } j = 1.
		\end{array}
		\right.
	\end{align*} 
\State Set active set be the complement of inactive set as  
	\begin{align*}
		I = \left( \bigcup_{\ell = 1}^{j - 1} \left\{i_{\ell} + 1, \ldots, i_{\ell + 1} - 1 \right\} \right) \cup \{i_{j} + 1, \ldots, m\}.
	\end{align*}
\State For each active set $I$, compute
	\begin{align*}
		(\beta^I, \beta^I_0) \gets & \argmin_{(\beta, \beta_0)} f^{\sigma}_{I}(\beta, \beta_0) + \phi(\beta, \beta_0), \\
		z^{\sigma}(I) \gets & \min_{(\beta, \beta_0)} f^{\sigma}_{I}(\beta, \beta_0) + \phi(\beta, \beta_0).
	\end{align*}
\EndFor		
\State \Return $(\hat{\beta}, \hat{\beta}_0)$ which corresponds to the minimum $z^{\sigma}(I)$ among all the $I$'s examined. 
\EndFunction
\end{algorithmic}
\end{algorithm}\\
Note that when $i_1 = 0$, the set $\{1, \ldots, i_1\} = \emptyset$, similarly, when $i_j = m$, the set $\{i_k + 1, \ldots, m\} = \emptyset$. It is clear to see that, for each $j = 1, \ldots, k$, there are $\binom{m}{j}$ distinct subsets $\{i_1, \ldots, i_j\}$ in $\{1, \ldots, m\}$. For each subset $\{i_1, \ldots, i_j\}$, Algorithm \ref{algo:gen-approximation} requires to solve a convex optimization problem, thus the total running time of Algorithm \ref{algo:gen-approximation} is $$\left( \sum_{i = 1}^k \binom{n}{i} \right) T = O\left(n^k T \right)$$ where $T$ is the running time of solving a convex optimization problem $$(\beta^I, \beta^I_0) \gets \argmin_{(\beta, \beta_0)} f^{\sigma}_{I}(\beta, \beta_0) + \phi(\beta, \beta_0).$$ Thus Algorithm \ref{algo:gen-approximation} is a polynomial-time algorithm.
%
%
\begin{theorem}[Approximation Ratio] \label{thm:AA} Algorithm~\ref{algo:gen-approximation} is an $\frac{n}{k}$-Approximation Algorithm, i.e., if $z^{approx}$ is the objective value of the $(\hat{\beta}, \hat{\beta_0})$ returned from Algorithm~\ref{algo:gen-approximation}, and $z^{\text{OPT}}$ is the global optimal value of \ref{eq:One-Node-ReLU}, then:
$$ z^{\text{OPT}} \leq z^{approx} \leq \frac{n}{k}z^{\text{OPT}}.$$
\end{theorem}

\paragraph{Comparison with related results from literature:} Note that there is an exact algorithm that solves \ref{eq:One-Node-ReLU} problem within $O(n^p)$ running time where $p$ is the dimension~\cite{arora2016understanding}. However, when $p$ is much greater than $n$ (i.e., in high dimensional cases), our generalized approximation algorithm could achieve a reasonable good solution in $O(n^k T)$ running time. 
In \cite{manurangsi2018computational}, P. Manurangsi and D. Reichman show that minimizing squared training error of a one-node neural network is NP-hard to approximate within the factor $n^{\frac{1}{(\log \log n)^{O(1)}}}$ (actually, $m$ samples $\{(x_i, y_i)\}_{i = 1}^m$ in their setting generates $nm$ samples in our setting based their polynomial-time reduction).  There is a significant gap between  upper bound from Algorithm~\ref{algo:gen-approximation} and this lower bound.

A very important consequence of Theorem~\ref{thm:AA} is the following result.
\begin{corollary}[Realizable case] When the \ref{eq:One-Node-ReLU} problem is realizable, i.e., there exists a true solution $(\beta^{\ast}, \beta^{\ast}_0)$ with 0 objective value, then Theorem~\ref{thm:AA} implies that the Algorithm~\ref{algo:gen-approximation} gives a polynomial-time approach that solves the \ref{eq:One-Node-ReLU} problem exactly to global optimal. 
\end{corollary}

\paragraph{Comparison with related results in literature} Soltanolkotabi in \cite{soltanolkotabi2017learning}, and Kalan et.al. in \cite{kalan2019fitting} studied the problem of learning one node ReLU neural network with i.i.d. random Gaussian distribution observation samples via gradient descent (GD) method and stochastic gradient descent (SGD) method in the realizable case. Soltanolkotabi showed that the gradient descent, when starting from origin converges at a linear rate to the true solution (with additive error) where the number of samples is sufficiently large. Kalan et.al. in \cite{kalan2019fitting} discussed the stochastic version that mini-batch stochastic gradient descent when suitably initialized, converges at a geometric rate to the true solution (with additive error). In contrast Algorithm~\ref{algo:gen-approximation} does not need the assumption that the data is i.i.d. random Gaussian distribution, there are no additive errors, and also deals with the case where $\beta_0$ is non-trivial. Finally, the initialization of SGD method requires some additional effort not needed for Algorithm~\ref{algo:gen-approximation}.

\subsection{Training \ref{eq:One-Node-ReLU} With Underlying Statistical Model}
In real life, it is natural to assume that the set of sample points follows some underlying statistical model. Here is our assumptions:  
\begin{assumption}[Underlying Statistical Model] \label{assumption}
	Suppose the set of sample points $$(X_1, Y_1), \ldots, (X_n, Y_n) \in \mathbb{R}^p \times \mathbb{R}$$ satisfies the correct underlying statistical model, i.e., for each $i = 1, \ldots, n$, 
\begin{align*}
	Y_i = \max\{0, X_i^{\top} \beta^{\ast} + \beta^{\ast}_0\} + \epsilon_i
\end{align*}
where $\beta^{\ast}, \beta^{\ast}_0$ is some fixed true parameter (may be distinct from $\left(\beta^{\text{OPT}}, \beta^{\text{OPT}}_0 \right)$ as the optimal solution of \ref{eq:One-Node-ReLU}). We further assume that  $\beta^{\ast}, \beta^{\ast}_0$ belongs to a convex compact set $\Theta \subseteq \mathbb{R}^p \times \mathbb{R}$, and for $i = 1, \ldots, n$, $X_i, \epsilon_i$ are i.i.d. random variables that are generated from some underlying fixed distribution $\mathcal{N}, \mathcal{D}$, respectively. Furthermore, $\mathcal{N}$ is a distribution that satisfies the following properties:
\begin{enumerate}
	\item $\mathbb{E}_{X \sim \mathcal{N}}[X] = 0_p, \text{ and } \text{Var}_{X \sim \mathcal{N}}(X) = \Sigma$ positive semi-definite.  
	\item {Unique Optimal Property:} Let $\text{Supp}_{\mathcal{N}} \subseteq \mathbb{R}^p$ be the support of distribution $\mathcal{N}$. For any $(\beta^{\ast}, \beta^{\ast}_0) \in \Theta$, there exists $p + 1$ vectors $v_1, \ldots, v_p, v_{p + 1} \in \text{Supp}_{\mathcal{N}}$ such that 
		\begin{align*}
			v_i^{\top} \beta^{\ast} + \beta^{\ast}_0 > 0, ~ \forall i = 1, \ldots, p, p + 1,
		\end{align*}
		in which $(v_1, 1), \ldots, (v_p, 1), (v_{p + 1}, 1) \in \mathbb{R}^{p + 1}$ are linearly independent. 
	\item Since $\beta^{\ast}, \beta^{\ast}_0$ is fixed then $$\mathbb{E}_{X \sim \mathcal{N}}[X^{\top} \beta^{\ast} + \beta^{\ast}_0] = \beta^{\ast}_0, \text{ and } \text{Var}_{X \sim \mathcal{N}}(X^{\top} \beta^{\ast} + \beta^{\ast}_0) = (\beta^{\ast})^{\top} \Sigma \beta^{\ast} =: \Delta^2.$$ 
\end{enumerate}
and $\mathcal{D}$ is a Gaussian distribution such that $\mathbb{E}_{\epsilon \sim \mathcal{D}} = 0$ and $\text{Var}_{\epsilon \sim \mathcal{D}} = \gamma^2 < \infty$. 
 
\end{assumption}
Note that Algorithm~\ref{algo:gen-approximation} with parameter $k = 1$,  provides a solution to the objective function in the following format:
\begin{align*}
	\min_{\beta \in \Theta} \underbrace{ \frac{1}{n} \bigg[ \sum_{i \in I(y)} \sigma(X_i^{\top} \beta + \beta_0, Y_i) + \sum_{i \in I^+ \backslash I(y)} (X_i^{\top} \beta + \beta_0 - Y_i)^2 + \sum_{i \in I^-} (\max\{0, X_i^{\top} \beta + \beta_0 \} - Y_i)^2 \bigg]}_{=: S^y_n (\beta, \beta_0)}
\end{align*}
where $I^+ = \{i: Y_i > 0\}, ~ I^- = \{i: Y_i \leq 0\}$, and $I(y) = \{i: 0 < Y_i \leq y\}$ for some $y > 0$. To see the exact correspondence: $\{0, 1, \dots, i_1\}$ is $I(y)$ and \{$i_1 +1, \dots, m\}$ is $[m]\setminus I(y)$. As we change $y$, we are effectively picking different values of $i_i$. 

First, using existing classical results in (\cite{mickey1963test}, p40) and \cite{jennrich1969asymptotic}, we obtain the following  results:

\begin{proposition} \label{prop:asy-sorting-obj}
	As $n \rightarrow \infty$, the objective function $S^y_n(\beta, \beta_0)$ converges to the following asymptotic objective function: 
\begin{align*}
	S^y_n (\beta, \beta_0) \rightarrow \mathbb{E}_{X \sim \mathcal{N}, \epsilon \sim \mathcal{D}}\bigg[ \psi_y (X^{\top} \beta + \beta_0, Y) \bigg], 
\end{align*}
for almost every sequence $\{(X_i, Y_i)\}_{i = 1}^n$ where $Y_i = \max\{0, X_i^{\top} \beta^{\ast} + \beta_0^{\ast} \} + \epsilon_i$ and 
\begin{align*}
	\psi_y (X^{\top} \beta + \beta_0, Y) := \left\{
	\begin{array}{llll}
		\sigma(X^{\top}\beta + \beta_0, Y) & \text{ if } 0 < Y \leq y \\
		(X^{\top} \beta + \beta_0 - Y)^2 & \text{ if } y < Y \\
		(\max\{0, X^{\top} \beta + \beta_0 \} - Y)^2 & \text{ if } Y \leq 0
	\end{array}
	\right..
\end{align*}
\end{proposition}

\begin{proposition} \label{prop:asy-obj}
	Given the set of sample points $\{(X_i, Y_i)\}_{i = 1}^n$ from the underlying statistical model, as $n \rightarrow \infty$, the least square estimator $(\beta^{\text{OPT}}, \beta^{\text{OPT}}_0)$ which depends on $\{(X_i, Y_i)\}_{i = 1}^n$ obtained from \ref{eq:One-Node-ReLU} problem converges to the true parameter $(\beta^{\ast}, \beta^{\ast}_0)$ almost surely, i.e., the estimator $(\beta^{\text{OPT}}, \beta^{\text{OPT}}_0)$ of $(\beta^{\ast}, \beta^{\ast}_0)$ is said to be strongly consistent. Moreover, as $n \rightarrow \infty$,
	\begin{align*}
		\frac{1}{n} \sum_{i = 1}^n \left( \max\{0, X_i^{\top} \beta + \beta_0 \} - Y_i \right)^2 \rightarrow \mathbb{E}_{X \sim \mathcal{N}, \epsilon \sim \mathcal{D}} \left[\left( \max\{0, X^{\top} \beta + \beta_0 \} - Y \right)^2 \right],
	\end{align*}
	and 
	\begin{align*}
		\min_{(\beta, \beta_0) \in \Theta} \mathbb{E}_{X \sim \mathcal{N}, \epsilon \sim \mathcal{D}} \left[\left( \max\{0, X^{\top} \beta + \beta_0 \} - Y \right)^2 \right] = \mathbb{E}_{X \sim \mathcal{N}, \epsilon \sim \mathcal{D}} \left[\left( \max\{0, X^{\top} \beta^{\ast} + \beta^{\ast}_0 \} - Y \right)^2 \right] = \gamma^2. 
	\end{align*}
\end{proposition}

Combining Proposition~\ref{prop:asy-sorting-obj} and Proposition~\ref{prop:asy-obj}, we show in Section~\ref{sec:stat} that Algorithm~\ref{algo:gen-approximation} guarantees an asymptotic bound as follows:

\begin{theorem}[Asymptotic Bound] \label{thm:asy-bound} Assuming the underlying statistical model~\ref{assumption}, let $z^{\text{asy}}$ be the optimal value of the asymptotic objective function $\mathbb{E}_{X \sim \mathcal{N}, \epsilon \sim \mathcal{D}} \left[ \psi_y (X^{\top} \beta + \beta_0, Y)\right]$ for all $y > 0$, i.e., 
		\begin{align*}
			z^{\text{asy}} = \min_{y \geq 0}\min_{(\beta, \beta_0) \in \Theta} \mathbb{E}_{X \sim \mathcal{N}, \epsilon \sim \mathcal{D}} \left[ \psi_y (X^{\top} \beta + \beta_0, Y)\right], 
		\end{align*}
		then $z^{\text{asy}}$ can be lower and upper bounded by the following:
		\begin{align*}
			\gamma^2 \leq z^{\text{asy}} \leq \frac{3 \gamma^2}{2} + \frac{2 + 2\Delta^2}{\sqrt{2 \pi}} \gamma
		\end{align*}
		where $\gamma$ and $\Delta$ are defined in the underlying statistical model~\ref{assumption}. 
\end{theorem}
Note that the upper bound for the asymptotic optimal value $z^{\text{asy}}$ only depends on the variance $\Delta^2$ and $\gamma^2$, therefore for any fixed underlying distribution $\mathcal{N}$ and $\mathcal{D}$, we have the following corollary:
\begin{corollary}[Asymptotic Approximation Ratio] \label{coro:asy-ratio}
	 Assuming the underlying statistical model~\ref{assumption}, as $n \rightarrow \infty$, the solution obtained from Approximation Algorithm~\ref{algo:gen-approximation} provides an asymptotic multiplicative approximation ratio 
	\begin{align*}
		\rho \leq \frac{3}{2} + \frac{2 + 2\Delta^2}{\sqrt{2 \pi}} \frac{1}{\gamma}
	\end{align*}
	which is independent of the sample size $n$. Moreover, this guarantee can be achieved by only computing $S^y_n(\beta, \beta_0)$ with $y = 0$. 
\end{corollary}
We note the following: as the variance of noise tends to zero, the multiplicative approximation ratio $\rho$ obtained in Corollary~\ref{coro:asy-ratio} goes to infinity. However, since the upper bound of $z^{\text{asy}}$ is in the order $O(\gamma)$, $z^{\text{asy}}$ will also tend to zero. 

\paragraph{Comparision with related results in literature}  \cite{hinton2006fast} gave a fast, greedy algorithm that can find a fairly good set of parameters quickly based on good initialization using ``complementary priors'' in a reasonable time. Later, \cite{zhang2016understanding} gave empirical evidence that simple two-layer neural networks have good sample expressivity in the over-parameterized case.  However, none of these papers provide any theoretical guarantees. Soltanolkotabi in \cite{soltanolkotabi2017learning}, and Kalan et.al. in \cite{kalan2019fitting} only study the problem in the zero noise case. Kakade et al. \cite{kakade2011efficient} provided algorithms for learning Generalized Linear and Single Index Models to obtain provable performance, which are both computationally and statistically efficient. In \cite{brutzkus2017globally}, Brutzkus, Globerson showed that when the input distribution is Gaussian, in noiseless case, a one-hidden layer neural network with ReLU activation function can be trained exactly in polynomial time with gradient descent. Du et al. in paper \cite{du2017gradient} (also see \cite{du2017convolutional}) showed that: learning a one-hidden layer ReLU neural network, (1) with a specific randomized initialization, the gradient descent converges to the ground truth with high probability, (2) the objective function does have a spurious local minimum (i.e., the local minimum plays a non-trivial role in the dynamics of gradient descent). Note that these two papers \cite{du2017gradient, du2017convolutional} need a good initialization. Goel, Klivans, and Meka in \cite{goel2018learning} presented an algorithm--\textit{Convotron}--which requires no special initialization or learning-rate tuning to converge to the global optimum. The proof of their results heavily depends on there being no constant term in the input to the ReLU function and the input distribution being centrally symmetric around the origin.

\section{Main Computational Results}\label{section:numerical-exp}
In this section, we empirically compare the numerical results of four methods: sorting method (a simplified version of Algorithm~\ref{algo:gen-approximation}, we describe this below in this section), sorting followed by an iterative heuristics (we describe this heuristic below in this section), gradient descent method starting at origin, sorting followed by gradient descent method, stochastic gradient method on synthetic instances for the \ref{eq:One-Node-ReLU} problem. A feasible solution $\hat{\beta}$ that obtained from above methods is evaluated in terms of its prediction error, objective value, recovery error and generalization error. In the rest of this section, details and settings are presented. 

\subsection{Hardware \& Software}
All numerical experiments are implemented on MacBookPro13 with 2 GHz Intel Core i5 CPU and 8 GB 1867 MHz LPDDR3 Memory. Each optimization step of sorting method (Algorithm~\ref{algo:gen-approximation}) and each optimization step of iterative method (Algorithm~\ref{algo:iter-method}) are solved using Gurobi 7.0.2. in python 3.5.3.
 
\subsection{Synthetic Instances} 
We perform numerical experiments on the following type of instances.  
\begin{enumerate}
	\item Given a vector $\mu \in \mathbb{R}^p$, and a positive semidefinite matrix $\Sigma \in \mathbb{R}^{p \times p}$, the true solution $\beta^{\ast}$ is generated from the Gaussian distribution $N(\mu, \Sigma)$. Specifically, $\beta^{\ast}$ in Figure~[\ref{fig:Figure_Compare_10_200}, \ref{fig:Figure_Compare_20_400}, \ref{fig:Figure_Compare_50_1000}] are generated from $N(0.5 \cdot \mathbf{0}_p, 10 \cdot I_p)$. 
	\item Both training set and testing set contain $n$ sample points. For each sample point $(X_i, Y_i) \in \mathbb{R}^{p} \times \mathbb{R}$ in training set, the observation sample $X_i = (X_{ij})_{j = 1}^p$ is generated by setting each component $X_{ij} = 1, X_{ij} = -1$ with probability $\frac{P}{2}$ and $X_{ij} = 0$ with probability $1 - P$ independently. In the rest of this sections, we denote the probability $ P = \mathbb{P}(\{X_{ij} = 1\} \cup \{X_{ij} = - 1\})$ as the level of \textit{sparsity}, i.e., the higher the level of sparsity, the more number of non-zero components exists in $X_i$ in expectation. 

		Moreover, in the realizable case we perturbed the data to guarantee that the global optimal solution is unique. Assuming that $\beta^*_i \neq 0$ for all $i \in [p]$, the first $p$  samples $X_i$ are obtained as $X_i \gets e_i \cdot \text{sgn}(\beta_i^{\ast})$ for all $i = 1, \ldots, p$ in training set, in which $e_i \in\mathbb{R}^p$ is a vector with 1 on its $i^{\text{th}}$ component and 0 on rest, and $\text{sgn}(x) = \left\{ \begin{array}{lll} 1 & \text{ if } x > 0 \\ - 1 & \text{ otherwise} \end{array} \right.$.

	\item Note the constant term $\beta_0^{\ast}$ can be achieved via adding one additional dimension with component 1 to each $X_i$. To simplify, we deiced to use $\beta^*_0 = 0$. The response sample $Y_i$ is therefore computed as  $Y_i = \max\{0, X_i^{\top} \beta^{\ast} \} + \epsilon_i$ with $\epsilon_i \sim N(0, \rho\sigma)$ where $\sigma$ and $\rho$ are set in the following way:
\begin{description}
	\item[$\sigma$:] (1) Let $Z_i \gets X_i^{\top} \beta^{\ast}$ for all $i = 1, \ldots, n$. (2) Compute the average of $\{Z_i\}_{i = 1}^n$ as $\bar{Z} \gets \frac{1}{n}\sum_{i = 1}^n Z_i$. (3) Set $\sigma^2 \gets \frac{1}{n} \sum_{i = 1}^n (Z_i - \bar{Z})^2$.
	\item[$\rho$:] 
Since the noise levels are commonly measured in signal-to-noise decibels (dB), we examined dB values $\{6, 10, 20, 30, \infty\}$, then the value of signal-to-noise ratio $\rho$ can be computed from 
		\begin{align*}
			\text{dB} \triangleq 10 \log_{10} \left( \frac{\sigma^2}{\rho^2 \sigma^2} \right) \in \{6, 10, 20, 30, \infty\}
		\end{align*}
		which corresponds to $\rho \approx \{0.5, 0.32, 0.1, 0.032, 0\}$.
\end{description} 
	\item For each sample point $(\tilde{X}_i, \tilde{Y}_i) \in \mathbb{R}^{p} \times \mathbb{R}$ in testing set, generate $\tilde{X}_i, \tilde{Y}_i$ in the same way as the training set.   
\end{enumerate}

\subsection{Algorithms}
In this section, we briefly describe the algorithms that were tested in the numerical experiments. 
\subsubsection{Method 1: Sorting Method (sorting)}\label{section:nr-sorting}
The sorting method is a simplied and slightly cruder version of Algorithm~\ref{algo:gen-approximation} with parameter $k = 1$. Essentially, in order to reduce the running time,  instead of $i_1$ taking all values from $1$ to $n$, we limit the values of $i_1$ that are used, see details in  Section~\ref{section:sorting-method}.


\subsubsection{Method 2: Sorting method followed by an Iterative Method (sorting + iterative)}\label{section:nr-sorting+Iter}

A natural iterative improving algorithm is the following: Fix $I$ and minimize $f^{\sigma}(I)$. Now, examine the solution and update the choice of $I$ which so the $f^{\sigma}$ and $f^{*}$ match for the current solution. Repeat until some stopping criteria is meet. See in  Algorithm \ref{algo:iter-method} in Section~\ref{section:iterative-method} for details. 

We use this heuristic to improve the solution obtained from the sorting method. 
After obtaining a feasible solution $\hat{\beta}^{\text{sorting}}$, we set the initial point of iterative heurisitic as $\hat{\beta}^{\text{sorting}}$. 

\subsubsection{Method 3: Gradient Descent Method (GD) }\label{section:nr-GD}
The gradient descent method used in numerical experiments is presented in Section~\ref{section:GD}, see Algorithm~\ref{algo:GD-method}. Given the initial point $\beta^0 \gets \mathbf{0}_p$, and set $\beta^t$ as the updated solution obtained in $(t - 1)^{\text{th}}$ iteration, the gradient used in $t^{\text{th}}$ iteration is:
\begin{align*}
	\frac{1}{n}\nabla_{\beta} L(\beta^t) = \frac{1}{n} \sum_{i = 1}^n (\max\{0, X_i^{\top}\beta^t\} - Y_i)(1 + \text{sgn}(X_i^{\top} \beta^t)) X_i
\end{align*}
where $L(\beta) \gets \sum_{i = 1}^n (\max\{0, X_i^{\top} \beta\} - Y_i)^2$ and $\text{sgn}(x) = \left\{ \begin{array}{lll} 1 & \text{ if } x > 0 \\ - 1 & \text{ if } x \leq 0 \end{array} \right. $.

\subsubsection{Method 4: Sorting followed by Gradient Descent Method (sorting + GD)}\label{section:nr-Sorting+GD}
Similar to the sorting followed by gradient descent method,  in this method we run the sorting method and then use the final solution of the sorting method as the starting point for gradient descent. 

\subsubsection{Method 5: Stochastic Gradient Descent Method (SGD)}
The initial point $\beta^0$ used in stochastic gradient descent (SGD) method is the same as gradient descent (GD) method, i.e. the origin. The only difference between SGD and GD is that: in $t^{\text{th}}$ iteration, we uniformly pick a mini-batch $B^t$ of size $m$ from the set of samples $\{(X_i, Y_i)\}_{i = 1}^n$ at random, then the gradient used in $t^{\text{th}}$ iteration is: 
\begin{align*}
	\frac{1}{m} \sum_{i \in S^t} (\max\{0, X_i^{\top}\beta^t\} - Y_i)(1 + \text{sgn}(X_i^{\top} \beta^t)) X_i, 
\end{align*}
with $\text{sgn}(x) = \left\{ \begin{array}{lll} 1 & \text{ if } x > 0 \\ - 1 & \text{ if } x \leq 0 \end{array} \right. $. See Algorithm~\ref{algo:SGD-method}. 

\subsection{Measures} 
The feasible solutions $\hat{\beta}$ obtained from the above five methods are evaluated in terms of their prediction error, objective value, recovery error, generalization error. The formal definitions of these three types of errors are:
\begin{description}
	\item[Prediction Error:] The prediction error is defined based on the provided solution $\hat{\beta}$, i.e.,   
$$\text{PE} \triangleq \sum_{i = 1}^n \left(\max\{0, X_i^{\top} \hat{\beta}\} - \max\{0, X_i^{\top} \beta^{\ast}\} \right)^2, \quad  \textup{where} \{X_i, Y_i\}_{i = 1}^{n} \textup{ is training data}.$$
	\item[Objective Value:] Note that the prediction error defined as above is not the objective value obtained from above five methods. Actually, in practice, when $\beta^{\ast}$ is unknown, the prediction error cannot be obtained exactly, thus one may use objective value (Obj) 
	\begin{align*}
		\text{Obj} \triangleq \sum_{i = 1}^n \left(\max\{0, X_i^{\top} \hat{\beta}\} - Y_i \right)^2, \quad  \textup{where} \{X_i, Y_i\}_{i = 1}^{n} \textup{ is training data}.
	\end{align*}
	as an alternative. 
	\item[Recovery Error:] The recovery error measures the distance between the solution $\hat{\beta}$ we obtained and the ground truth $\beta^{\ast}$, which is $$\text{RE} \triangleq \|\hat{\beta} - \beta^{\ast}\|_2.$$
	\item[Generalization Error:] The generalization error measures how good the solution $\hat{\beta}$ is, when using the objective function with resect to testing set, i.e. 
$$\text{GE} \triangleq \sum_{i = 1}^n \left(\max\{0, \tilde{X}_i^{\top} \hat{\beta}\} - \tilde{Y}_i\right)^2, \quad  \textup{where} \{X_i, Y_i\}_{i = 1}^{n} \textup{ is testing data}.$$
\end{description}
We note that in order to see the comparison between different methods more clearly, the prediction error and generalization error  are not divided by the size of sample set $n$. 

\subsection{Numerical Results: Notation and Parameters} \label{section:NR-notations}
The numerical results in Figure~[\ref{fig:Figure_Compare_50_1000}] of this section, and in Figure~[\ref{fig:Figure_Compare_10_200}, \ref{fig:Figure_Compare_20_400}] in Appendix~\ref{section:app-computational-results} present how these measures (prediction error, recovery error, generalization error) and the running time change depending on different standard deviation of noise empirically. The detailed realizable cases in Appendix~\ref{section:realizable-cases} provide an empirical result of the performances of previous four methods. Below we present  notations and the parameters that used for numerical experiments:
\begin{itemize}
	\item Each line presented in Figure~[\ref{fig:Figure_Compare_10_200}, \ref{fig:Figure_Compare_20_400}, \ref{fig:Figure_Compare_50_1000}] represents the average of the measures or running time obtained from 20 instances under the same settings. 
	\item The first column of each Table in Appendix~\ref{section:realizable-cases} is a tuple of 4 elements $(p, n, \rho; \text{index})$ which represents the dimension of $\beta$, the number of training samples, the ratio used for noise $\epsilon_i$, and the index of the instance with such settings respectively.
	\item For the Sorting Method (Algorithm~\ref{algo:sorting-method}), $N$ (the number of splits) used is $10$. 
	\item For the Sorting (Algorithm~\ref{algo:sorting-method}) + Iterative Method (Algorithm~\ref{algo:iter-method}), $N$ (the number of split) is set to $10$, and let $\hat{\beta}^{\text{sorting}}$ be the solution obtained from Sorting Method, then the parameters of Iterative Method are set to be:
		\begin{align*}
			(\{(X_i, Y_i)\}_{i = 1}^n,\beta^0, T) \gets (\{(X_i, Y_i)\}_{i = 1}^n, \hat{\beta}^{\text{sorting}}, 20)
		\end{align*}
		where $\beta^0$ denotes the starting point, $T$ denotes the maximum number of iterations. 
	\item For the Gradient Descent Method (Algorithm~\ref{algo:GD-method}), the parameters are set to be
		\begin{align*}
			(\{(X_i, Y_i)\}_{i = 1}^n, \beta^0, T, \epsilon, \eta_0, \gamma, \alpha) \gets (\{(X_i, Y_i)\}_{i = 1}^n, \mathbf{0}_p, 1000, 0.01, 1, 0.03, 0.6)
		\end{align*}
		where $\beta^0$ denotes the starting point, $T$ denotes the maximum number of iterations, $\epsilon$ is a termination criteria parameter, $\eta_0$ denotes the initial stepsize, $\gamma, \alpha$ are parameters used to adjust stepsize in each iteration.  
	\item For the Sorting (Algorithm~\ref{algo:sorting-method}) + Gradient Descent Method (Algorithm~\ref{algo:GD-method}),  $N$ (the number of split) is set to be $10$, and let $\hat{\beta}^{\text{sorting}}$ is as above, and the parameters of the Gradient Descent Method are set to be:
		\begin{align*}
			(\{(X_i, Y_i)\}_{i = 1}^n, \beta^0, T, \epsilon, \eta_0, \gamma, \alpha) \gets (\{(X_i, Y_i)\}_{i = 1}^n, \hat{\beta}^{\text{sorting}}, 1000, 0.01, 1, 0.03, 0.6).
		\end{align*}
	\item For the Stochastic Gradient Descent Method (Algorithm~\ref{algo:SGD-method}), parameters are set to be:
		\begin{align*}
			(\{(X_i, Y_i)\}_{i = 1}^n, \beta^0, T, \epsilon, \eta_0, \gamma, \alpha, m) \gets (\{(X_i, Y_i)\}_{i = 1}^n, \mathbf{0}_p, 1000, 0.01, 1, 0.03, 0.6, \lfloor 0.1 n \rfloor).
		\end{align*}
\end{itemize}

\begin{figure*}
    \centering
    \begin{subfigure}[b]{0.24\textwidth}
        \centering
        \includegraphics[width=\textwidth]{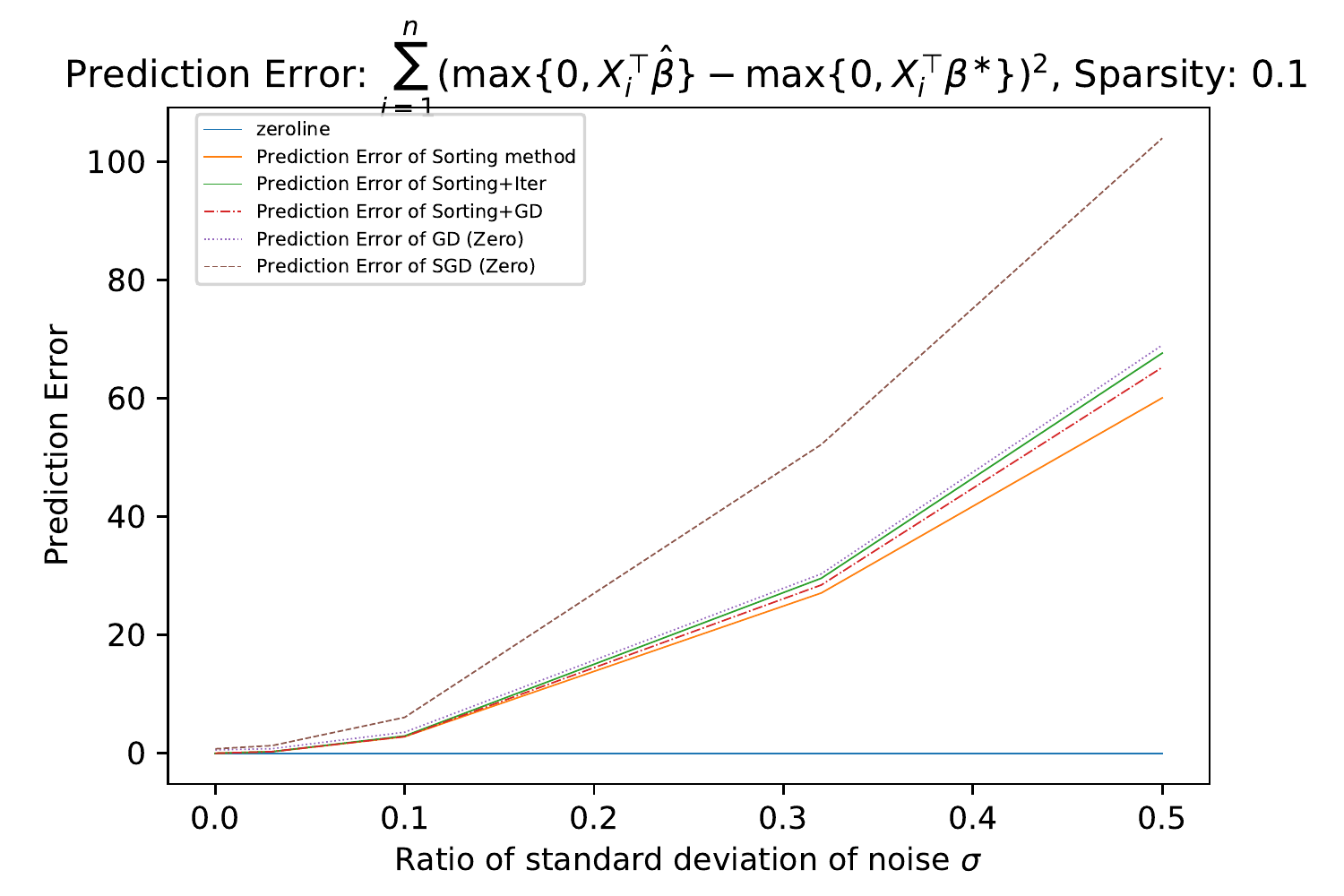}
        \caption[Network1]%
        {{\small Prediction Error}}    
        \label{fig:Figure_Compare_PE_50_1000_10}
    \end{subfigure}
    \hfill
    \begin{subfigure}[b]{0.24\textwidth}   
        \centering 
        \includegraphics[width=\textwidth]{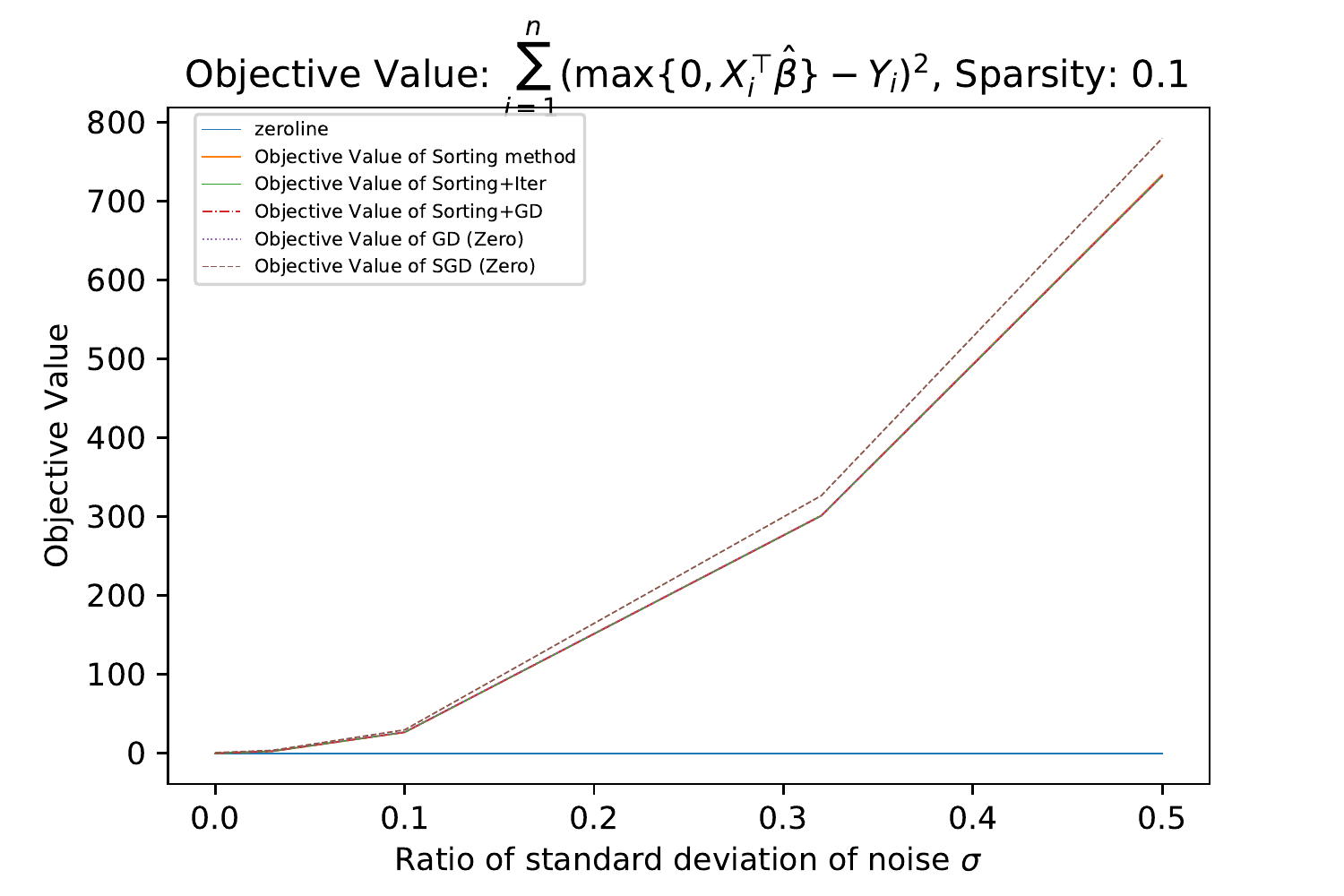}
        \caption[]%
        {{\small Objective Value}}    
        \label{fig:Figure_Compare_OB_50_1000_10}
    \end{subfigure}
    \hfill
    \begin{subfigure}[b]{0.24\textwidth}  
        \centering 
        \includegraphics[width=\textwidth]{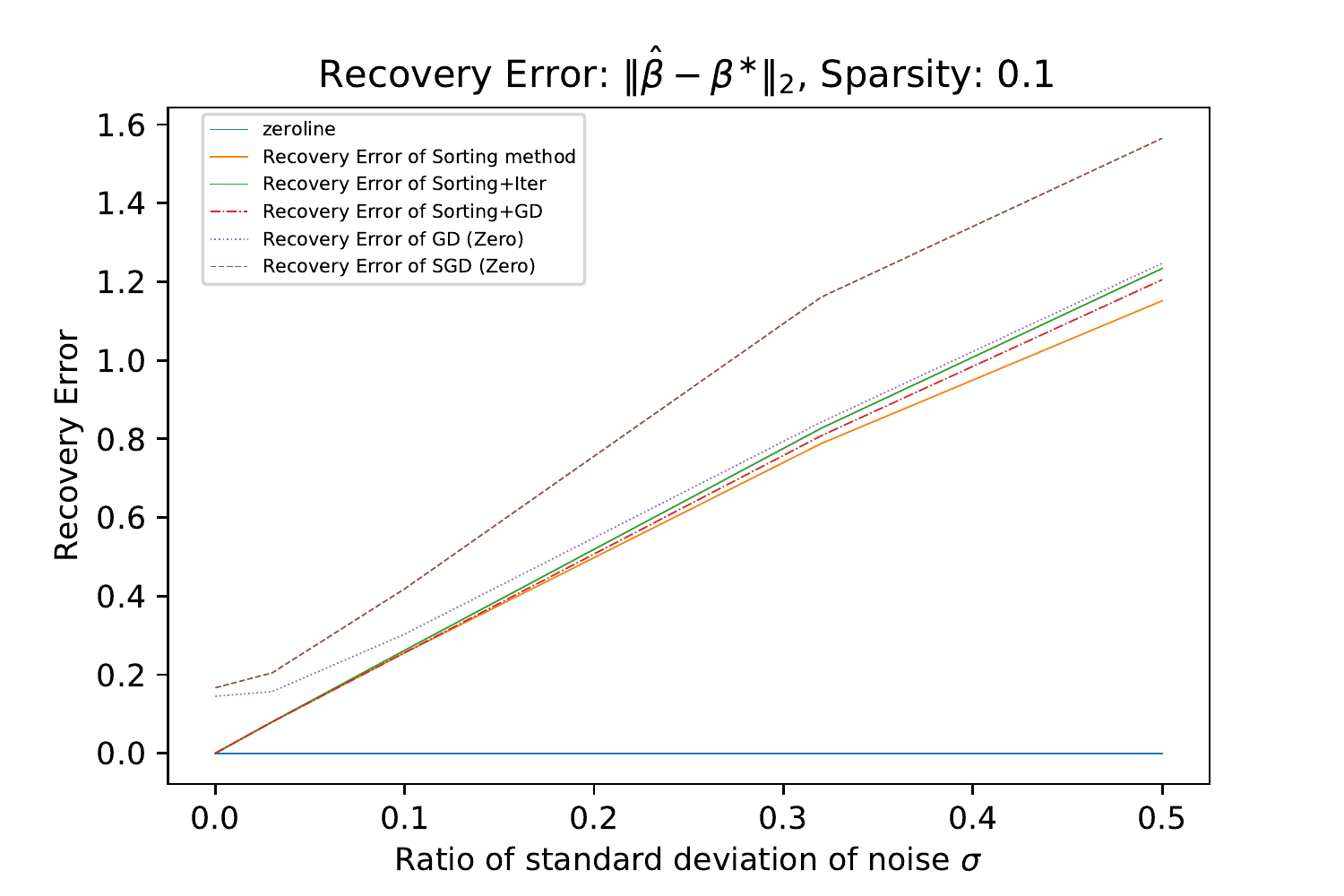}
        \caption[]%
        {{\small Recovery Error}}    
        \label{fig:Figure_Compare_RE_50_1000_10}
    \end{subfigure}
    \hfill
    \begin{subfigure}[b]{0.24\textwidth}   
        \centering 
        \includegraphics[width=\textwidth]{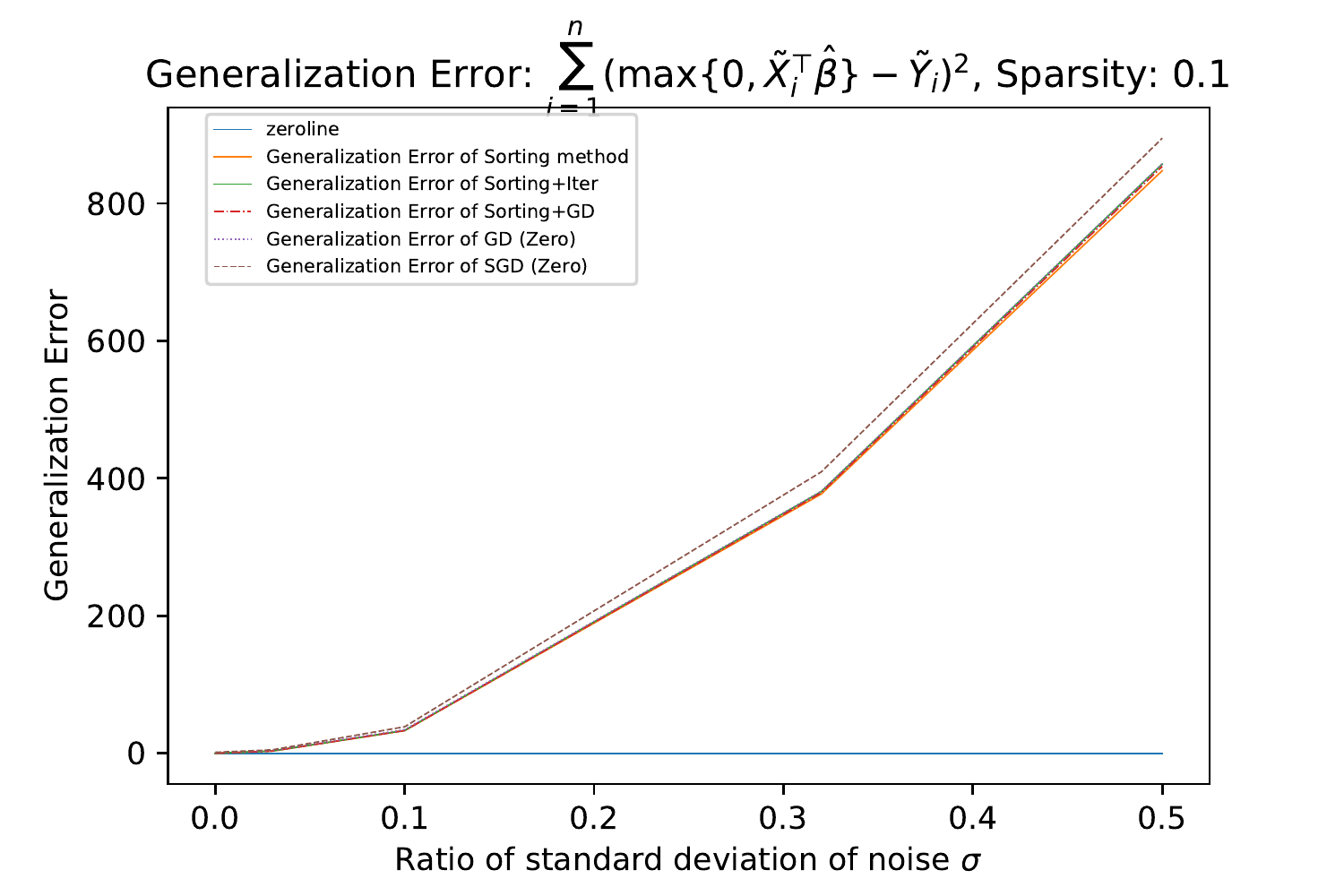}
        \caption[]%
        {{\small Generalization Error}}    
        \label{fig:Figure_Compare_GE_50_1000_10}
    \end{subfigure}

    \vskip\baselineskip
    
    \centering
    \begin{subfigure}[b]{0.24\textwidth}
        \centering
        \includegraphics[width=\textwidth]{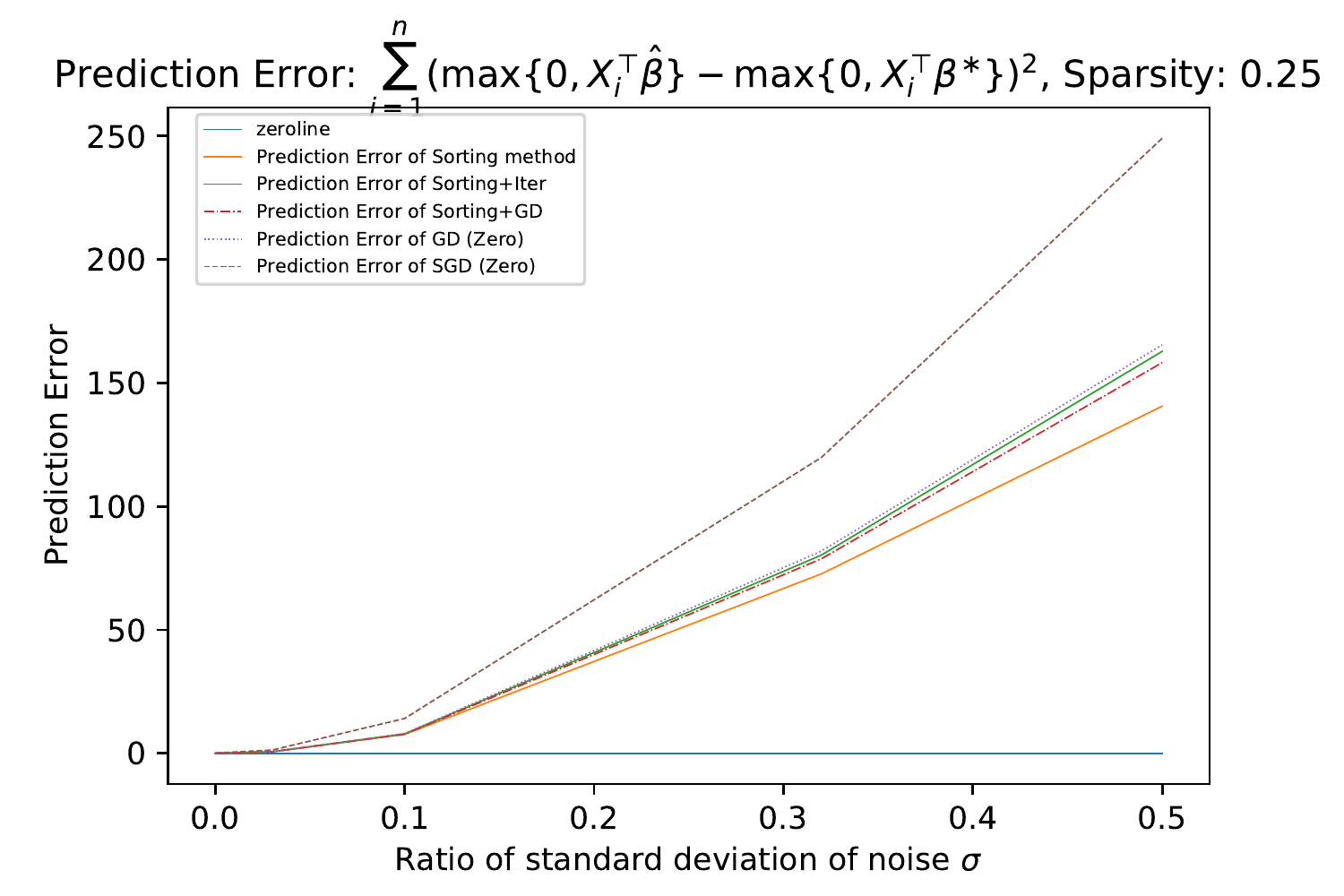}
        \caption[Network1]%
        {{\small Prediction Error}}    
        \label{fig:Figure_Compare_PE_50_1000_25}
    \end{subfigure}
    \hfill
    \begin{subfigure}[b]{0.24\textwidth}   
        \centering 
        \includegraphics[width=\textwidth]{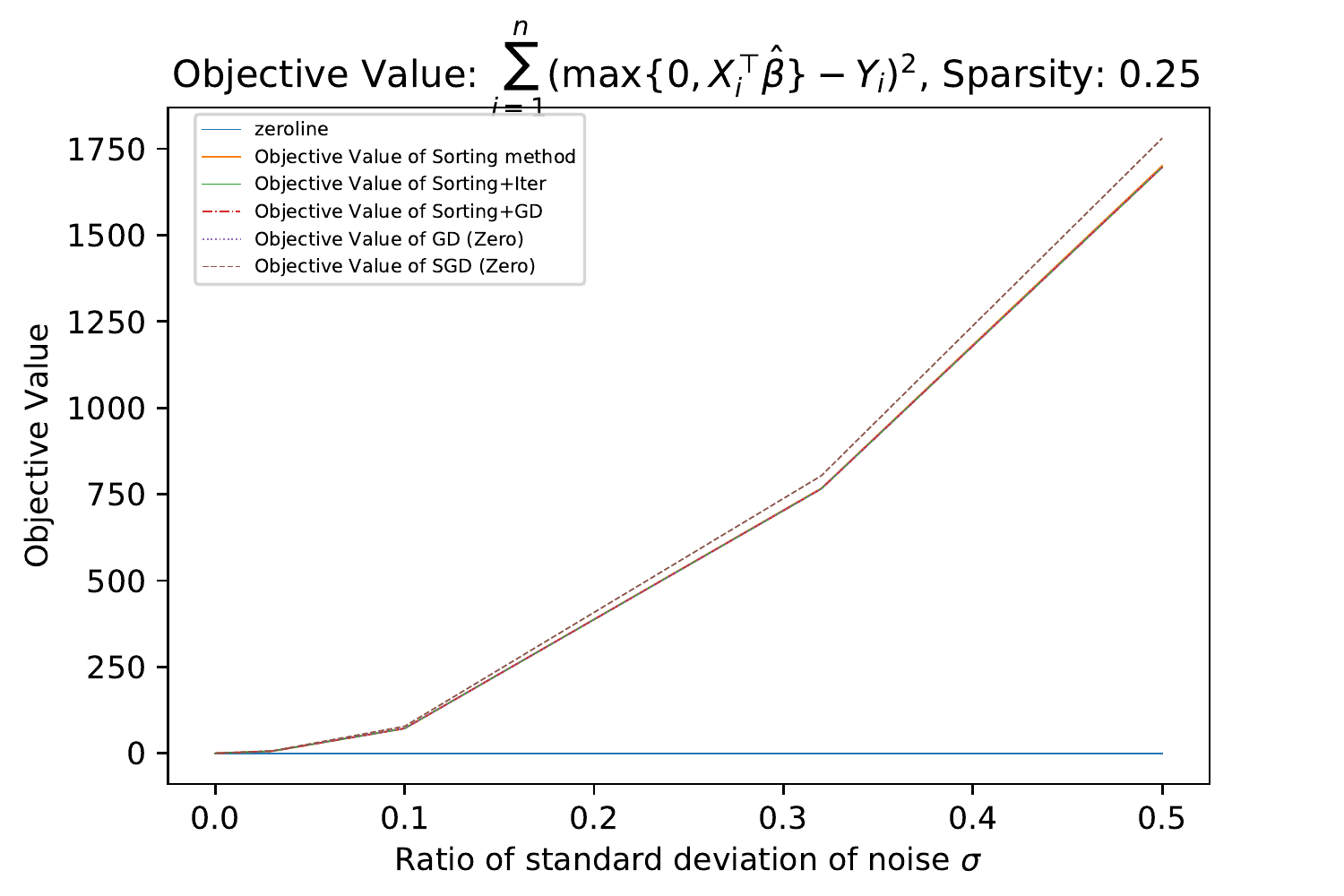}
        \caption[]%
        {{\small Objective Value}}    
        \label{fig:Figure_Compare_OB_50_1000_25}
    \end{subfigure}
    \hfill
    \begin{subfigure}[b]{0.24\textwidth}  
        \centering 
        \includegraphics[width=\textwidth]{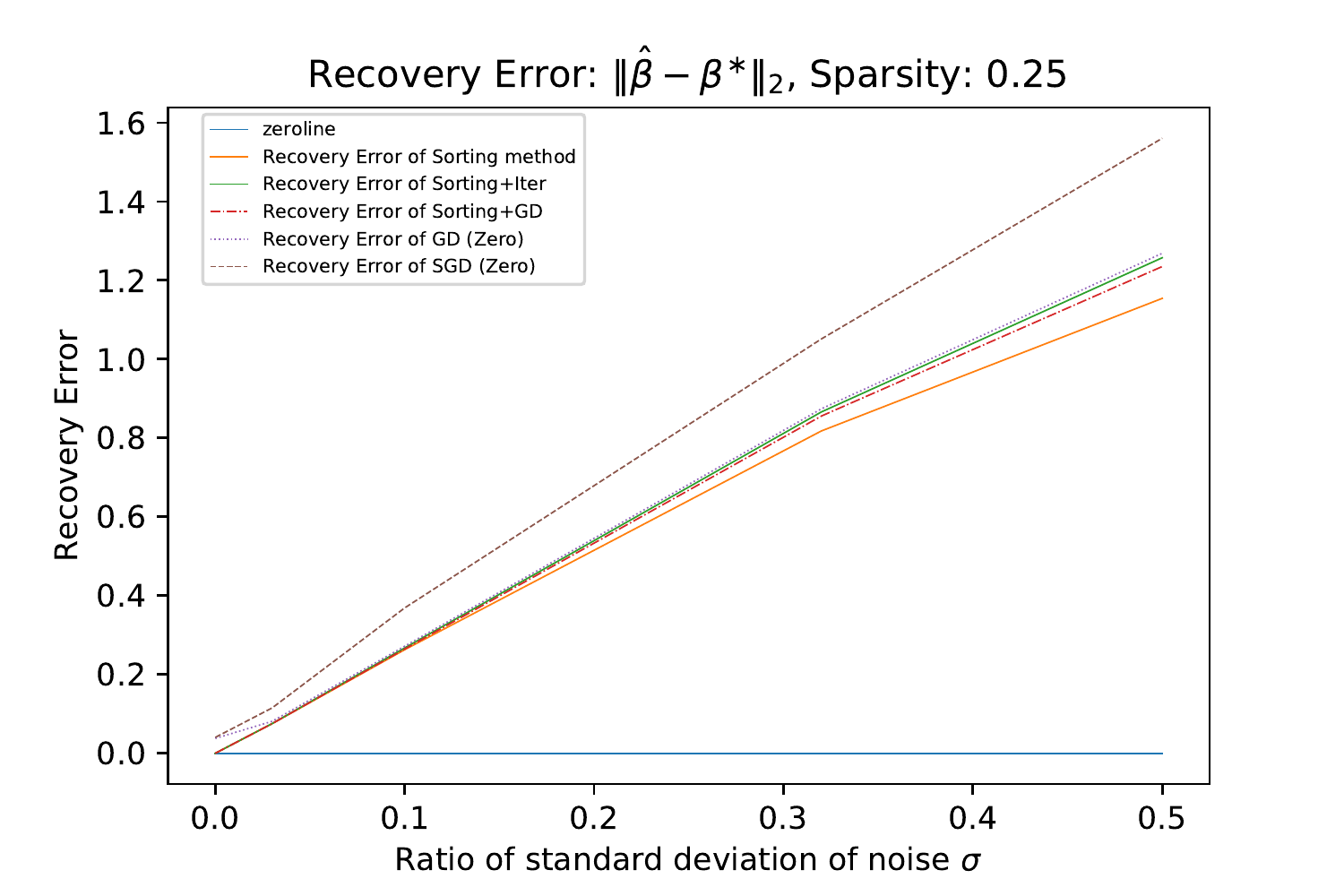}
        \caption[]%
        {{\small Recovery Error}}    
        \label{fig:Figure_Compare_RE_50_1000_25}
    \end{subfigure}
    \hfill
    \begin{subfigure}[b]{0.24\textwidth}   
        \centering 
        \includegraphics[width=\textwidth]{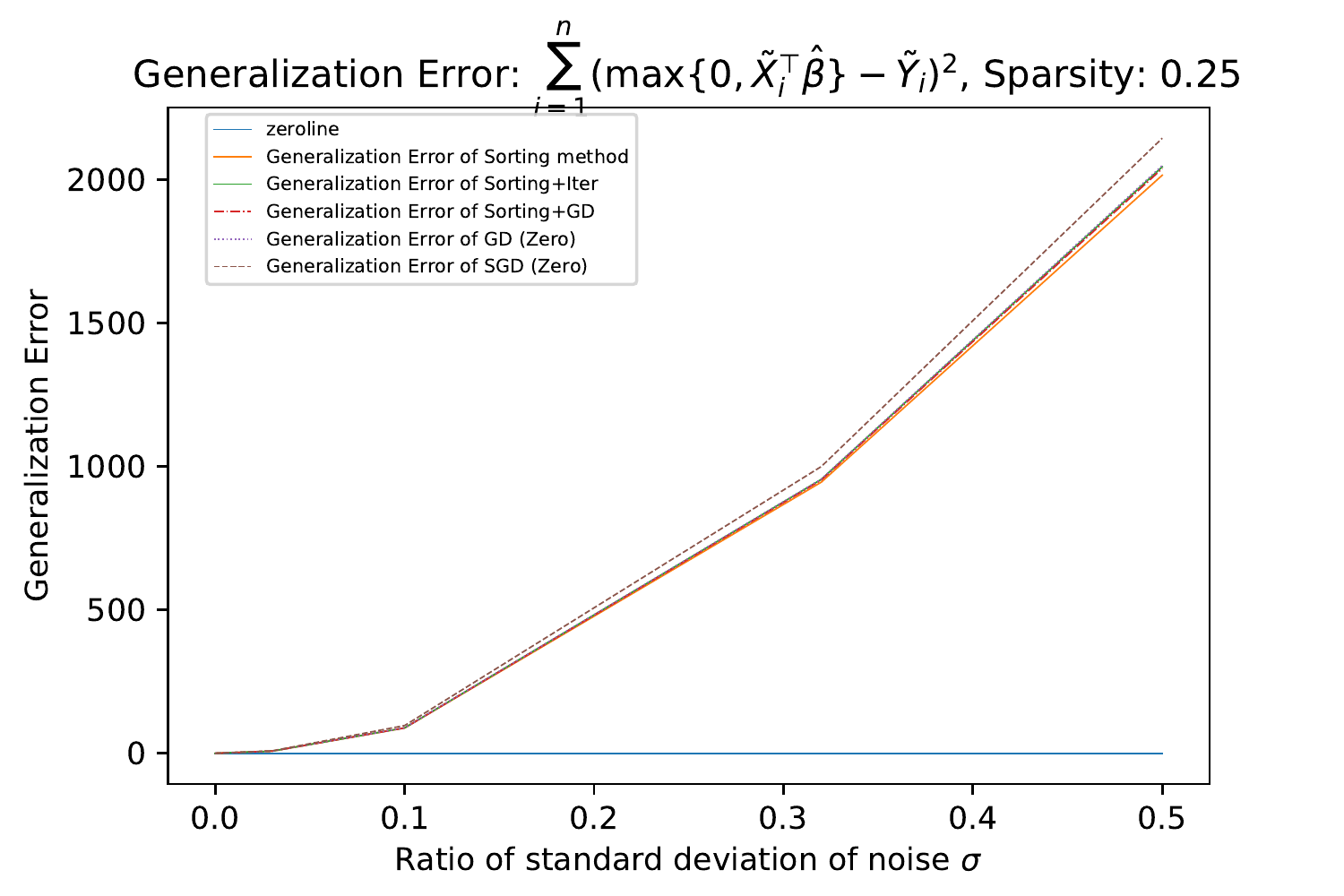}
        \caption[]%
        {{\small Generalization Error}}    
        \label{fig:Figure_Compare_GE_50_1000_25}
    \end{subfigure}
    
    \vskip\baselineskip
    
    \centering
    \begin{subfigure}[b]{0.24\textwidth}
        \centering
        \includegraphics[width=\textwidth]{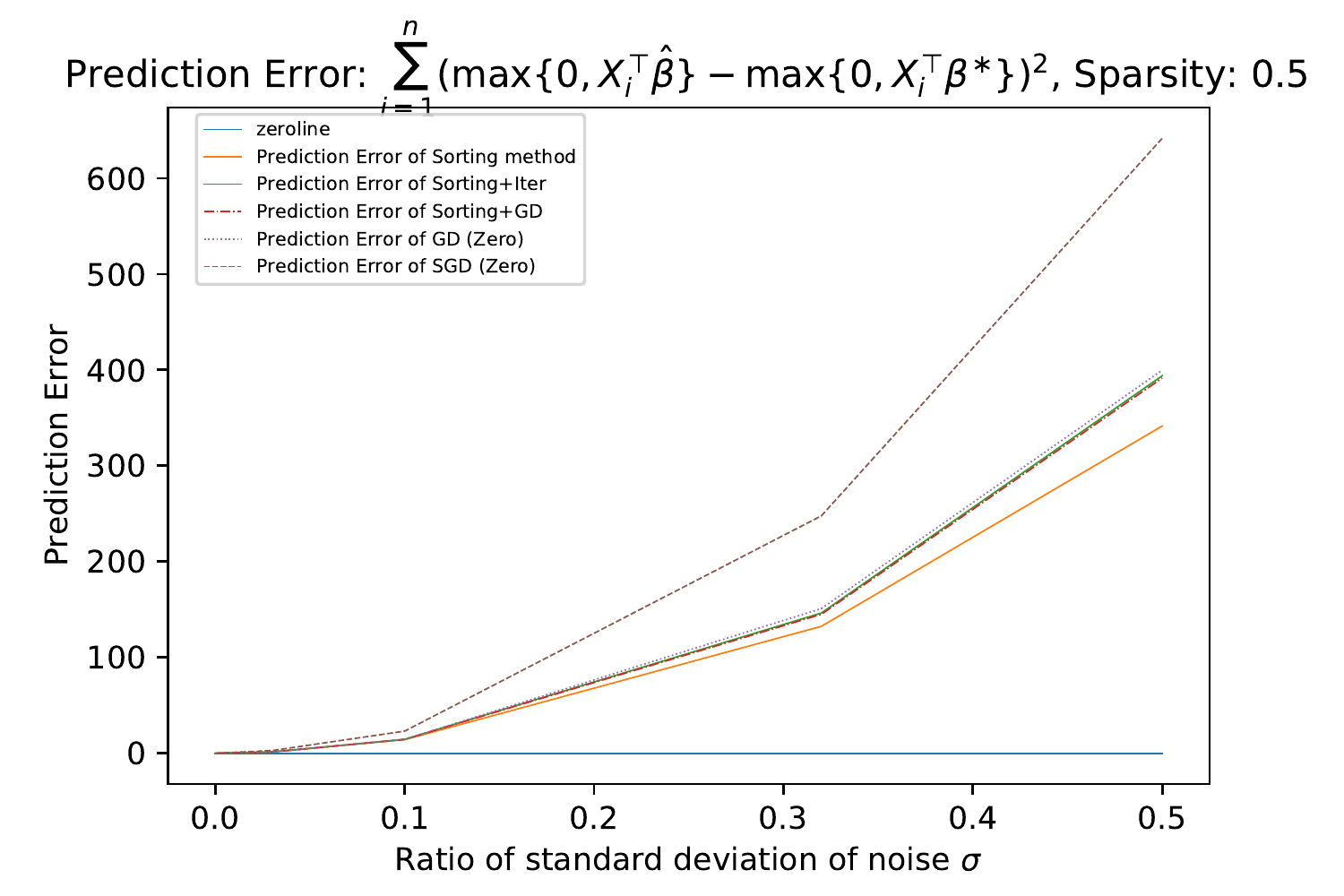}
        \caption[Network1]%
        {{\small Prediction Error}}    
        \label{fig:Figure_Compare_PE_50_1000_50}
    \end{subfigure}
    \hfill
    \begin{subfigure}[b]{0.24\textwidth}   
        \centering 
        \includegraphics[width=\textwidth]{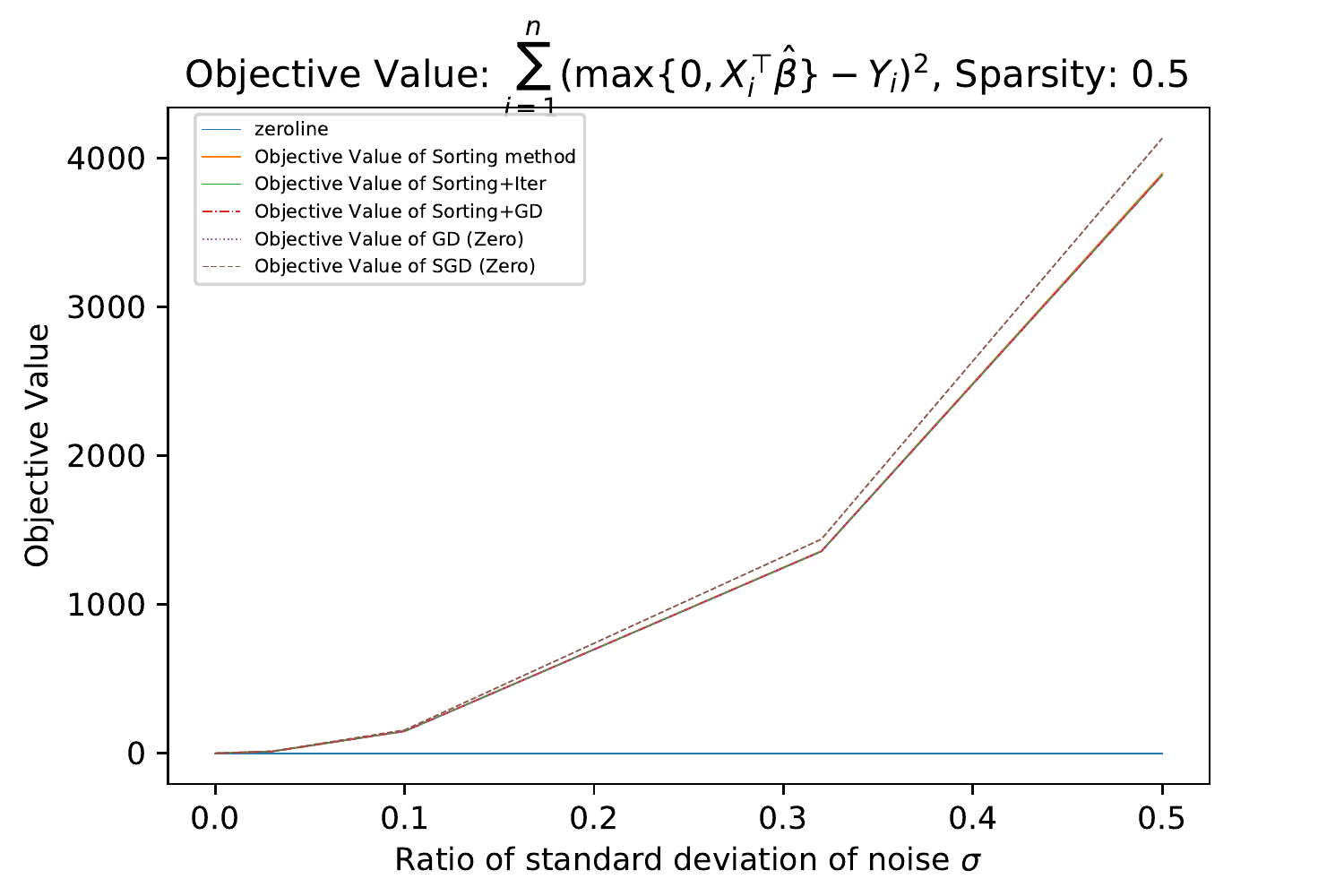}
        \caption[]%
        {{\small Objective Value}}    
        \label{fig:Figure_Compare_OB_50_1000_50}
    \end{subfigure}
    \hfill
    \begin{subfigure}[b]{0.24\textwidth}  
        \centering 
        \includegraphics[width=\textwidth]{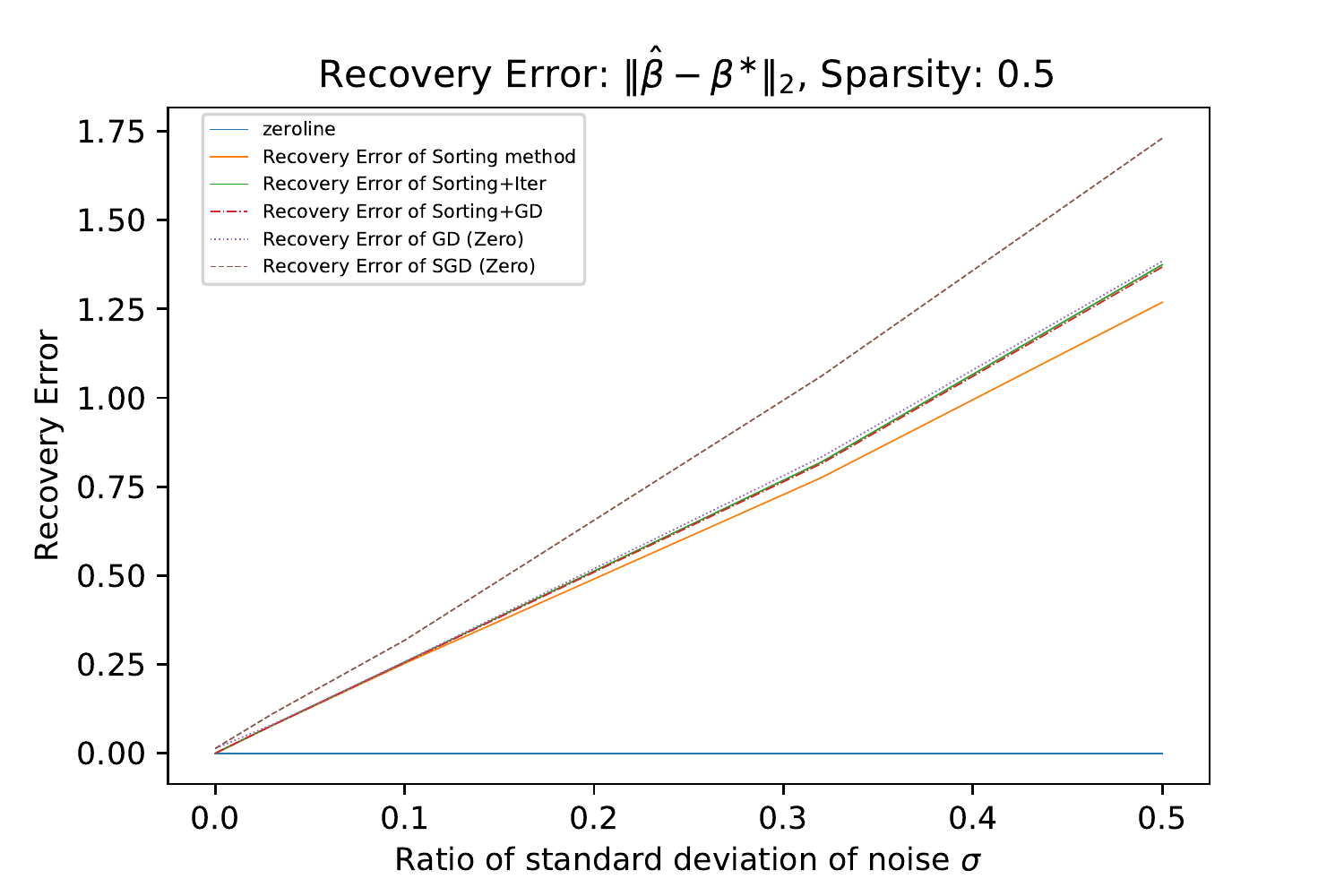}
        \caption[]%
        {{\small Recovery Error}}    
        \label{fig:Figure_Compare_RE_50_1000_50}
    \end{subfigure}
    \hfill
    \begin{subfigure}[b]{0.24\textwidth}   
        \centering 
        \includegraphics[width=\textwidth]{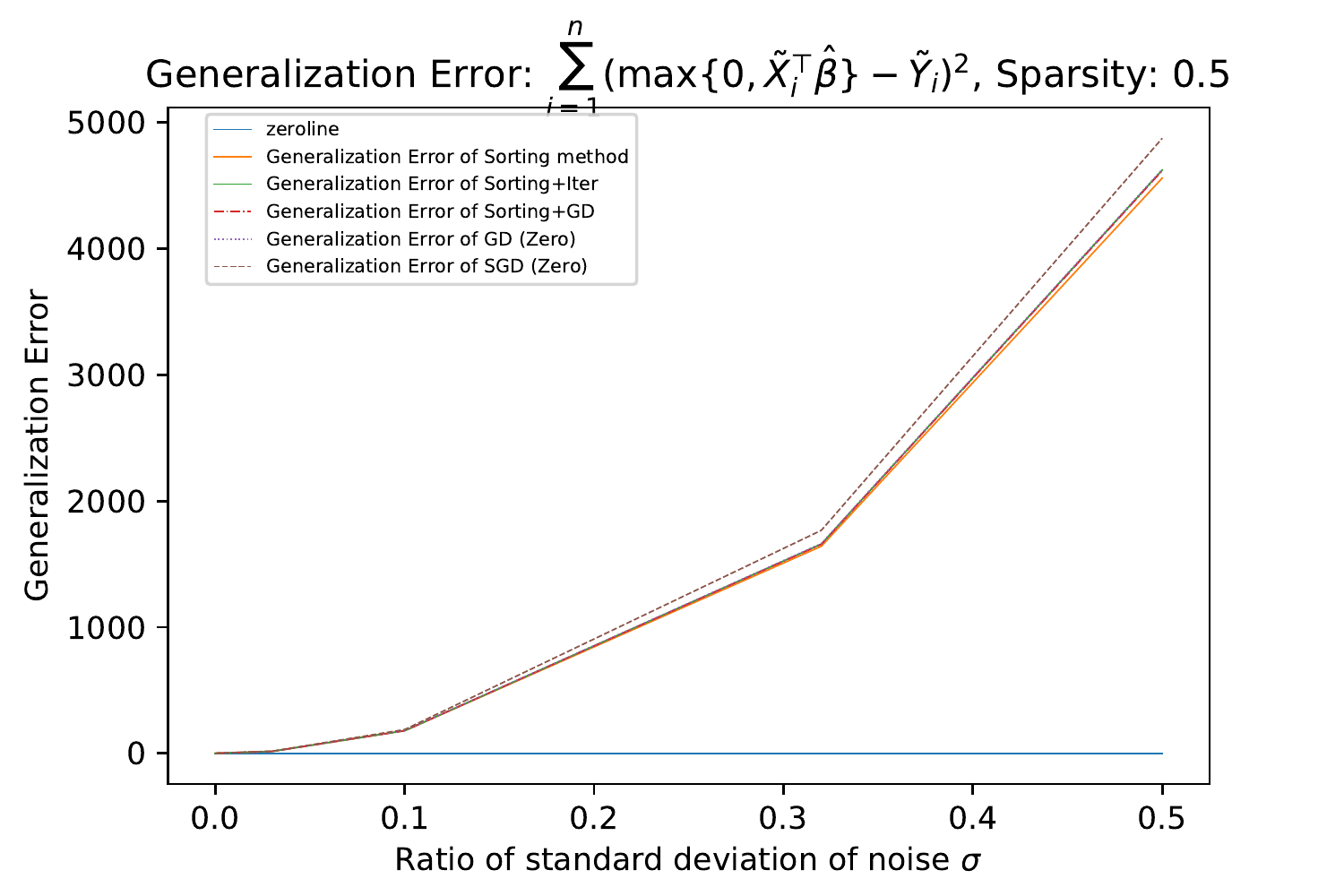}
        \caption[]%
        {{\small Generalization Error}}    
        \label{fig:Figure_Compare_GE_50_1000_50}
    \end{subfigure}
    
    \vskip\baselineskip
    
    \centering
    \begin{subfigure}[b]{0.24\textwidth}
        \centering
        \includegraphics[width=\textwidth]{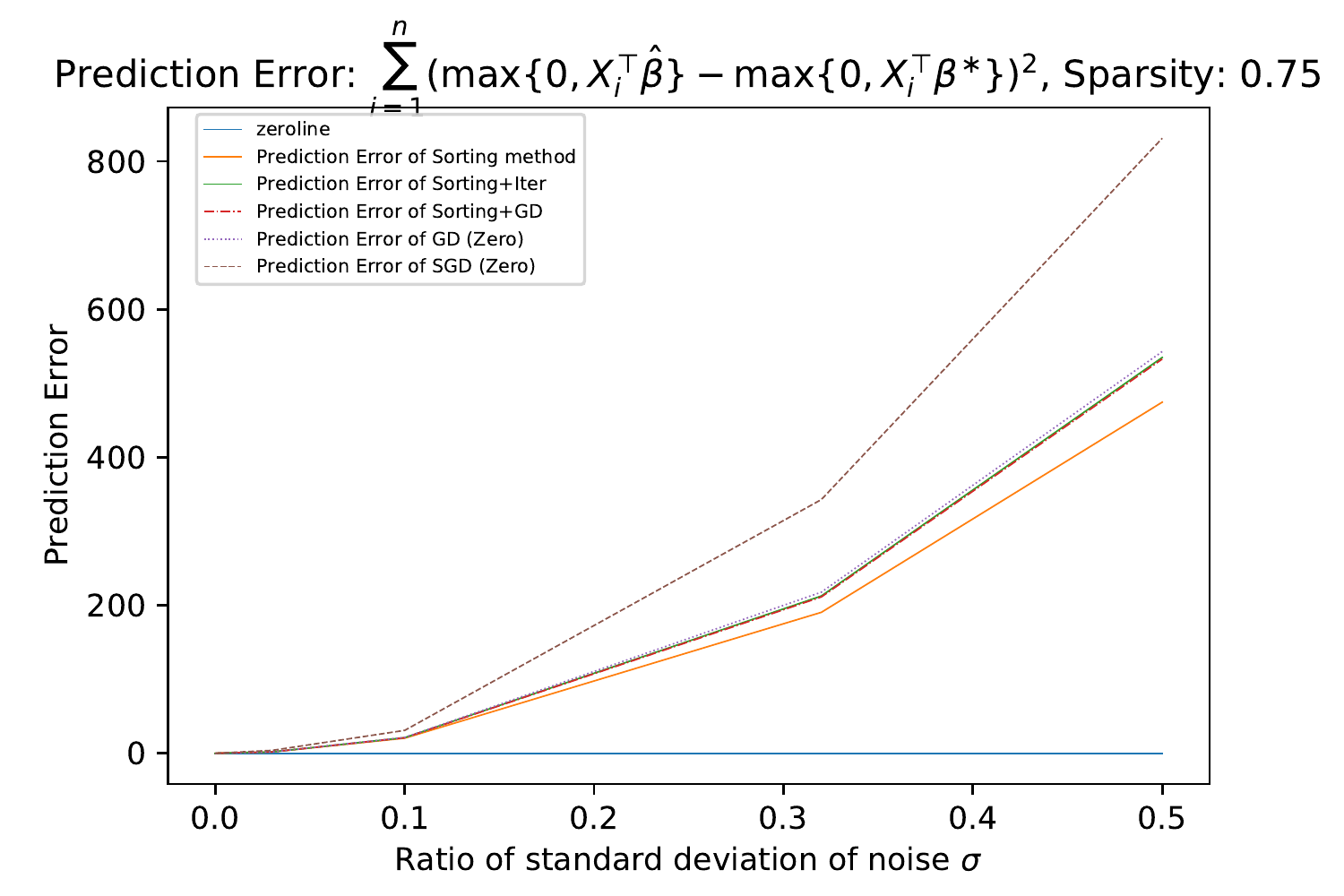}
        \caption[Network1]%
        {{\small Prediction Error}}    
        \label{fig:Figure_Compare_PE_50_1000_75}
    \end{subfigure}
    \hfill
    \begin{subfigure}[b]{0.24\textwidth}   
        \centering 
        \includegraphics[width=\textwidth]{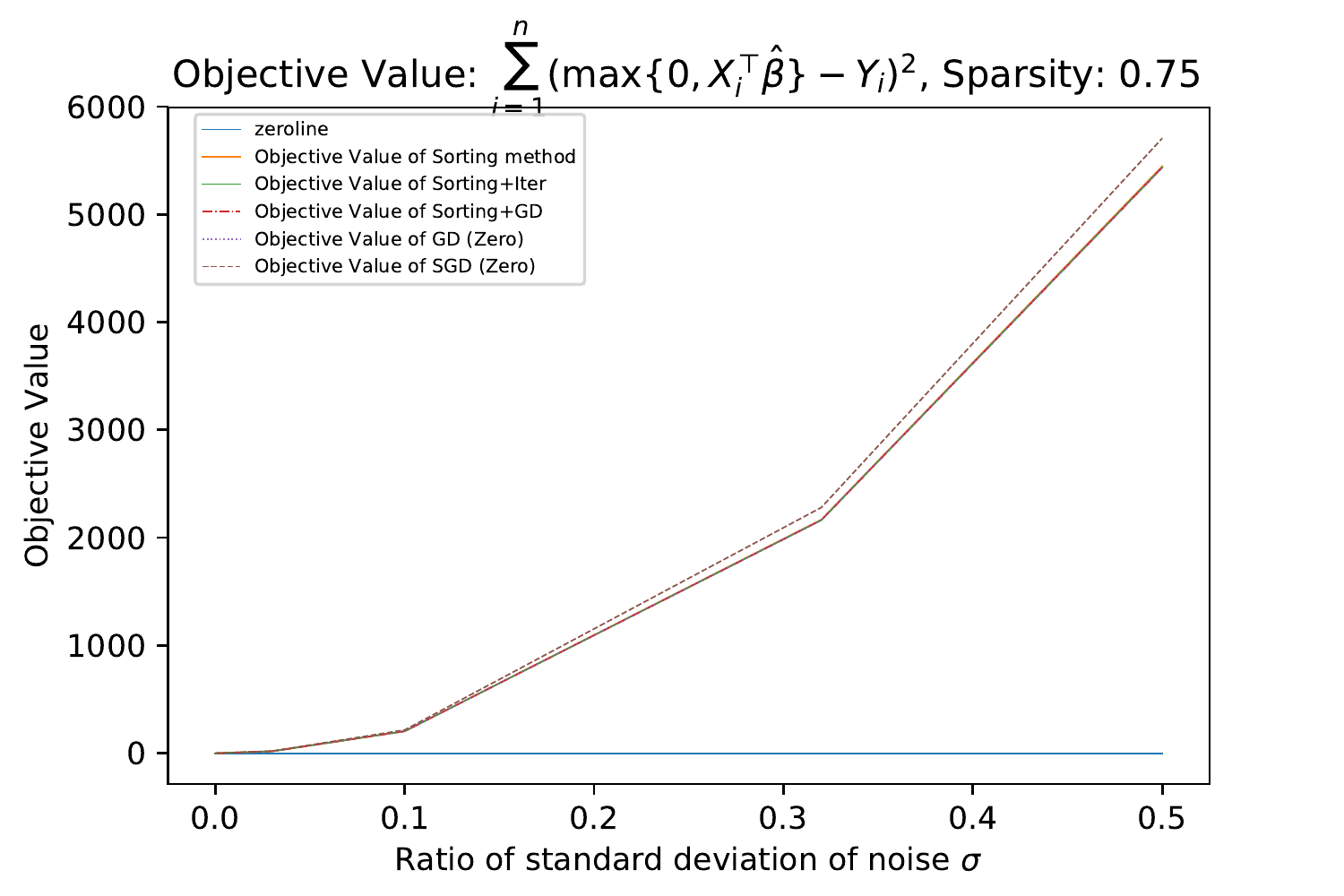}
        \caption[]%
        {{\small Objective Value}}    
        \label{fig:Figure_Compare_OB_50_1000_75}
    \end{subfigure}
    \hfill
    \begin{subfigure}[b]{0.24\textwidth}  
        \centering 
        \includegraphics[width=\textwidth]{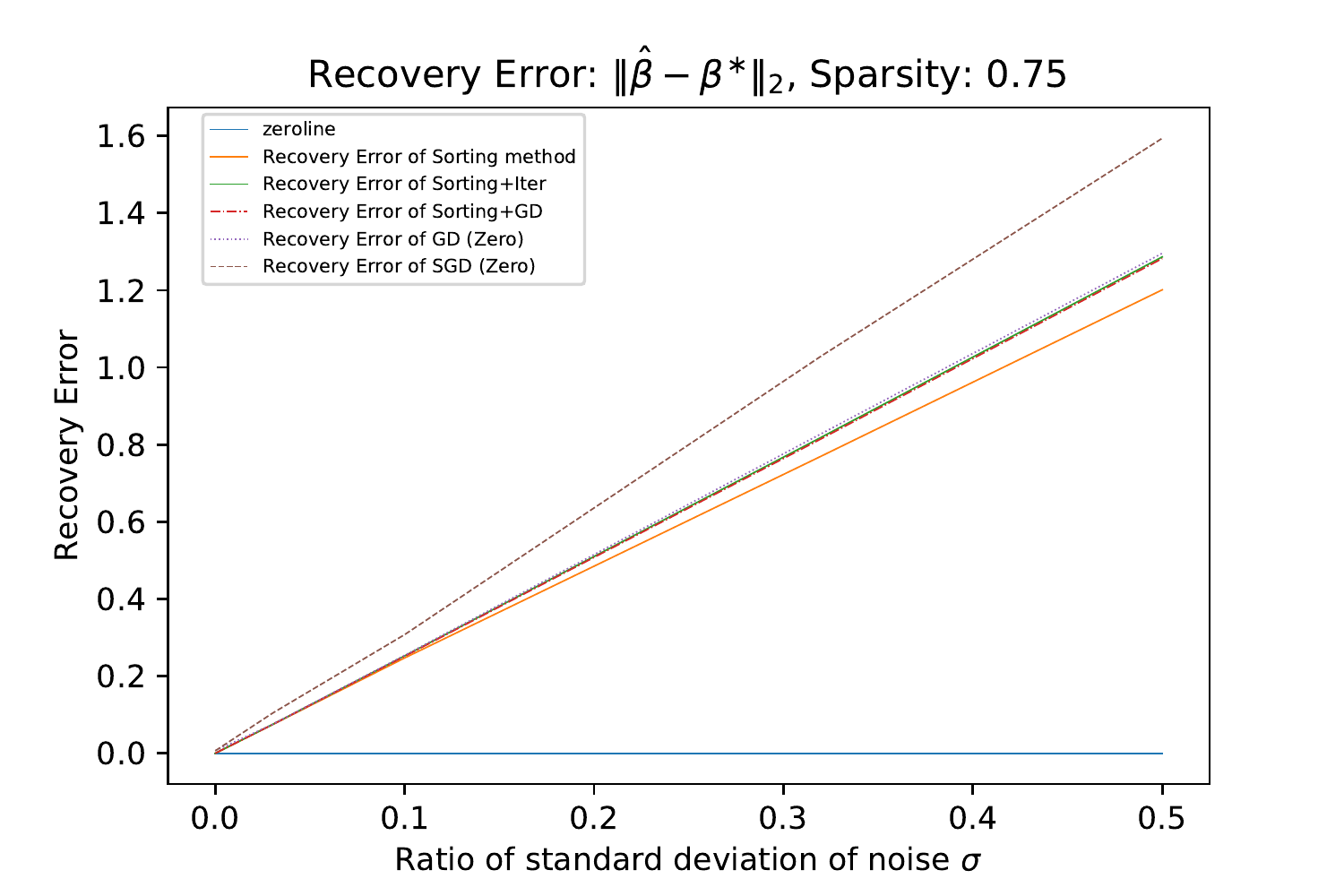}
        \caption[]%
        {{\small Recovery Error}}    
        \label{fig:Figure_Compare_RE_50_1000_75}
    \end{subfigure}
    \hfill
    \begin{subfigure}[b]{0.24\textwidth}   
        \centering 
        \includegraphics[width=\textwidth]{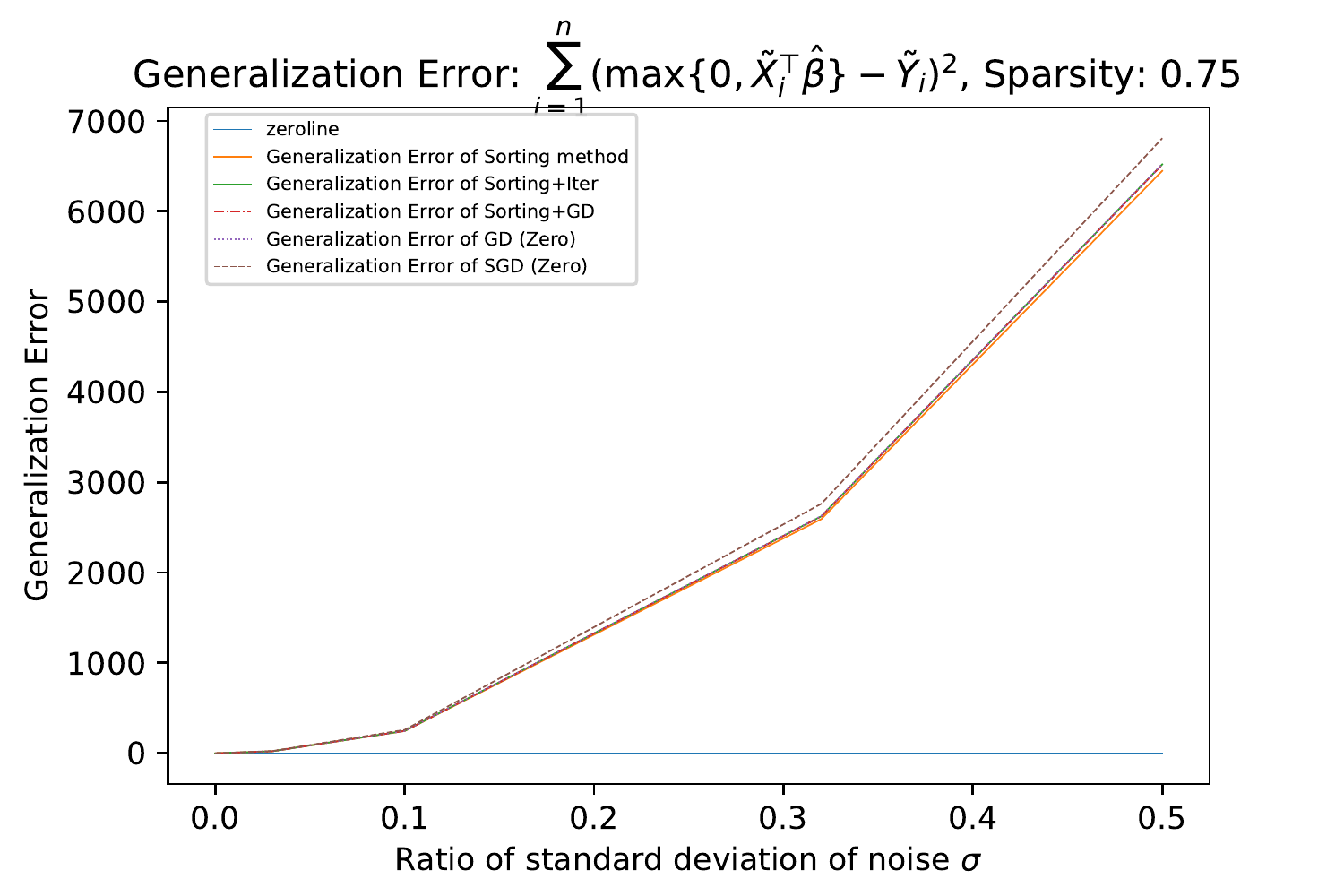}
        \caption[]%
        {{\small Generalization Error}}    
        \label{fig:Figure_Compare_GE_50_1000_75}
    \end{subfigure}
    
    \vskip\baselineskip
    
    \centering
    \begin{subfigure}[b]{0.24\textwidth}
        \centering
        \includegraphics[width=\textwidth]{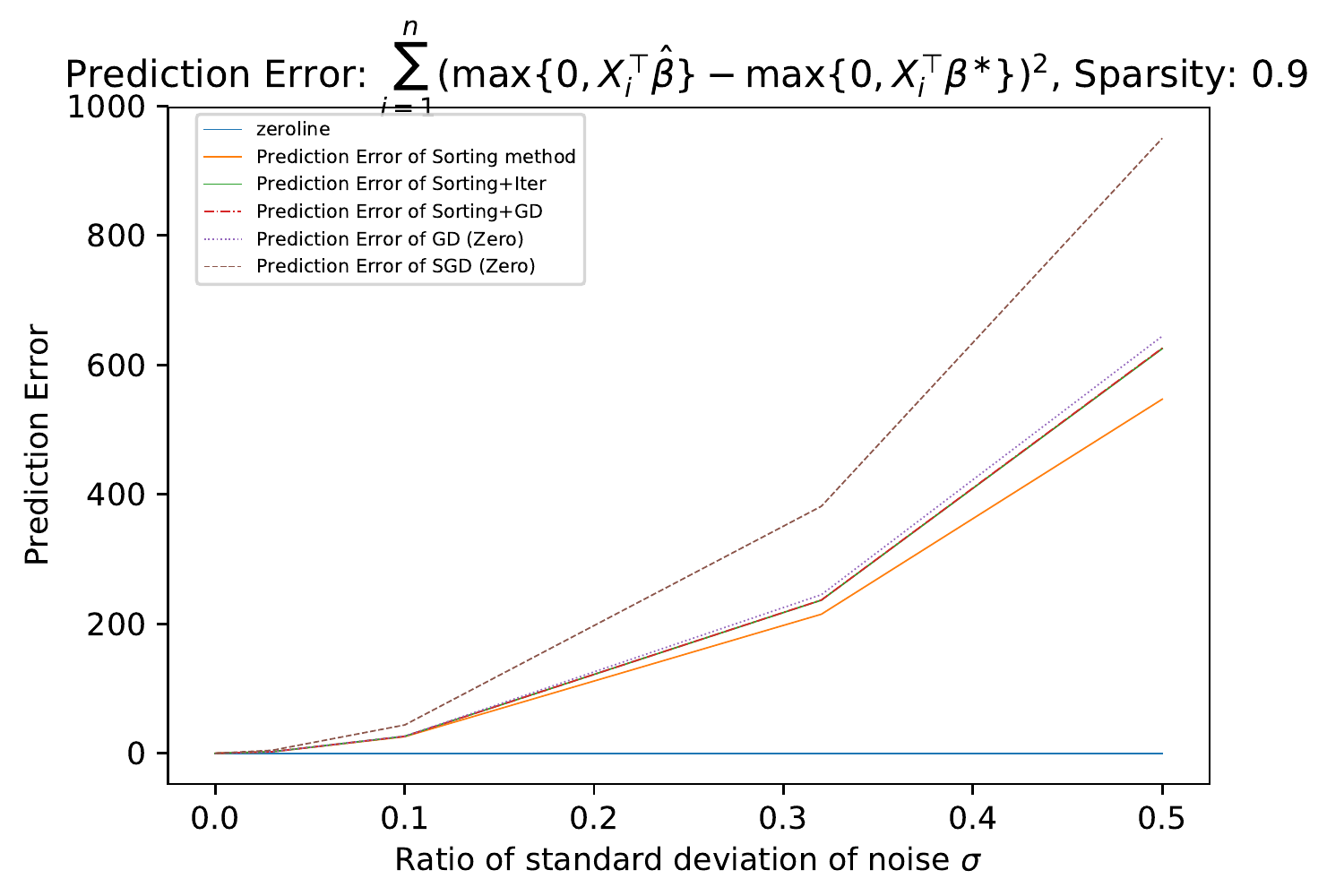}
        \caption[Network1]%
        {{\small Prediction Error}}    
        \label{fig:Figure_Compare_PE_50_1000_90}
    \end{subfigure}
    \hfill
    \begin{subfigure}[b]{0.24\textwidth}   
        \centering 
        \includegraphics[width=\textwidth]{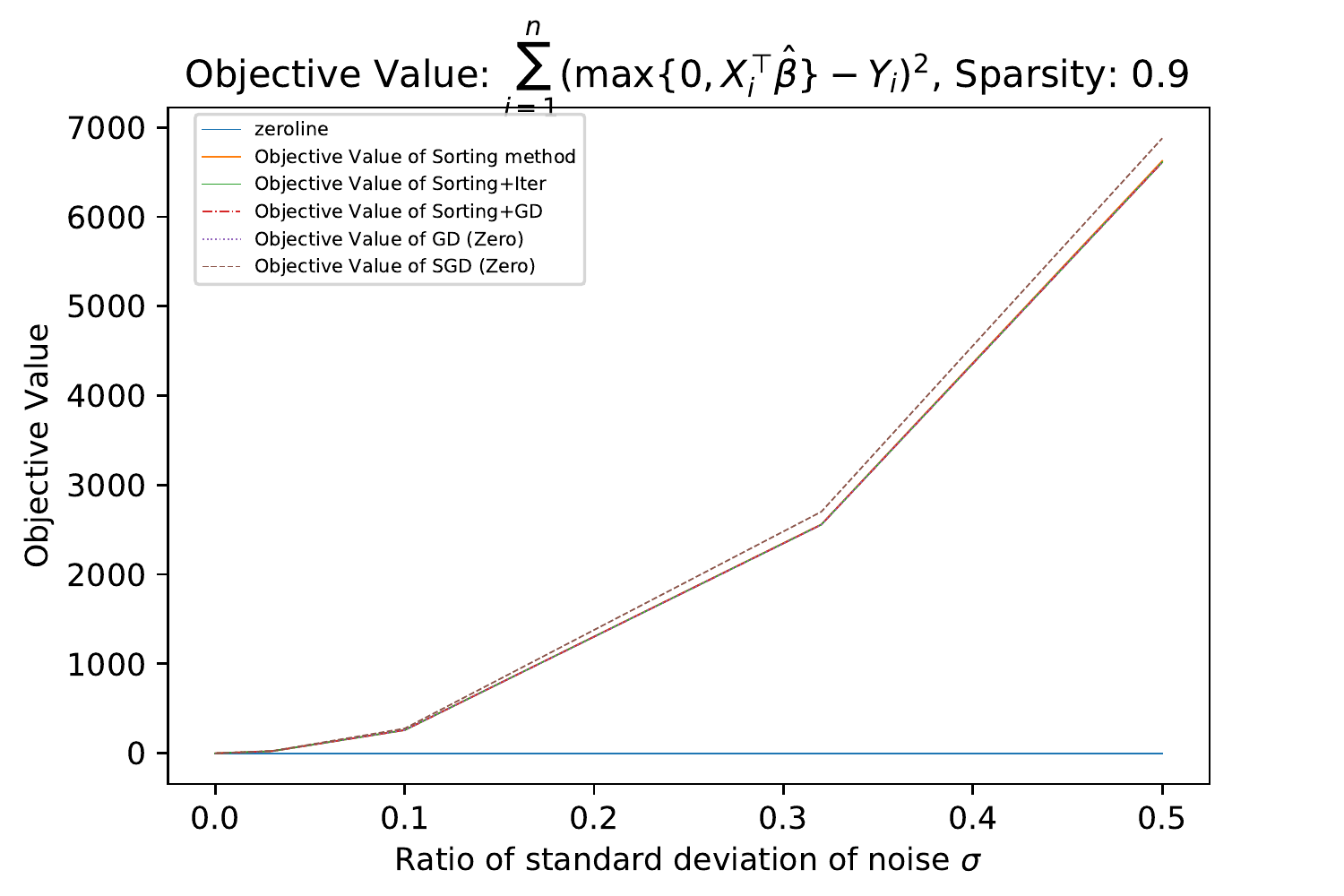}
        \caption[]%
        {{\small Objective Value}}    
        \label{fig:Figure_Compare_OB_50_1000_90}
    \end{subfigure}
    \hfill
    \begin{subfigure}[b]{0.24\textwidth}  
        \centering 
        \includegraphics[width=\textwidth]{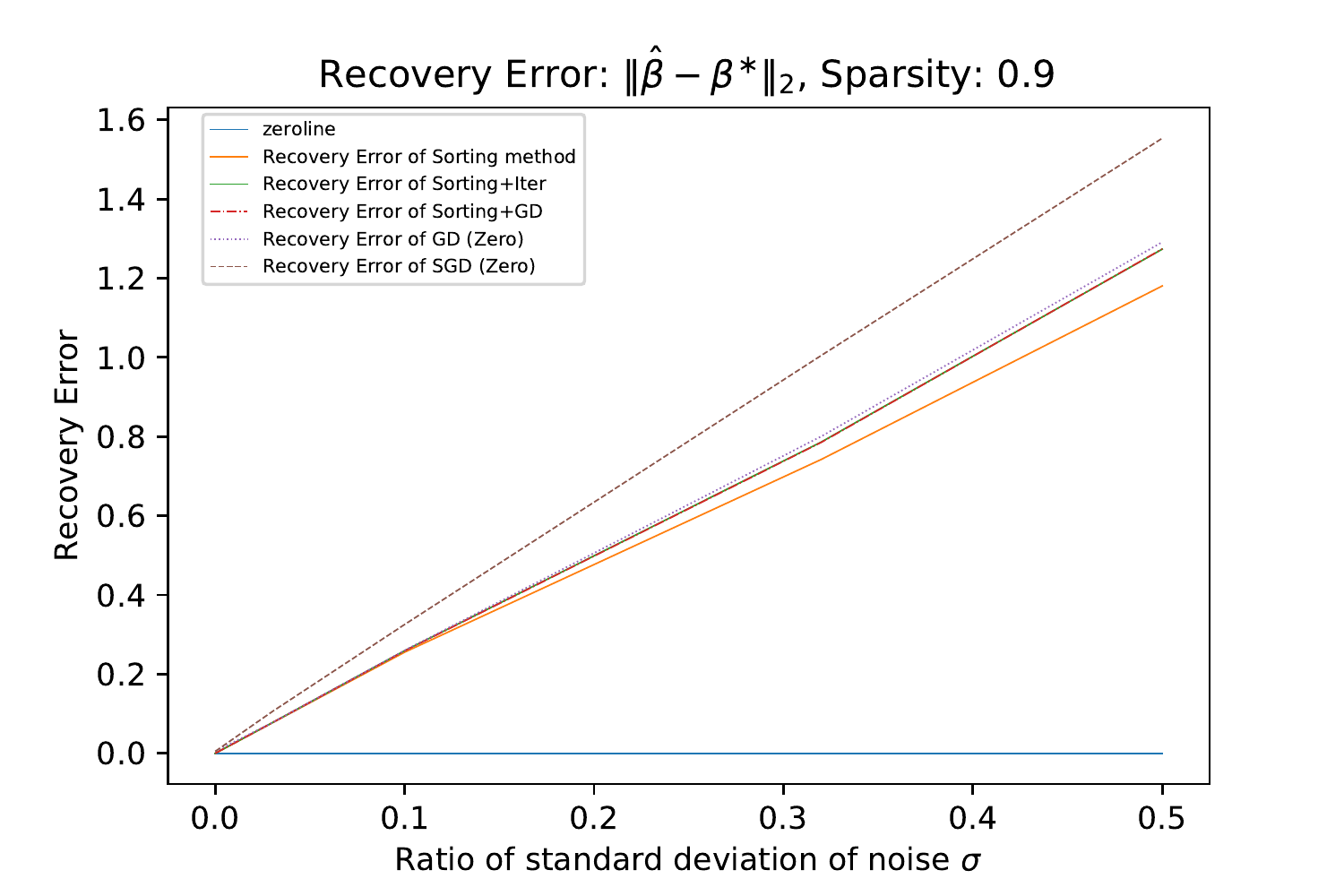}
        \caption[]%
        {{\small Recovery Error}}    
        \label{fig:Figure_Compare_RE_50_1000_90}
    \end{subfigure}
    \hfill
    \begin{subfigure}[b]{0.24\textwidth}   
        \centering 
        \includegraphics[width=\textwidth]{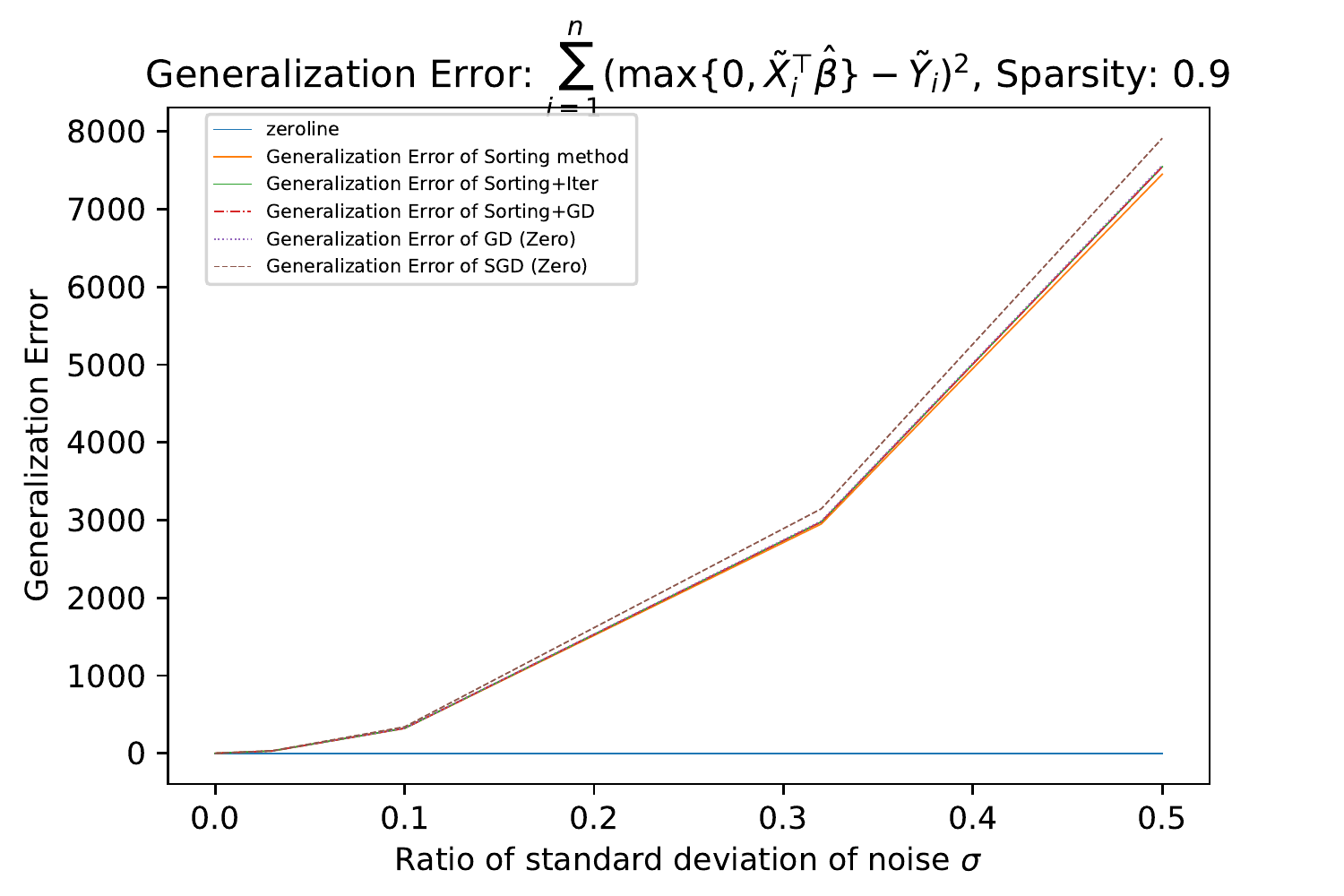}
        \caption[]%
        {{\small Generalization Error}}    
        \label{fig:Figure_Compare_GE_50_1000_90}
    \end{subfigure}
    
    \caption[Numerical Results of sample size $(p, n)= (50, 1000)$]
    {\small Numerical Results of sample size $(p, n) = (50, 1000)$ and $\beta^{\ast} \sim N(0.5 \cdot \mathbf{1}_p, 50 \cdot I_p)$ with sparsity $\{0.1, 0.25, 0.5, 0.75, 0.9\}$} 
    \label{fig:Figure_Compare_50_1000}
\end{figure*}

\subsection{Summary of Numerical Experiments}
Based on the results reported in Figure~[\ref{fig:Figure_Compare_10_200}, \ref{fig:Figure_Compare_20_400}, \ref{fig:Figure_Compare_50_1000}] and Tables in Appendix~\ref{section:realizable-cases}, some preliminary conclusions can be draw as follows:
\begin{description}
	\item[Prediction Error:] The empirical prediction error satisfies the following: 
	\begin{align*}
		\text{PE}^{\text{sorting}} \leq \text{PE}^{\text{sorting + iter}} \leq \text{PE}^{\text{sorting + GD}} \leq \text{PE}^{\text{GD}} \leq \text{PE}^{\text{SGD}}
	\end{align*}
	where the differences between $\text{PE}^{\text{sorting + iter}}, \text{PE}^{\text{sorting + GD}}, \text{PE}^{\text{GD}}$ are relative small than the differences between $\text{PE}^{\text{sorting}}, \text{PE}^{\text{sorting + iter}}$ and $\text{PE}^{\text{GD}}, \text{PE}^{\text{SGD}}$. These empirical results show that the when the output samples $\{Y_i\}$ follows the correct underlying model (which may not be for some real applications), the sorting method performs well in practice. 
	\item[Objective Value:] In most of the cases, objective value satisfies 
	\begin{align*}
		& \text{Obj}^{\text{sorting + iter}} \leq \text{Obj}^{\text{sorting + GD}} \leq \text{Obj}^{\text{sorting}} \leq \text{Obj}^{\text{GD}} \leq \text{Obj}^{\text{SGD}}.
	\end{align*}
	The difference between the SGD method and the GD method are large in general since SGD cannot always find out the local minimum solution in reasonable time. The gaps between GD method and the rest three methods (sorting, sorting + GD, sorting + iter) are relatively larger than the differences within the rest three methods. The objective value of sorting method, when the standard deviation of noise grows, increases most. The sorting + iterative method and sorting + gradient descent method perform almost the same for objective value, which implies that: (1) doing iterative method after the sorting method really benefits the optimization (comparing with sorting method with smaller objective value); (2) initializing with $\hat{\beta}^{\text{sorting}}$ will improve the performances of gradient descent (comparing with GD/SGD with smaller objective value). 
	
	\item[Recovery Error:] When the standard deviation of noise is small, the recovery error satisfies that $$\text{RE}^{\text{sorting + iter}} \leq \text{RE}^{\text{sorting + GD}} \leq \text{RE}^{\text{sorting}} \leq \text{RE}^{\text{GD}} \leq \text{RE}^{\text{SGD}}.$$ As the standard deviation of noise increases, the recovery error obtained from gradient descent method will not increases as much as the rest three types of methods, and finally becomes the best at the point with $\rho = 0.32$. 
	
	\item[Generalization Error:] The performances of generalization error is very similar to the performances of prediction error. Hence the sorting + iterative method has the strongest generalization power. 
	
	\item[Running Time:] Empirically, the running time of sorting method, sorting + iterative method, sorting + GD method and SGD method satisfies the following:
	\begin{align*}
		T^{\text{SGD}} \leq T^{\text{sorting}} \leq T^{\text{sorting + iter}} \approx T^{\text{sorting + GD}}
	\end{align*}
	in most of the cases. One possible result of the least running time of SGD method is that SGD cannot find out the local minimum and stops early with fewer iterations.  For GD method with $\mathbf{0}_p$ initial point, as the size of instances increases, its running time increases faster than the rest four methods. Moreover, the sparsity level, in empirical, has significant influence on the running time of GD method.   
\end{description}

\section{Conclusions}\label{sec:conclusion}
After showing that that~\ref{eq:One-Node-ReLU} is NP-hard, we presented an approximation algorithm for this problem. We showed that for arbitrary data this algorithm gives a multiplicative guarantee of $\frac{n}{k}$ where $n$ is the number of samples and $k$ is a fixed integer. An important consequence of this result is that in the realizable case~\ref{eq:One-Node-ReLU} can be solved in polynomial time.  In the more natural ``statistical model" of training data, where the data comes from an underlying single node with relu function where the output is perturbed with a Gaussian noise, we are able to show that the algorithm promises guarantees that are independent of $n$. \emph{To the best of our knowledge, these are best theoretical performance guarantees for the solving \ref{eq:One-Node-ReLU}, especially in the realizable case and in the case of statistical data model. }

Computational experiments show that Algorithm together with a heuristic performs better that gradient descent and stochastic gradient algorithm. Very importantly, starting gradient descent from the solution of the approximation algorithm, performs significantly better than gradient descent algorithm. In our opinion, this is a very important empirical observation in the following sense: \emph{there is value in coming up with specialized approximation algorithms for various non-convex problems (for which we intend to use gradient descent), since such algorithms due to their theoretical guarantees provide a good starting point for gradient descent, usually a requirement for the gradient descent algorithm to work well.} 

Many open questions remain. In the case of arbitrary training data model, there is big gap between multiplicative guarantee of $\frac{n}{k}$ and known lower bound of $n^{\frac{1}{loglog n}}$. In the statistical model, we believe that our Algorithm is optimal, i.e. performance guarantees cannot be improved. Proving or disproving this conjecture is important. Another important direction of research is to extend these results to multi-node networks.  

\section{Proofs of Results Presented in Section~\ref{sec:theory}}\label{sec:proofs}
\subsection{Proof of Proposition \ref{prop:convex} \label{section:Prop-convex}}
\begin{proof}
Since $\phi(\beta, \beta_0) = \sum_{i \in \{m + 1, \ldots, n\}} (\max\{0, X_i^{\top} \beta + \beta_0\} - Y_i)^2$, then it is sufficient to show that $(\max\{0, X_i^{\top} \beta + \beta_0\} - Y_i)^2$ is convex for each $i = m + 1, \ldots, n$. Let $\theta(x) = (\max\{0, x\} - Y_i)^2 = (\max\{0, x\})^2 + Y_i^2 - 2Y_i \max\{0, x\}$ with $Y_i < 0$. Note that $\theta(x)$ is convex over $x \in \mathbb{R}$. Let $L(\beta, \beta_0) = X_i^{\top} \beta + \beta_0$ be an affine function. Then $(\max\{0, X_i^{\top} \beta + \beta_0\} - Y_i)^2 = \theta(L(\beta, \beta_0))$ is convex. 
\end{proof}

\subsection{Proof of Theorem \ref{thm:NP-hard} \label{section:NP-hard}}

In order to prove Theorem~\ref{thm:NP-hard}, we show that the subset sum problem can be polynomially reduced to a special case of \ref{eq:One-Node-ReLU} problem. We begin a definition of the subset sum problem. 
\begin{definition}
\textbf{Subset sum problem:} Given $p$ non-negative integers $a_1, \ldots, a_p$, the subset sum problem is to find out whether there exists a subset $S \subseteq [p]$ such that $\sum_{i \in S} a_i = \frac{1}{2}\sum_{i =1}^pa_i$.
\end{definition}

Note that the subset sum problem is equivalent to find out a feasible solution $x \in \{0,1\}^p$ such that $\sum_{i = 1}^n a_i x_i = \frac{1}{2}\sum_{i =1}^pa_i$. Therefore, the following $\{\pm 1\}-$subset sum problem is still NP-hard. 

\begin{definition}
\textbf{$\{\pm 1\}-$subset sum problem:} Given  $p$ nonnegative integers $a_1, \ldots, a_p$, the $\{\pm 1\}-$subset sum problem is to decide if there exists a solution $x \in \{\pm 1\}^p$ such that $\sum_{i = 1}^p a_i x_i = \frac{1}{2}\sum_{i =1}^pa_i$.
\end{definition}

\begin{proposition}
The decision problem $\{\pm 1\}-$subset sum problem is in NP-complete. 
\end{proposition}
\begin{proof} 
Clearly, {$\{\pm 1\}-$subset sum problem} is in NP.  In order to show that  \textbf{$\{\pm 1\}-$subset sum problem} is in NP-complete, we show that the instance of subset sum corresponding to $(a_1, \dots, a_p)$ is feasible 
if and only if the \textbf{$\{\pm 1\}-$subset sum} instance $(a_1, \dots, a_p, a_{p+1})$ with $a_{p +1} = \sum_{i =1}^p a_i$ is feasible.

Clearly if the subset set instance is feasible, then there exists a subset $S \subseteq [p]$ such that $\sum_{i \in S} a_i = \frac{1}{2}\sum_{i = 1}^p a_i$. Then setting $x_i = 1$ for $i \in S \cup \{p +1\}$ and $x_i = -1$ for $i \in [p]\setminus S$ gives us: $\sum_{i = 1}^{p+1}a_i x_i = \frac{1}{2}\sum_{i = 1}^{p+1}a_i$. 

On the other hand if the $\{\pm 1\}-$subset sum is feasible, there exists some $x_i \in \{-1, 1\}^{p+1}$ such that $\sum_{i = 1}^{p+1}a_i x_i = \frac{1}{2}\sum_{i = 1}^{p+1}a_i$. First observe that $x_{p +1}$ cannot be $-1$ since then we would have that $\sum_{i = 1}^p a_i x_i = 2 \sum_{i = 1}^p a_i$. Thus, we have that $\sum_{i = 1}^pa_ix_i = 0$ implying that there exists $S \subseteq [p]$ such that $\sum_{i \in S} a_i = \frac{1}{2}\sum_{i = 1}^p a_i$.
\end{proof}

Now we show the equivalence between $\{\pm 1\}-$subset sum problem and a special case of \ref{eq:One-Node-ReLU} problem. Consider the following auxiliary function 
\begin{align*}
	\theta(x, \beta_0) = & (\max\{0, x + \beta_0 \} - 1)^2 + (\max\{0, - x + \beta_0 \} - 1)^2 
\end{align*}
(See Figure \ref{fig:theta}).
\begin{figure}[h]
	\begin{center}
		\includegraphics[width= 0.6 \textwidth]{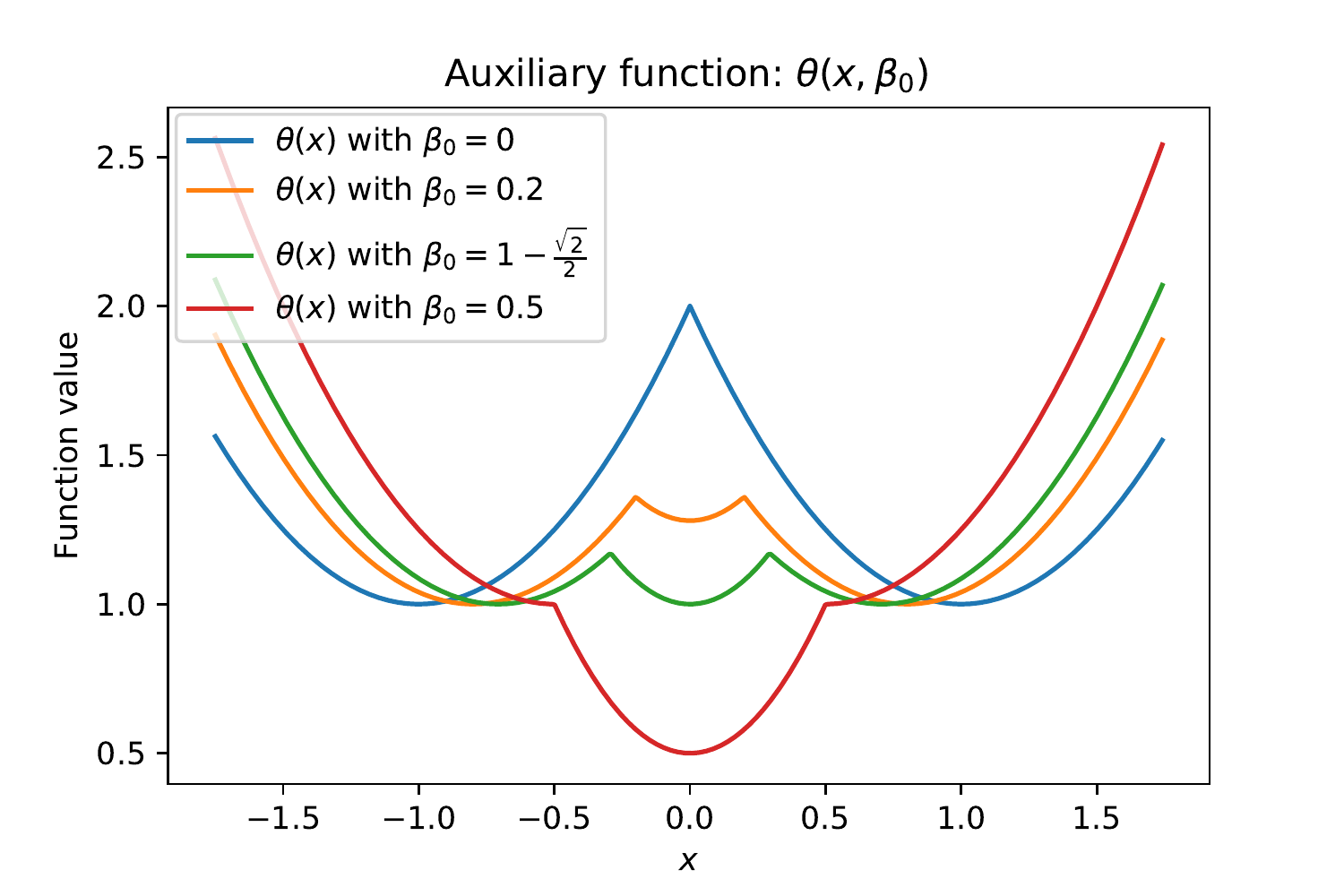}
	\end{center}
	\caption{Function $\theta(x, \beta_0)$}
	\label{fig:theta}
\end{figure}
For a fixed $x$, let $g(\beta_0) = \min_x \theta(x, \beta_0)$. We construct our affine \ref{eq:One-Node-ReLU} problem as follows: 
\begin{align*}
	\min_{\beta, \beta_0 \in \mathbb{R}^{p + 1}} \underbrace{\bigg( \max \bigg\{0, \sum_{i = 1}^p a_i \beta_i + \beta_0 \bigg\} - \frac{1}{2}\sum_{i = 1}^pa_i \bigg)^2 + \bigg( \max \bigg\{0, \sum_{i = 1}^p 2 \cdot a_i \beta_i + \beta_0 \bigg\} - \sum_{i = 1}^pa_i \bigg)^2}_{=: \tau(\beta, \beta_0)} \\
+\sum_{i = 1}^p \theta( e_j^{\top} \beta, \beta_0) +  \bigg( \max\{0, \beta_0\} + 10 p \bigg)^2 . \tag{ReLU} \label{eq:subset-sum}
\end{align*}
Observe that solving (\ref{eq:subset-sum}) is equivalent to training a \ref{eq:One-Node-ReLU} where the data samples are:
\begin{enumerate}
\item $X_1 = [a_1, \dots, a_p]$, $Y_1 = \frac{1}{2}\sum_{i =1}^p a_i$
\item $X_2 = [2 \cdot a_1, \dots, 2 \cdot a_p]$, $Y_1 = \sum_{i =1}^p a_i$
\item $X_{2i + 1} = e_i$, $Y_{2i + 1} = 1$, $X_{2i +2} = -e_i$, $Y_{2i +2} = 1$ for $i \in \{1, \dots, p\}$. 
\item $X_{2p +3} = 0$, $Y_{2p + 3} = 10p$.
\end{enumerate}
Now we verify Theorem~\ref{thm:NP-hard}  by showing that the $\{\pm 1\}-$subset sum problem iff the training error in solving (\ref{eq:subset-sum}) is $p + 100p^2$.

Thus 
\begin{itemize}
	\item Suppose the $\{\pm 1\}-$subset sum problem with non-negative parameters $a_1, \ldots, a_p$ has a feasible solution $x \in \{\pm 1\}^p$ such that $\sum_{i = 1}^p a_i x_i = \frac{1}{2}\sum_{i =1}^p a_i$. Let $\beta = x$ and $\beta_0 = 0$, we have that the objective function value of  (\ref{eq:subset-sum}) is $p + 100p^2$.
	\item Suppose the $\{\pm 1\}-$subset sum problem does not have a feasible solution. Let $\beta, \beta_0$ be the optimal solution to (\ref{eq:subset-sum}).
Then, observe that 
\begin{align*}
g(\beta_0) = \left\{
					\begin{array}{lll}
						2\beta_0^2 - 4 \beta_0 + 2 & \text{ when } \beta_0 \geq 1 - \frac{\sqrt{2}}{2} \\
						1 & \text{ when } \beta_0 \leq 1 - \frac{\sqrt{2}}{2} ~ (\leq \frac{1}{2})
					\end{array}
					\right. .
				\end{align*}
We consider four cases:
\begin{enumerate}
\item $\beta_0 \geq 1 - \frac{\sqrt{2}}{2}$: In this case $\tau(\beta, \beta_0) + p \cdot g(\beta_0) + (\max\{0, \beta_0\} + 10p)^2 \geq 0 + p(2\beta_0^2 - 4 \beta_0 + 2) + (\beta_0 + 10p)^2 = (2p + 1)\beta_0^2 + 16p \beta_0 + 2p + 100p^2 > p + 100p^2$.
\item $0 < \beta_0 \leq 1 - \frac{\sqrt{2}}{2}$: In this case $\tau(\beta, \beta_0) + p \cdot g(\beta_0) + (\max\{0, \beta_0\} + 10p)^2 \geq 0 + p*1 + (\beta_0 + 10p)^2 > p + 100p^2$.
\item $\beta_0 < 0$: Note that $\frac{1}{2}\sum_{i = 1}^p a_i >0$ and therefore $\tau(\beta, \beta_0) =0$ iff $\beta_0 = 0$. In particular in this case $\tau(\beta, \beta_0) > 0$. Therefore, we have $\tau(\beta, \beta_0) + p \cdot g(\beta_0) + (\max\{0, \beta_0\} + 10p)^2 >  0 + p \cdot 1 + 100p^2$.
\item $\beta_ 0 = 0$: In this case, observe that $$\tau(\beta, \beta_0) + \sum_{i =1}^p \theta (e_j ^{\top}\beta, \beta_0) + (\max\{0, \beta_0\} + 10p)^2 \geq  0 + p \cdot 1 + 100p^2.$$ However, for equality to hold  in the above inequality, we must have $\theta (e_j ^{\top}\beta, \beta_0) = 1$ for $j \in [p]$ and $\tau(\beta, \beta_0 ) = 0$, which implies we must have $\beta_j \in \{-1, 1\}$ and $\sum_{i = 1}^p a_i \beta_i = \frac{1}{2} \sum_{i = 1}^p a_i$. However since there is no solution to the $\{\pm 1\}-$subset sum problem, we obtain that $\tau(\beta, \beta_0) + \sum_{i =1}^p \theta (e_j ^{\top}\beta, \beta_0) + (\max\{0, \beta_0\} + 10p)^2 >  0 + p \cdot 1 + 100p^2.$
\end{enumerate}
\end{itemize}

\subsection{Proof of Proposition \ref{prop:LB} \label{section:Prop-LB}}
\begin{proof}
	Re-written $$ \|\max\{\mathbf{0}, X^{\top}\beta + \beta_0 \mathbf{1}\} - Y\|_2^2 = \sum_{i = 1}^m(\max\{0, X_i^{\top} \beta + \beta_0 \} - Y_i)^2 + \phi(\beta, \beta_0).$$ Note that $Y_i > 0$ for all $i \in [m]$, then we have: 
	\begin{align*}
		& (\max\{0, X_i^{\top} \beta + \beta_0 \} - Y_i)^2 \leq (X_i^{\top} \beta + \beta_0 - Y_i)^2 \\
		& (\max\{0, X_i^{\top} \beta + \beta_0 \} - Y_i)^2 \leq \sigma(X_i^{\top} \beta + \beta_0, Y_i)
	\end{align*}
	holds for all $i \in [m]$. Since for any $I \subseteq [m]$
	$$\sum_{i = 1}^m(\max\{0, X_i^{\top} \beta + \beta_0 \} - Y_i)^2 + \phi(\beta, \beta_0) \leq \sum_{i \in I} (X_i^{\top} \beta + \beta_0 - Y_i)^2 + \sum_{i \in [m]\backslash I} \sigma(X_i^{\top} \beta + \beta_0, Y_i) + \phi(\beta, \beta_0), $$ 
	then taking minimum on both side implies $\min_{(\beta, \beta_0) \in \mathbb{R}^p \times \mathbb{R}} \|\max\{\mathbf{0}, X^{\top}\beta + \beta_0 \mathbf{1}\} - Y\|_2^2 \leq z^{\sigma}(I)$.
	
	Moreover, recall $I^{\text{OPT}}$ is the active set corresponding to a global optimal solution $(\beta^{\text{OPT}}, \beta^{\text{OPT}}_0)$ as defined above. We have:
	\begin{align*}
		z^{\sigma} (I^{\text{OPT}}) \leq & ~ \sum_{i \in I^{\text{OPT}}} (\underbrace{ X_i^{\top} \beta^{\text{OPT}} + \beta^{\text{OPT}}_0}_{\geq 0} - Y_i)^2 + \sum_{i \in [m] \backslash I^{\text{OPT}}} \sigma(\underbrace{ X_i^{\top} \beta^{\text{OPT}} + \beta^{\text{OPT}}_0}_{< 0}, Y_i) + \phi(\beta^{\text{OPT}}, \beta^{\text{OPT}}_0) \\
		= & ~ \sum_{i \in I^{\text{OPT}}} (X_i^{\top} \beta^{\text{OPT}} + \beta^{\text{OPT}}_0 - Y_i)^2 + \sum_{i \in [m] \backslash I^{\text{OPT}}} Y_i^2 + \phi(\beta^{\text{OPT}}, \beta^{\text{OPT}}_0) \\
		= & ~ z^{\text{OPT}}. 
	\end{align*}
	Combine with $z^{\text{OPT}} = \|\max\{\mathbf{0}, X^{\top}\beta + \beta_0 \mathbf{1}\} - Y\|_2^2 \leq z^{\sigma}(I^{\text{OPT}})$, we have $z^{\text{OPT}} = z^{\sigma}(I^{\text{OPT}})$.
\end{proof}

\subsection{Proof of Proposition \ref{prop:LB} \label{section:AA}}
\begin{proof}
Recall that $(\beta^{\text{OPT}}, \beta^{\text{OPT}}_0)$ is a global optimal solution, and $z^{\text{OPT}}$ is the global optimal value of \ref{eq:One-Node-ReLU}. Let $I^{\text{OPT}} = \left\{i: X_i^{\top} \beta^{\text{OPT}} + \beta^{\text{OPT}}_0 > 0 \right\} \subseteq [m]$ be the active set corresponds to $(\beta^{\text{OPT}}, \beta^{\text{OPT}}_0)$. Based on the input condition of Algorithm~\ref{algo:gen-approximation}, the response samples $\{Y_i\}$ satisfies:
\begin{align*}
	0 < Y_1 \leq Y_2 \leq \ldots \leq Y_m.
\end{align*}
Given $k$ as a predefined integral parameter, pick $k$ indices $i_1, i_2, \ldots, i_k$ such that $0 \leq i_1 < \ldots < i_k \leq m$, from Algorithm~\ref{algo:gen-approximation}, we have:
\begin{align*}
	[m] \backslash \hat{I} := & \{1, \ldots, i_1\} \cup \left(\bigcup_{\ell = 2}^{k }\{i_{\ell}\} \right) & \text{ be our inactive set}, \\
	\hat{I} := & \left( \bigcup_{\ell = 1}^{k - 1} \left\{ i_{\ell} + 1, \ldots, i_{\ell + 1} - 1 \right\} \right) \cup \{i_k + 1, \ldots, m\} & \text{ be our active set.}
\end{align*}

Suppose $I^{\text{OPT}}$ is of size $|I^{\text{OPT}}| \geq m - k + 1$, let $\{s_{\ell}\}_{\ell = 1}^p$ with $p \leq k - 1$ be the set of increasingly-sorted indices that are not in $I^{\text{OPT}}$. Let $j = p + 1 \leq k$, set $i_1 = 0$, $i_{\ell} = s_{\ell - 1},$ for all $\ell = 2, \ldots, j$. Then we see that Algorithm~\ref{algo:gen-approximation} would discover  the optimal solution and thus solve the \ref{eq:One-Node-ReLU} problem exactly. 

Therefore, henceforth we assume that $|I^{\text{OPT}}| \leq m - k$.

Now pick $i_1, \ldots, i_k$ as the largest increasingly-sorted indices that not in $I^{\text{OPT}}$. Therefore we have: (1) $\hat{I} \subseteq I^{\text{OPT}}$, (2) $\bigcup_{\ell = 1}^{k} \{i_{\ell}\} \subseteq [m] \backslash I^{\text{OPT}}$, and (3) $i_k - 1 \in I^{\text{OPT}} \text{ if }  i_{k - 1} \neq i_k - 1$, these three conditions further implies that $$I^{\text{OPT}} \backslash \hat{I} \subseteq \{1, \ldots, i_1 - 1\}.$$ Since the approximation algorithm examines this solution, we will use this ``solution" to obtain an upper bound on the quality of solution produced by the Algorithm.

Thus the objective value $z^{\sigma}(\hat{I})$ is further upper bounded as follows:
\begin{align*}
	z^{\sigma}(\hat{I}) = & ~ \min_{(\beta, \beta_0) \in \mathbb{R}^p \times \mathbb{R}} \sum_{i \in \hat{I}} (X_i^{\top} \beta + \beta_0 - Y_i)^2 + \sum_{i \in [m] \backslash \hat{I}} \sigma (X_i^{\top} \beta + \beta_0, Y_i) + \phi(\beta, \beta_0) \\
	\leq & ~ \sum_{i \in \hat{I}} ( \underbrace{ X_i^{\top} \beta^{\text{OPT}} + \beta^{\text{OPT}}_0}_{\geq 0} - Y_i)^2 + \sum_{i \in I^{\text{OPT}} \backslash \hat{I}} \sigma ( \underbrace{ X_i^{\top} \beta^{\text{OPT}} + \beta^{\text{OPT}}_0 }_{\geq 0}, Y_i) \\ 
	& ~ + \sum_{i \in [m] \backslash I^{\text{OPT}}} \sigma ( \underbrace{ X_i^{\top} \beta^{\text{OPT}} + \beta^{\text{OPT}}_0}_{ < 0}, Y_i) + \phi(\beta^{\text{OPT}}, \beta^{\text{OPT}}_0) \\
	= & ~ \sum_{i \in \hat{I}} (  X_i^{\top} \beta^{\text{OPT}} + \beta^{\text{OPT}}_0 - Y_i)^2 + \sum_{i \in I^{\text{OPT}} \backslash \hat{I}} \sigma (  X_i^{\top} \beta^{\text{OPT}} + \beta^{\text{OPT}}_0 , Y_i) + \sum_{i \in [m] \backslash I^{\text{OPT}}} Y_i^2 + \phi(\beta^{\text{OPT}}, \beta^{\text{OPT}}_0).
\end{align*}
Split $I^{\text{OPT}} \backslash \hat{I}$ into the following two parts:
\begin{align*}
	\tilde{I}_+ := & ~ \left\{i \in I^{\text{OPT}} \backslash \hat{I}: X_i^{\top} \beta^{\text{OPT}} + \beta_0^{\text{OPT}} > 2 Y_i \right\}, \\
	\tilde{I}_- := & ~ \left\{i \in I^{\text{OPT}} \backslash \hat{I}: 2Y_i \geq X_i^{\top} \beta^{\text{OPT}} + \beta_0^{\text{OPT}} \geq 0 \right\},
\end{align*}
the second term of above equals to:
\begin{align*}
	\sum_{i \in I^{\text{OPT}} \backslash \hat{I}} \sigma (  X_i^{\top} \beta^{\text{OPT}} + \beta^{\text{OPT}}_0 , Y_i) = \sum_{i \in \tilde{I}_+} (X_i^{\top} \beta^{\text{OPT}} + \beta_0^{\text{OPT}} - Y_i)^2 +  \sum_{i \in \tilde{I}_-} Y_i^2.
\end{align*} 
Therefore, 
\begin{align*}
	z^{\sigma}(\hat{I}) \leq \sum_{i \in \hat{I} \cup \tilde{I}_+} (X_i^{\top} \beta^{\text{OPT}} + \beta_0^{\text{OPT}} - Y_i)^2 + \sum_{i \in \tilde{I}_- \cup \left( [m] \backslash I^{\text{OPT}} \right)} Y_i^2 + \phi(\beta^{\text{OPT}}, \beta_0^{\text{OPT}}). \tag{UB} \label{eq:UB} 
\end{align*}
Since $I^{\text{OPT}} = \hat{I} \cup \tilde{I}_+ \cup \tilde{I}_-$, then the global optimal value of \ref{eq:One-Node-ReLU} can be represented as
\begin{align*}
	z^{\text{OPT}} = & ~ \sum_{i \in I^{\text{OPT}}} (X_i^{\top} \beta^{\text{OPT}} + \beta_0^{\text{OPT}} - Y_i)^2 + \sum_{i \in [m] \backslash I^{\text{OPT}}} Y_i^2 + \phi(\beta^{\text{OPT}}, \beta_0^{\text{OPT}}) \\
	= & ~ \sum_{i \in \hat{I} \cup \tilde{I}_+} (X_i^{\top} \beta^{\text{OPT}} + \beta_0^{\text{OPT}} - Y_i)^2 + \sum_{i \in \tilde{I}_-} (X_i^{\top} \beta^{\text{OPT}} + \beta_0^{\text{OPT}} - Y_i)^2 + \sum_{i \in [m] \backslash I^{\text{OPT}}} Y_i^2 + \phi(\beta^{\text{OPT}}, \beta_0^{\text{OPT}}). 
\end{align*} 
Note $\{i_1, \ldots, i_k\}$ is a subset of $[m] \backslash I^{\text{OPT}}$ based on our choice  $i_1, \ldots, i_k$, then the term $D$ satisfies:
\begin{eqnarray}
	D := \sum_{i \in \hat{I} \cup \tilde{I}_+} (X_i^{\top} \beta^{\text{OPT}} + \beta_0^{\text{OPT}} - Y_i)^2 + \sum_{i \in [m] \backslash I^{\text{OPT}}} Y_i^2 + \phi(\beta^{\text{OPT}}, \beta_0^{\text{OPT}}) \geq \sum_{j = 1}^k Y_{i_j}^2. \label{eq:Dineq}
\end{eqnarray}
Since \ref{eq:UB} and $z^{\text{OPT}}$ can be represented as:
\begin{align*}
	(\text{UB}) := & ~ D + \sum_{i \in \tilde{I}_-} Y_i^2 \\
	z^{\text{OPT}} := & ~ D + \sum_{i \in \tilde{I}_-} (X_i^{\top} \beta^{\text{OPT}} + \beta_0^{\text{OPT}} - Y_i)^2
\end{align*}
then the approximation ratio $\rho$ guaranteed by Algorithm~\ref{algo:gen-approximation} is upper bounded as follows:
\begin{align*}
	\rho := \frac{z^{\sigma}(\hat{I})}{z^{\text{OPT}}} \leq \frac{(UB)}{z^{\text{OPT}}} = \frac{D + \sum_{i \in \tilde{I}_-} Y_i^2}{D + \sum_{i \in \tilde{I}_-} (X_i^{\top} \beta^{\text{OPT}} + \beta_0^{\text{OPT}} - Y_i)^2} \leq \frac{D + \sum_{i \in \tilde{I}_-} Y_i^2}{D} \leq \frac{n}{k}
\end{align*}
where the final inequality holds because of the following: with $\{Y_i\}_{i = 1}^m$ increasingly-sorted, the term $\sum_{i \in \tilde{I}_-} Y_i^2$ can be upper bounded by
	\begin{align*}
		\sum_{i \in \tilde{I}_-} Y_i^2 \leq & ~ |\tilde{I}_-| \cdot Y_{i_1}^2 & \text{ by } \tilde{I}_- \subseteq I^{\text{OPT}} \backslash \hat{I} \subseteq \{1, \ldots, i_1 - 1\}, \\
		\leq & ~ \frac{|\tilde{I}_-|}{k} \cdot \sum_{j = 1}^k Y_{i_j}^2 & \text{ by } Y_{i_j} \geq Y_{i_1} \text{ for all } j = 1, \ldots, k, \\
		\leq & ~ \frac{|\tilde{I}_-|}{k} \cdot D & \text{ by previous inequality of } D, \\
		\leq & ~ \frac{n - k}{k} \cdot D & \text{ by } \tilde{I}_- \subseteq I^{\text{OPT}} \text{ and } |I^{\text{OPT}}| \leq m - k \leq n - k, 
	\end{align*}
	then replacing $\sum_{i \in \tilde{I}_-} Y_i^2$ by $\frac{n - k}{k} \cdot D$ gives the final approximation ratio. 
\end{proof}

\subsection{Proof of Theorem~\ref{thm:asy-bound}}\label{sec:stat}

We first verify Proposition~\ref{prop:asy-sorting-obj} and Proposition~\ref{prop:asy-obj}.  Proposition~\ref{prop:asy-sorting-obj} is a consequence of the following result:
\begin{theorem}[Mickey, 1963] \label{thm:mickey}
	 Let $g$ be a function on $\mathcal{X} \times \Theta$ where $\mathcal{X}$ is a Euclidean space and $\Theta$ is a compact set of Euclidean space. Let $g(x, \theta)$ be a continuous function of $\theta$ for each $x \in \mathcal{X}$ and a measurable function of $x$ for each $\theta$. Assume assume that $|g(x, \theta)| \leq h(x)$ for all $x \in \mathcal{X}$ and $\theta \in \Theta$, where $h$ is integrable with respect to a probability distribution function $F$ on $\mathcal{X}$. If $x_1, x_2, \ldots$ is a random sample from $F$ then for almost every sequence $\{x_t\}$ 
	\begin{align*}
		\frac{1}{n} \sum_{t = 1}^n g(x_t, \theta) \rightarrow \int g(x, \theta) d F(x)
	\end{align*}
	uniformly for all $\theta \in \Theta$.  
\end{theorem}

\begin{proof}\textit{of Proposition~\ref{prop:asy-sorting-obj}} 
	Let $\psi_y(X^{\top}\beta + \beta_0, Y)$ be defined as in Proposition~\ref{prop:asy-sorting-obj}. Let $\mathcal{X} = \mathbb{R}^p \times \mathbb{R}$ be a Euclidean space, and let $\Theta$ be the same convex compact set in Assumption~\ref{assumption}. We have $\psi_y(X^{\top}\beta + \beta_0 , Y)$ is a continuous function of $(\beta, \beta_0)$ for each $(X, Y) \in \mathcal{X}$ and a measurable function of $(X, Y)$ for each $(\beta, \beta_0) \in \Theta$. Moreover, since $\Theta$ is a convex compact set, then there exists a constant $d_{\Theta} > 0$ such that $|\theta_i| \leq d_{\Theta}$ for all $i = 0, 1, \ldots, p$. Define function $h(X, Y)$ as
	\begin{align*}
		h(X, Y) = 2 \left( \sum_{i = 1}^p |[X]_i| \cdot d_{\Theta} + d_{\Theta} \right)^2 + 2 Y^2 
	\end{align*} 
	where $[X]_i$ denotes the $i^{\text{th}}$ component of $X$ for $i = 1, \ldots, p$. Thus we have $h(X, Y) \geq |\psi_y (X^{\top}\beta + \beta_0, Y)|$ holds for all $(X, Y) \in \mathcal{X}$ and $(\beta, \beta_0) \in \Theta$, where $h(X, Y)$ is integrable with respect to a probability distribution $\mathcal{N} \times \mathcal{D}$ on $\mathcal{X}$. Since all the conditions in  Theorem~\ref{thm:mickey} holds, Proposition~\ref{prop:asy-sorting-obj} holds.
\end{proof}

Proposition~\ref{prop:asy-obj} is a consequence of the following result:

\begin{theorem}[Jennrich, 1969] \label{thm:jennrich}
	 Under the statistical model: $y_t = f(x_t, \theta_0) + \epsilon_t$ for all $t = 1, \ldots, n$ when $x_t$ is $i^{\text{th}}$ ``fixed'' input vector and $\{\epsilon_t\}$ are i.i.d. distributed errors with zero mean and same finite unknown variance. Any vector $\hat{\theta}_n \in \Theta$ which minimizes the residual sum of squares 
	\begin{align*}
		S_n(\theta) := \frac{1}{n} \sum_{t = 1}^n (f(x_t, \theta) - y_t)^2 
	\end{align*}
	is said to be strongly consistent of $\theta_0$ (i.e., $\hat{\theta}_n \rightarrow \theta_0$ almost surely as $n \rightarrow \infty$) under the following condition:   $D_n(\theta, \theta')$ convergence uniformly to a continuous function $D(\theta, \theta')$ and $D(\theta, \theta_0) = 0$ if and only if $\theta = \theta_0$ where 
	\begin{align*}
		D_n(\theta, \theta') = \frac{1}{n} \sum_{t = 1}^n (f(x_t, \theta) - f(x_t, \theta'))^2. 
	\end{align*}
\end{theorem}

\begin{proof}\textit{of Proposition~\ref{prop:asy-obj}} 
Based on Theorem~\ref{thm:mickey}, with the similar proof of Proposition~\ref{prop:asy-sorting-obj}, we have:
	\begin{align*}
		& ~ \underbrace{ \frac{1}{n} \sum_{i = 1}^n \left( \max\{0, X_i^{\top} \beta + \beta_0 \} - \max\{0, X_i^{\top} \beta^{\ast} + \beta_0^{\ast} \} \right)^2}_{ =: D_n((\beta, \beta_0), (\beta^{\ast}, \beta^{\ast}_0) ) } \\
		\rightarrow & ~ \underbrace{ \mathbb{E}_{X \sim \mathcal{N}, \epsilon \sim \mathcal{D}} \left[ \left( \max\{0, X^{\top} \beta + \beta_0 \} - \max\{0, X^{\top} \beta^{\ast} + \beta_0^{\ast} \} \right)^2  \right] }_{=:  D((\beta, \beta_0), (\beta^{\ast}, \beta^{\ast}_0) )}
	\end{align*}
	uniformly for almost every sequence $\{X_i, Y_i\}$. 
Moreover, a direct consequence of the second property of distribution $\mathcal{N}$ ({Unique Optimal Property}) implies that $D((\beta, \beta_0), (\beta^{\ast}, \beta^{\ast}_0)) = 0$ if and only if $(\beta, \beta_0) = (\beta^{\ast}, \beta^{\ast}_0)$. Thus, since all conditions of Theorem~\ref{thm:jennrich} hold, Proposition~\ref{prop:asy-obj} holds.
\end{proof}

\begin{proof}\textit{of Theorem~\ref{thm:asy-bound}} The optimal value of the asymptotic objective function from sorting algorithm can be upper bounded by replacing optimal solution with the true parameter $\beta^{\ast}$ as follows:
\begin{align*}
	\min_{\beta \in \Theta} \mathbb{E}_{X \sim \mathcal{N}, \epsilon \sim \mathcal{D}}[ \psi_y (X^{\top} \beta + \beta_0, Y) ] \leq \mathbb{E}_{X \sim \mathcal{N}, \epsilon \sim \mathcal{D}}[ \psi_y (X^{\top} \beta^{\ast} + \beta^{\ast}, Y)], 
\end{align*}
where $\mathbb{E}_{X \sim \mathcal{N}, \epsilon \sim \mathcal{D}}[ \psi_y (X^{\top} \beta^{\ast} + \beta^{\ast}_0, Y) ]$ can be split into the sum from (\ref{eq:T-1}) to (\ref{eq:T-7}):
\begin{align*}
		& \mathbb{E}_{X \sim \mathcal{N}, \epsilon \sim \mathcal{D}}\big[ \psi_y (X^{\top} \beta^{\ast} + \beta^{\ast}_0, Y) \big] & \\
		= & ~ \mathbb{E}\big[ Y^2 ~|~ 0 < Y \leq y, 0 < X^{\top} \beta^{\ast} + \beta^{\ast}_0 \leq 2 Y \big] \mathbb{P}(0 < Y \leq y, 0 < X^{\top} \beta^{\ast} + \beta^{\ast}_0 \leq 2 Y) & \tag{$T_1$} \label{eq:T-1}\\
		& + \mathbb{E}\big[ Y^2 ~|~ 0 < Y \leq y, X^{\top} \beta^{\ast} + \beta^{\ast}_0 \leq 0 \big] \mathbb{P}(0 < Y \leq y, X^{\top} \beta^{\ast} + \beta^{\ast}_0 \leq 0) & \tag{$T_2$} \label{eq:T-2} \\
		& + \mathbb{E}\big[ (X^{\top} \beta^{\ast} + \beta^{\ast}_0 - Y)^2 ~|~ 0 < Y \leq y, 2Y < X^{\top} \beta^{\ast} + \beta^{\ast}_0 \big] \mathbb{P}(0 < Y \leq y, 2Y < X^{\top} \beta^{\ast} + \beta^{\ast}_0 )  & \tag{$T_3$} \label{eq:T-3} \\
		& + \mathbb{E}\big[ (X^{\top} \beta^{\ast} + \beta^{\ast}_0 - Y)^2 ~ | ~ y < Y, 0 < X^{\top} \beta^{\ast} + \beta^{\ast}_0 \big] \mathbb{P} (y < Y, 0 < X^{\top} \beta^{\ast} + \beta^{\ast}_0 ) & \tag{$T_4$} \label{eq:T-4}\\
		& + \mathbb{E}\big[ (X^{\top} \beta^{\ast} + \beta^{\ast}_0 - Y)^2 ~ | ~ y < Y, X^{\top} \beta^{\ast} + \beta^{\ast}_0 \leq 0 \big] \mathbb{P} (y < Y, X^{\top} \beta^{\ast} + \beta^{\ast}_0 \leq 0) & \tag{$T_5$} \label{eq:T-5} \\
		& + \mathbb{E}\big[ \epsilon^2 ~ | ~ Y \leq 0, X^{\top} \beta^{\ast} + \beta^{\ast}_0 \leq 0 \big] \mathbb{P}(Y \leq 0, X^{\top} \beta^{\ast} + \beta^{\ast}_0 \leq 0) & \tag{$T_6$} \label{eq:T-6}\\
		& + \mathbb{E}\big[ \epsilon^2 ~ | ~ Y \leq 0, 0 < X^{\top} \beta^{\ast} + \beta^{\ast}_0 \big] \mathbb{P}(Y \leq 0, 0 < X^{\top} \beta^{\ast} + \beta^{\ast}_0). & \tag{$T_7$} \label{eq:T-7}
\end{align*}
Since term (\ref{eq:T-1}) - (\ref{eq:T-7}) can be reformulated as follows:
\begin{align*}
		(T_1) & = \mathbb{E}[ Y^2 ~|~ 0 < Y \leq y, 0 < X^{\top} \beta^{\ast} + \beta^{\ast}_0  \leq 2 Y ] \cdot \mathbb{P}(0 < Y \leq y, 0 < X^{\top} \beta^{\ast} + \beta^{\ast}_0  \leq 2 Y), \\
		(T_2) & = \mathbb{E}[ \epsilon^2 ~|~ 0 < Y \leq y, X^{\top} \beta^{\ast} + \beta^{\ast}_0  \leq 0 ] \cdot \mathbb{P}(0 < Y \leq y, X^{\top} \beta^{\ast} + \beta^{\ast}_0  \leq 0), \\
		(T_3) & = \mathbb{E}[ \epsilon^2 ~|~ 0 < Y \leq y, 2Y < X^{\top} \beta^{\ast} + \beta^{\ast}_0 ] \cdot \mathbb{P}(0 < Y \leq y, 2Y < X^{\top} \beta^{\ast} + \beta^{\ast}_0 ) ,\\
		(T_4) & = \mathbb{E}[ \epsilon^2 ~ | ~ y < Y, 0 < X^{\top} \beta^{\ast} + \beta^{\ast}_0  ] \cdot \mathbb{P} (y < Y, 0 < X^{\top} \beta^{\ast} + \beta^{\ast}_0 ), \\
		(T_5) & = \mathbb{E}\big[ (X^{\top} \beta^{\ast} + \beta^{\ast}_0  - \epsilon)^2 ~ | ~ y < \epsilon, X^{\top} \beta^{\ast} + \beta^{\ast}_0  \leq 0 \big] \cdot \mathbb{P} (y < \epsilon, X^{\top} \beta^{\ast} + \beta^{\ast}_0  \leq 0), \\
		(T_6) & = \mathbb{E}\big[ \epsilon^2 ~ | ~ Y \leq 0, X^{\top} \beta^{\ast} + \beta^{\ast}_0  \leq 0 \big] \mathbb{P}(Y \leq 0, X^{\top} \beta^{\ast} + \beta^{\ast}_0  \leq 0), \\
		(T_7) & = \mathbb{E}\big[ \epsilon^2 ~ | ~ Y \leq 0, 0 < X^{\top} \beta^{\ast} + \beta^{\ast}_0  \big] \mathbb{P}(Y \leq 0, 0 < X^{\top} \beta^{\ast} + \beta^{\ast}_0 ),
\end{align*}
note that $(T_1)$ is upper bounded by $y^2$, $(T_2) + (T_3) + (T_4) + (T_6) + (T_7) \leq \text{Var}(\epsilon) = \gamma^2$, and by Lemma~\ref{lemma:T-5} (proved below) and setting $y = 0$, 
\begin{align*}
	z^{\text{asy}} \leq \mathbb{E}_{X \sim \mathcal{N}, \epsilon \sim \mathcal{D}}[ \psi_0 (X^{\top} \beta^{\ast} + \beta^{\ast}_0, Y) ] \leq \gamma^2 + \frac{\gamma^2}{2} + \frac{2 + 2\Delta^2}{\sqrt{2 \pi}} \gamma.  
\end{align*}
To lower bound $z^{\text{asy}}$, note that $\psi_y (X^{\top} \beta + \beta_0, Y) \geq \left( \max\{0 ,X^{\top} \beta + \beta_0\} - Y\right)^2$ holds for any $(\beta, \beta_0) \in \Theta$ and any $(X, Y) \in \mathcal{X}$, and by Proposition~\ref{prop:asy-obj}, the optimal value of asymptotic version of \ref{eq:One-Node-ReLU} problem is $\gamma^2$, thus $z^{\ast}$ is lower bounded by $\gamma^2$. Combine lower and upper bounds together, we have
\begin{align*}
	\gamma^2 \leq z^{\text{asy}} \leq \frac{3 \gamma^2}{2} + \frac{2 + 2\Delta^2}{\sqrt{2 \pi}} \gamma. 
\end{align*}
\end{proof}

\begin{lemma} \label{lemma:T-5}
	Assume the underlying statistical model~\ref{assumption} holds, we have
	\begin{align*}
		(T_5) \leq \frac{\gamma^2}{2} + \frac{2 + 2\Delta^2}{\sqrt{2 \pi}} \gamma.
	\end{align*}  
\end{lemma}
\begin{proof} 	Assume the underlying statistical model~\ref{assumption}, we have $X^{\top} \beta^{\ast} + \beta^{\ast}_0$ satisfies $\mathbb{E}[X^{\top} \beta^{\ast} + \beta^{\ast}_0] = \beta^{\ast}_0, \text{Var}(X^{\top} \beta^{\ast} + \beta^{\ast}_0) =(\beta^{\ast})^{\top} \Sigma \beta^{\ast} = \Delta^2$, thus 
	\begin{align*}
		& (T_5) \\
		\leq & ~ \mathbb{E}\big[ (X^{\top} \beta^{\ast} + \beta^{\ast}_0 - \epsilon)^2 ~ | ~ y < \epsilon, X^{\top} \beta^{\ast} + \beta^{\ast}_0 \leq 0 \big] \cdot \mathbb{P} (y < \epsilon, X^{\top} \beta^{\ast} + \beta^{\ast}_0 \leq 0) \\
		= & ~ \int_{\substack{v \in \mathbb{R} \\ u \in \mathbb{R}}} (u - v)^2 f(\epsilon = v, X^{\top} \beta^{\ast} + \beta^{\ast}_0 = u ~ | y < \epsilon, X^{\top} \beta^{\ast} + \beta^{\ast}_0 \leq 0) du dv \cdot \mathbb{P} (y < \epsilon, X^{\top} \beta^{\ast} + \beta^{\ast}_0 \leq 0) \tag{$\ast$} \label{eq:T_5-ast}
	\end{align*}
	where $f(\epsilon = v, X^{\top} \beta^{\ast} + \beta^{\ast}_0 = u ~ | y < \epsilon, X^{\top} \beta^{\ast} + \beta^{\ast}_0 \leq 0)$ is the conditional joint density function of variables $\epsilon, X^{\top} \beta^{\ast} + \beta^{\ast}_0$. Then
	\begin{align*}
		(\ast) = & ~ \int_{\substack{v > y \\ u \leq 0}} (u^2 - 2 uv + v^2) f(\epsilon = v) f(X^{\top} \beta^{\ast} + \beta^{\ast}_0 = u) du dv \\
		= & ~ \int_{v > y} f(\epsilon = v) dv \cdot \int_{u \leq 0} u^2 f(X^{\top} \beta^{\ast} + \beta^{\ast}_0 = u) du \\
		& ~ - 2 \int_{v > y} v f(\epsilon = v) dv \cdot \int_{u \leq 0} u f(X^{\top} \beta^{\ast} + \beta^{\ast}_0 = u) du \\
		& ~ + \int_{v > y} v^2 f(\epsilon = v) dv \cdot \int_{u \leq 0} f(X^{\top} \beta^{\ast} + \beta^{\ast}_0 = u) du
	\end{align*}
	where 
	\begin{align*}
		\int_{u \leq 0} u^2 f(X^{\top} \beta^{\ast} + \beta^{\ast}_0 = u) du & \leq \Delta^2, ~ \\
		- 1 - \Delta^2 \leq \int_{u \leq 0} u f(X^{\top} \beta^{\ast} + \beta^{\ast}_0 = u) du & \leq 1 + \Delta^2 , ~ \\
		\int_{u \leq 0} f(X^{\top} \beta^{\ast} + \beta^{\ast}_0 = u) du & \leq 1. 
	\end{align*}
	Suppose the noise $\epsilon$ follows Gaussian distribution $N(0, \gamma^2)$, then 
	\begin{align*}
		(\ast) \leq & ~ \int_{v > y} f(\epsilon = v) dv \cdot \Delta^2 + 2 \int_{v > y} v f(\epsilon = v) dv \cdot \left( 1 + \Delta^2 \right) + \int_{v > y} v^2 f(\epsilon = v) dv \cdot 1 \\
		= & ~ \frac{1}{2} \text{erf}\left( \frac{y}{\sqrt{2} \gamma} \right) \cdot \Delta^2 + \frac{2}{\sqrt{2 \pi}} e^{- \frac{y^2}{2 \gamma^2}} \gamma \cdot \left( 1 + \Delta^2 \right) + \left( \frac{1}{\sqrt{2 \pi}} y \gamma e^{- \frac{y^2}{2 \gamma^2}} + \frac{\gamma^2}{2} \text{erfc} \left( \frac{y}{\sqrt{2} \gamma} \right) \right) \cdot 1 \\
		\leq & ~ \frac{\Delta^2 y}{\sqrt{2 \pi} \gamma} + \frac{2 + 2\Delta^2}{\sqrt{2 \pi}} \gamma e^{- \frac{y^2}{2 \gamma^2}} + \frac{1}{\sqrt{2 \pi}} y \gamma e^{- \frac{y^2}{2 \gamma^2}} + \frac{\gamma^2}{2} e^{- \frac{y^2}{2 \gamma^2}} 
	\end{align*} 
	where the final inequality holds since
	\begin{align*}
		& \text{erf}(z) := \frac{2}{\sqrt{\pi}} \int_0^z e^{- t^2} dt \leq \frac{2 z}{\sqrt{\pi}}, & \text{erfc}(z) := 1 - \text{erf}(z) \leq e^{- z^2}.
	\end{align*}
	Since the above inequality holds for any $y \geq 0$, then set $y = 0$, we have
	\begin{align*}
		(\ast) \leq \frac{\gamma^2}{2} + \frac{2 + 2\Delta^2}{\sqrt{2 \pi}} \gamma. 
	\end{align*} 
\end{proof}

\newpage
\bibliographystyle{plain}
\bibliography{refs}

\newpage
\appendix
\section{Appendix: Sorting Method}\label{section:sorting-method}
The sorting method that used in Section~\ref{section:nr-sorting} is a special case of Algorithm~\ref{algo:gen-approximation} which follows Algorithm~\ref{algo:sorting-method}.
\begin{algorithm}
\caption{Sorting Method} \label{algo:sorting-method}
\begin{algorithmic}[1]
\BState \emph{Input}: Set of sample points $\{(X_i, Y_i)\}_{i = 1}^n \in \mathbb{R}^p \times \mathbb{R}$, integer $1 \leq N \leq n$.
\BState \emph{Output}: A feasible solution $\hat{\beta}$.
\Function{Sorting Method}{$\{(X_i, Y_i)\}_{i = 1}^n$} \label{function:sorting-method}
\State Without loss of generality, sort $\{Y_i\}_{i = 1}^n$ as $Y_1 \leq Y_2 \leq \ldots \leq Y_n$.
\For{$t = 0, 1, \ldots, N$}
\State Set $\mathcal{I}^t \gets \{\lfloor \frac{t}{N} n \rfloor + 1, \ldots, N\} \subseteq \{1, \ldots, N\}$ for $t = 0, 1, \ldots, N-1$, and $\mathcal{I}^N \gets \emptyset$. 
\State Set $\beta^{t} \gets \arg \min_{\beta \in \mathbb{R}^p} f_{\mathcal{I}^t}^{\sigma}(\beta)$.
\State Compute the objective value of the \ref{eq:One-Node-ReLU} with $\beta^t$ as 
	\begin{align*}
		\text{OPT}^t \gets \sum_{i = 1}^n (\max\{0, X_i^{\top} \beta^t\} - Y_i)^2.
	\end{align*}
\EndFor
\State \Return $\hat{\beta}$ where $\hat{\beta}$ is a feasible solution with the minimum $\text{OPT}^t$.
\EndFunction
\end{algorithmic}
\end{algorithm}
Based on the result from Paper \cite{dey2018relu}, the above sorting method is a special case of Algorithm~\ref{algo:gen-approximation} with parameter $k = 1$ and subset
\begin{align*}
	\{i\} = \left\{
	\begin{array}{lll}
		\{\lfloor \frac{t}{N} n \rfloor \} & \text{ if } t = 1, \ldots, N \\
		\emptyset & \text{ if } t = 0
 	\end{array}
	\right.
\end{align*}
which implies the term corresponds to $\lfloor \frac{t}{N} n \rfloor^{\text{th}}$ index in the objective function of \ref{eq:One-Node-ReLU} is not in the quadratic part (i.e., not active) but in the $\sigma$ function part.

\section{Appendix: Iterative Method}\label{section:iterative-method}
Given any feasible solution $\beta$ of the \ref{eq:One-Node-ReLU} problem, let the \textit{iterative set} $\mathcal{I}(\beta) \gets \{i \in [n]: X_i^{\top} \beta > 0\}$ be the set of indices that in the linearity part of ReLU function $\max\{0, X_i^{\top} \beta\}$. The iterative method that used in Section~\ref{section:nr-sorting+Iter} follows Algorithm~\ref{algo:iter-method}.
\begin{algorithm}
\caption{Iterative Heuristic} \label{algo:iter-method}
\begin{algorithmic}[1]
\BState \emph{Input}: Set of sample points $\{(X_i, Y_i)\}_{i = 1}^n \in \mathbb{R}^p \times \mathbb{R}$, initial feasible solution $\beta^0 \in \mathbb{R}^p$, maximum number of iterations $T$.  
\BState \emph{Output}: A feasible solution $\hat{\beta}$.
\Function{Iterative Heuristic}{$\{(X_i, Y_i)\}_{i = 1}^n, \beta^0, T$} \label{function:Iterative}
\State Initialize $t = 0$.
\State Set the past iterative set set $\mathcal{I}^{-1} \gets \emptyset$. 
\State Set the initial iterative set set $\mathcal{I}^0 \gets \mathcal{I}(\beta^0) := \{i \in [n] : X_i^{\top} \beta^0 > 0 \}$
\State Denote the iterative set in $t^{\text{th}}$ iteration be $\mathcal{I}^t$.
\While{$t < T$ and $\mathcal{I}^{t} \neq \mathcal{I}^{t - 1}$}
\State Set $\beta^{t + 1} \gets \arg\min_{\beta \in \mathbb{R}^p} f_{\mathcal{I}^t}^{\sigma} (\beta)$.
\State Set $\mathcal{I}^{t + 1} \gets \mathcal{I}(\beta^{t + 1}).$
\State Set $t \gets t + 1$.
\EndWhile
\State \Return $\hat{\beta}$ where $\hat{\beta}$ is the final feasible solution obtained in the loop. 
\EndFunction
\end{algorithmic}
\end{algorithm}
Based on the result from Paper \cite{dey2018relu}, the iterative heuristic method guarantees the decreasing of objective value in each iteration, i.e., $\min_{\beta \in \mathbb{R}^p} f_{\mathcal{I}^t}^{\sigma} (\beta) \leq \min_{\beta \in \mathbb{R}^p} f_{\mathcal{I}^{t + 1}}^{\sigma} (\beta)$ for $t = 0, 1, 2, \ldots$. Moreover, the algorithm~\ref{algo:iter-method} terminates in finite number of iterations. 

\section{Appendix: Gradient Descent Method} \label{section:GD}
The gradient descent method that used in Section \ref{section:nr-GD} and \ref{section:nr-Sorting+GD} is Algorithm~\ref{algo:GD-method}. 
\begin{algorithm}
\caption{Gradient Descent Method} \label{algo:GD-method}
\begin{algorithmic}[1]
\BState \emph{Input}: Set of sample points $\{(X_i, Y_i)\}_{i = 1}^n \in \mathbb{R}^p \times \mathbb{R}$, initial feasible solution $\beta^0 \in \mathbb{R}^p$, maximum number of iterations $T$, termination criteria parameter $\epsilon > 0$, initial stepsize $\eta_0 > 0$, stepsize parameter $\gamma > 0$, back track parameter $\alpha \in (0,1)$.  
\BState \emph{Output}: A feasible solution $\hat{\beta}$.
\Function{Gradient Descent Method}{$\{(X_i, Y_i)\}_{i = 1}^n, \beta^0, T, \epsilon, \eta_0, \gamma, \alpha$} \label{function:GD-method}
\State Initialize $t = 0$, $L^{- 1} \gets + \infty$, $L^0 \gets \sum_{i = 1}^n (\max\{0, X_i^{\top} \beta^0\} - Y_i)^2$.
\State Set $\beta^t$ as the solution obtained in $t^{\text{th}}$ iteration.
\State Set $\eta_t$ as the stepsize used in $t^{\text{th}}$ iteration.
\State Set $L(\beta) \gets \sum_{i = 1}^n (\max\{0, X_i^{\top} \beta\} - Y_i)^2$.
\While{$t < T$ and $L(\beta^{t - 1}) - L(\beta^t) > \epsilon$}
\State Set temporary solution $\bar{\beta}$ be $\bar{\beta} \gets \beta^t - \eta_t \cdot \frac{1}{n} \nabla_{\beta} L(\beta^t)$.
\While{$L(\bar{\beta}) \geq L(\beta^t)$}
\State Update $\eta_t \gets \alpha \cdot \eta_t$
\State Update $\bar{\beta} \gets \beta^t - \eta_t \cdot \frac{1}{n} \nabla_{\beta} L(\beta^t)$.
\EndWhile
\State Set $\beta^{t + 1} \gets \bar{\beta}$. 
\State Set $\eta_t \gets \frac{\eta_0}{1 + \gamma t}$.
\State Set $t \gets t + 1$.
\EndWhile
\State \Return $\hat{\beta}$ where $\hat{\beta}$ is the final feasible solution obtained in the loop. 
\EndFunction
\end{algorithmic}
\end{algorithm}
Note that the outer while-loop follows a standard gradient descent method with gradient $\frac{1}{n} \nabla_{\beta} L(\beta^t)$ and stepsize $\eta_t$, and the inner while-loop uses a back search method that guarantee the decreasing of objective value in each outer iteration. 

\section{Appendix: Stochastic Gradient Descent Method}
The stochastic gradient descent method used in this paper is presented below. This algorithm follows a similar updating rule of the gradient descent method (Algorithm~\ref{algo:GD-method}), the only difference is that in each iteration, the stochastic gradient descent method uniformly picks a mini-batch of size $m$ from the given set of samples $\{X_i\}_{i = 1}^n$. 
\begin{algorithm}
\caption{Stochastic Gradient Descent Method} \label{algo:SGD-method}
\begin{algorithmic}[1]
\BState \emph{Input}: Set of sample points $\{(X_i, Y_i)\}_{i = 1}^n \in \mathbb{R}^p \times \mathbb{R}$, initial feasible solution $\beta^0 \in \mathbb{R}^p$, maximum number of iterations $T$, termination criteria parameter $\epsilon > 0$, initial stepsize $\eta_0 > 0$, stepsize parameter $\gamma > 0$, back track parameter $\alpha \in (0,1)$, size of mini-batch $1 \leq m \leq n$.  
\BState \emph{Output}: A feasible solution $\hat{\beta}$.
\Function{Stochastic Gradient Descent Method}{$\{(X_i, Y_i)\}_{i = 1}^n, \beta^0, T, \epsilon, \eta_0, \gamma, \alpha, m$} \label{function:SGD-method}
\State Initialize $t = 0$, $S^0$ uniformly picked from $\{1, \ldots, n\}$ with size $m$, $L^{-1} \gets + \infty$, $L^0 \gets \sum_{i \in S^0} (\max\{0, X_i^{\top} \beta^0\} - Y_i)^2$.
\State Set $\beta^t$ as the solution obtained in $t^{\text{th}}$ iteration.
\State Set $\eta_t$ as the stepsize used in $t^{\text{th}}$ iteration.
\State Set $S^t$ as the mini-batch of size $m$ in $t^{\text{th}}$ iteration.
\State Set $L(S, \beta) \gets \sum_{i \in S} (\max\{0, X_i^{\top} \beta\} - Y_i)^2$.
\While{$t < T$ and $L(S^{t - 1}, \beta^{t - 1}) - L(S^t, \beta^t) > \epsilon$}
\State Set $S^{t + 1}$ uniformly from $\{1, \ldots, n\}$ with size $m$. 
\State Set temporary solution $\bar{\beta}$ be $\bar{\beta} \gets \beta^t - \eta_t \cdot \frac{1}{m} \nabla_{\beta} L(S^{t + 1}, \beta^t)$.
\While{$L(\bar{\beta}) \geq L(\beta^t)$}
\State Update $\eta_t \gets \alpha \cdot \eta_t$
\State Update $\bar{\beta} \gets \beta^t - \eta_t \cdot \frac{1}{m} \nabla_{\beta} L(S^{t + 1}, \beta^t)$.
\EndWhile
\State Set $\beta^{t + 1} \gets \bar{\beta}$. 
\State Set $\eta_t \gets \frac{\eta_0}{1 + \gamma t}$.
\State Set $t \gets t + 1$.
\EndWhile
\State \Return $\hat{\beta}$ where $\hat{\beta}$ is the final feasible solution obtained in the loop. 
\EndFunction
\end{algorithmic}
\end{algorithm}

\section{Appendix: Main Computational Results, Continued} \label{section:app-computational-results}
Figure~[\ref{fig:Figure_Compare_10_200}, \ref{fig:Figure_Compare_20_400}] are the continued numerical results that presented in section~\ref{section:NR-notations}. 
\begin{figure*}
    \centering
    \begin{subfigure}[b]{0.24\textwidth}
        \centering
        \includegraphics[width=\textwidth]{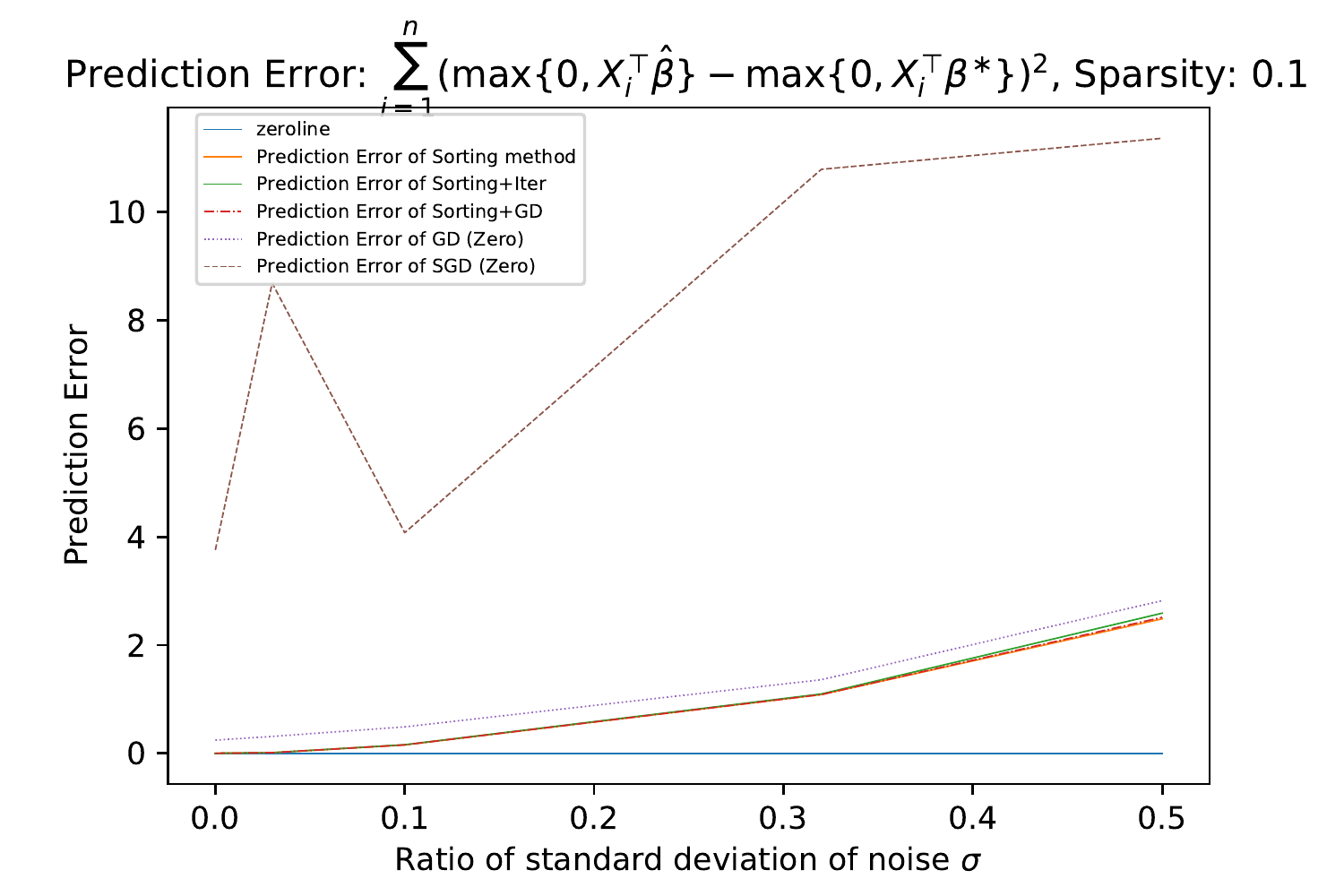}
        \caption[Network1]%
        {{\small Prediction Error}}    
        \label{fig:Figure_Compare_PE_10_200_10}
    \end{subfigure}
    \hfill
    \begin{subfigure}[b]{0.24\textwidth}   
        \centering 
        \includegraphics[width=\textwidth]{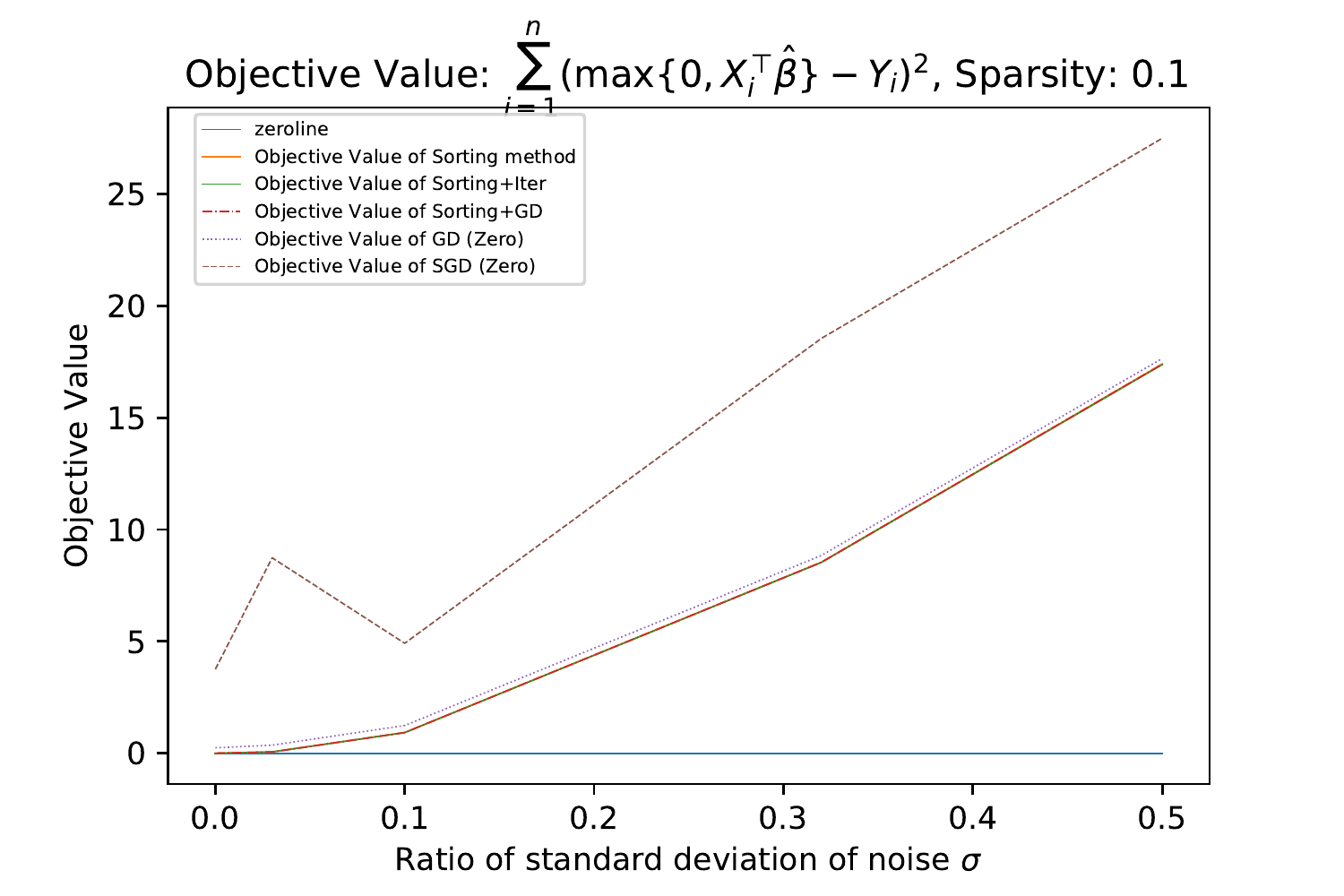}
        \caption[]%
        {{\small Objective Value}}    
        \label{fig:Figure_Compare_OB_10_200_10}
    \end{subfigure}
    \hfill
    \begin{subfigure}[b]{0.24\textwidth}  
        \centering 
        \includegraphics[width=\textwidth]{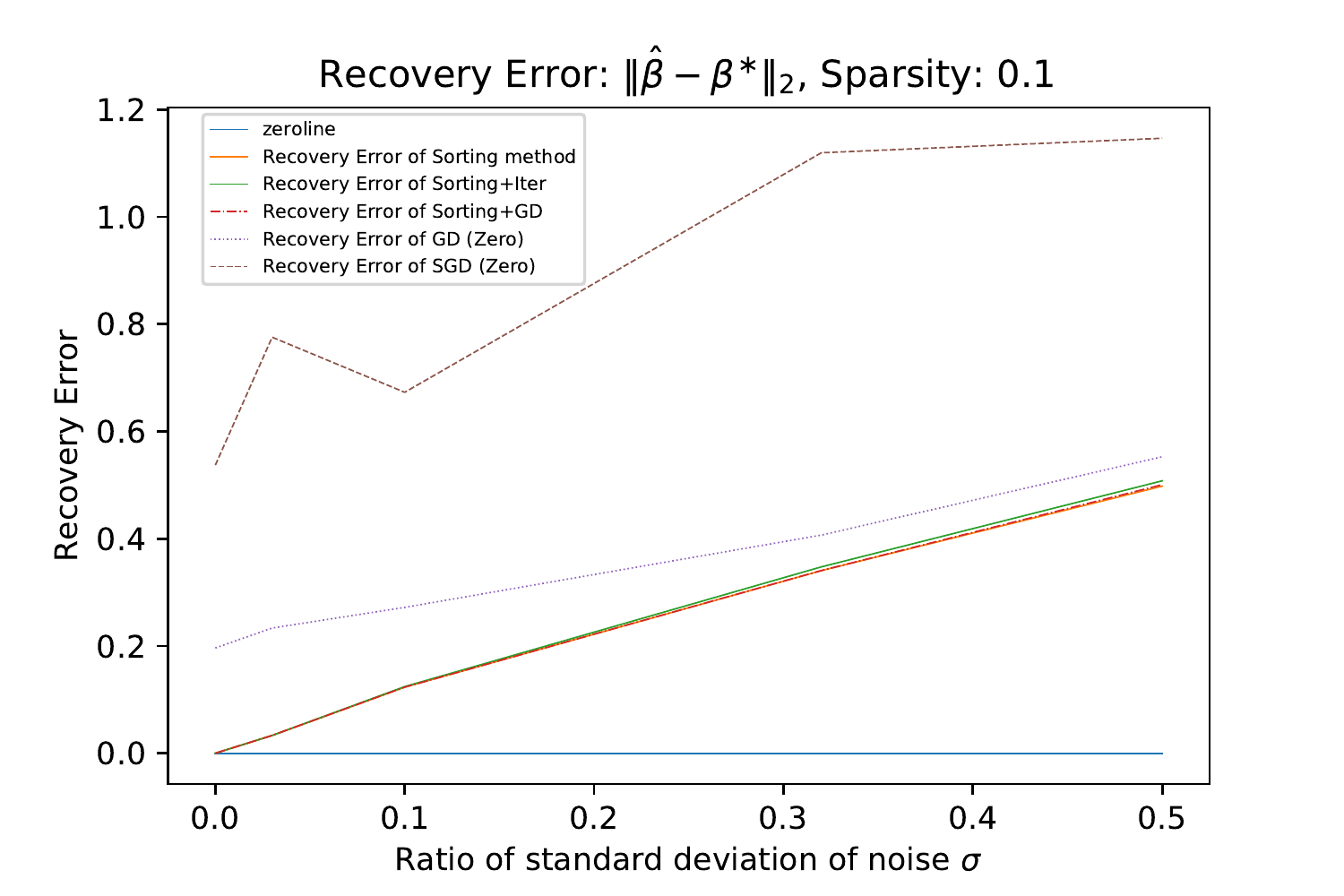}
        \caption[]%
        {{\small Recovery Error}}    
        \label{fig:Figure_Compare_RE_10_200_10}
    \end{subfigure}
    \hfill
    \begin{subfigure}[b]{0.24\textwidth}   
        \centering 
        \includegraphics[width=\textwidth]{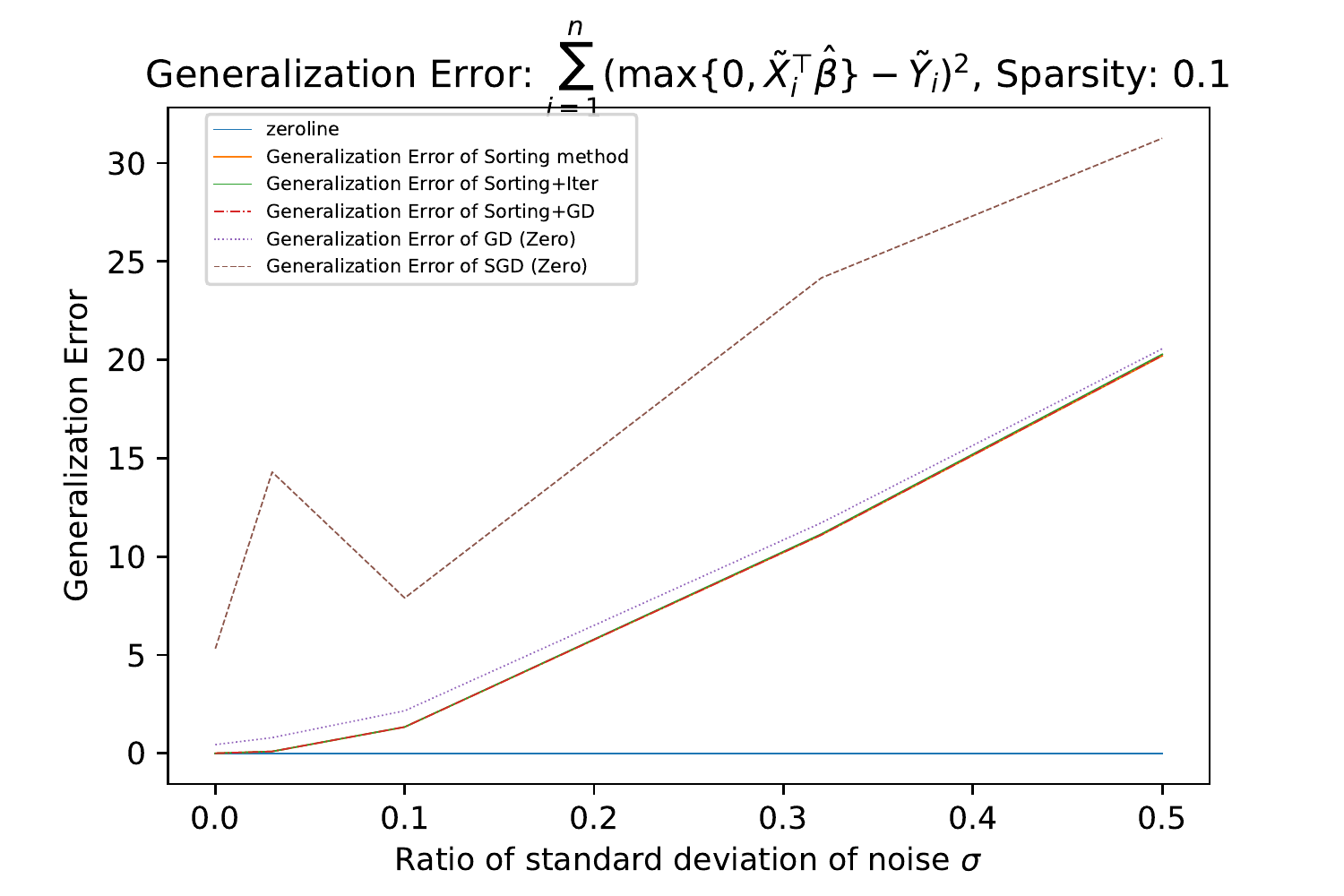}
        \caption[]%
        {{\small Generalization Error}}    
        \label{fig:Figure_Compare_GE_10_200_10}
    \end{subfigure}

    \vskip\baselineskip
    
    \centering
    \begin{subfigure}[b]{0.24\textwidth}
        \centering
        \includegraphics[width=\textwidth]{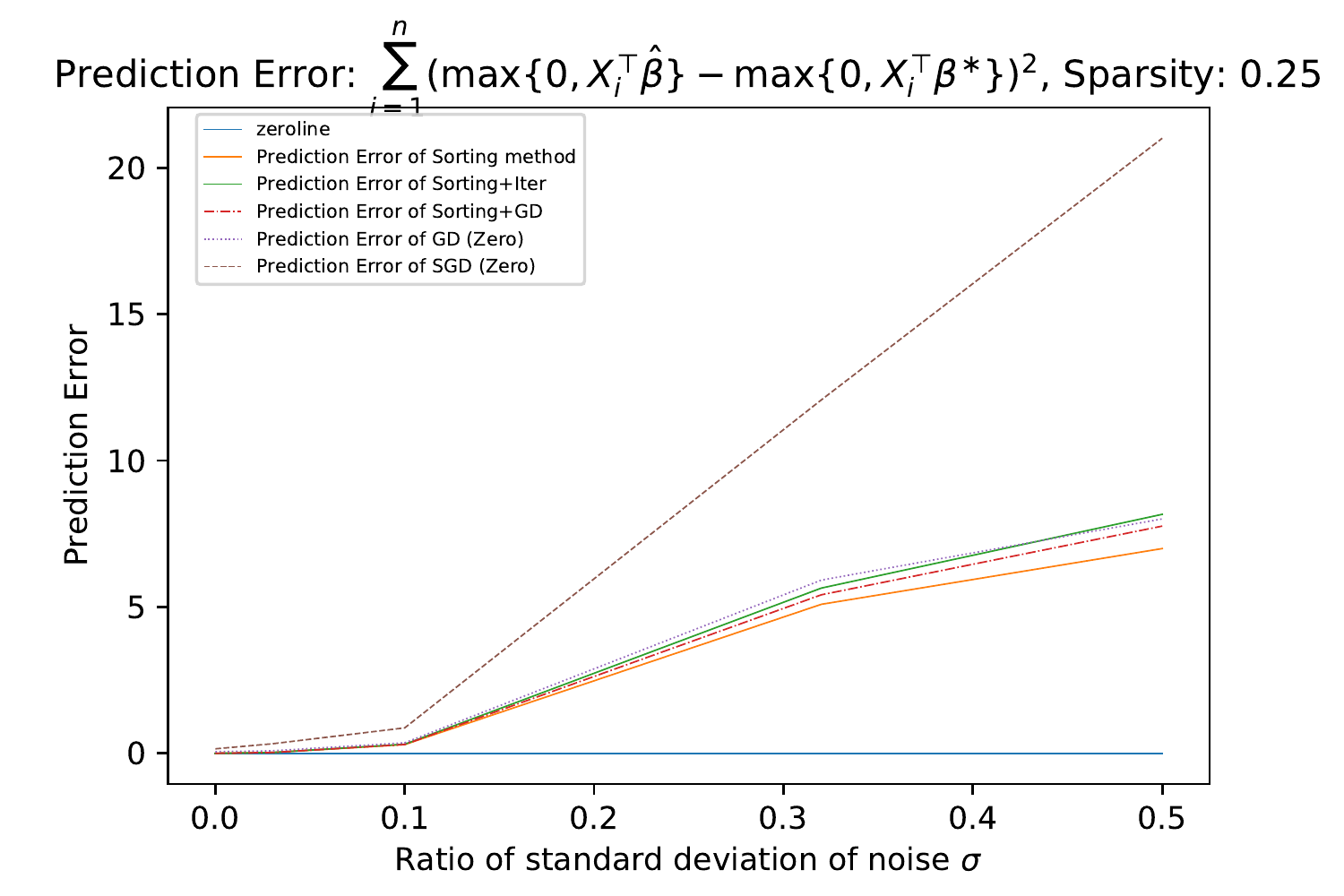}
        \caption[Network1]%
        {{\small Prediction Error}}    
        \label{fig:Figure_Compare_PE_10_200_25}
    \end{subfigure}
    \hfill
    \begin{subfigure}[b]{0.24\textwidth}   
        \centering 
        \includegraphics[width=\textwidth]{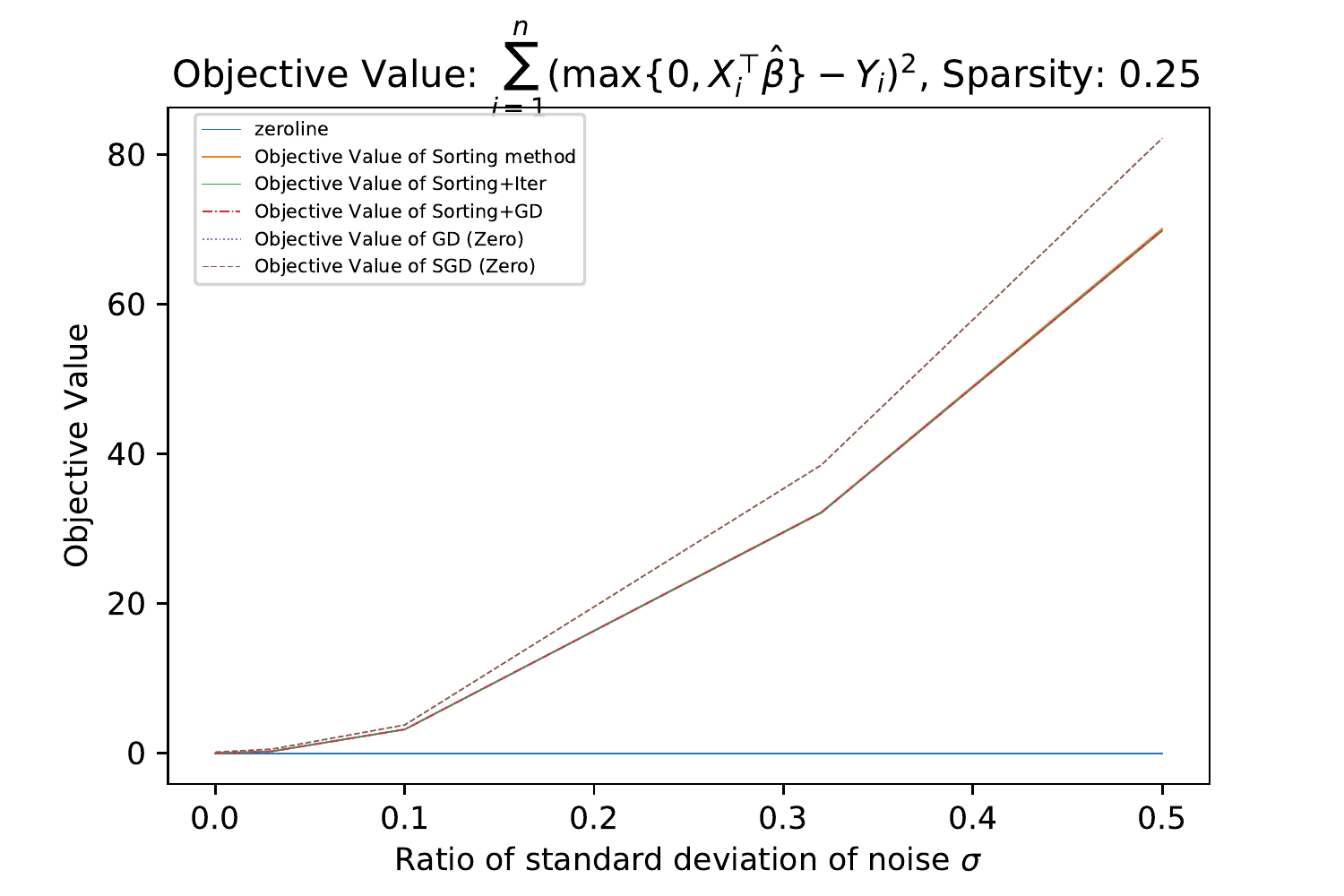}
        \caption[]%
        {{\small Objective Value}}    
        \label{fig:Figure_Compare_OB_10_200_25}
    \end{subfigure}
    \hfill
    \begin{subfigure}[b]{0.24\textwidth}  
        \centering 
        \includegraphics[width=\textwidth]{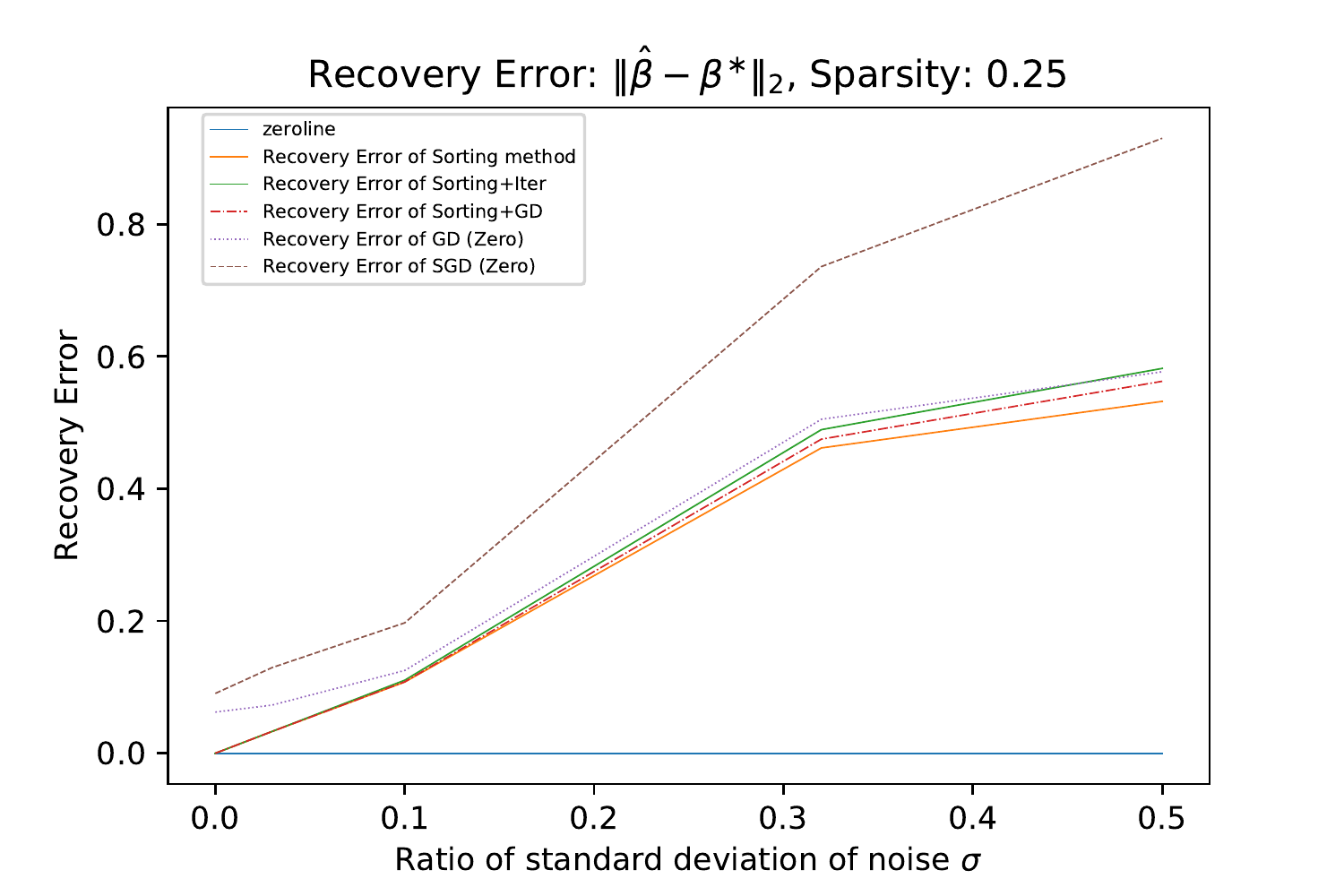}
        \caption[]%
        {{\small Recovery Error}}    
        \label{fig:Figure_Compare_RE_10_200_25}
    \end{subfigure}
    \hfill
    \begin{subfigure}[b]{0.24\textwidth}   
        \centering 
        \includegraphics[width=\textwidth]{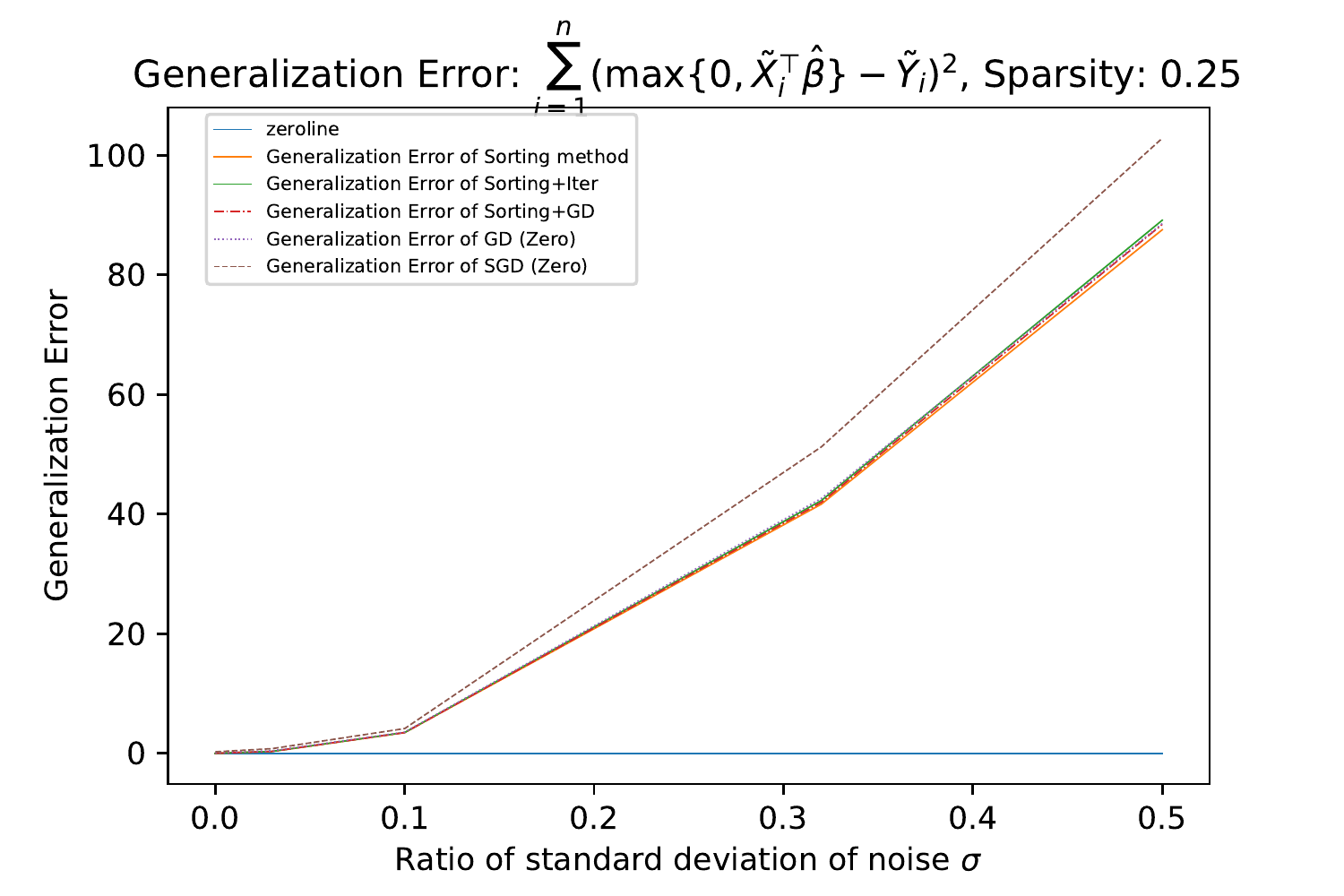}
        \caption[]%
        {{\small Generalization Error}}    
        \label{fig:Figure_Compare_GE_10_200_25}
    \end{subfigure}
    
    \vskip\baselineskip
    
    \centering
    \begin{subfigure}[b]{0.24\textwidth}
        \centering
        \includegraphics[width=\textwidth]{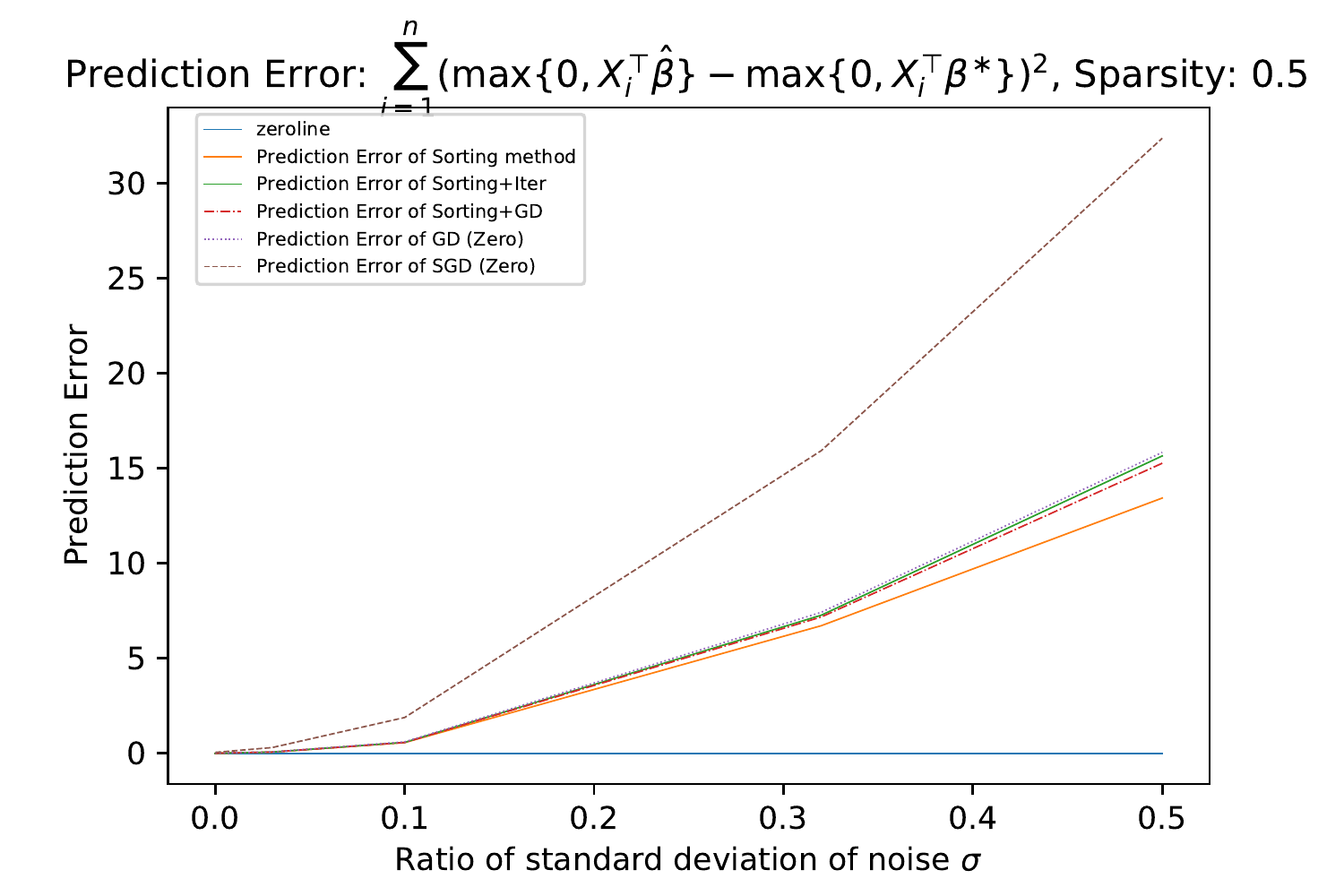}
        \caption[Network1]%
        {{\small Prediction Error}}    
        \label{fig:Figure_Compare_PE_10_200_50}
    \end{subfigure}
    \hfill
    \begin{subfigure}[b]{0.24\textwidth}   
        \centering 
        \includegraphics[width=\textwidth]{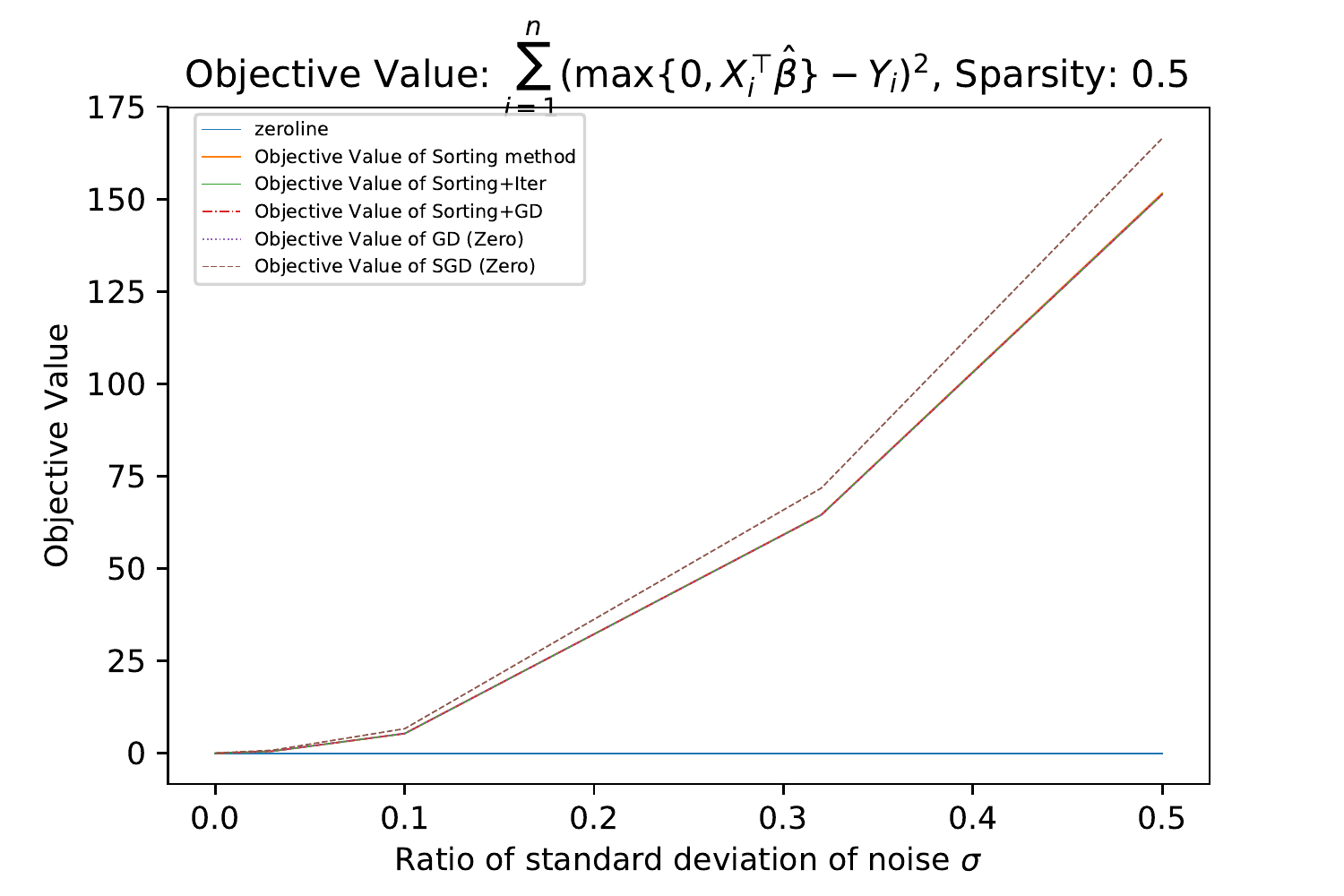}
        \caption[]%
        {{\small Objective Value}}    
        \label{fig:Figure_Compare_OB_10_200_50}
    \end{subfigure}
    \hfill
    \begin{subfigure}[b]{0.24\textwidth}  
        \centering 
        \includegraphics[width=\textwidth]{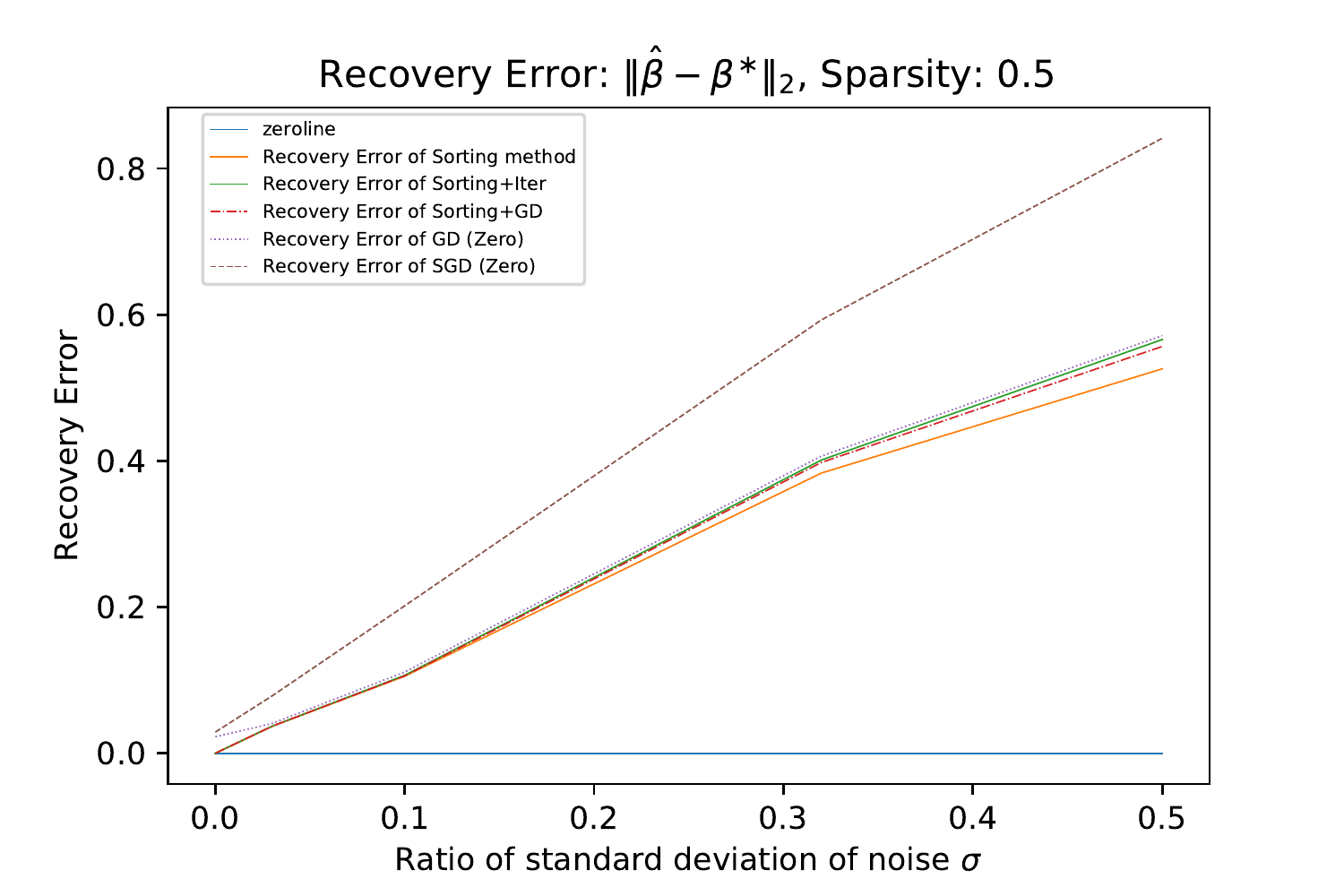}
        \caption[]%
        {{\small Recovery Error}}    
        \label{fig:Figure_Compare_RE_10_200_50}
    \end{subfigure}
    \hfill
    \begin{subfigure}[b]{0.24\textwidth}   
        \centering 
        \includegraphics[width=\textwidth]{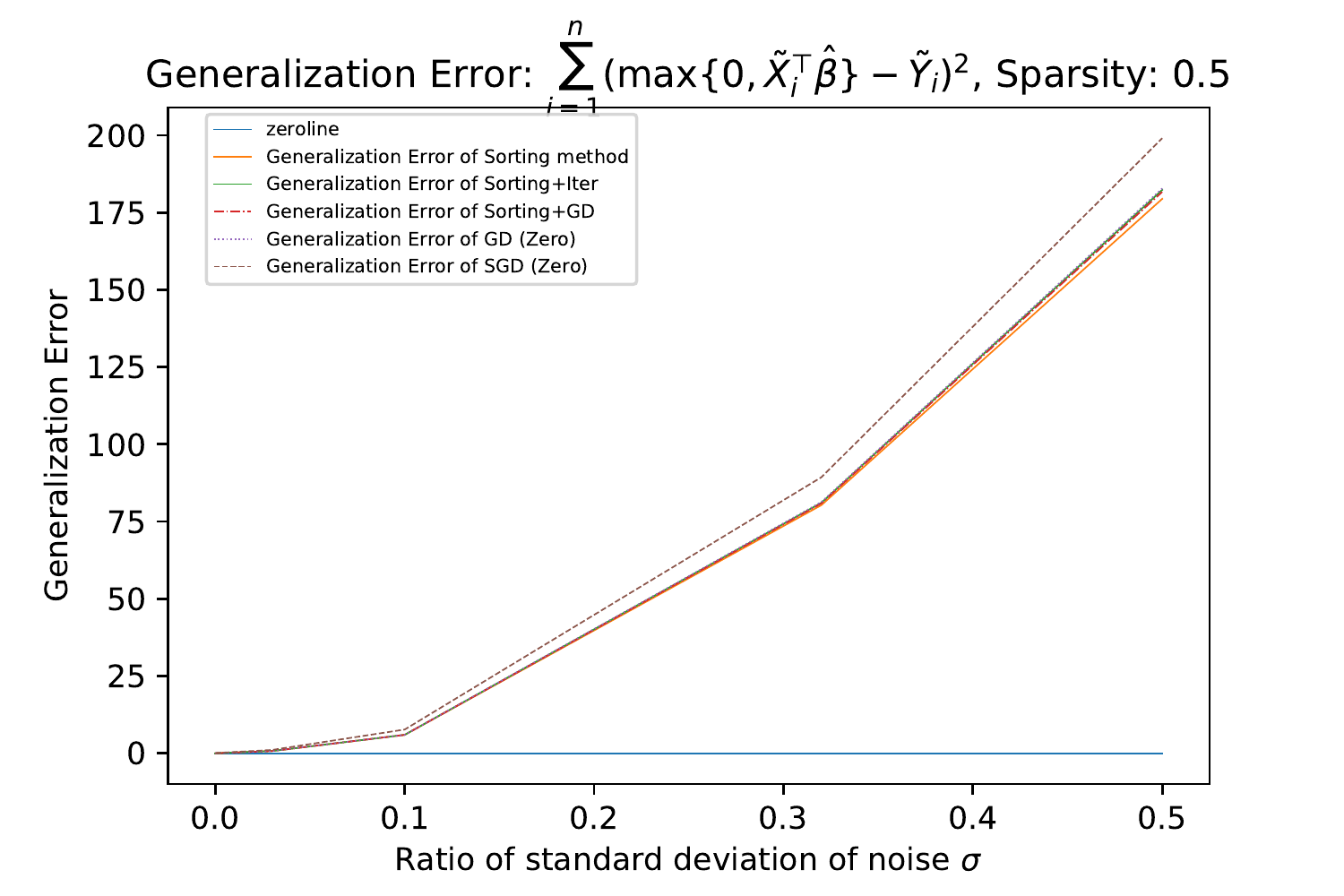}
        \caption[]%
        {{\small Generalization Error}}    
        \label{fig:Figure_Compare_GE_10_200_50}
    \end{subfigure}
    
    \vskip\baselineskip
    
    \centering
    \begin{subfigure}[b]{0.24\textwidth}
        \centering
        \includegraphics[width=\textwidth]{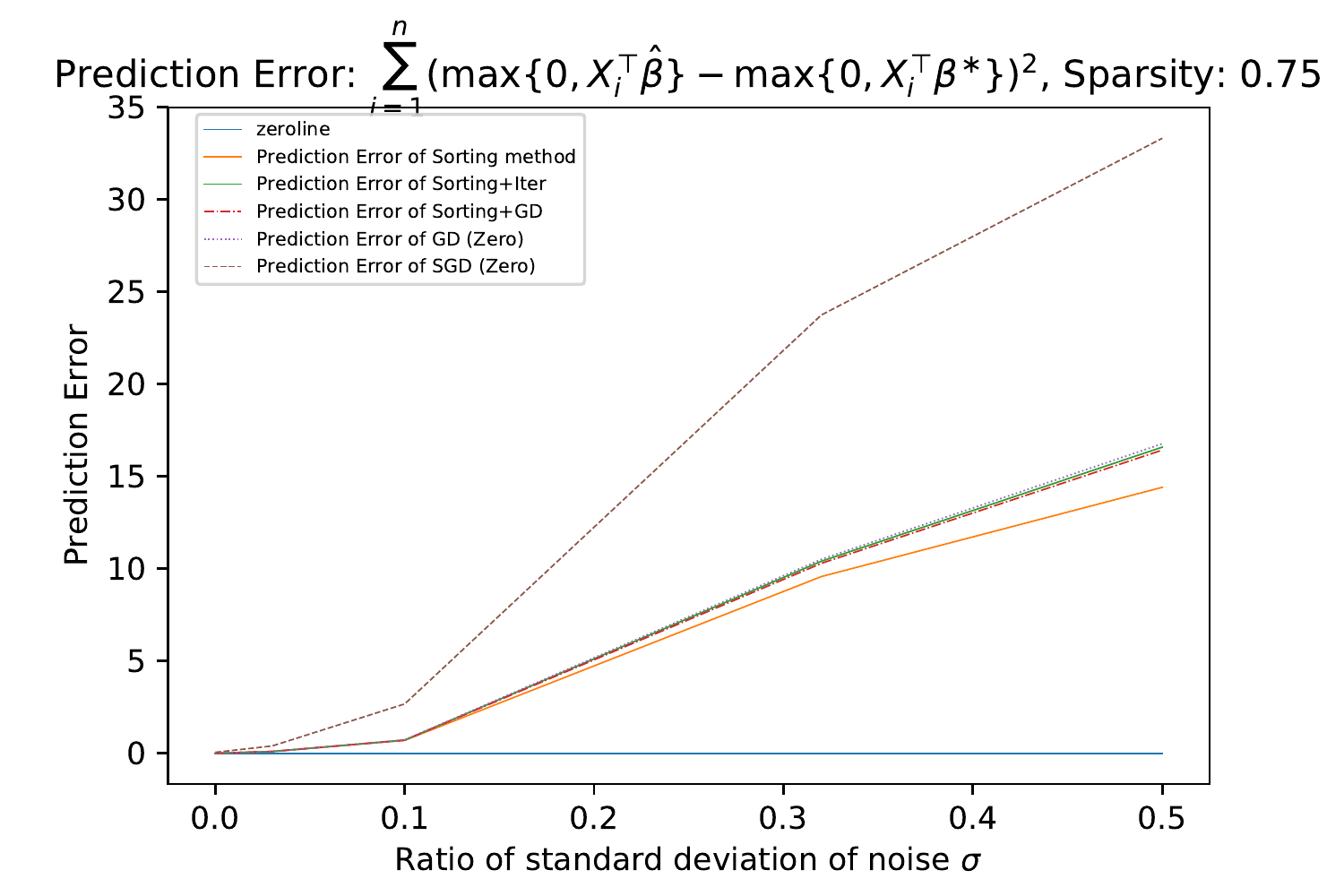}
        \caption[Network1]%
        {{\small Prediction Error}}    
        \label{fig:Figure_Compare_PE_10_200_75}
    \end{subfigure}
    \hfill
    \begin{subfigure}[b]{0.24\textwidth}   
        \centering 
        \includegraphics[width=\textwidth]{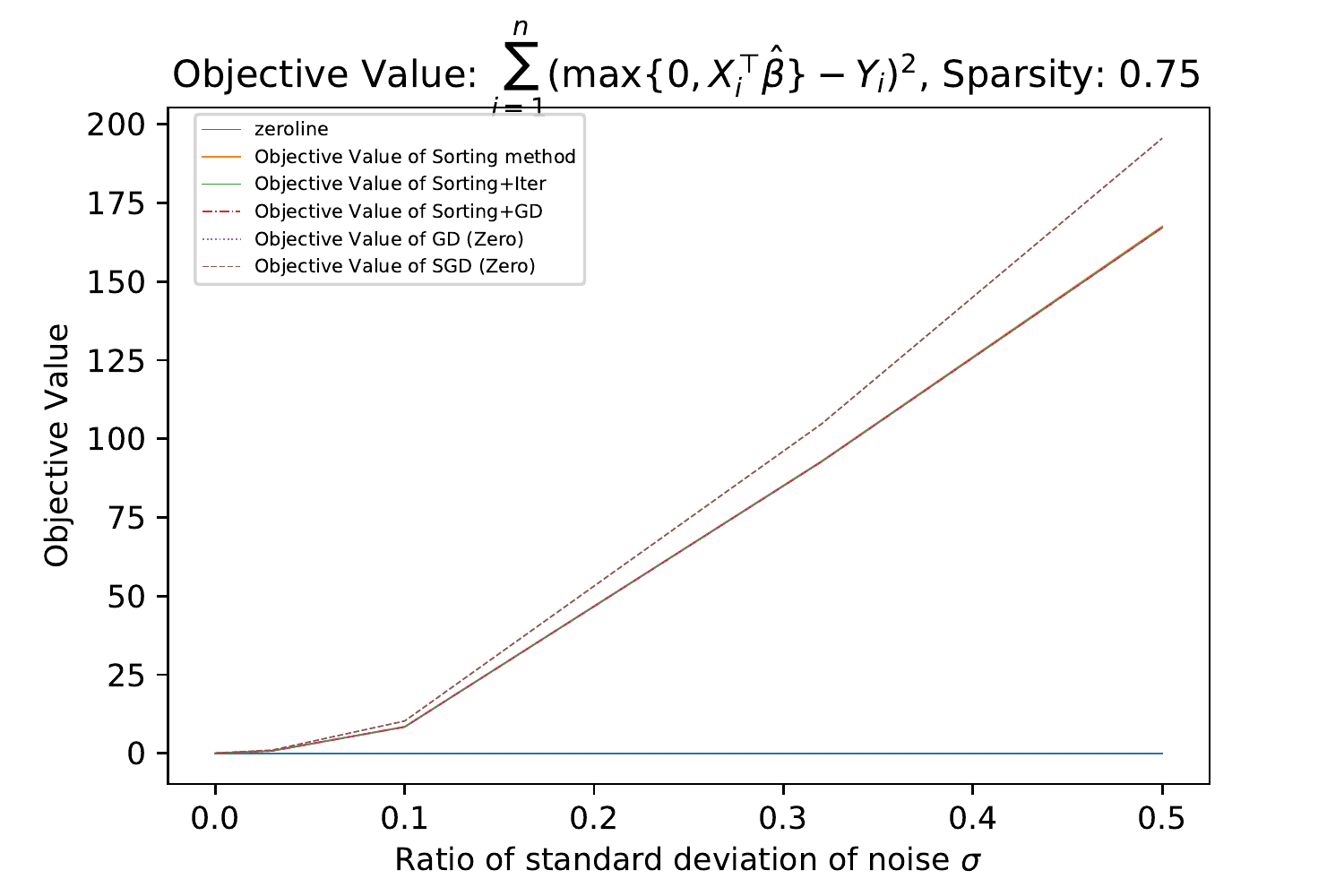}
        \caption[]%
        {{\small Objective Value}}    
        \label{fig:Figure_Compare_OB_10_200_75}
    \end{subfigure}
    \hfill
    \begin{subfigure}[b]{0.24\textwidth}  
        \centering 
        \includegraphics[width=\textwidth]{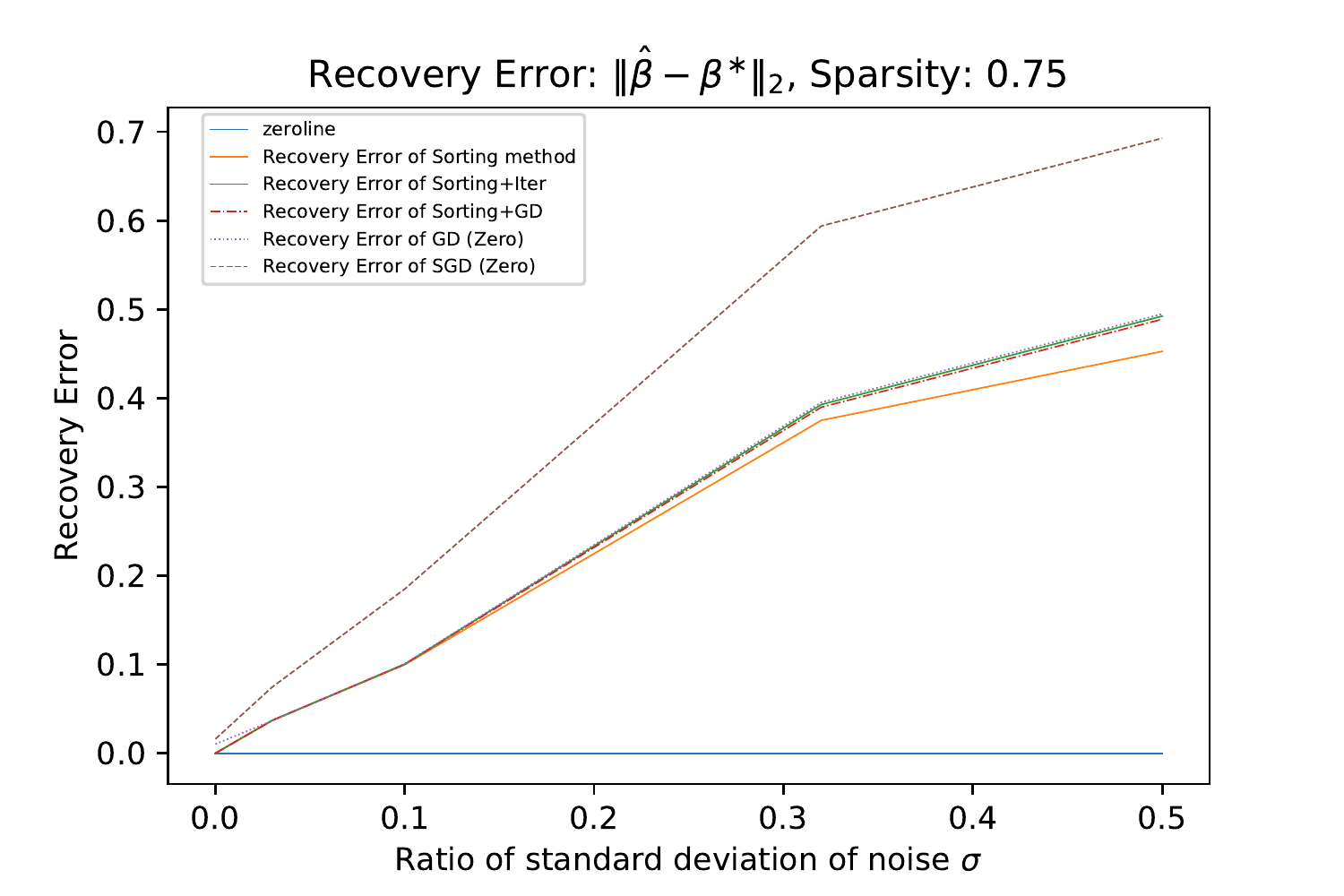}
        \caption[]%
        {{\small Recovery Error}}    
        \label{fig:Figure_Compare_RE_10_200_75}
    \end{subfigure}
    \hfill
    \begin{subfigure}[b]{0.24\textwidth}   
        \centering 
        \includegraphics[width=\textwidth]{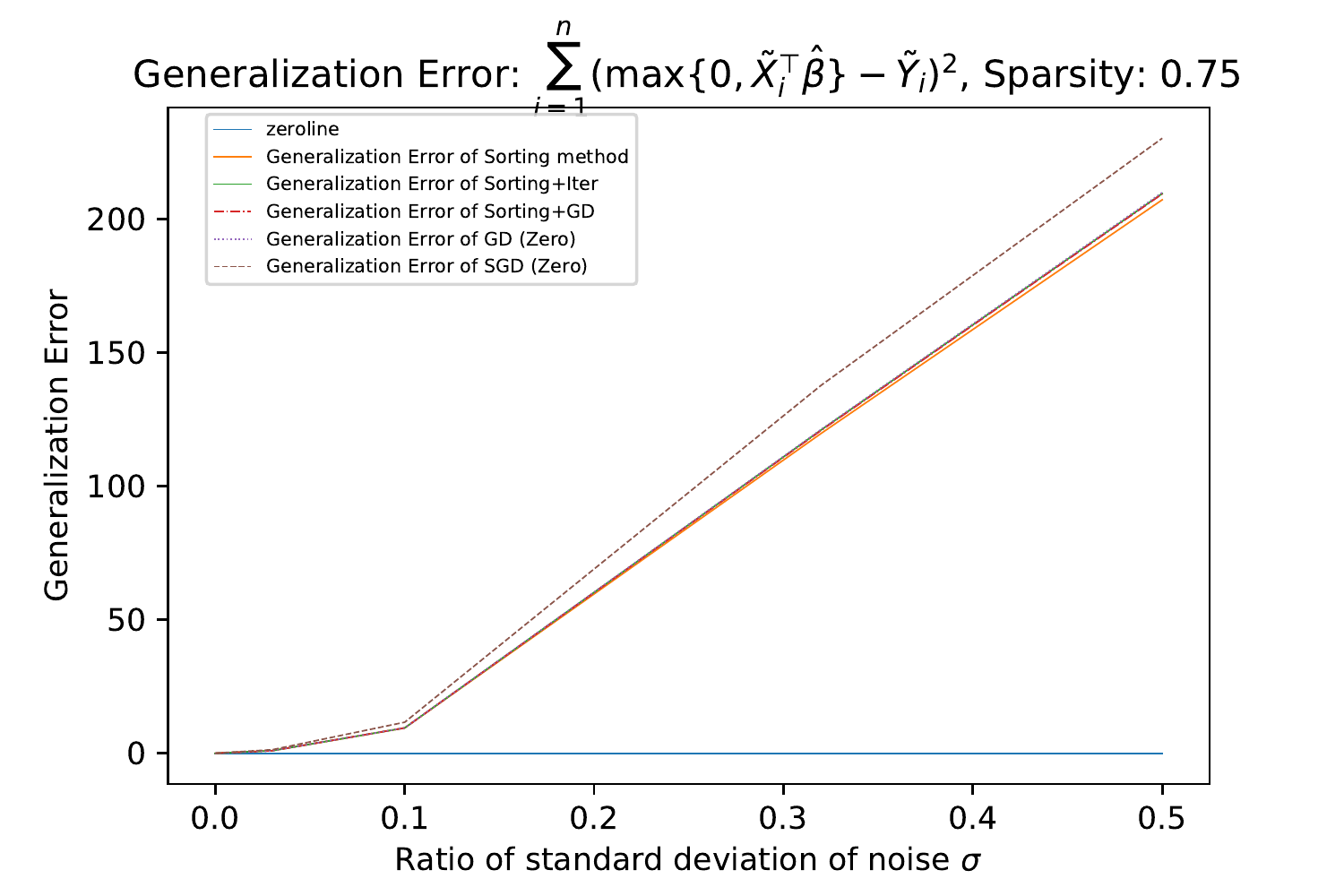}
        \caption[]%
        {{\small Generalization Error}}    
        \label{fig:Figure_Compare_GE_10_200_75}
    \end{subfigure}
    
    \vskip\baselineskip
    
    \centering
    \begin{subfigure}[b]{0.24\textwidth}
        \centering
        \includegraphics[width=\textwidth]{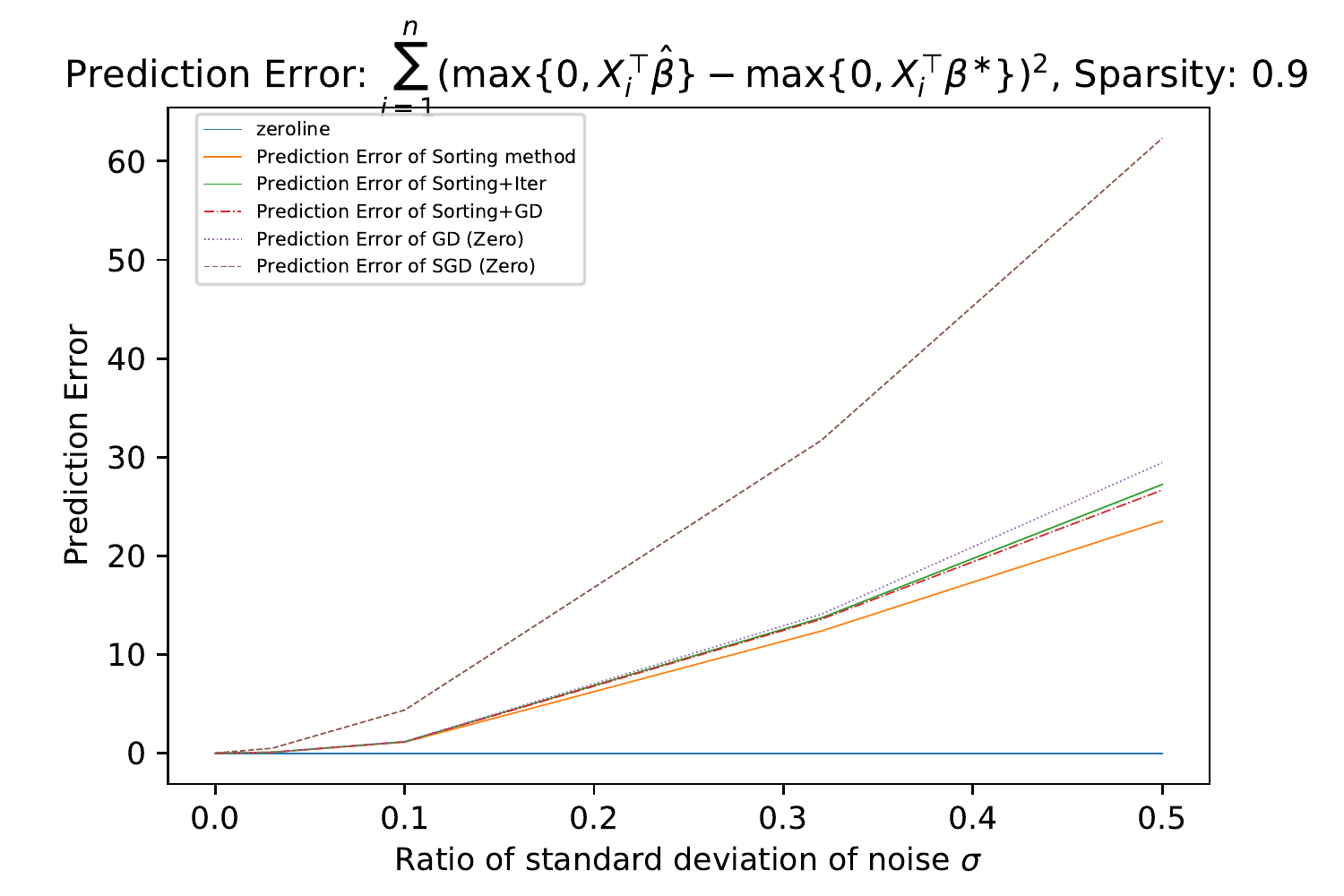}
        \caption[Network1]%
        {{\small Prediction Error}}    
        \label{fig:Figure_Compare_PE_10_200_90}
    \end{subfigure}
    \hfill
    \begin{subfigure}[b]{0.24\textwidth}   
        \centering 
        \includegraphics[width=\textwidth]{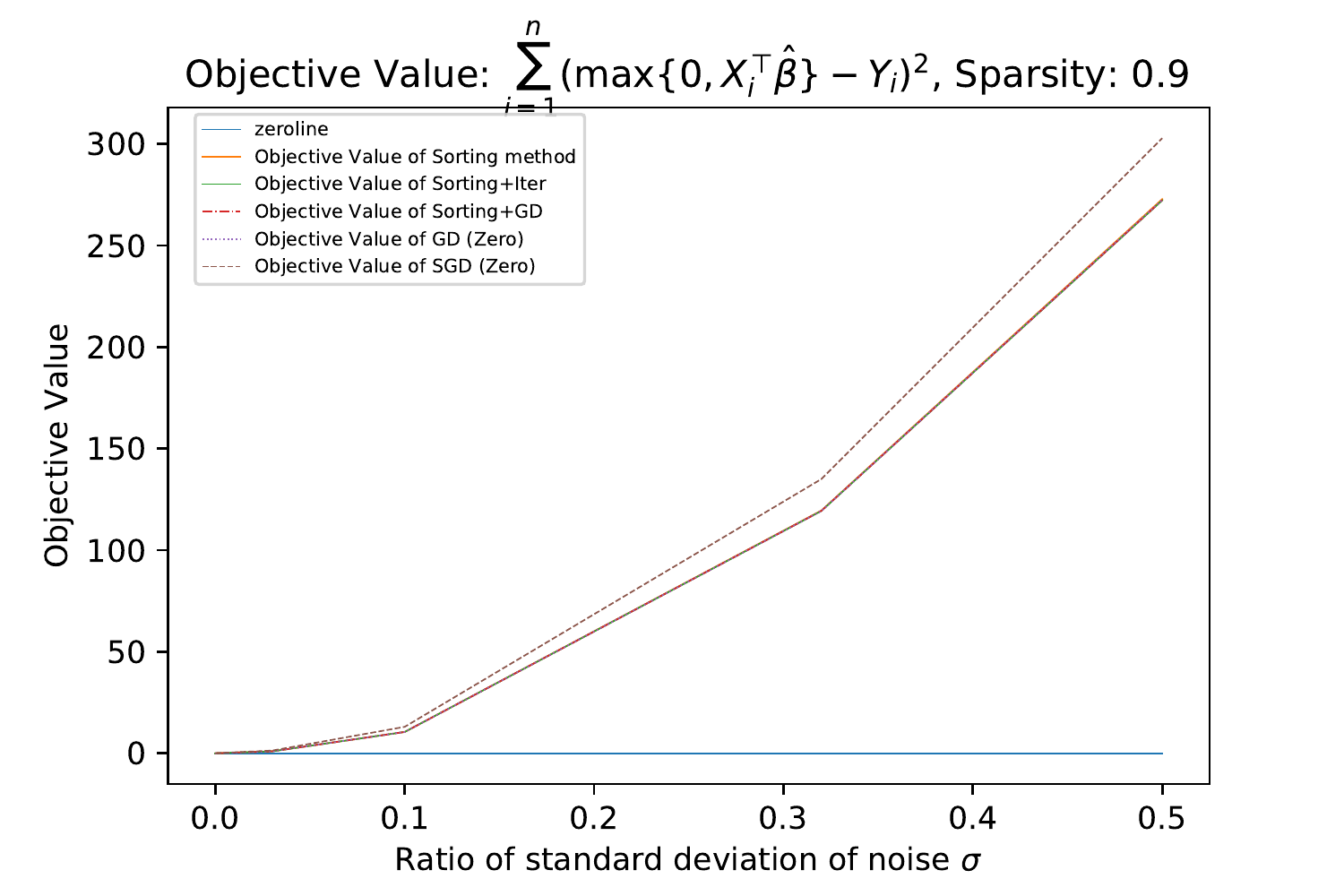}
        \caption[]%
        {{\small Objective Value}}    
        \label{fig:Figure_Compare_OB_10_200_90}
    \end{subfigure}
    \hfill
    \begin{subfigure}[b]{0.24\textwidth}  
        \centering 
        \includegraphics[width=\textwidth]{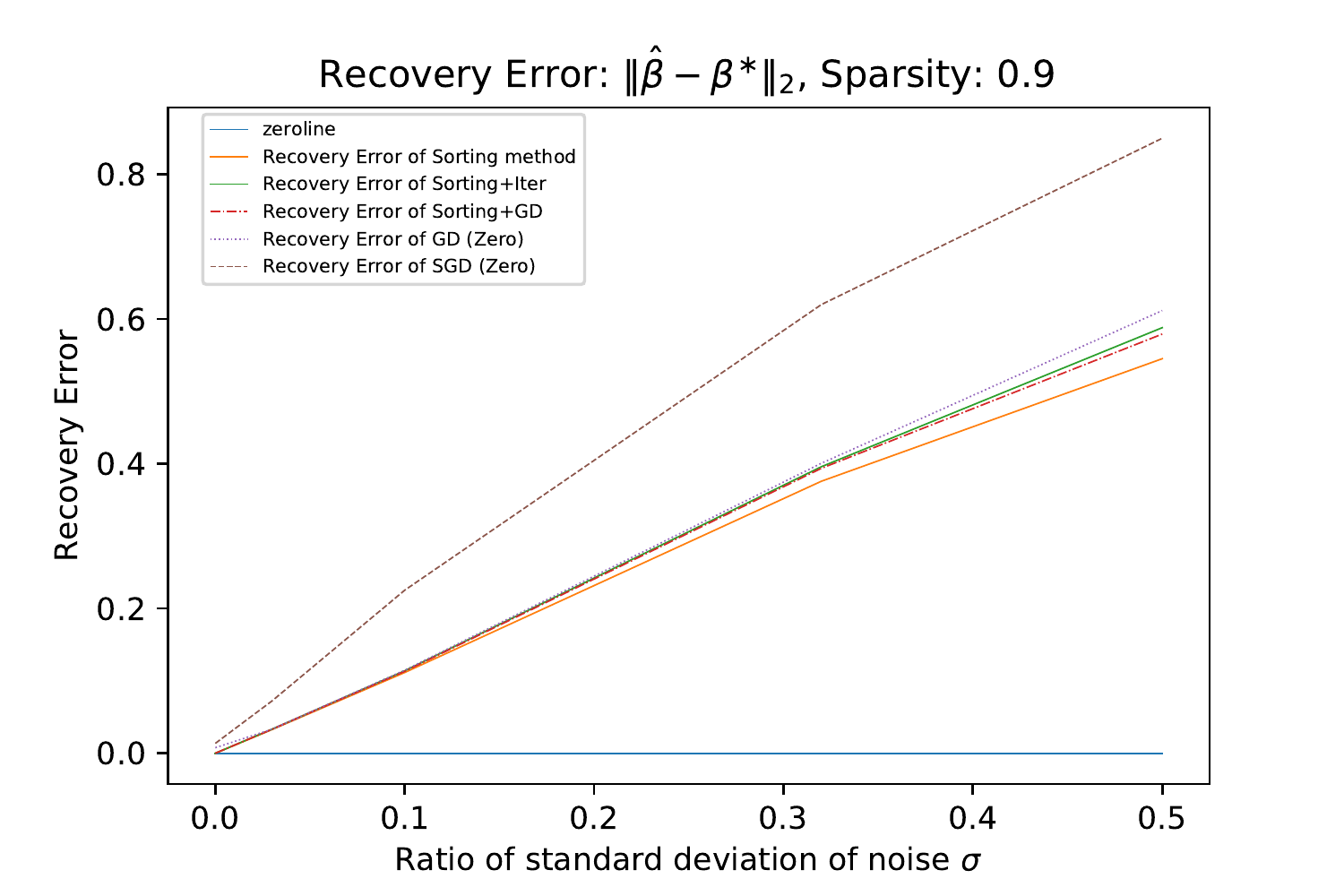}
        \caption[]%
        {{\small Recovery Error}}    
        \label{fig:Figure_Compare_RE_10_200_90}
    \end{subfigure}
    \hfill
    \begin{subfigure}[b]{0.24\textwidth}   
        \centering 
        \includegraphics[width=\textwidth]{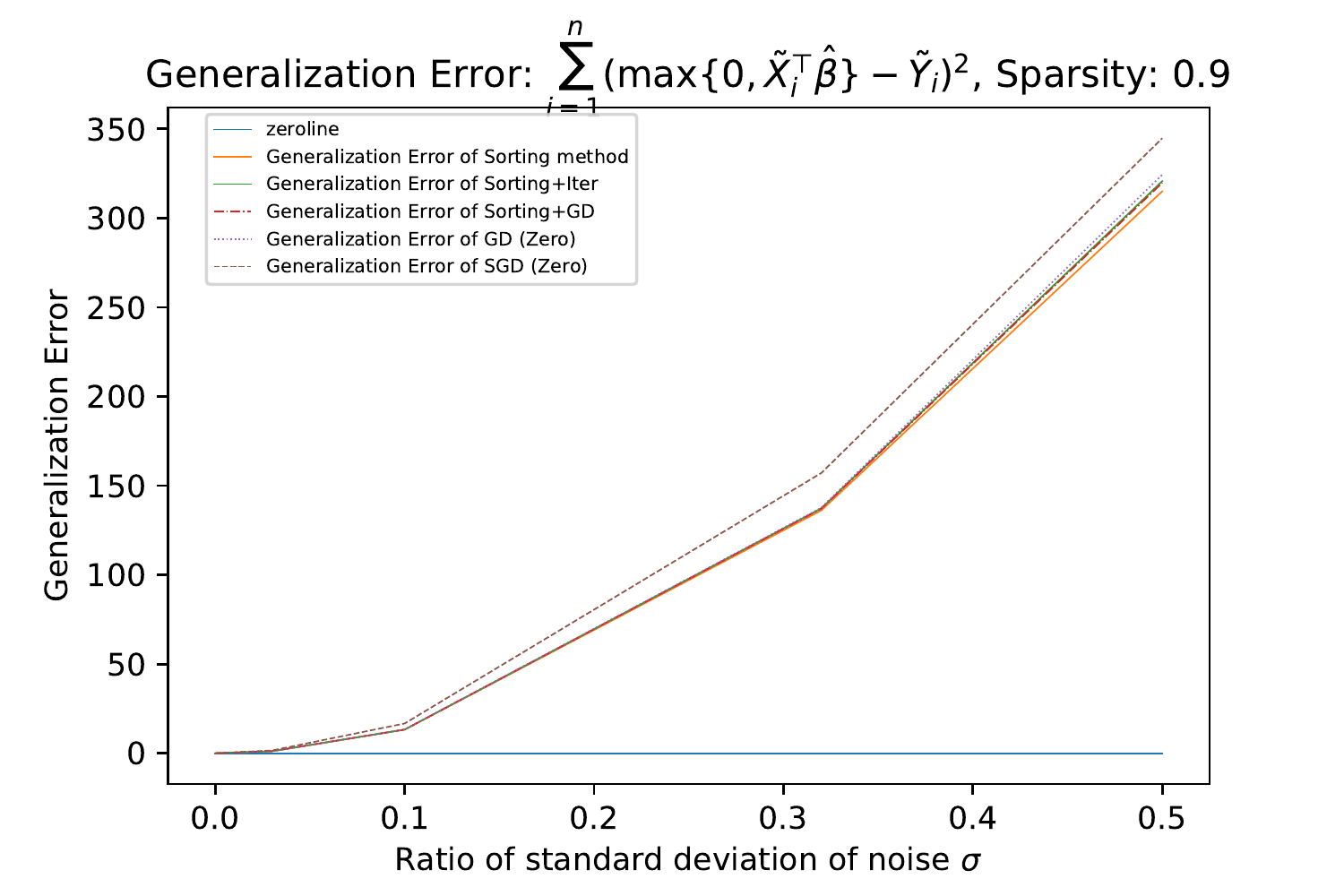}
        \caption[]%
        {{\small Generalization Error}}    
        \label{fig:Figure_Compare_GE_10_200_90}
    \end{subfigure}
    
    \caption[Numerical Results of sample size $(p, n)= (10, 200)$]
    {\small Numerical Results of sample size $(p, n) = (10, 200)$ and $\beta^{\ast} \sim N(0.5 \cdot \mathbf{1}_p, 10 \cdot I_p)$ with sparsity $\{0.1, 0.25, 0.5, 0.75, 0.9\}$} 
    \label{fig:Figure_Compare_10_200}
\end{figure*}

\begin{figure*}
    \centering
    \begin{subfigure}[b]{0.24\textwidth}
        \centering
        \includegraphics[width=\textwidth]{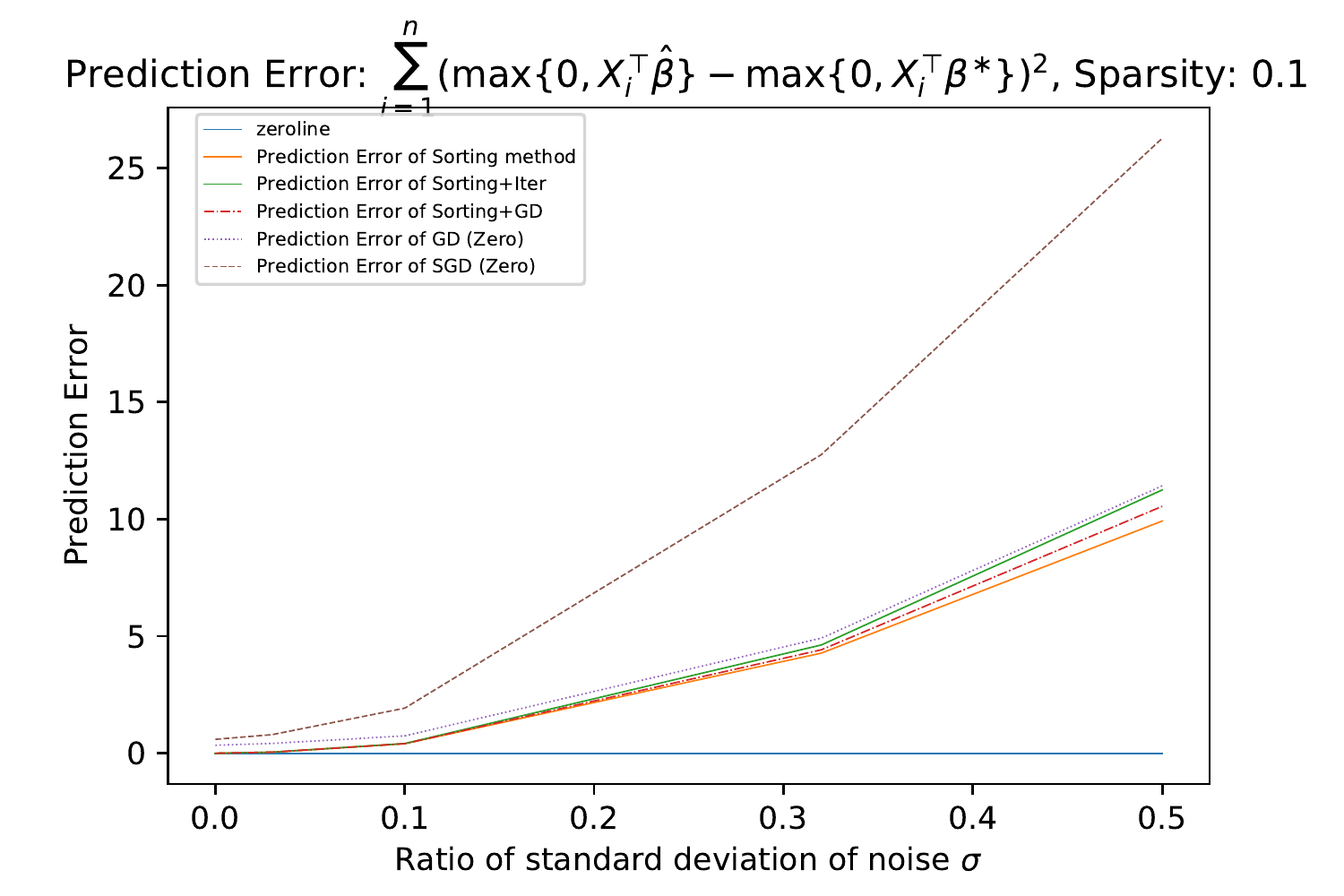}
        \caption[Network1]%
        {{\small Prediction Error}}    
        \label{fig:Figure_Compare_PE_20_400_10}
    \end{subfigure}
    \hfill
    \begin{subfigure}[b]{0.24\textwidth}   
        \centering 
        \includegraphics[width=\textwidth]{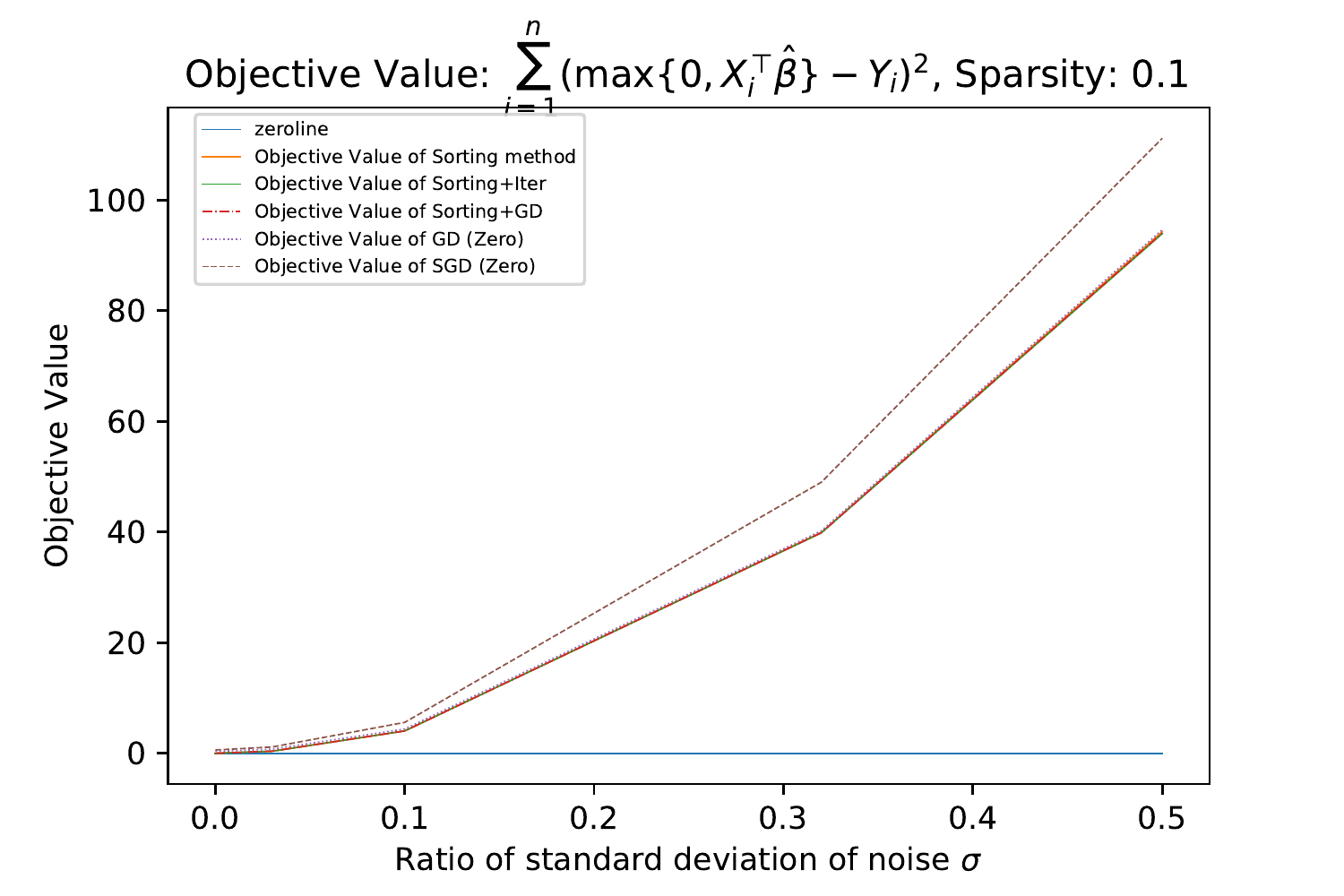}
        \caption[]%
        {{\small Objective Value}}    
        \label{fig:Figure_Compare_OB_20_400_10}
    \end{subfigure}
    \hfill
    \begin{subfigure}[b]{0.24\textwidth}  
        \centering 
        \includegraphics[width=\textwidth]{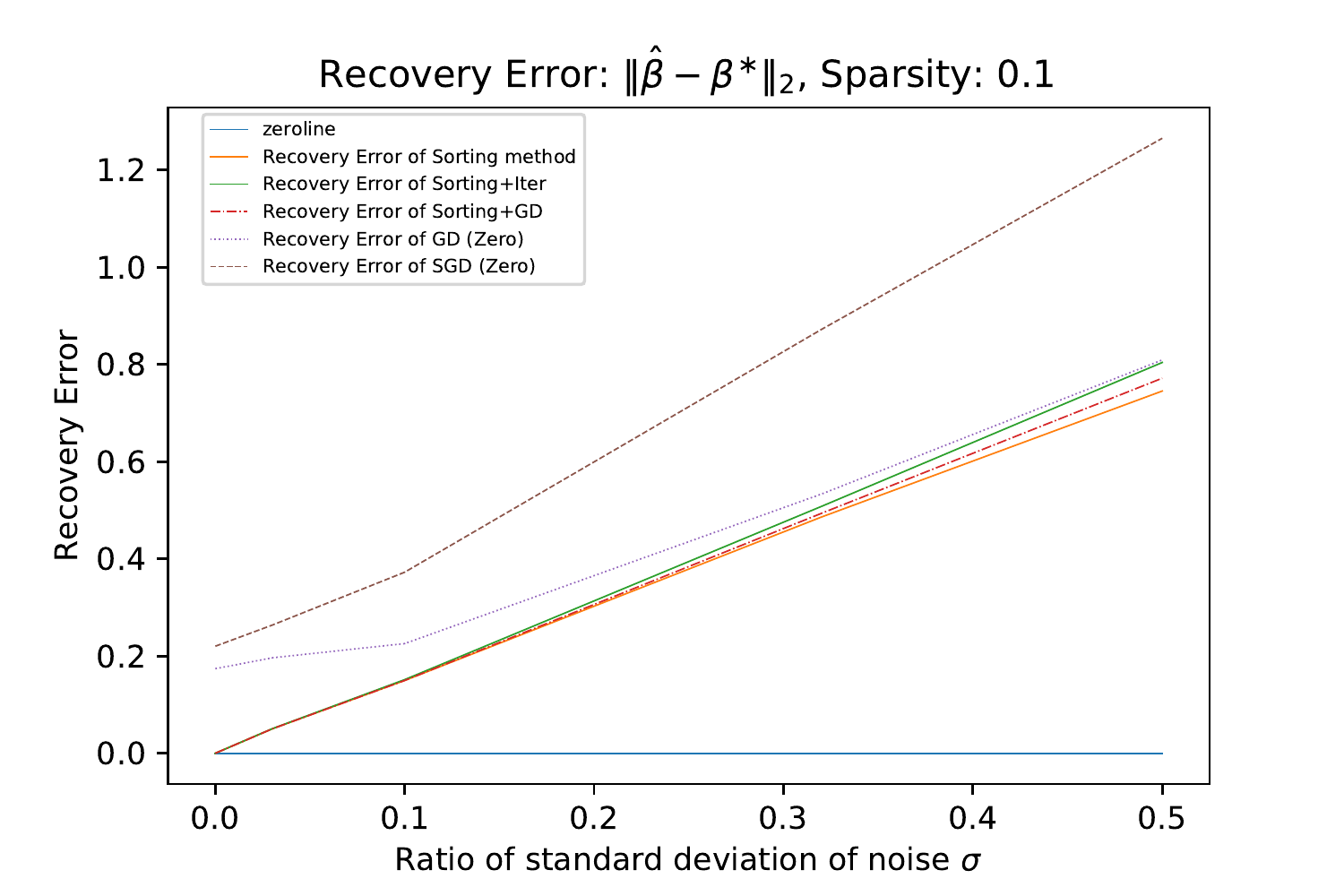}
        \caption[]%
        {{\small Recovery Error}}    
        \label{fig:Figure_Compare_RE_20_400_10}
    \end{subfigure}
    \hfill
    \begin{subfigure}[b]{0.24\textwidth}   
        \centering 
        \includegraphics[width=\textwidth]{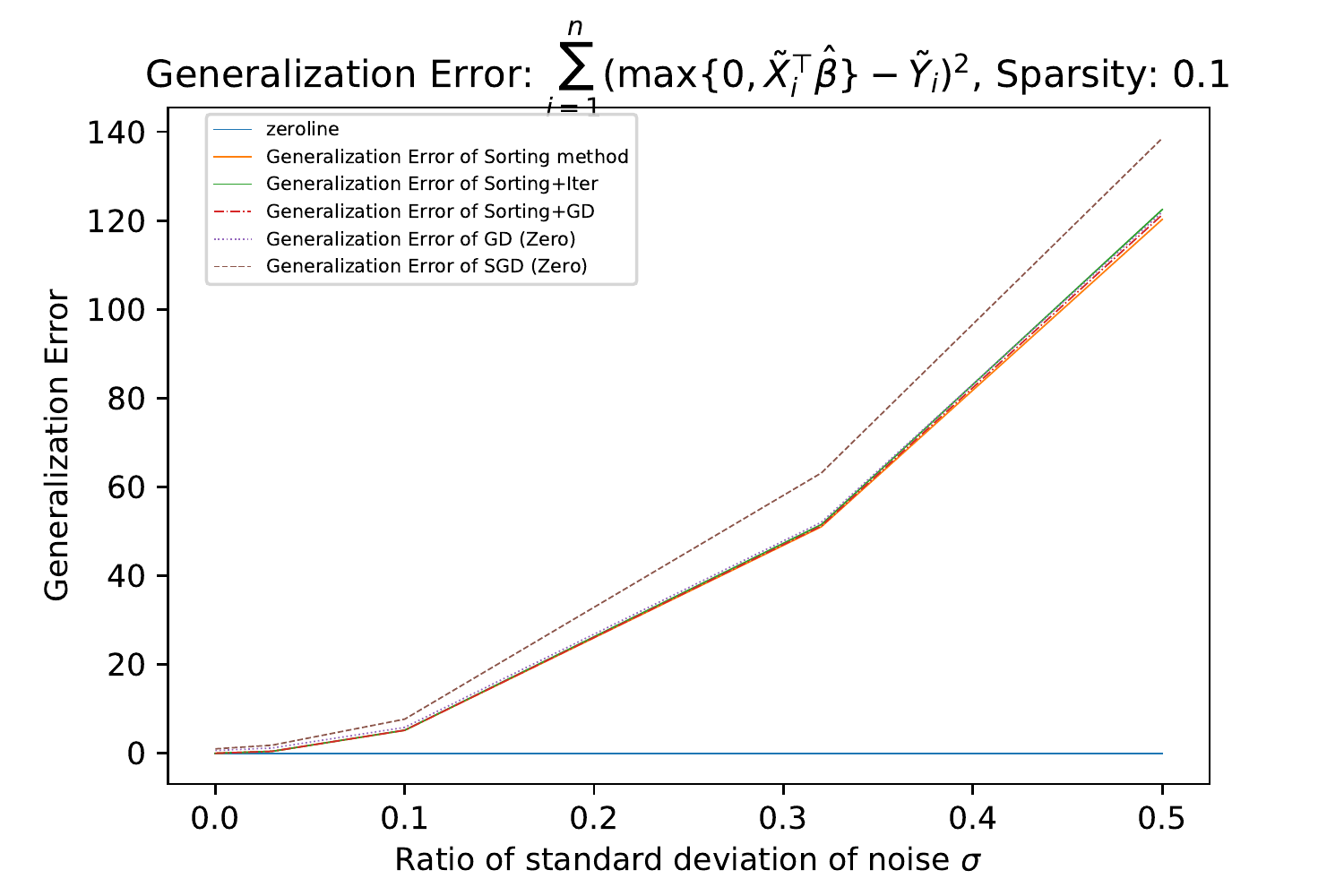}
        \caption[]%
        {{\small Generalization Error}}    
        \label{fig:Figure_Compare_GE_20_400_10}
    \end{subfigure}

    \vskip\baselineskip
    
    \centering
    \begin{subfigure}[b]{0.24\textwidth}
        \centering
        \includegraphics[width=\textwidth]{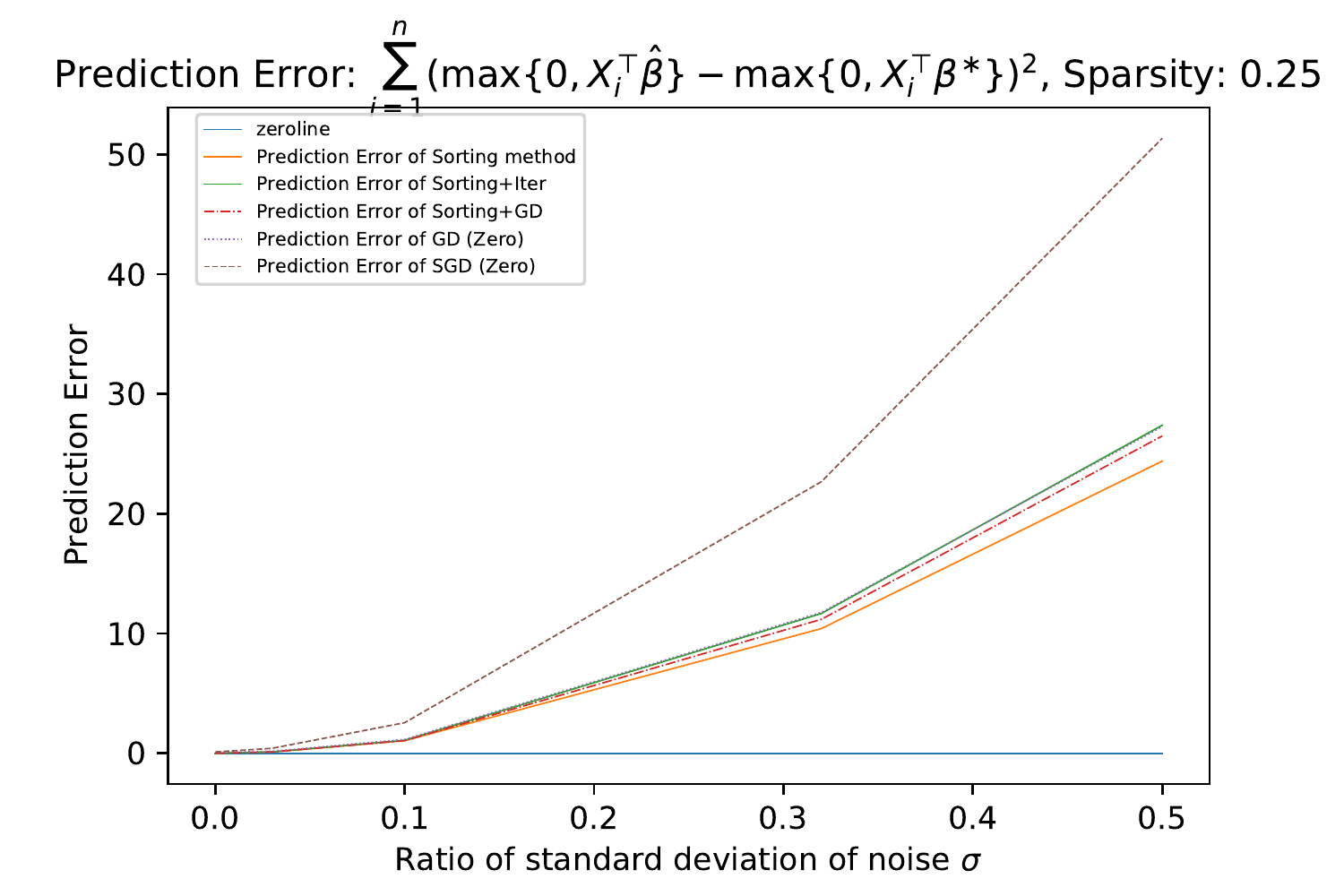}
        \caption[Network1]%
        {{\small Prediction Error}}    
        \label{fig:Figure_Compare_PE_20_400_25}
    \end{subfigure}
    \hfill
    \begin{subfigure}[b]{0.24\textwidth}   
        \centering 
        \includegraphics[width=\textwidth]{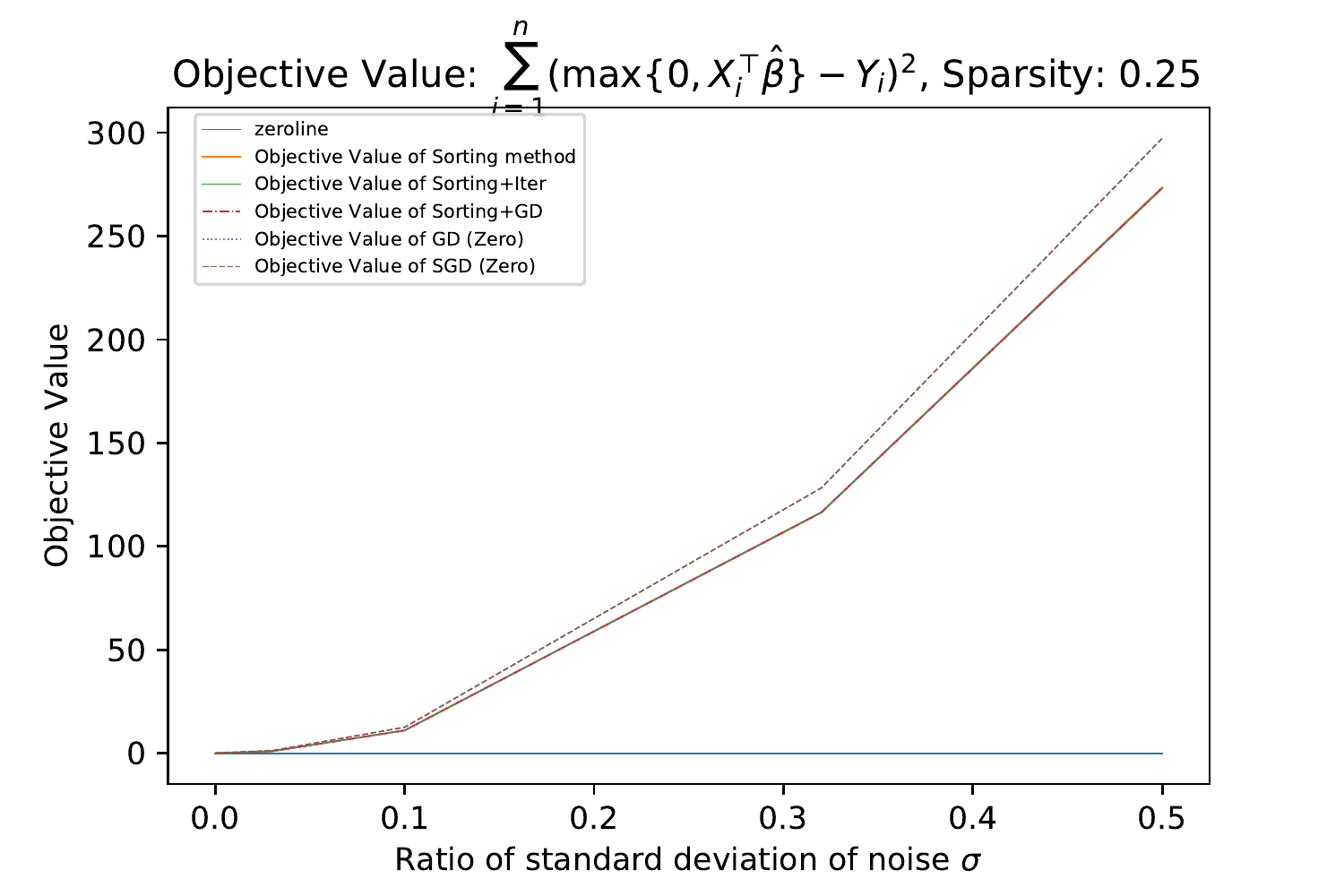}
        \caption[]%
        {{\small Objective Value}}    
        \label{fig:Figure_Compare_OB_20_400_25}
    \end{subfigure}
    \hfill
    \begin{subfigure}[b]{0.24\textwidth}  
        \centering 
        \includegraphics[width=\textwidth]{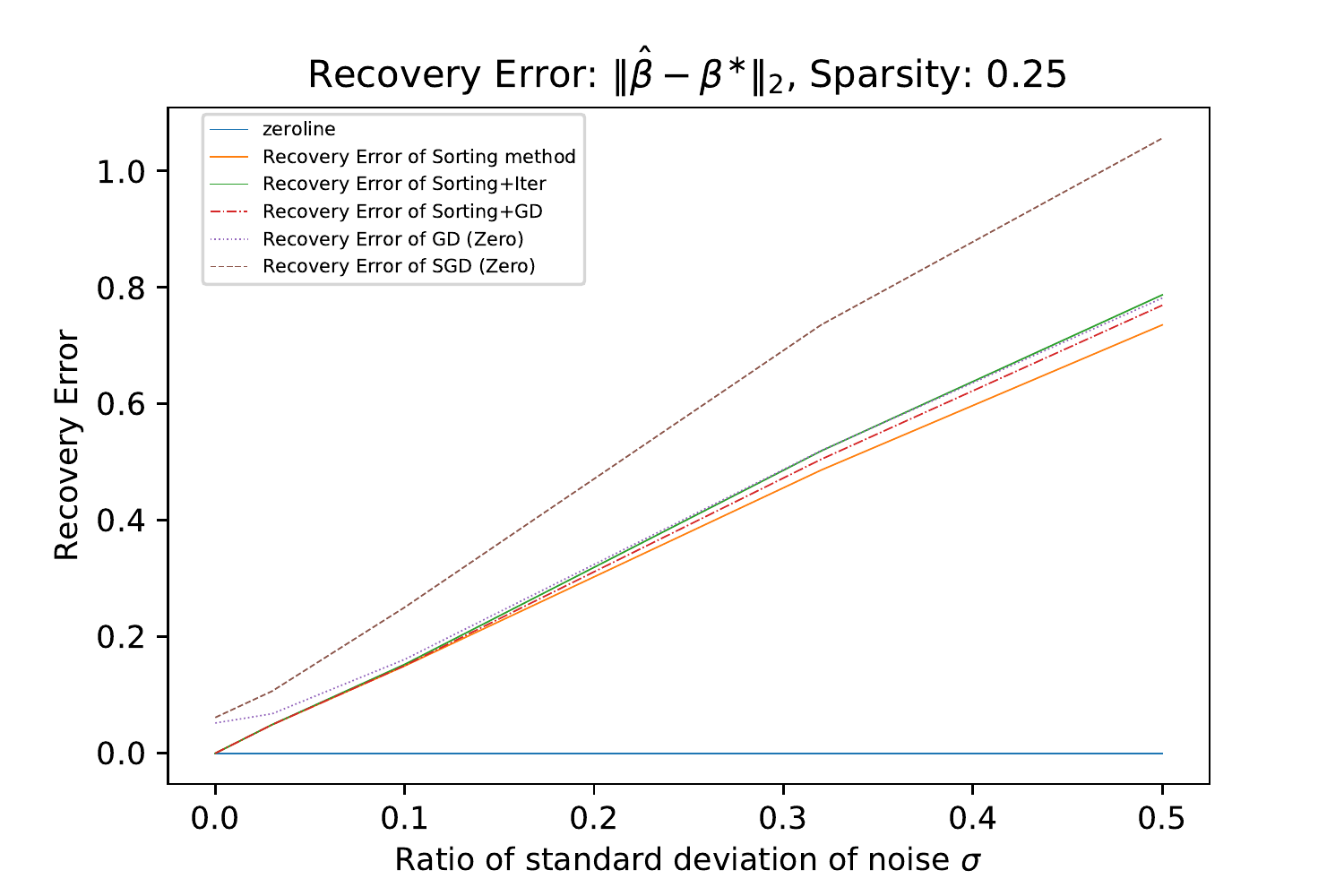}
        \caption[]%
        {{\small Recovery Error}}    
        \label{fig:Figure_Compare_RE_20_400_25}
    \end{subfigure}
    \hfill
    \begin{subfigure}[b]{0.24\textwidth}   
        \centering 
        \includegraphics[width=\textwidth]{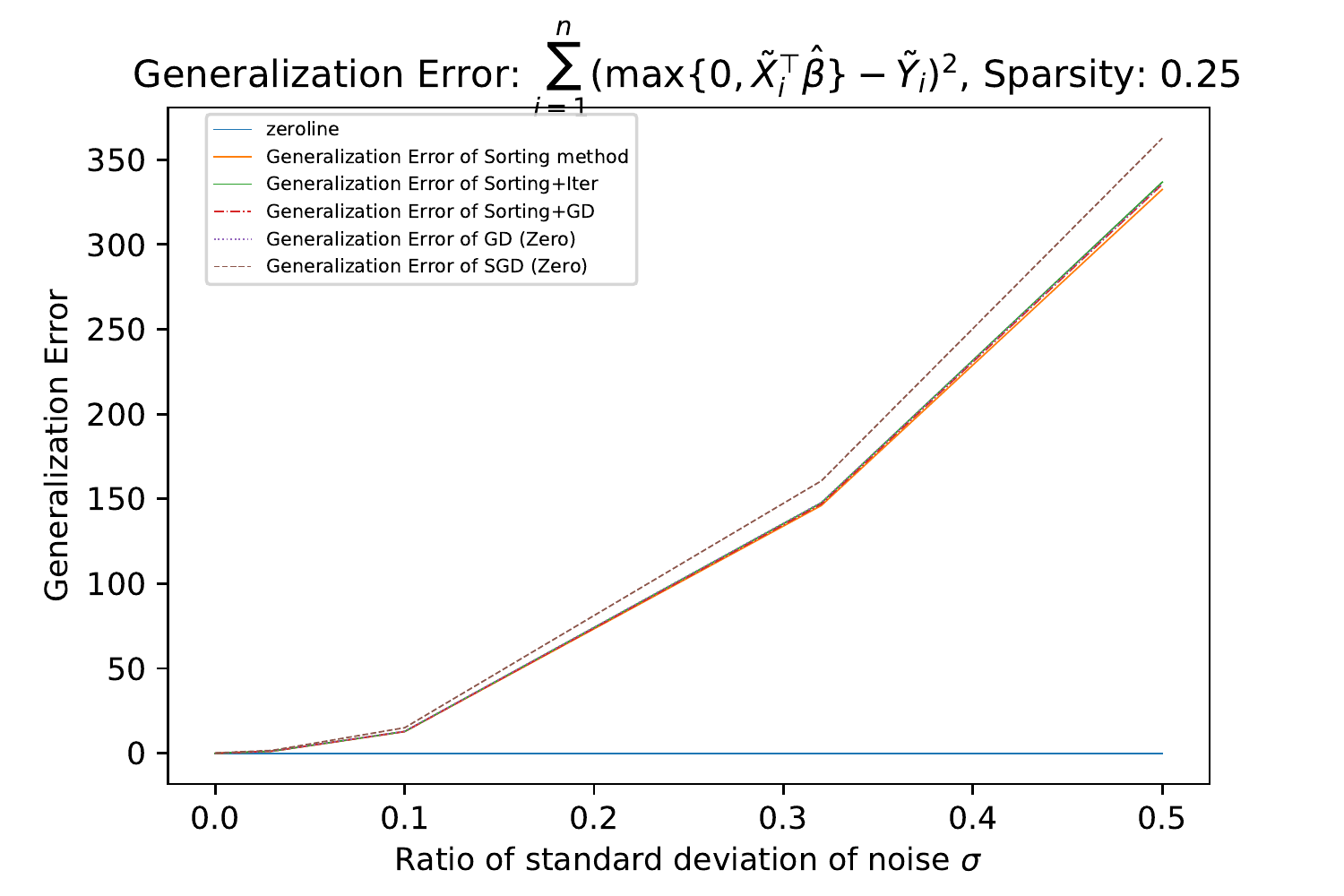}
        \caption[]%
        {{\small Generalization Error}}    
        \label{fig:Figure_Compare_GE_20_400_25}
    \end{subfigure}
    
    \vskip\baselineskip
    
    \centering
    \begin{subfigure}[b]{0.24\textwidth}
        \centering
        \includegraphics[width=\textwidth]{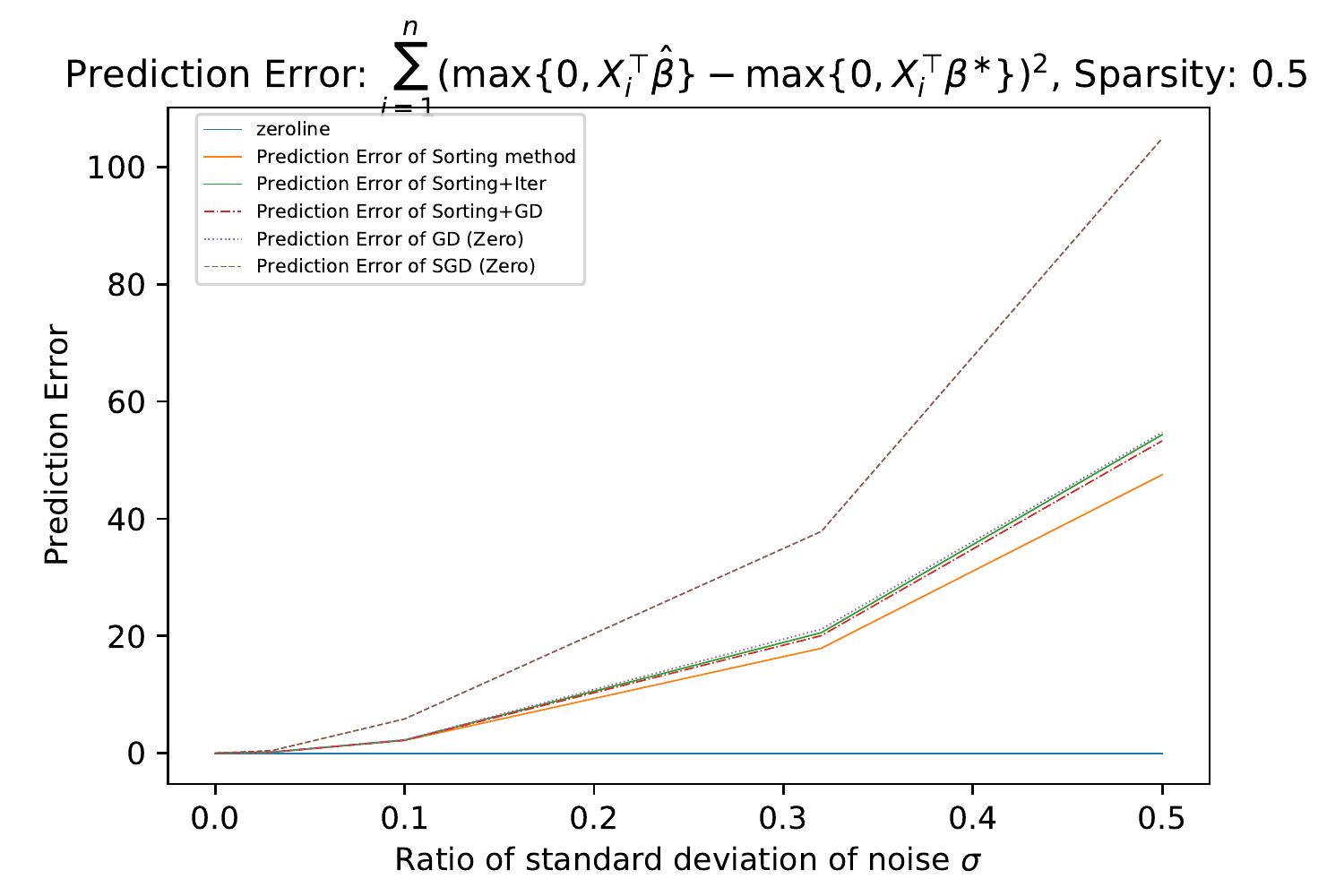}
        \caption[Network1]%
        {{\small Prediction Error}}    
        \label{fig:Figure_Compare_PE_20_400_50}
    \end{subfigure}
    \hfill
    \begin{subfigure}[b]{0.24\textwidth}   
        \centering 
        \includegraphics[width=\textwidth]{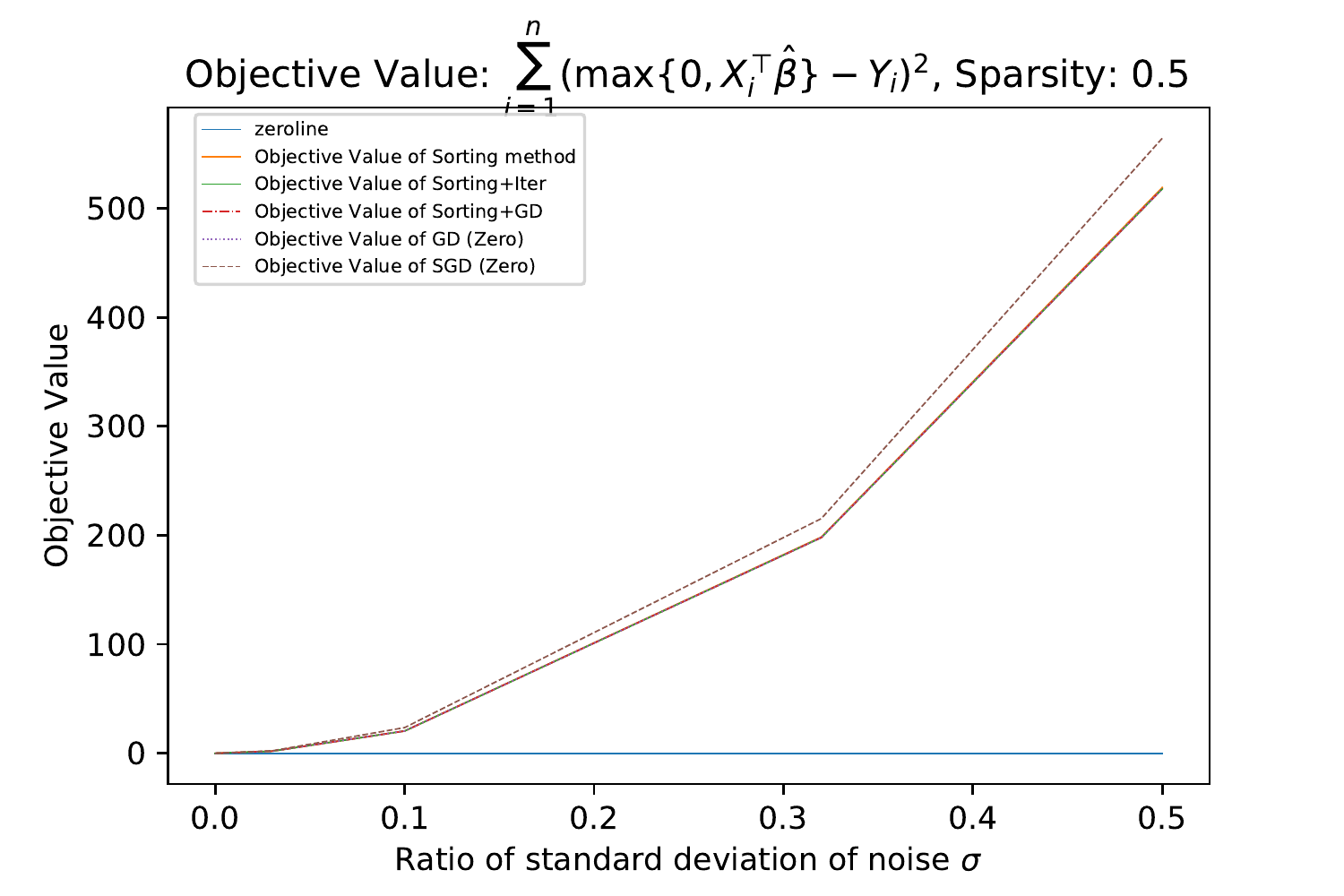}
        \caption[]%
        {{\small Objective Value}}    
        \label{fig:Figure_Compare_OB_20_400_50}
    \end{subfigure}
    \hfill
    \begin{subfigure}[b]{0.24\textwidth}  
        \centering 
        \includegraphics[width=\textwidth]{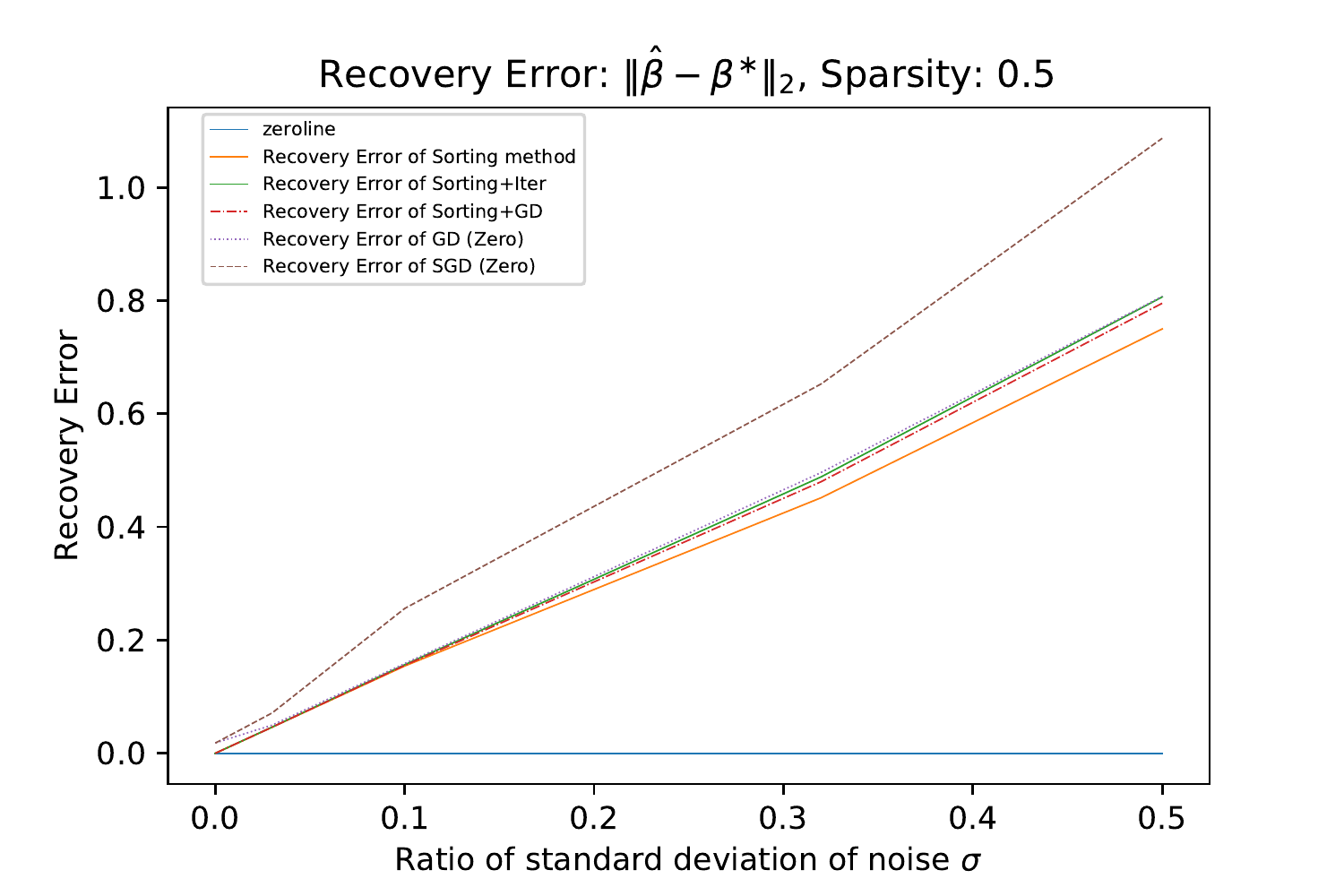}
        \caption[]%
        {{\small Recovery Error}}    
        \label{fig:Figure_Compare_RE_20_400_50}
    \end{subfigure}
    \hfill
    \begin{subfigure}[b]{0.24\textwidth}   
        \centering 
        \includegraphics[width=\textwidth]{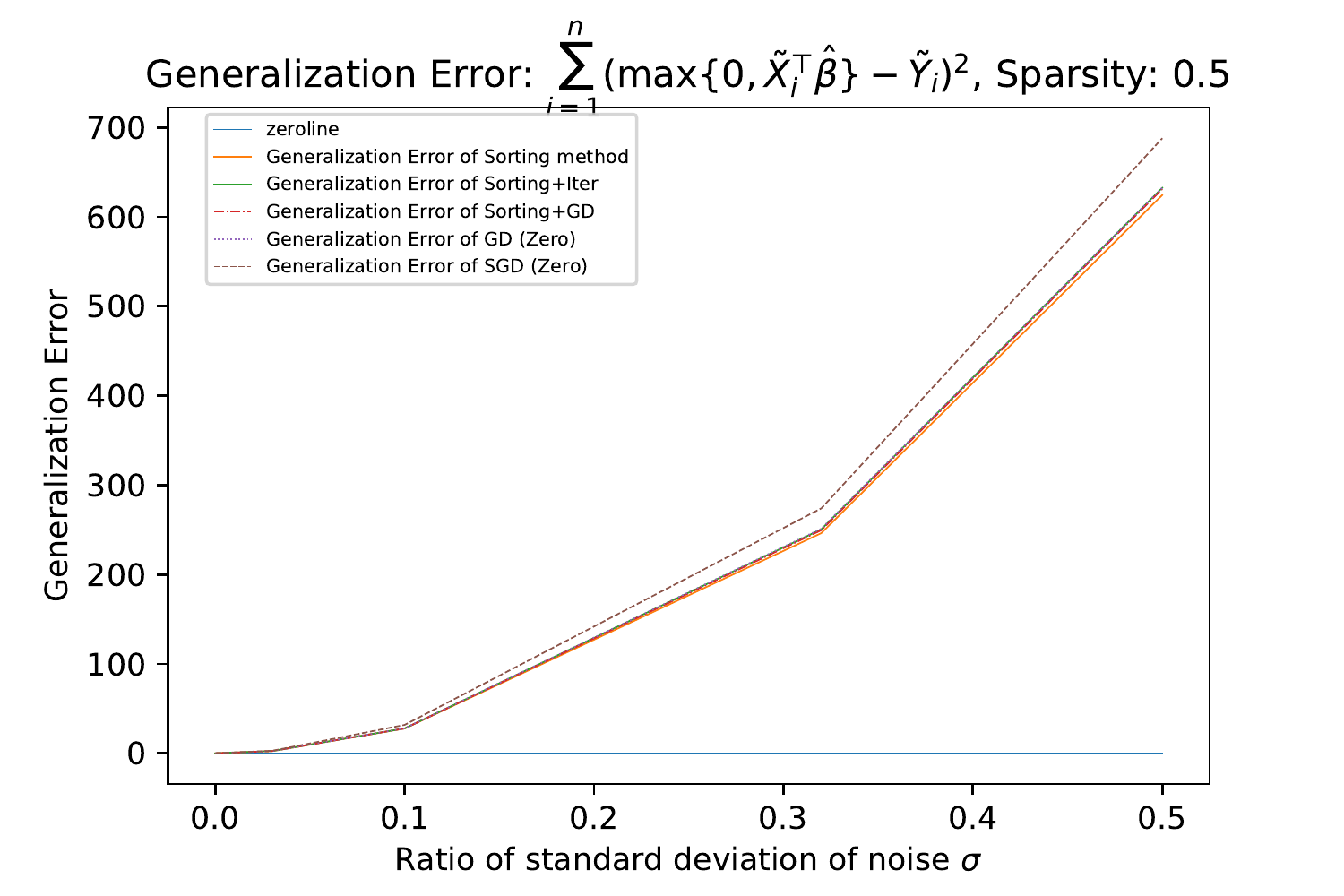}
        \caption[]%
        {{\small Generalization Error}}    
        \label{fig:Figure_Compare_GE_20_400_50}
    \end{subfigure}
    
    \vskip\baselineskip
    
    \centering
    \begin{subfigure}[b]{0.24\textwidth}
        \centering
        \includegraphics[width=\textwidth]{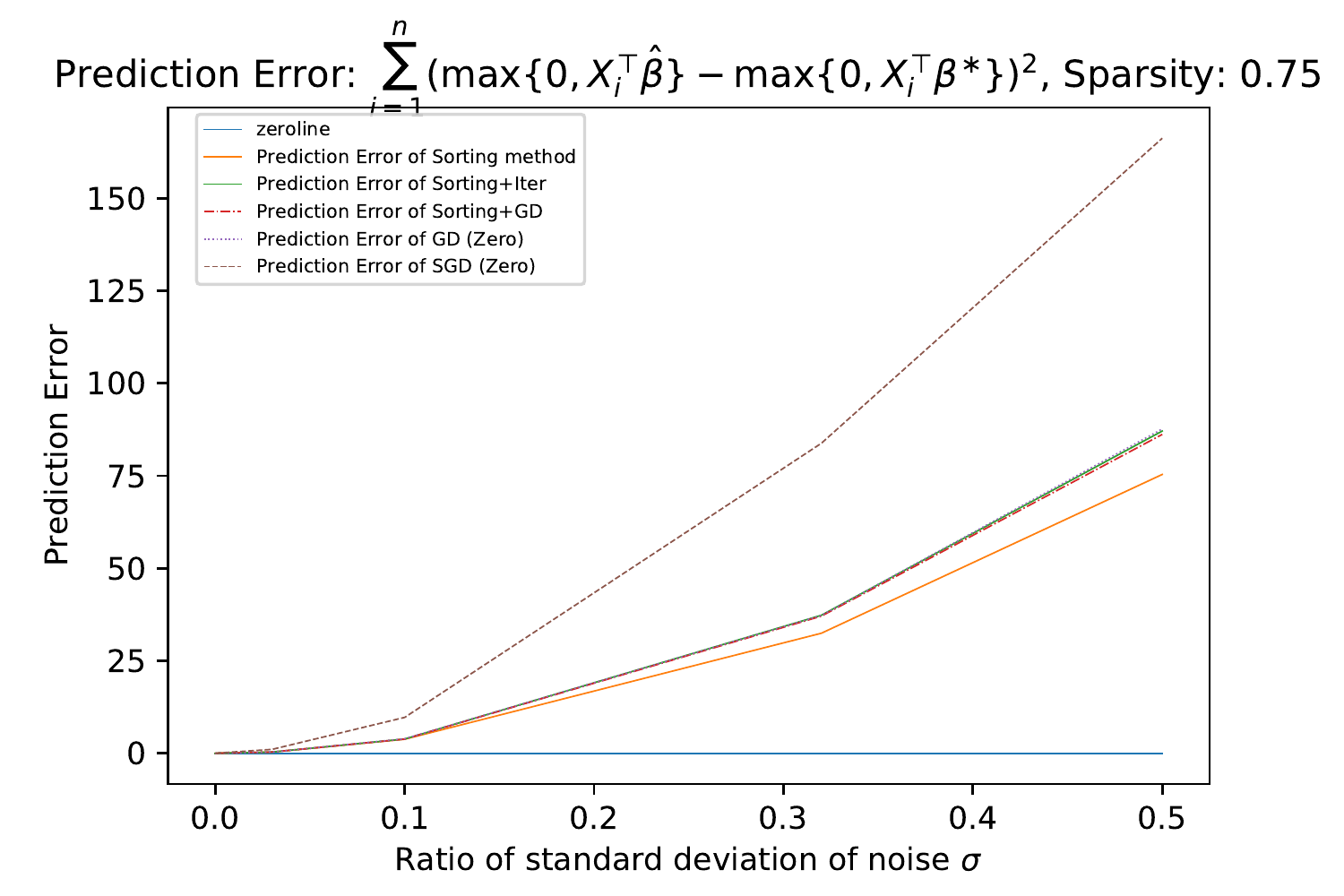}
        \caption[Network1]%
        {{\small Prediction Error}}    
        \label{fig:Figure_Compare_PE_20_400_75}
    \end{subfigure}
    \hfill
    \begin{subfigure}[b]{0.24\textwidth}   
        \centering 
        \includegraphics[width=\textwidth]{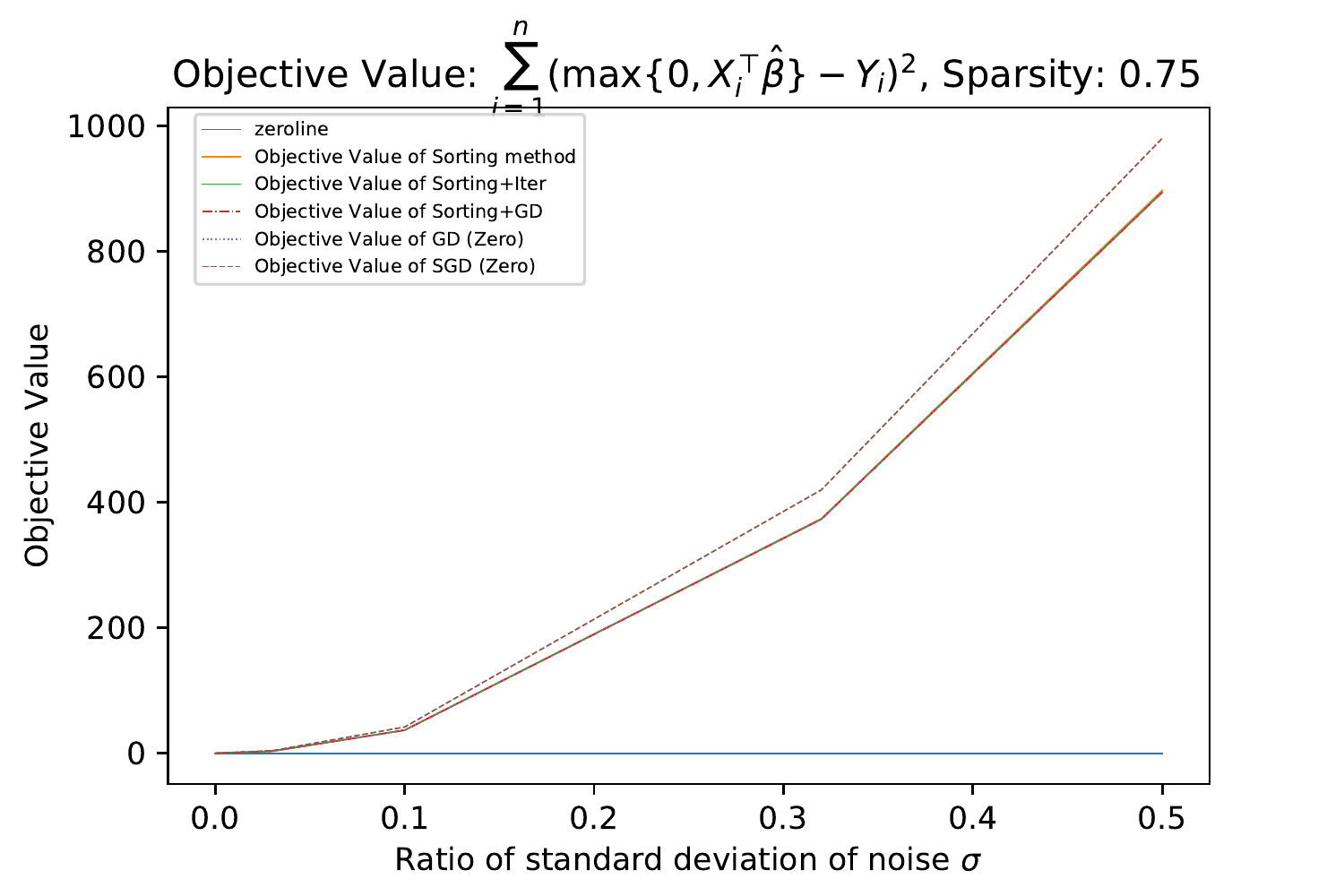}
        \caption[]%
        {{\small Objective Value}}    
        \label{fig:Figure_Compare_OB_20_400_75}
    \end{subfigure}
    \hfill
    \begin{subfigure}[b]{0.24\textwidth}  
        \centering 
        \includegraphics[width=\textwidth]{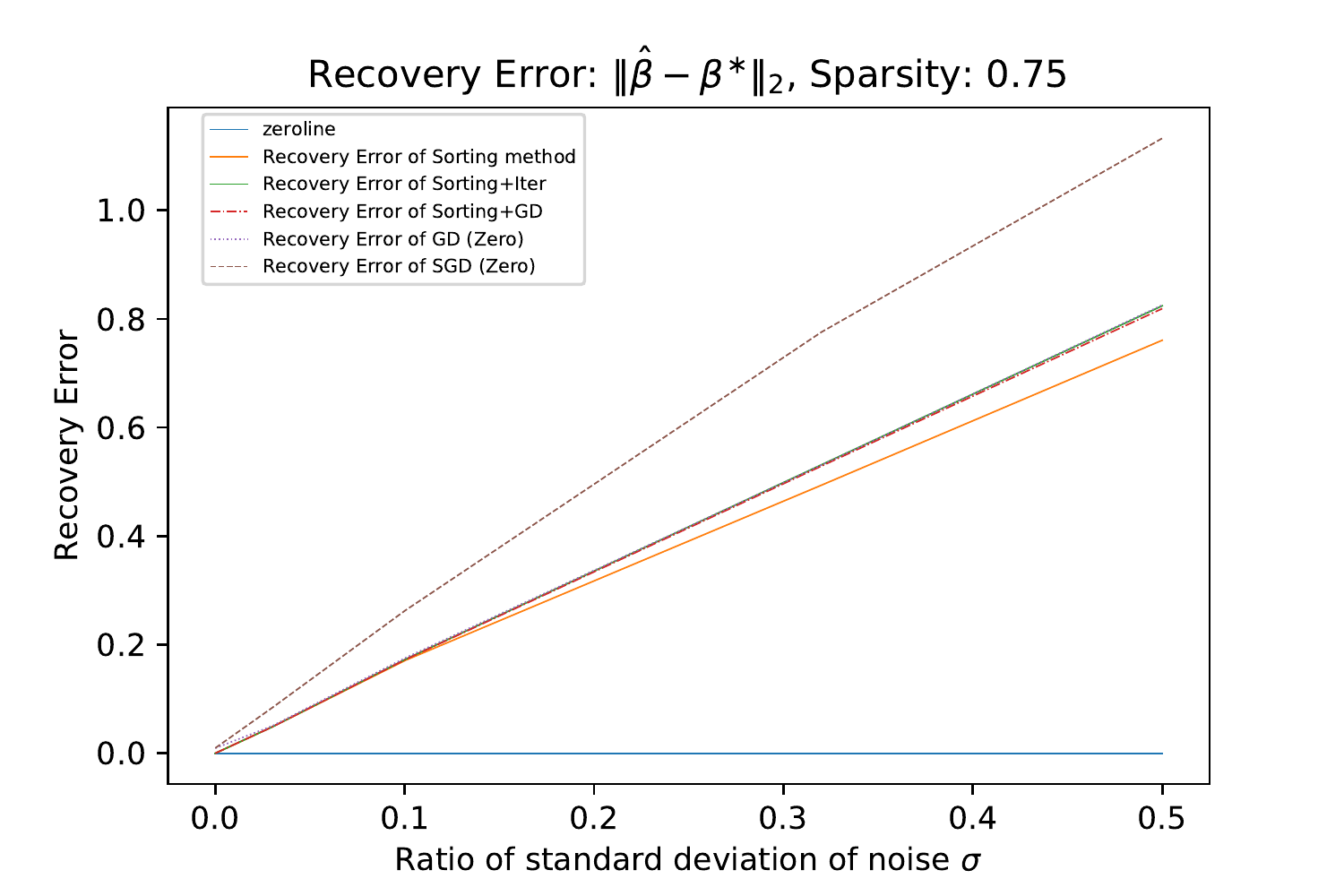}
        \caption[]%
        {{\small Recovery Error}}    
        \label{fig:Figure_Compare_RE_20_400_75}
    \end{subfigure}
    \hfill
    \begin{subfigure}[b]{0.24\textwidth}   
        \centering 
        \includegraphics[width=\textwidth]{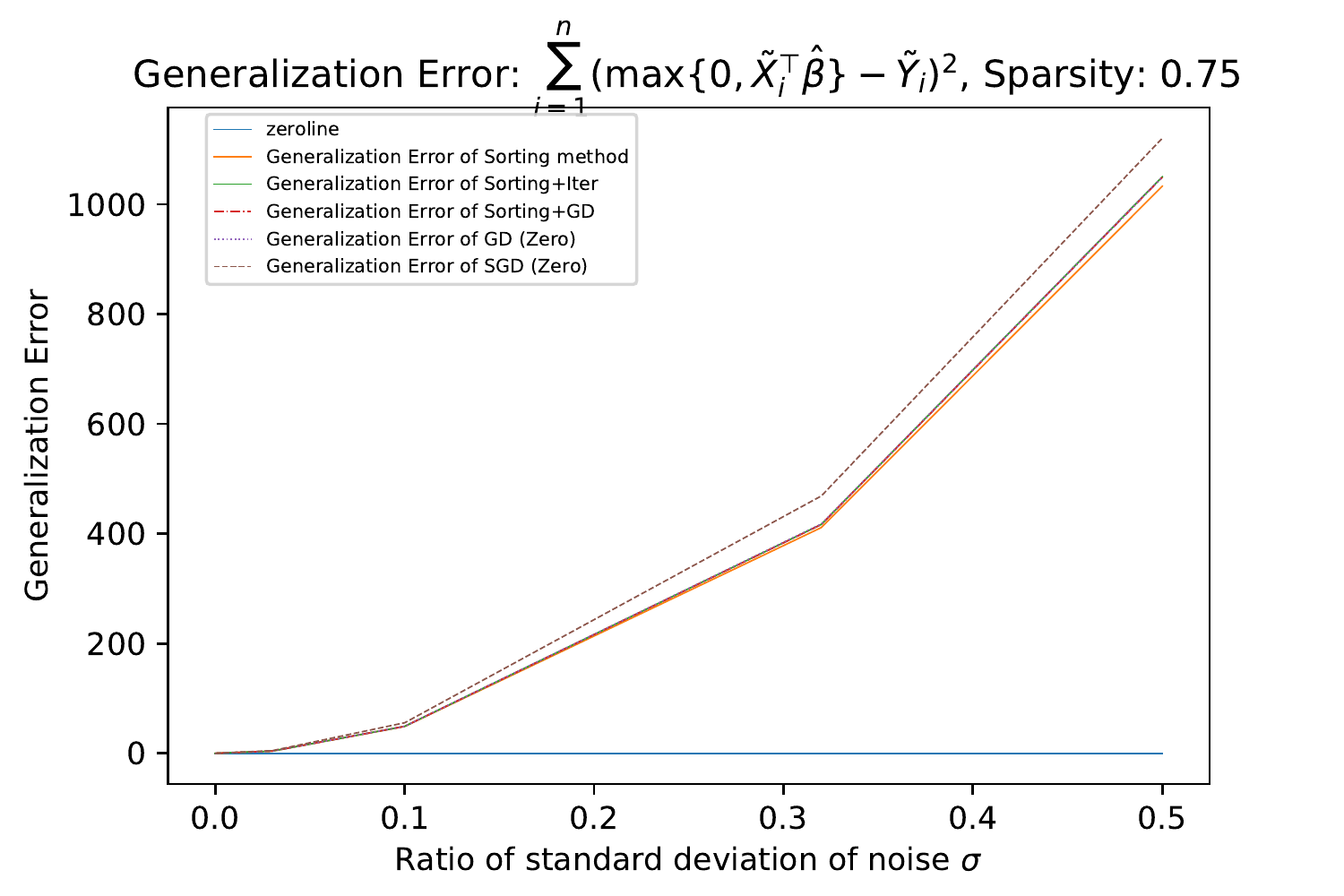}
        \caption[]%
        {{\small Generalization Error}}    
        \label{fig:Figure_Compare_GE_20_400_75}
    \end{subfigure}
    
    \vskip\baselineskip
    
    \centering
    \begin{subfigure}[b]{0.24\textwidth}
        \centering
        \includegraphics[width=\textwidth]{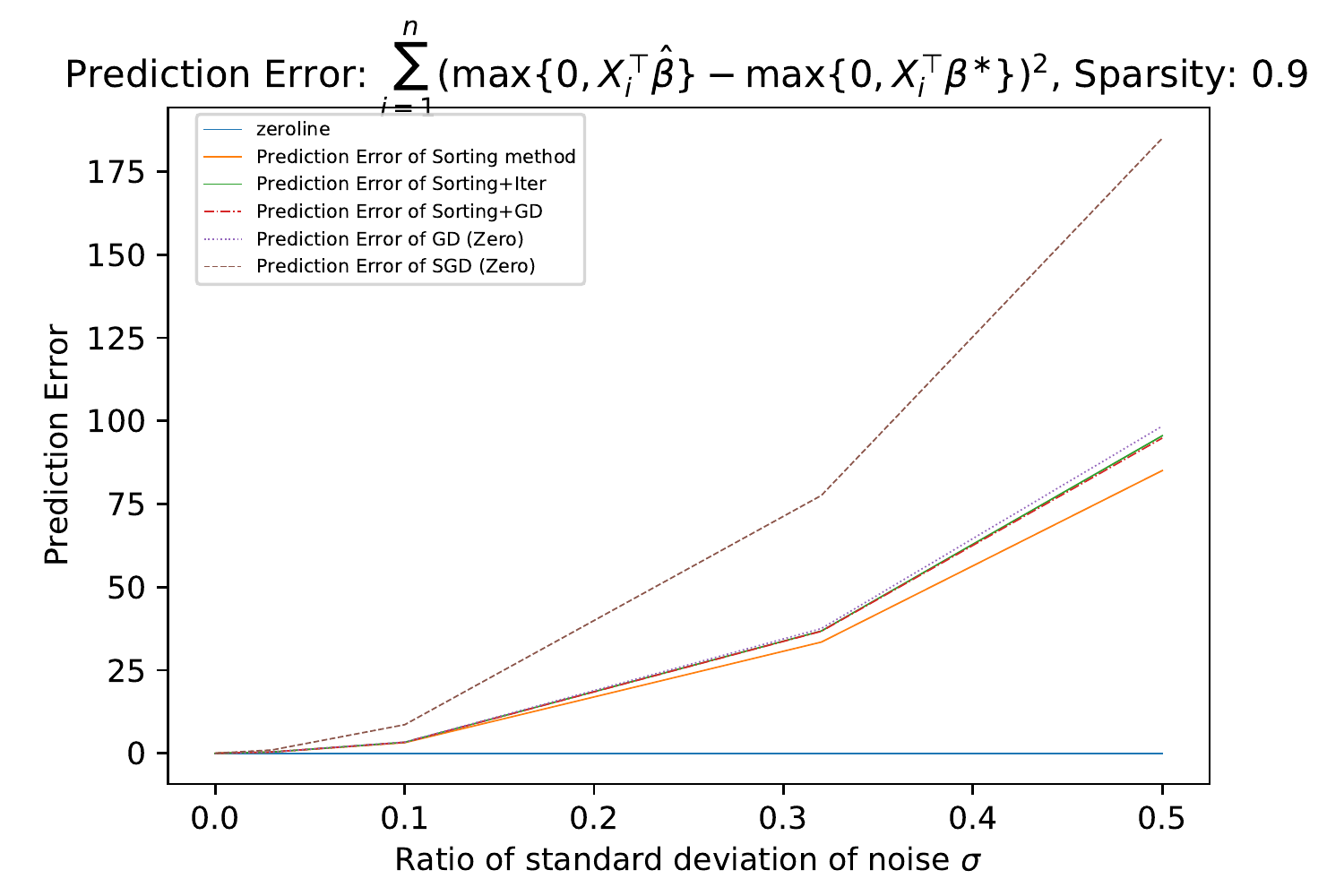}
        \caption[Network1]%
        {{\small Prediction Error}}    
        \label{fig:Figure_Compare_PE_20_400_90}
    \end{subfigure}
    \hfill
    \begin{subfigure}[b]{0.24\textwidth}   
        \centering 
        \includegraphics[width=\textwidth]{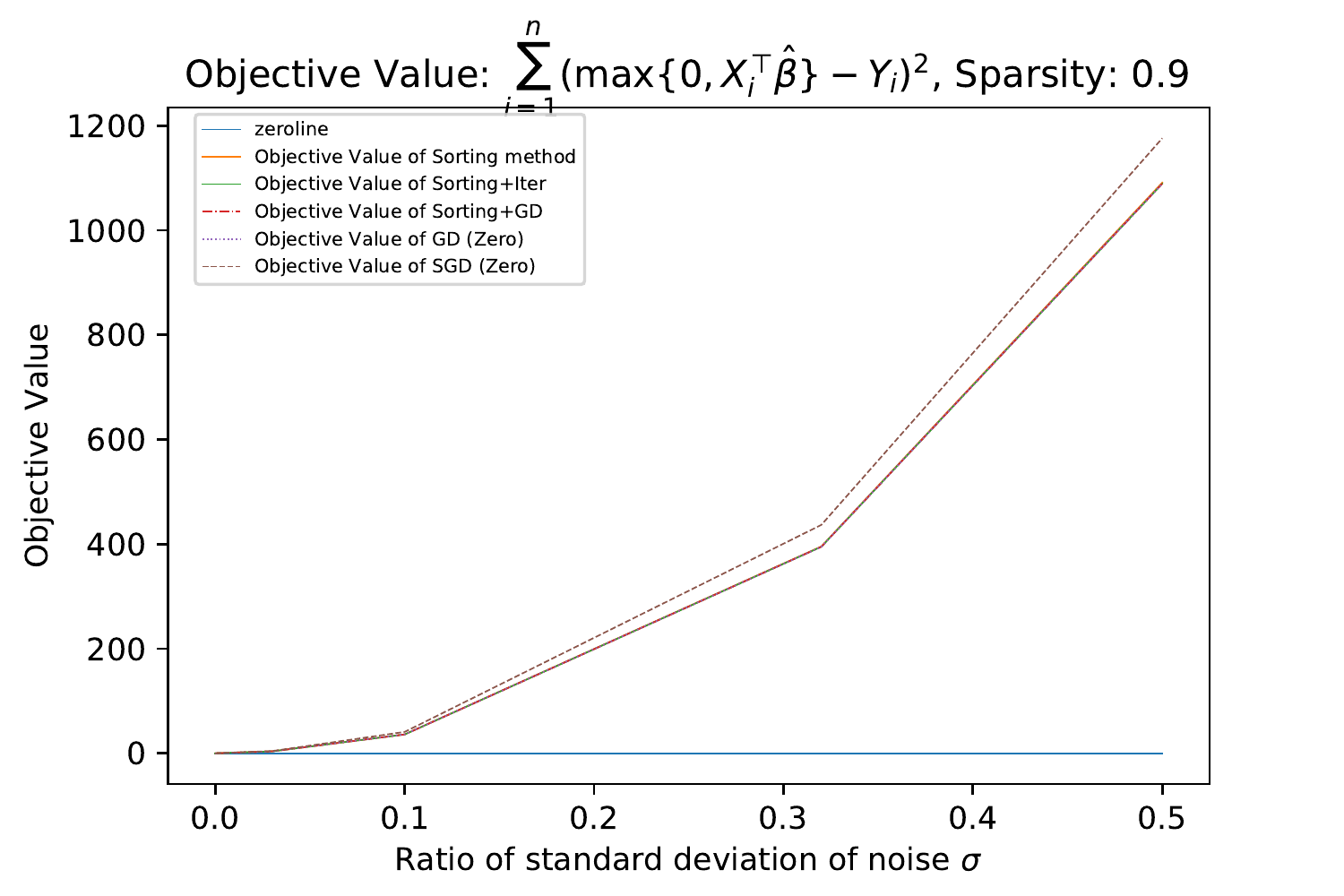}
        \caption[]%
        {{\small Objective Value}}    
        \label{fig:Figure_Compare_OB_20_400_90}
    \end{subfigure}
    \hfill
    \begin{subfigure}[b]{0.24\textwidth}  
        \centering 
        \includegraphics[width=\textwidth]{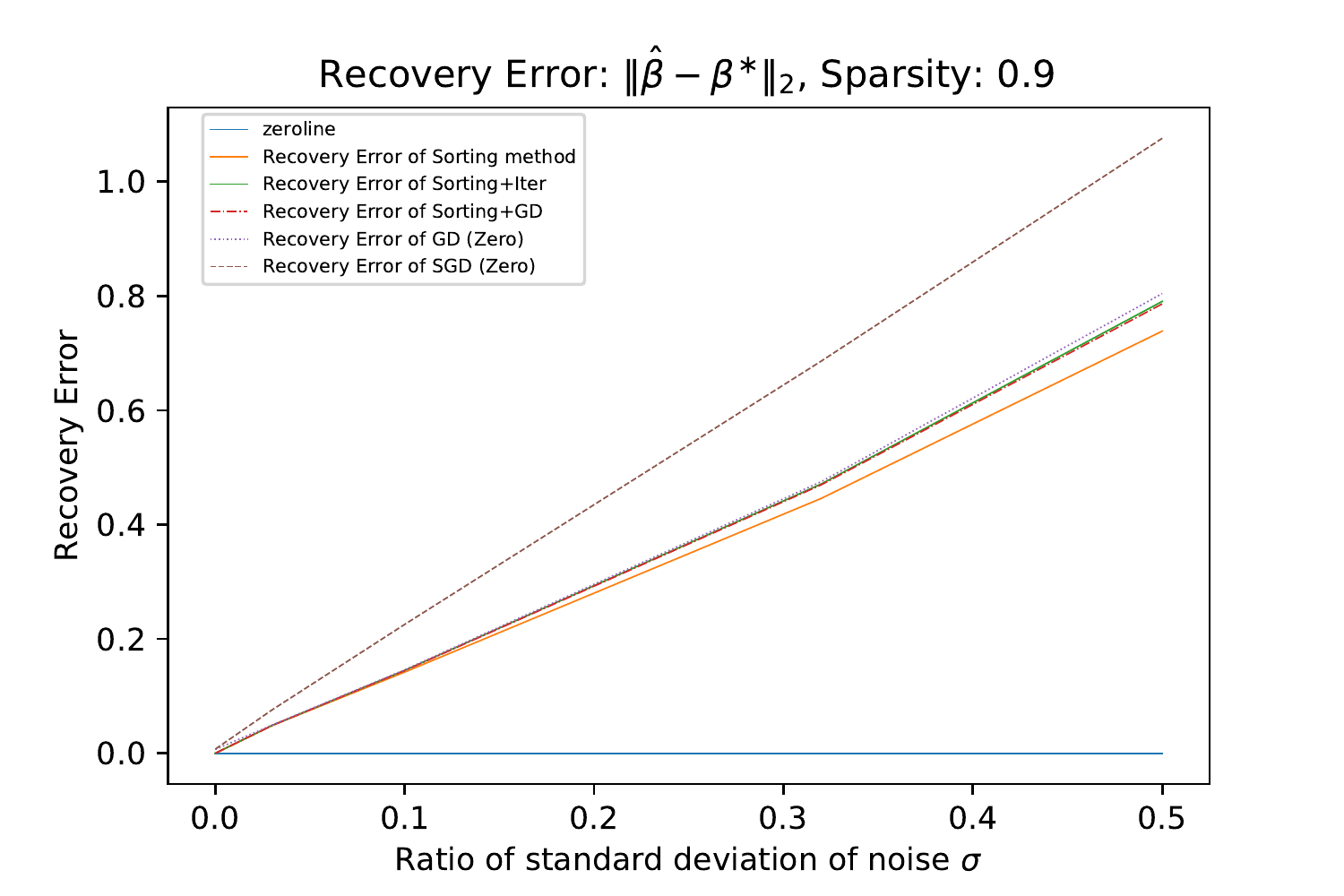}
        \caption[]%
        {{\small Recovery Error}}    
        \label{fig:Figure_Compare_RE_20_400_90}
    \end{subfigure}
    \hfill
    \begin{subfigure}[b]{0.24\textwidth}   
        \centering 
        \includegraphics[width=\textwidth]{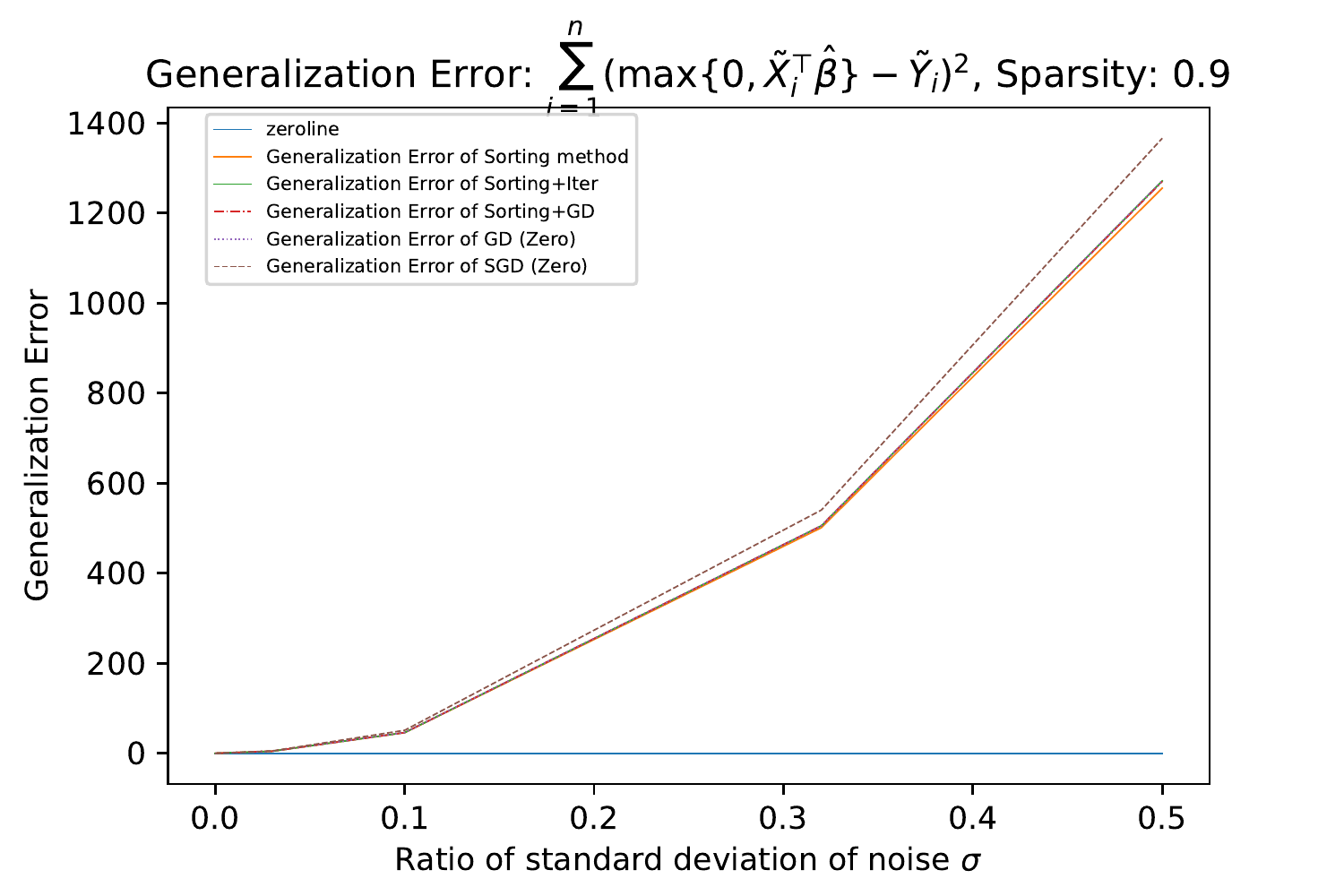}
        \caption[]%
        {{\small Generalization Error}}    
        \label{fig:Figure_Compare_GE_20_400_90}
    \end{subfigure}
    
    \caption[Numerical Results of sample size $(p, n)= (20, 400)$]
    {\small Numerical Results of sample size $(p, n) = (20, 400)$ and $\beta^{\ast} \sim N(0.5 \cdot \mathbf{1}_p, 20 \cdot I_p)$ with sparsity $\{0.1, 0.25, 0.5, 0.75, 0.9\}$} 
    \label{fig:Figure_Compare_20_400}
\end{figure*}

\section{Appendix: Realizable Cases} \label{section:realizable-cases}
Note that in realizable cases, since the observation samples $\{X_i\}_{i = 1}^n$ are constructed to guarantee the full column rank, i.e., the global optimal solution is unique, then finding a solution with 0 prediction error is equivalent to achieving 0 recovery error. In Figure~[\ref{fig:Figure_Compare_RE_20_400_10}, \ref{fig:Figure_Compare_RE_50_1000_10}, \ref{fig:Figure_Compare_RE_50_1000_25}], the averages of the recovery errors of realizable cases are not zero, however their corresponding prediction errors are very small, this may happen when the methods cannot find out the global optimal solutions. The details of realizable cases are presented in Table~[\ref{tab:realizable_10_200_10}, \ref{tab:realizable_10_200_25}, \ref{tab:realizable_10_200_50}, \ref{tab:realizable_10_200_75}, \ref{tab:realizable_10_200_90}, \ref{tab:realizable_20_400_10}, \ref{tab:realizable_20_400_25}, \ref{tab:realizable_20_400_50}, \ref{tab:realizable_20_400_75}, \ref{tab:realizable_20_400_90}, \ref{tab:realizable_50_1000_10}, \ref{tab:realizable_50_1000_25}, \ref{tab:realizable_50_1000_50}, \ref{tab:realizable_50_1000_75}, \ref{tab:realizable_50_1000_90}].

\begin{table}[ht]
\caption {\small Realizable Case $p = 10, n = 200, \text{sparsity} = 0.1$ with $\beta^{\ast} \sim N(0.5 \cdot \mathbf{1}_p, 10 \cdot I_p)$} \label{tab:realizable_10_200_10}
\begin{center}
\resizebox{0.7 \textwidth}{!}{\begin{tabular}{|l|l|l|l|l|l|l|l|l|l|l|} \hline
	\multirow{2}*{\bf Settings }& \multicolumn{2}{|l|}{\bf Sorting } & \multicolumn{2}{|l|}{\bf Sorting + Iter} & \multicolumn{2}{|l|}{\bf Sorting + GD } & \multicolumn{2}{|l|}{\bf GD } & \multicolumn{2}{|l|}{\bf SGD }  \\ \cline{2 - 11}
	& Prediction & Recovery & Prediction & Recovery & Prediction & Recovery &  Prediction & Recovery  & Prediction & Recovery \\ \specialrule{.15em}{.05em}{.05em} 

(10, 200, 0.0; 1) & 0.0 & 0.0 & 0.0 & 0.0 & 0.0 & 0.0 & 0.373 & 0.265 & 1.803 & 0.601 \\  \hline 
(10, 200, 0.0; 2) & 0.0 & 0.0 & 0.0 & 0.0 & 0.0 & 0.0 & 0.191 & 0.17 & 0.5 & 0.265 \\  \hline 
(10, 200, 0.0; 3) & 0.0 & 0.0 & 0.0 & 0.0 & 0.0 & 0.0 & 0.284 & 0.24 & 7.214 & 1.025 \\  \hline 
(10, 200, 0.0; 4) & 0.0 & 0.0 & 0.0 & 0.0 & 0.0 & 0.0 & 0.152 & 0.138 & 0.433 & 0.23 \\  \hline 
(10, 200, 0.0; 5) & 0.0 & 0.0 & 0.0 & 0.0 & 0.0 & 0.0 & 0.335 & 0.261 & 1.398 & 0.467 \\  \hline 
(10, 200, 0.0; 6) & 0.0 & 0.0 & 0.0 & 0.0 & 0.0 & 0.0 & 0.135 & 0.127 & 0.226 & 0.167 \\  \hline 
(10, 200, 0.0; 7) & 0.0 & 0.0 & 0.0 & 0.0 & 0.0 & 0.0 & 0.165 & 0.149 & 0.461 & 0.226 \\  \hline 
(10, 200, 0.0; 8) & 0.0 & 0.0 & 0.0 & 0.0 & 0.0 & 0.0 & 0.145 & 0.141 & 0.334 & 0.2 \\  \hline 
(10, 200, 0.0; 9) & 0.0 & 0.0 & 0.0 & 0.0 & 0.0 & 0.0 & 0.196 & 0.161 & 0.745 & 0.299 \\  \hline 
(10, 200, 0.0; 10) & 0.0 & 0.0 & 0.0 & 0.0 & 0.0 & 0.0 & 0.426 & 0.294 & 39.097 & 2.593 \\  \hline 
(10, 200, 0.0; 11) & 0.0 & 0.0 & 0.0 & 0.0 & 0.0 & 0.0 & 0.306 & 0.24 & 0.588 & 0.314 \\  \hline 
(10, 200, 0.0; 12) & 0.0 & 0.0 & 0.0 & 0.0 & 0.0 & 0.0 & 0.339 & 0.281 & 4.137 & 0.832 \\  \hline 
(10, 200, 0.0; 13) & 0.0 & 0.0 & 0.0 & 0.0 & 0.0 & 0.0 & 0.189 & 0.165 & 0.191 & 0.168 \\  \hline 
(10, 200, 0.0; 14) & 0.0 & 0.0 & 0.0 & 0.0 & 0.0 & 0.0 & 0.386 & 0.25 & 0.619 & 0.297 \\  \hline 
(10, 200, 0.0; 15) & 0.0 & 0.0 & 0.0 & 0.0 & 0.0 & 0.0 & 0.192 & 0.16 & 0.664 & 0.291 \\  \hline 
(10, 200, 0.0; 16) & 0.0 & 0.0 & 0.0 & 0.0 & 0.0 & 0.0 & 0.291 & 0.218 & 1.646 & 0.501 \\  \hline 
(10, 200, 0.0; 17) & 0.0 & 0.0 & 0.0 & 0.0 & 0.0 & 0.0 & 0.205 & 0.186 & 0.66 & 0.312 \\  \hline 
(10, 200, 0.0; 18) & 0.0 & 0.0 & 0.0 & 0.0 & 0.0 & 0.0 & 0.18 & 0.161 & 1.234 & 0.421 \\  \hline 
(10, 200, 0.0; 19) & 0.0 & 0.0 & 0.0 & 0.0 & 0.0 & 0.0 & 0.219 & 0.184 & 0.715 & 0.333 \\  \hline 
(10, 200, 0.0; 20) & 0.0 & 0.0 & 0.0 & 0.0 & 0.0 & 0.0 & 0.142 & 0.136 & 12.466 & 1.2 \\  \hline 

\end{tabular}}
\end{center}
\end{table}

\begin{table}[ht]
\caption {\small Realizable Case $p = 10, n = 200, \text{sparsity} = 0.25$ with $\beta^{\ast} \sim N(0.5 \cdot \mathbf{1}_p, 10 \cdot I_p)$} \label{tab:realizable_10_200_25}
\begin{center}
\resizebox{0.7 \textwidth}{!}{\begin{tabular}{|l|l|l|l|l|l|l|l|l|l|l|} \hline
	\multirow{2}*{\bf Settings }& \multicolumn{2}{|l|}{\bf Sorting } & \multicolumn{2}{|l|}{\bf Sorting + Iter} & \multicolumn{2}{|l|}{\bf Sorting + GD } & \multicolumn{2}{|l|}{\bf GD } & \multicolumn{2}{|l|}{\bf SGD }  \\ \cline{2 - 11}
	& Prediction & Recovery & Prediction & Recovery & Prediction & Recovery &  Prediction & Recovery  & Prediction & Recovery \\ \specialrule{.15em}{.05em}{.05em} 

(10, 200, 0.0; 1) & 0.0 & 0.0 & 0.0 & 0.0 & 0.0 & 0.0 & 0.049 & 0.058 & 0.107 & 0.084 \\  \hline 
(10, 200, 0.0; 2) & 0.0 & 0.0 & 0.0 & 0.0 & 0.0 & 0.0 & 0.045 & 0.054 & 0.088 & 0.073 \\  \hline 
(10, 200, 0.0; 3) & 0.0 & 0.0 & 0.0 & 0.0 & 0.0 & 0.0 & 0.09 & 0.088 & 0.147 & 0.111 \\  \hline 
(10, 200, 0.0; 4) & 0.0 & 0.0 & 0.0 & 0.0 & 0.0 & 0.0 & 0.025 & 0.037 & 0.068 & 0.065 \\  \hline 
(10, 200, 0.0; 5) & 0.0 & 0.0 & 0.0 & 0.0 & 0.0 & 0.0 & 0.06 & 0.066 & 0.929 & 0.247 \\  \hline 
(10, 200, 0.0; 6) & 0.0 & 0.0 & 0.0 & 0.0 & 0.0 & 0.0 & 0.046 & 0.058 & 0.098 & 0.087 \\  \hline 
(10, 200, 0.0; 7) & 0.0 & 0.0 & 0.0 & 0.0 & 0.0 & 0.0 & 0.044 & 0.055 & 0.065 & 0.06 \\  \hline 
(10, 200, 0.0; 8) & 0.0 & 0.0 & 0.0 & 0.0 & 0.0 & 0.0 & 0.062 & 0.07 & 0.11 & 0.085 \\  \hline 
(10, 200, 0.0; 9) & 0.0 & 0.0 & 0.0 & 0.0 & 0.0 & 0.0 & 0.05 & 0.059 & 0.197 & 0.111 \\  \hline 
(10, 200, 0.0; 10) & 0.0 & 0.0 & 0.0 & 0.0 & 0.0 & 0.0 & 0.083 & 0.084 & 0.061 & 0.069 \\  \hline 
(10, 200, 0.0; 11) & 0.0 & 0.0 & 0.0 & 0.0 & 0.0 & 0.0 & 0.08 & 0.087 & 0.125 & 0.099 \\  \hline 
(10, 200, 0.0; 12) & 0.0 & 0.0 & 0.0 & 0.0 & 0.0 & 0.0 & 0.056 & 0.064 & 0.063 & 0.067 \\  \hline 
(10, 200, 0.0; 13) & 0.0 & 0.0 & 0.0 & 0.0 & 0.0 & 0.0 & 0.054 & 0.065 & 0.083 & 0.072 \\  \hline 
(10, 200, 0.0; 14) & 0.0 & 0.0 & 0.0 & 0.0 & 0.0 & 0.0 & 0.042 & 0.053 & 0.285 & 0.115 \\  \hline 
(10, 200, 0.0; 15) & 0.0 & 0.0 & 0.0 & 0.0 & 0.0 & 0.0 & 0.032 & 0.039 & 0.032 & 0.037 \\  \hline 
(10, 200, 0.0; 16) & 0.0 & 0.0 & 0.0 & 0.0 & 0.0 & 0.0 & 0.073 & 0.081 & 0.083 & 0.072 \\  \hline 
(10, 200, 0.0; 17) & 0.0 & 0.0 & 0.0 & 0.0 & 0.0 & 0.0 & 0.05 & 0.057 & 0.071 & 0.073 \\  \hline 
(10, 200, 0.0; 18) & 0.0 & 0.0 & 0.0 & 0.0 & 0.0 & 0.0 & 0.035 & 0.044 & 0.099 & 0.081 \\  \hline 
(10, 200, 0.0; 19) & 0.0 & 0.0 & 0.0 & 0.0 & 0.0 & 0.0 & 0.038 & 0.048 & 0.107 & 0.077 \\  \hline 
(10, 200, 0.0; 20) & 0.0 & 0.0 & 0.0 & 0.0 & 0.0 & 0.0 & 0.071 & 0.076 & 0.207 & 0.124 \\  \hline

\end{tabular}}
\end{center}
\end{table}

\begin{table}[ht]
\caption {\small Realizable Case $p = 10, n = 200, \text{sparsity} = 0.5$ with $\beta^{\ast} \sim N(0.5 \cdot \mathbf{1}_p, 10 \cdot I_p)$} \label{tab:realizable_10_200_50}
\begin{center}
\resizebox{0.7 \textwidth}{!}{\begin{tabular}{|l|l|l|l|l|l|l|l|l|l|l|} \hline
	\multirow{2}*{\bf Settings }& \multicolumn{2}{|l|}{\bf Sorting } & \multicolumn{2}{|l|}{\bf Sorting + Iter} & \multicolumn{2}{|l|}{\bf Sorting + GD } & \multicolumn{2}{|l|}{\bf GD } & \multicolumn{2}{|l|}{\bf SGD }  \\ \cline{2 - 11}
	& Prediction & Recovery & Prediction & Recovery & Prediction & Recovery &  Prediction & Recovery  & Prediction & Recovery \\ \specialrule{.15em}{.05em}{.05em} 
(10, 200, 0.0; 1) & 0.0 & 0.0 & 0.0 & 0.0 & 0.0 & 0.0 & 0.015 & 0.022 & 0.043 & 0.031 \\  \hline 
(10, 200, 0.0; 2) & 0.0 & 0.0 & 0.0 & 0.0 & 0.0 & 0.0 & 0.01 & 0.017 & 0.013 & 0.02 \\  \hline 
(10, 200, 0.0; 3) & 0.0 & 0.0 & 0.0 & 0.0 & 0.0 & 0.0 & 0.016 & 0.025 & 0.03 & 0.029 \\  \hline 
(10, 200, 0.0; 4) & 0.0 & 0.0 & 0.0 & 0.0 & 0.0 & 0.0 & 0.009 & 0.017 & 0.025 & 0.026 \\  \hline 
(10, 200, 0.0; 5) & 0.0 & 0.0 & 0.0 & 0.0 & 0.0 & 0.0 & 0.017 & 0.024 & 0.019 & 0.024 \\  \hline 
(10, 200, 0.0; 6) & 0.0 & 0.0 & 0.0 & 0.0 & 0.0 & 0.0 & 0.015 & 0.022 & 0.017 & 0.02 \\  \hline 
(10, 200, 0.0; 7) & 0.0 & 0.0 & 0.0 & 0.0 & 0.0 & 0.0 & 0.014 & 0.023 & 0.075 & 0.046 \\  \hline 
(10, 200, 0.0; 8) & 0.0 & 0.0 & 0.0 & 0.0 & 0.0 & 0.0 & 0.02 & 0.028 & 0.016 & 0.024 \\  \hline 
(10, 200, 0.0; 9) & 0.0 & 0.0 & 0.0 & 0.0 & 0.0 & 0.0 & 0.025 & 0.035 & 0.048 & 0.044 \\  \hline 
(10, 200, 0.0; 10) & 0.0 & 0.0 & 0.0 & 0.0 & 0.0 & 0.0 & 0.015 & 0.024 & 0.017 & 0.026 \\  \hline 
(10, 200, 0.0; 11) & 0.0 & 0.0 & 0.0 & 0.0 & 0.0 & 0.0 & 0.015 & 0.024 & 0.002 & 0.006 \\  \hline 
(10, 200, 0.0; 12) & 0.0 & 0.0 & 0.0 & 0.0 & 0.0 & 0.0 & 0.008 & 0.016 & 0.007 & 0.013 \\  \hline 
(10, 200, 0.0; 13) & 0.0 & 0.0 & 0.0 & 0.0 & 0.0 & 0.0 & 0.015 & 0.024 & 0.018 & 0.022 \\  \hline 
(10, 200, 0.0; 14) & 0.0 & 0.0 & 0.0 & 0.0 & 0.0 & 0.0 & 0.008 & 0.015 & 0.016 & 0.018 \\  \hline 
(10, 200, 0.0; 15) & 0.0 & 0.0 & 0.0 & 0.0 & 0.0 & 0.0 & 0.014 & 0.023 & 0.009 & 0.015 \\  \hline 
(10, 200, 0.0; 16) & 0.0 & 0.0 & 0.0 & 0.0 & 0.0 & 0.0 & 0.007 & 0.014 & 0.024 & 0.036 \\  \hline 
(10, 200, 0.0; 17) & 0.0 & 0.0 & 0.0 & 0.0 & 0.0 & 0.0 & 0.025 & 0.034 & 0.008 & 0.018 \\  \hline 
(10, 200, 0.0; 18) & 0.0 & 0.0 & 0.0 & 0.0 & 0.0 & 0.0 & 0.011 & 0.02 & 0.026 & 0.027 \\  \hline 
(10, 200, 0.0; 19) & 0.0 & 0.0 & 0.0 & 0.0 & 0.0 & 0.0 & 0.011 & 0.019 & 0.007 & 0.015 \\  \hline 
(10, 200, 0.0; 20) & 0.0 & 0.0 & 0.0 & 0.0 & 0.0 & 0.0 & 0.016 & 0.025 & 0.51 & 0.12 \\  \hline 

\end{tabular}}
\end{center}
\end{table}

\begin{table}[ht]
\caption {\small Realizable Case $p = 10, n = 200, \text{sparsity} = 0.75$ with $\beta^{\ast} \sim N(0.5 \cdot \mathbf{1}_p, 10 \cdot I_p)$} \label{tab:realizable_10_200_75}
\begin{center}
\resizebox{0.7 \textwidth}{!}{\begin{tabular}{|l|l|l|l|l|l|l|l|l|l|l|} \hline
	\multirow{2}*{\bf Settings }& \multicolumn{2}{|l|}{\bf Sorting } & \multicolumn{2}{|l|}{\bf Sorting + Iter} & \multicolumn{2}{|l|}{\bf Sorting + GD } & \multicolumn{2}{|l|}{\bf GD } & \multicolumn{2}{|l|}{\bf SGD }  \\ \cline{2 - 11}
	& Prediction & Recovery & Prediction & Recovery & Prediction & Recovery &  Prediction & Recovery  & Prediction & Recovery \\ \specialrule{.15em}{.05em}{.05em} 

(10, 200, 0.0; 1) & 0.0 & 0.0 & 0.0 & 0.0 & 0.0 & 0.0 & 0.002 & 0.006 & 0.006 & 0.009 \\  \hline 
(10, 200, 0.0; 2) & 0.0 & 0.0 & 0.0 & 0.0 & 0.0 & 0.0 & 0.008 & 0.016 & 0.009 & 0.012 \\  \hline 
(10, 200, 0.0; 3) & 0.0 & 0.0 & 0.0 & 0.0 & 0.0 & 0.0 & 0.005 & 0.01 & 0.01 & 0.013 \\  \hline 
(10, 200, 0.0; 4) & 0.0 & 0.0 & 0.0 & 0.0 & 0.0 & 0.0 & 0.004 & 0.009 & 0.005 & 0.009 \\  \hline 
(10, 200, 0.0; 5) & 0.0 & 0.0 & 0.0 & 0.0 & 0.0 & 0.0 & 0.006 & 0.012 & 0.019 & 0.017 \\  \hline 
(10, 200, 0.0; 6) & 0.0 & 0.0 & 0.0 & 0.0 & 0.0 & 0.0 & 0.004 & 0.009 & 0.013 & 0.015 \\  \hline 
(10, 200, 0.0; 7) & 0.0 & 0.0 & 0.0 & 0.0 & 0.0 & 0.0 & 0.005 & 0.011 & 0.031 & 0.019 \\  \hline 
(10, 200, 0.0; 8) & 0.0 & 0.0 & 0.0 & 0.0 & 0.0 & 0.0 & 0.001 & 0.005 & 0.006 & 0.009 \\  \hline 
(10, 200, 0.0; 9) & 0.0 & 0.0 & 0.0 & 0.0 & 0.0 & 0.0 & 0.005 & 0.01 & 0.005 & 0.008 \\  \hline 
(10, 200, 0.0; 10) & 0.0 & 0.0 & 0.0 & 0.0 & 0.0 & 0.0 & 0.007 & 0.012 & 0.008 & 0.012 \\  \hline 
(10, 200, 0.0; 11) & 0.0 & 0.0 & 0.0 & 0.0 & 0.0 & 0.0 & 0.005 & 0.011 & 0.009 & 0.012 \\  \hline 
(10, 200, 0.0; 12) & 0.0 & 0.0 & 0.0 & 0.0 & 0.0 & 0.0 & 0.001 & 0.005 & 0.003 & 0.007 \\  \hline 
(10, 200, 0.0; 13) & 0.0 & 0.0 & 0.0 & 0.0 & 0.0 & 0.0 & 0.007 & 0.014 & 0.001 & 0.004 \\  \hline 
(10, 200, 0.0; 14) & 0.0 & 0.0 & 0.0 & 0.0 & 0.0 & 0.0 & 0.003 & 0.007 & 0.007 & 0.009 \\  \hline 
(10, 200, 0.0; 15) & 0.0 & 0.0 & 0.0 & 0.0 & 0.0 & 0.0 & 0.003 & 0.008 & 0.007 & 0.01 \\  \hline 
(10, 200, 0.0; 16) & 0.0 & 0.0 & 0.0 & 0.0 & 0.0 & 0.0 & 0.004 & 0.009 & 0.012 & 0.012 \\  \hline 
(10, 200, 0.0; 17) & 0.0 & 0.0 & 0.0 & 0.0 & 0.0 & 0.0 & 0.008 & 0.014 & 0.85 & 0.101 \\  \hline 
(10, 200, 0.0; 18) & 0.0 & 0.0 & 0.0 & 0.0 & 0.0 & 0.0 & 0.006 & 0.013 & 0.017 & 0.019 \\  \hline 
(10, 200, 0.0; 19) & 0.0 & 0.0 & 0.0 & 0.0 & 0.0 & 0.0 & 0.003 & 0.008 & 0.006 & 0.009 \\  \hline 
(10, 200, 0.0; 20) & 0.0 & 0.0 & 0.0 & 0.0 & 0.0 & 0.0 & 0.01 & 0.016 & 0.014 & 0.013 \\  \hline

\end{tabular}}
\end{center}
\end{table}

\begin{table}[ht]
\caption {\small Realizable Case $p = 10, n = 200, \text{sparsity} = 0.9$ with $\beta^{\ast} \sim N(0.5 \cdot \mathbf{1}_p, 10 \cdot I_p)$} \label{tab:realizable_10_200_90}
\begin{center}
\resizebox{0.7 \textwidth}{!}{\begin{tabular}{|l|l|l|l|l|l|l|l|l|l|l|} \hline
	\multirow{2}*{\bf Settings }& \multicolumn{2}{|l|}{\bf Sorting } & \multicolumn{2}{|l|}{\bf Sorting + Iter} & \multicolumn{2}{|l|}{\bf Sorting + GD } & \multicolumn{2}{|l|}{\bf GD } & \multicolumn{2}{|l|}{\bf SGD }  \\ \cline{2 - 11}
	& Prediction & Recovery & Prediction & Recovery & Prediction & Recovery &  Prediction & Recovery  & Prediction & Recovery \\ \specialrule{.15em}{.05em}{.05em} 

(10, 200, 0.0; 1) & 0.0 & 0.0 & 0.0 & 0.0 & 0.0 & 0.0 & 0.006 & 0.011 & 0.005 & 0.01 \\  \hline 
(10, 200, 0.0; 2) & 0.0 & 0.0 & 0.0 & 0.0 & 0.0 & 0.0 & 0.009 & 0.016 & 0.028 & 0.02 \\  \hline 
(10, 200, 0.0; 3) & 0.0 & 0.0 & 0.0 & 0.0 & 0.0 & 0.0 & 0.005 & 0.011 & 0.044 & 0.022 \\  \hline 
(10, 200, 0.0; 4) & 0.0 & 0.0 & 0.0 & 0.0 & 0.0 & 0.0 & 0.001 & 0.002 & 0.012 & 0.013 \\  \hline 
(10, 200, 0.0; 5) & 0.0 & 0.0 & 0.0 & 0.0 & 0.0 & 0.0 & 0.004 & 0.009 & 0.002 & 0.005 \\  \hline 
(10, 200, 0.0; 6) & 0.0 & 0.0 & 0.0 & 0.0 & 0.0 & 0.0 & 0.0 & 0.001 & 0.007 & 0.008 \\  \hline 
(10, 200, 0.0; 7) & 0.0 & 0.0 & 0.0 & 0.0 & 0.0 & 0.0 & 0.005 & 0.011 & 0.019 & 0.019 \\  \hline 
(10, 200, 0.0; 8) & 0.0 & 0.0 & 0.0 & 0.0 & 0.0 & 0.0 & 0.004 & 0.009 & 0.011 & 0.013 \\  \hline 
(10, 200, 0.0; 9) & 0.0 & 0.0 & 0.0 & 0.0 & 0.0 & 0.0 & 0.001 & 0.005 & 0.006 & 0.008 \\  \hline 
(10, 200, 0.0; 10) & 0.0 & 0.0 & 0.0 & 0.0 & 0.0 & 0.0 & 0.002 & 0.006 & 0.007 & 0.009 \\  \hline 
(10, 200, 0.0; 11) & 0.0 & 0.0 & 0.0 & 0.0 & 0.0 & 0.0 & 0.002 & 0.006 & 0.004 & 0.007 \\  \hline 
(10, 200, 0.0; 12) & 0.0 & 0.0 & 0.0 & 0.0 & 0.0 & 0.0 & 0.001 & 0.005 & 0.01 & 0.011 \\  \hline 
(10, 200, 0.0; 13) & 0.0 & 0.0 & 0.0 & 0.0 & 0.0 & 0.0 & 0.005 & 0.011 & 0.026 & 0.019 \\  \hline 
(10, 200, 0.0; 14) & 0.0 & 0.0 & 0.0 & 0.0 & 0.0 & 0.0 & 0.001 & 0.005 & 0.004 & 0.006 \\  \hline 
(10, 200, 0.0; 15) & 0.0 & 0.0 & 0.0 & 0.0 & 0.0 & 0.0 & 0.002 & 0.006 & 0.01 & 0.011 \\  \hline 
(10, 200, 0.0; 16) & 0.0 & 0.0 & 0.0 & 0.0 & 0.0 & 0.0 & 0.004 & 0.01 & 0.02 & 0.013 \\  \hline 
(10, 200, 0.0; 17) & 0.0 & 0.0 & 0.0 & 0.0 & 0.0 & 0.0 & 0.0 & 0.001 & 0.118 & 0.031 \\  \hline 
(10, 200, 0.0; 18) & 0.0 & 0.0 & 0.0 & 0.0 & 0.0 & 0.0 & 0.003 & 0.009 & 0.034 & 0.018 \\  \hline 
(10, 200, 0.0; 19) & 0.0 & 0.0 & 0.0 & 0.0 & 0.0 & 0.0 & 0.002 & 0.006 & 0.023 & 0.016 \\  \hline 
(10, 200, 0.0; 20) & 0.0 & 0.0 & 0.0 & 0.0 & 0.0 & 0.0 & 0.006 & 0.011 & 0.014 & 0.012 \\  \hline 

\end{tabular}}
\end{center}
\end{table}

\begin{table}[ht]
\caption {\small Realizable Case $p = 20, n = 400, \text{sparsity} = 0.1$ with $\beta^{\ast} \sim N(0.5 \cdot \mathbf{1}_p, 10 \cdot I_p)$} \label{tab:realizable_20_400_10}
\begin{center}
\resizebox{0.7 \textwidth}{!}{\begin{tabular}{|l|l|l|l|l|l|l|l|l|l|l|} \hline
	\multirow{2}*{\bf Settings }& \multicolumn{2}{|l|}{\bf Sorting } & \multicolumn{2}{|l|}{\bf Sorting + Iter} & \multicolumn{2}{|l|}{\bf Sorting + GD } & \multicolumn{2}{|l|}{\bf GD } & \multicolumn{2}{|l|}{\bf SGD }  \\ \cline{2 - 11}
	& Prediction & Recovery & Prediction & Recovery & Prediction & Recovery &  Prediction & Recovery  & Prediction & Recovery \\ \specialrule{.15em}{.05em}{.05em} 

(20, 400, 0.0; 1) & 0.0 & 0.0 & 0.0 & 0.0 & 0.0 & 0.0 & 0.256 & 0.142 & 0.391 & 0.178 \\  \hline 
(20, 400, 0.0; 2) & 0.0 & 0.0 & 0.0 & 0.0 & 0.0 & 0.0 & 0.436 & 0.214 & 0.79 & 0.281 \\  \hline 
(20, 400, 0.0; 3) & 0.0 & 0.0 & 0.0 & 0.0 & 0.0 & 0.0 & 0.295 & 0.157 & 0.401 & 0.18 \\  \hline 
(20, 400, 0.0; 4) & 0.0 & 0.0 & 0.0 & 0.0 & 0.0 & 0.0 & 0.332 & 0.171 & 0.325 & 0.158 \\  \hline 
(20, 400, 0.0; 5) & 0.0 & 0.0 & 0.0 & 0.0 & 0.0 & 0.0 & 0.288 & 0.157 & 0.598 & 0.22 \\  \hline 
(20, 400, 0.0; 6) & 0.0 & 0.0 & 0.0 & 0.0 & 0.0 & 0.0 & 0.43 & 0.196 & 0.633 & 0.232 \\  \hline 
(20, 400, 0.0; 7) & 0.0 & 0.0 & 0.0 & 0.0 & 0.0 & 0.0 & 0.257 & 0.146 & 0.431 & 0.19 \\  \hline 
(20, 400, 0.0; 8) & 0.0 & 0.0 & 0.0 & 0.0 & 0.0 & 0.0 & 0.48 & 0.226 & 0.667 & 0.259 \\  \hline 
(20, 400, 0.0; 9) & 0.0 & 0.0 & 0.0 & 0.0 & 0.0 & 0.0 & 0.457 & 0.204 & 1.063 & 0.31 \\  \hline 
(20, 400, 0.0; 10) & 0.0 & 0.0 & 0.0 & 0.0 & 0.0 & 0.0 & 0.226 & 0.135 & 0.404 & 0.177 \\  \hline 
(20, 400, 0.0; 11) & 0.0 & 0.0 & 0.0 & 0.0 & 0.0 & 0.0 & 0.301 & 0.16 & 0.481 & 0.189 \\  \hline 
(20, 400, 0.0; 12) & 0.0 & 0.0 & 0.0 & 0.0 & 0.0 & 0.0 & 0.383 & 0.183 & 1.135 & 0.31 \\  \hline 
(20, 400, 0.0; 13) & 0.0 & 0.0 & 0.0 & 0.0 & 0.0 & 0.0 & 0.33 & 0.162 & 1.007 & 0.278 \\  \hline 
(20, 400, 0.0; 14) & 0.0 & 0.0 & 0.0 & 0.0 & 0.0 & 0.0 & 0.336 & 0.173 & 0.478 & 0.201 \\  \hline 
(20, 400, 0.0; 15) & 0.0 & 0.0 & 0.0 & 0.0 & 0.0 & 0.0 & 0.305 & 0.163 & 0.5 & 0.199 \\  \hline 
(20, 400, 0.0; 16) & 0.0 & 0.0 & 0.0 & 0.0 & 0.0 & 0.0 & 0.272 & 0.152 & 0.5 & 0.198 \\  \hline 
(20, 400, 0.0; 17) & 0.0 & 0.0 & 0.0 & 0.0 & 0.0 & 0.0 & 0.313 & 0.169 & 0.274 & 0.146 \\  \hline 
(20, 400, 0.0; 18) & 0.0 & 0.0 & 0.0 & 0.0 & 0.0 & 0.0 & 0.288 & 0.154 & 0.338 & 0.169 \\  \hline 
(20, 400, 0.0; 19) & 0.0 & 0.0 & 0.0 & 0.0 & 0.0 & 0.0 & 0.548 & 0.255 & 1.103 & 0.355 \\  \hline 
(20, 400, 0.0; 20) & 0.0 & 0.0 & 0.0 & 0.0 & 0.0 & 0.0 & 0.313 & 0.16 & 0.383 & 0.172 \\  \hline 

\end{tabular}}
\end{center}
\end{table}

\begin{table}[ht]
\caption {\small Realizable Case $p = 20, n = 400, \text{sparsity} = 0.25$ with $\beta^{\ast} \sim N(0.5 \cdot \mathbf{1}_p, 10 \cdot I_p)$} \label{tab:realizable_20_400_25}
\begin{center}
\resizebox{0.7 \textwidth}{!}{\begin{tabular}{|l|l|l|l|l|l|l|l|l|l|l|} \hline
	\multirow{2}*{\bf Settings }& \multicolumn{2}{|l|}{\bf Sorting } & \multicolumn{2}{|l|}{\bf Sorting + Iter} & \multicolumn{2}{|l|}{\bf Sorting + GD } & \multicolumn{2}{|l|}{\bf GD } & \multicolumn{2}{|l|}{\bf SGD }  \\ \cline{2 - 11}
	& Prediction & Recovery & Prediction & Recovery & Prediction & Recovery &  Prediction & Recovery  & Prediction & Recovery \\ \specialrule{.15em}{.05em}{.05em} 

(20, 400, 0.0; 1) & 0.0 & 0.0 & 0.0 & 0.0 & 0.0 & 0.0 & 0.064 & 0.05 & 0.061 & 0.046 \\  \hline 
(20, 400, 0.0; 2) & 0.0 & 0.0 & 0.0 & 0.0 & 0.0 & 0.0 & 0.051 & 0.04 & 0.039 & 0.035 \\  \hline 
(20, 400, 0.0; 3) & 0.0 & 0.0 & 0.0 & 0.0 & 0.0 & 0.0 & 0.069 & 0.05 & 0.119 & 0.063 \\  \hline 
(20, 400, 0.0; 4) & 0.0 & 0.0 & 0.0 & 0.0 & 0.0 & 0.0 & 0.061 & 0.045 & 0.078 & 0.051 \\  \hline 
(20, 400, 0.0; 5) & 0.0 & 0.0 & 0.0 & 0.0 & 0.0 & 0.0 & 0.059 & 0.046 & 0.065 & 0.046 \\  \hline 
(20, 400, 0.0; 6) & 0.0 & 0.0 & 0.0 & 0.0 & 0.0 & 0.0 & 0.061 & 0.046 & 0.248 & 0.089 \\  \hline 
(20, 400, 0.0; 7) & 0.0 & 0.0 & 0.0 & 0.0 & 0.0 & 0.0 & 0.11 & 0.074 & 0.245 & 0.109 \\  \hline 
(20, 400, 0.0; 8) & 0.0 & 0.0 & 0.0 & 0.0 & 0.0 & 0.0 & 0.054 & 0.043 & 0.12 & 0.06 \\  \hline 
(20, 400, 0.0; 9) & 0.0 & 0.0 & 0.0 & 0.0 & 0.0 & 0.0 & 0.079 & 0.055 & 0.087 & 0.057 \\  \hline 
(20, 400, 0.0; 10) & 0.0 & 0.0 & 0.0 & 0.0 & 0.0 & 0.0 & 0.065 & 0.052 & 0.078 & 0.054 \\  \hline 
(20, 400, 0.0; 11) & 0.0 & 0.0 & 0.0 & 0.0 & 0.0 & 0.0 & 0.093 & 0.067 & 0.111 & 0.066 \\  \hline 
(20, 400, 0.0; 12) & 0.0 & 0.0 & 0.0 & 0.0 & 0.0 & 0.0 & 0.082 & 0.058 & 0.118 & 0.073 \\  \hline 
(20, 400, 0.0; 13) & 0.0 & 0.0 & 0.0 & 0.0 & 0.0 & 0.0 & 0.067 & 0.05 & 0.085 & 0.057 \\  \hline 
(20, 400, 0.0; 14) & 0.0 & 0.0 & 0.0 & 0.0 & 0.0 & 0.0 & 0.05 & 0.042 & 0.06 & 0.042 \\  \hline 
(20, 400, 0.0; 15) & 0.0 & 0.0 & 0.0 & 0.0 & 0.0 & 0.0 & 0.065 & 0.05 & 0.12 & 0.067 \\  \hline 
(20, 400, 0.0; 16) & 0.0 & 0.0 & 0.0 & 0.0 & 0.0 & 0.0 & 0.093 & 0.065 & 0.117 & 0.07 \\  \hline 
(20, 400, 0.0; 17) & 0.0 & 0.0 & 0.0 & 0.0 & 0.0 & 0.0 & 0.097 & 0.065 & 0.133 & 0.072 \\  \hline 
(20, 400, 0.0; 18) & 0.0 & 0.0 & 0.0 & 0.0 & 0.0 & 0.0 & 0.067 & 0.049 & 0.09 & 0.057 \\  \hline 
(20, 400, 0.0; 19) & 0.0 & 0.0 & 0.0 & 0.0 & 0.0 & 0.0 & 0.073 & 0.054 & 0.084 & 0.057 \\  \hline 
(20, 400, 0.0; 20) & 0.0 & 0.0 & 0.0 & 0.0 & 0.0 & 0.0 & 0.049 & 0.039 & 0.108 & 0.057 \\  \hline  

\end{tabular}}
\end{center}
\end{table}

\begin{table}[ht]
\caption {\small Realizable Case $p = 20, n = 400, \text{sparsity} = 0.5$ with $\beta^{\ast} \sim N(0.5 \cdot \mathbf{1}_p, 10 \cdot I_p)$} \label{tab:realizable_20_400_50}
\begin{center}
\resizebox{0.7 \textwidth}{!}{\begin{tabular}{|l|l|l|l|l|l|l|l|l|l|l|} \hline
	\multirow{2}*{\bf Settings }& \multicolumn{2}{|l|}{\bf Sorting } & \multicolumn{2}{|l|}{\bf Sorting + Iter} & \multicolumn{2}{|l|}{\bf Sorting + GD } & \multicolumn{2}{|l|}{\bf GD } & \multicolumn{2}{|l|}{\bf SGD }  \\ \cline{2 - 11}
	& Prediction & Recovery & Prediction & Recovery & Prediction & Recovery &  Prediction & Recovery  & Prediction & Recovery \\ \specialrule{.15em}{.05em}{.05em} 

(20, 400, 0.0; 1) & 0.0 & 0.0 & 0.0 & 0.0 & 0.0 & 0.0 & 0.028 & 0.026 & 0.021 & 0.02 \\  \hline 
(20, 400, 0.0; 2) & 0.0 & 0.0 & 0.0 & 0.0 & 0.0 & 0.0 & 0.018 & 0.018 & 0.026 & 0.021 \\  \hline 
(20, 400, 0.0; 3) & 0.0 & 0.0 & 0.0 & 0.0 & 0.0 & 0.0 & 0.015 & 0.016 & 0.023 & 0.018 \\  \hline 
(20, 400, 0.0; 4) & 0.0 & 0.0 & 0.0 & 0.0 & 0.0 & 0.0 & 0.018 & 0.019 & 0.034 & 0.024 \\  \hline 
(20, 400, 0.0; 5) & 0.0 & 0.0 & 0.0 & 0.0 & 0.0 & 0.0 & 0.01 & 0.013 & 0.014 & 0.014 \\  \hline 
(20, 400, 0.0; 6) & 0.0 & 0.0 & 0.0 & 0.0 & 0.0 & 0.0 & 0.02 & 0.019 & 0.017 & 0.016 \\  \hline 
(20, 400, 0.0; 7) & 0.0 & 0.0 & 0.0 & 0.0 & 0.0 & 0.0 & 0.01 & 0.013 & 0.011 & 0.012 \\  \hline 
(20, 400, 0.0; 8) & 0.0 & 0.0 & 0.0 & 0.0 & 0.0 & 0.0 & 0.022 & 0.021 & 0.022 & 0.018 \\  \hline 
(20, 400, 0.0; 9) & 0.0 & 0.0 & 0.0 & 0.0 & 0.0 & 0.0 & 0.016 & 0.018 & 0.005 & 0.009 \\  \hline 
(20, 400, 0.0; 10) & 0.0 & 0.0 & 0.0 & 0.0 & 0.0 & 0.0 & 0.016 & 0.017 & 0.029 & 0.02 \\  \hline 
(20, 400, 0.0; 11) & 0.0 & 0.0 & 0.0 & 0.0 & 0.0 & 0.0 & 0.022 & 0.021 & 0.043 & 0.025 \\  \hline 
(20, 400, 0.0; 12) & 0.0 & 0.0 & 0.0 & 0.0 & 0.0 & 0.0 & 0.019 & 0.019 & 0.017 & 0.017 \\  \hline 
(20, 400, 0.0; 13) & 0.0 & 0.0 & 0.0 & 0.0 & 0.0 & 0.0 & 0.017 & 0.017 & 0.014 & 0.015 \\  \hline 
(20, 400, 0.0; 14) & 0.0 & 0.0 & 0.0 & 0.0 & 0.0 & 0.0 & 0.019 & 0.02 & 0.034 & 0.025 \\  \hline 
(20, 400, 0.0; 15) & 0.0 & 0.0 & 0.0 & 0.0 & 0.0 & 0.0 & 0.015 & 0.016 & 0.015 & 0.014 \\  \hline 
(20, 400, 0.0; 16) & 0.0 & 0.0 & 0.0 & 0.0 & 0.0 & 0.0 & 0.017 & 0.017 & 0.011 & 0.012 \\  \hline 
(20, 400, 0.0; 17) & 0.0 & 0.0 & 0.0 & 0.0 & 0.0 & 0.0 & 0.024 & 0.022 & 0.065 & 0.031 \\  \hline 
(20, 400, 0.0; 18) & 0.0 & 0.0 & 0.0 & 0.0 & 0.0 & 0.0 & 0.017 & 0.017 & 0.011 & 0.014 \\  \hline 
(20, 400, 0.0; 19) & 0.0 & 0.0 & 0.0 & 0.0 & 0.0 & 0.0 & 0.017 & 0.018 & 0.022 & 0.018 \\  \hline 
(20, 400, 0.0; 20) & 0.0 & 0.0 & 0.0 & 0.0 & 0.0 & 0.0 & 0.019 & 0.019 & 0.02 & 0.016 \\  \hline 

\end{tabular}}
\end{center}
\end{table}

\begin{table}[ht]
\caption {\small Realizable Case $p = 20, n = 400, \text{sparsity} = 0.75$ with $\beta^{\ast} \sim N(0.5 \cdot \mathbf{1}_p, 10 \cdot I_p)$} \label{tab:realizable_20_400_75}
\begin{center}
\resizebox{0.7 \textwidth}{!}{\begin{tabular}{|l|l|l|l|l|l|l|l|l|l|l|} \hline
	\multirow{2}*{\bf Settings }& \multicolumn{2}{|l|}{\bf Sorting } & \multicolumn{2}{|l|}{\bf Sorting + Iter} & \multicolumn{2}{|l|}{\bf Sorting + GD } & \multicolumn{2}{|l|}{\bf GD } & \multicolumn{2}{|l|}{\bf SGD }  \\ \cline{2 - 11}
	& Prediction & Recovery & Prediction & Recovery & Prediction & Recovery &  Prediction & Recovery  & Prediction & Recovery \\ \specialrule{.15em}{.05em}{.05em} 

(20, 400, 0.0; 1) & 0.0 & 0.0 & 0.0 & 0.0 & 0.0 & 0.0 & 0.01 & 0.012 & 0.02 & 0.013 \\  \hline 
(20, 400, 0.0; 2) & 0.0 & 0.0 & 0.0 & 0.0 & 0.0 & 0.0 & 0.006 & 0.008 & 0.006 & 0.006 \\  \hline 
(20, 400, 0.0; 3) & 0.0 & 0.0 & 0.0 & 0.0 & 0.0 & 0.0 & 0.003 & 0.005 & 0.005 & 0.006 \\  \hline 
(20, 400, 0.0; 4) & 0.0 & 0.0 & 0.0 & 0.0 & 0.0 & 0.0 & 0.009 & 0.011 & 0.027 & 0.015 \\  \hline 
(20, 400, 0.0; 5) & 0.0 & 0.0 & 0.0 & 0.0 & 0.0 & 0.0 & 0.009 & 0.011 & 0.008 & 0.009 \\  \hline 
(20, 400, 0.0; 6) & 0.0 & 0.0 & 0.0 & 0.0 & 0.0 & 0.0 & 0.009 & 0.012 & 0.016 & 0.013 \\  \hline 
(20, 400, 0.0; 7) & 0.0 & 0.0 & 0.0 & 0.0 & 0.0 & 0.0 & 0.005 & 0.008 & 0.009 & 0.008 \\  \hline 
(20, 400, 0.0; 8) & 0.0 & 0.0 & 0.0 & 0.0 & 0.0 & 0.0 & 0.003 & 0.005 & 0.013 & 0.01 \\  \hline 
(20, 400, 0.0; 9) & 0.0 & 0.0 & 0.0 & 0.0 & 0.0 & 0.0 & 0.007 & 0.01 & 0.01 & 0.009 \\  \hline 
(20, 400, 0.0; 10) & 0.0 & 0.0 & 0.0 & 0.0 & 0.0 & 0.0 & 0.007 & 0.009 & 0.007 & 0.007 \\  \hline 
(20, 400, 0.0; 11) & 0.0 & 0.0 & 0.0 & 0.0 & 0.0 & 0.0 & 0.008 & 0.01 & 0.014 & 0.01 \\  \hline 
(20, 400, 0.0; 12) & 0.0 & 0.0 & 0.0 & 0.0 & 0.0 & 0.0 & 0.011 & 0.012 & 0.01 & 0.009 \\  \hline 
(20, 400, 0.0; 13) & 0.0 & 0.0 & 0.0 & 0.0 & 0.0 & 0.0 & 0.005 & 0.008 & 0.007 & 0.008 \\  \hline 
(20, 400, 0.0; 14) & 0.0 & 0.0 & 0.0 & 0.0 & 0.0 & 0.0 & 0.007 & 0.009 & 0.012 & 0.009 \\  \hline 
(20, 400, 0.0; 15) & 0.0 & 0.0 & 0.0 & 0.0 & 0.0 & 0.0 & 0.005 & 0.008 & 0.003 & 0.004 \\  \hline 
(20, 400, 0.0; 16) & 0.0 & 0.0 & 0.0 & 0.0 & 0.0 & 0.0 & 0.005 & 0.007 & 0.011 & 0.009 \\  \hline 
(20, 400, 0.0; 17) & 0.0 & 0.0 & 0.0 & 0.0 & 0.0 & 0.0 & 0.006 & 0.009 & 0.011 & 0.01 \\  \hline 
(20, 400, 0.0; 18) & 0.0 & 0.0 & 0.0 & 0.0 & 0.0 & 0.0 & 0.005 & 0.008 & 0.011 & 0.009 \\  \hline 
(20, 400, 0.0; 19) & 0.0 & 0.0 & 0.0 & 0.0 & 0.0 & 0.0 & 0.01 & 0.012 & 0.01 & 0.01 \\  \hline 
(20, 400, 0.0; 20) & 0.0 & 0.0 & 0.0 & 0.0 & 0.0 & 0.0 & 0.012 & 0.014 & 0.032 & 0.018 \\  \hline   

\end{tabular}}
\end{center}
\end{table}

\begin{table}[ht]
\caption {\small Realizable Case $p = 20, n = 400, \text{sparsity} = 0.9$ with $\beta^{\ast} \sim N(0.5 \cdot \mathbf{1}_p, 10 \cdot I_p)$} \label{tab:realizable_20_400_90}
\begin{center}
\resizebox{0.7 \textwidth}{!}{\begin{tabular}{|l|l|l|l|l|l|l|l|l|l|l|} \hline
	\multirow{2}*{\bf Settings }& \multicolumn{2}{|l|}{\bf Sorting } & \multicolumn{2}{|l|}{\bf Sorting + Iter} & \multicolumn{2}{|l|}{\bf Sorting + GD } & \multicolumn{2}{|l|}{\bf GD } & \multicolumn{2}{|l|}{\bf SGD }  \\ \cline{2 - 11}
	& Prediction & Recovery & Prediction & Recovery & Prediction & Recovery &  Prediction & Recovery  & Prediction & Recovery \\ \specialrule{.15em}{.05em}{.05em} 

(20, 400, 0.0; 1) & 0.0 & 0.0 & 0.0 & 0.0 & 0.0 & 0.0 & 0.004 & 0.007 & 0.005 & 0.006 \\  \hline 
(20, 400, 0.0; 2) & 0.0 & 0.0 & 0.0 & 0.0 & 0.0 & 0.0 & 0.004 & 0.006 & 0.009 & 0.008 \\  \hline 
(20, 400, 0.0; 3) & 0.0 & 0.0 & 0.0 & 0.0 & 0.0 & 0.0 & 0.005 & 0.007 & 0.006 & 0.005 \\  \hline 
(20, 400, 0.0; 4) & 0.0 & 0.0 & 0.0 & 0.0 & 0.0 & 0.0 & 0.002 & 0.004 & 0.02 & 0.012 \\  \hline 
(20, 400, 0.0; 5) & 0.0 & 0.0 & 0.0 & 0.0 & 0.0 & 0.0 & 0.003 & 0.005 & 0.006 & 0.006 \\  \hline 
(20, 400, 0.0; 6) & 0.0 & 0.0 & 0.0 & 0.0 & 0.0 & 0.0 & 0.003 & 0.005 & 0.003 & 0.004 \\  \hline 
(20, 400, 0.0; 7) & 0.0 & 0.0 & 0.0 & 0.0 & 0.0 & 0.0 & 0.005 & 0.008 & 0.01 & 0.009 \\  \hline 
(20, 400, 0.0; 8) & 0.0 & 0.0 & 0.0 & 0.0 & 0.0 & 0.0 & 0.006 & 0.008 & 0.006 & 0.006 \\  \hline 
(20, 400, 0.0; 9) & 0.0 & 0.0 & 0.0 & 0.0 & 0.0 & 0.0 & 0.003 & 0.006 & 0.012 & 0.008 \\  \hline 
(20, 400, 0.0; 10) & 0.0 & 0.0 & 0.0 & 0.0 & 0.0 & 0.0 & 0.005 & 0.008 & 0.007 & 0.006 \\  \hline 
(20, 400, 0.0; 11) & 0.0 & 0.0 & 0.0 & 0.0 & 0.0 & 0.0 & 0.004 & 0.007 & 0.014 & 0.009 \\  \hline 
(20, 400, 0.0; 12) & 0.0 & 0.0 & 0.0 & 0.0 & 0.0 & 0.0 & 0.006 & 0.008 & 0.009 & 0.007 \\  \hline 
(20, 400, 0.0; 13) & 0.0 & 0.0 & 0.0 & 0.0 & 0.0 & 0.0 & 0.006 & 0.008 & 0.006 & 0.006 \\  \hline 
(20, 400, 0.0; 14) & 0.0 & 0.0 & 0.0 & 0.0 & 0.0 & 0.0 & 0.007 & 0.008 & 0.018 & 0.011 \\  \hline 
(20, 400, 0.0; 15) & 0.0 & 0.0 & 0.0 & 0.0 & 0.0 & 0.0 & 0.004 & 0.006 & 0.007 & 0.007 \\  \hline 
(20, 400, 0.0; 16) & 0.0 & 0.0 & 0.0 & 0.0 & 0.0 & 0.0 & 0.003 & 0.005 & 0.004 & 0.005 \\  \hline 
(20, 400, 0.0; 17) & 0.0 & 0.0 & 0.0 & 0.0 & 0.0 & 0.0 & 0.006 & 0.009 & 0.006 & 0.006 \\  \hline 
(20, 400, 0.0; 18) & 0.0 & 0.0 & 0.0 & 0.0 & 0.0 & 0.0 & 0.005 & 0.007 & 0.004 & 0.005 \\  \hline 
(20, 400, 0.0; 19) & 0.0 & 0.0 & 0.0 & 0.0 & 0.0 & 0.0 & 0.001 & 0.003 & 0.013 & 0.008 \\  \hline 
(20, 400, 0.0; 20) & 0.0 & 0.0 & 0.0 & 0.0 & 0.0 & 0.0 & 0.005 & 0.007 & 0.011 & 0.008 \\  \hline    

\end{tabular}}
\end{center}
\end{table}

\begin{table}[ht]
\caption {\small Realizable Case $p = 50, n = 1000, \text{sparsity} = 0.1$ with $\beta^{\ast} \sim N(0.5 \cdot \mathbf{1}_p, 10 \cdot I_p)$} \label{tab:realizable_50_1000_10}
\begin{center}
\resizebox{0.7 \textwidth}{!}{\begin{tabular}{|l|l|l|l|l|l|l|l|l|l|l|} \hline
	\multirow{2}*{\bf Settings }& \multicolumn{2}{|l|}{\bf Sorting } & \multicolumn{2}{|l|}{\bf Sorting + Iter} & \multicolumn{2}{|l|}{\bf Sorting + GD } & \multicolumn{2}{|l|}{\bf GD } & \multicolumn{2}{|l|}{\bf SGD }  \\ \cline{2 - 11}
	& Prediction & Recovery & Prediction & Recovery & Prediction & Recovery &  Prediction & Recovery  & Prediction & Recovery \\ \specialrule{.15em}{.05em}{.05em} 

(50, 1000, 0.0; 1) & 0.0 & 0.0 & 0.0 & 0.0 & 0.0 & 0.0 & 0.715 & 0.17 & 1.104 & 0.21 \\  \hline 
(50, 1000, 0.0; 2) & 0.0 & 0.0 & 0.0 & 0.0 & 0.0 & 0.0 & 0.601 & 0.154 & 0.787 & 0.175 \\  \hline 
(50, 1000, 0.0; 3) & 0.0 & 0.0 & 0.0 & 0.0 & 0.0 & 0.0 & 0.661 & 0.167 & 0.741 & 0.175 \\  \hline 
(50, 1000, 0.0; 4) & 0.0 & 0.0 & 0.0 & 0.0 & 0.0 & 0.0 & 0.551 & 0.143 & 0.631 & 0.151 \\  \hline 
(50, 1000, 0.0; 5) & 0.0 & 0.0 & 0.0 & 0.0 & 0.0 & 0.0 & 0.528 & 0.144 & 0.661 & 0.16 \\  \hline 
(50, 1000, 0.0; 6) & 0.0 & 0.0 & 0.0 & 0.0 & 0.0 & 0.0 & 0.347 & 0.105 & 0.417 & 0.112 \\  \hline 
(50, 1000, 0.0; 7) & 0.0 & 0.0 & 0.0 & 0.0 & 0.0 & 0.0 & 0.493 & 0.13 & 0.586 & 0.141 \\  \hline 
(50, 1000, 0.0; 8) & 0.0 & 0.0 & 0.0 & 0.0 & 0.0 & 0.0 & 0.526 & 0.141 & 0.814 & 0.176 \\  \hline 
(50, 1000, 0.0; 9) & 0.0 & 0.0 & 0.0 & 0.0 & 0.0 & 0.0 & 0.46 & 0.128 & 0.848 & 0.173 \\  \hline 
(50, 1000, 0.0; 10) & 0.0 & 0.0 & 0.0 & 0.0 & 0.0 & 0.0 & 0.571 & 0.15 & 0.703 & 0.167 \\  \hline 
(50, 1000, 0.0; 11) & 0.0 & 0.0 & 0.0 & 0.0 & 0.0 & 0.0 & 0.641 & 0.157 & 0.883 & 0.183 \\  \hline 
(50, 1000, 0.0; 12) & 0.0 & 0.0 & 0.0 & 0.0 & 0.0 & 0.0 & 0.601 & 0.157 & 0.785 & 0.179 \\  \hline 
(50, 1000, 0.0; 13) & 0.0 & 0.0 & 0.0 & 0.0 & 0.0 & 0.0 & 0.581 & 0.144 & 0.871 & 0.176 \\  \hline 
(50, 1000, 0.0; 14) & 0.0 & 0.0 & 0.0 & 0.0 & 0.0 & 0.0 & 0.469 & 0.126 & 0.569 & 0.139 \\  \hline 
(50, 1000, 0.0; 15) & 0.0 & 0.0 & 0.0 & 0.0 & 0.0 & 0.0 & 0.535 & 0.14 & 0.882 & 0.181 \\  \hline 
(50, 1000, 0.0; 16) & 0.0 & 0.0 & 0.0 & 0.0 & 0.0 & 0.0 & 0.7 & 0.168 & 0.85 & 0.184 \\  \hline 
(50, 1000, 0.0; 17) & 0.0 & 0.0 & 0.0 & 0.0 & 0.0 & 0.0 & 0.613 & 0.156 & 0.881 & 0.182 \\  \hline 
(50, 1000, 0.0; 18) & 0.0 & 0.0 & 0.0 & 0.0 & 0.0 & 0.0 & 0.551 & 0.142 & 0.802 & 0.17 \\  \hline 
(50, 1000, 0.0; 19) & 0.0 & 0.0 & 0.0 & 0.0 & 0.0 & 0.0 & 0.475 & 0.131 & 0.589 & 0.148 \\  \hline 
(50, 1000, 0.0; 20) & 0.0 & 0.0 & 0.0 & 0.0 & 0.0 & 0.0 & 0.566 & 0.149 & 0.608 & 0.154 \\  \hline 

\end{tabular}}
\end{center}
\end{table}

\begin{table}[ht]
\caption {\small Realizable Case $p = 50, n = 1000, \text{sparsity} = 0.25$ with $\beta^{\ast} \sim N(0.5 \cdot \mathbf{1}_p, 10 \cdot I_p)$} \label{tab:realizable_50_1000_25}
\begin{center}
\resizebox{0.7 \textwidth}{!}{\begin{tabular}{|l|l|l|l|l|l|l|l|l|l|l|} \hline
	\multirow{2}*{\bf Settings }& \multicolumn{2}{|l|}{\bf Sorting } & \multicolumn{2}{|l|}{\bf Sorting + Iter} & \multicolumn{2}{|l|}{\bf Sorting + GD } & \multicolumn{2}{|l|}{\bf GD } & \multicolumn{2}{|l|}{\bf SGD }  \\ \cline{2 - 11}
	& Prediction & Recovery & Prediction & Recovery & Prediction & Recovery &  Prediction & Recovery  & Prediction & Recovery \\ \specialrule{.15em}{.05em}{.05em} 

(50, 1000, 0.0; 1) & 0.0 & 0.0 & 0.0 & 0.0 & 0.0 & 0.0 & 0.096 & 0.04 & 0.128 & 0.045 \\  \hline 
(50, 1000, 0.0; 2) & 0.0 & 0.0 & 0.0 & 0.0 & 0.0 & 0.0 & 0.073 & 0.033 & 0.075 & 0.031 \\  \hline 
(50, 1000, 0.0; 3) & 0.0 & 0.0 & 0.0 & 0.0 & 0.0 & 0.0 & 0.092 & 0.038 & 0.073 & 0.033 \\  \hline 
(50, 1000, 0.0; 4) & 0.0 & 0.0 & 0.0 & 0.0 & 0.0 & 0.0 & 0.09 & 0.038 & 0.071 & 0.032 \\  \hline 
(50, 1000, 0.0; 5) & 0.0 & 0.0 & 0.0 & 0.0 & 0.0 & 0.0 & 0.095 & 0.039 & 0.107 & 0.04 \\  \hline 
(50, 1000, 0.0; 6) & 0.0 & 0.0 & 0.0 & 0.0 & 0.0 & 0.0 & 0.112 & 0.044 & 0.145 & 0.051 \\  \hline 
(50, 1000, 0.0; 7) & 0.0 & 0.0 & 0.0 & 0.0 & 0.0 & 0.0 & 0.086 & 0.037 & 0.078 & 0.034 \\  \hline 
(50, 1000, 0.0; 8) & 0.0 & 0.0 & 0.0 & 0.0 & 0.0 & 0.0 & 0.091 & 0.037 & 0.136 & 0.045 \\  \hline 
(50, 1000, 0.0; 9) & 0.0 & 0.0 & 0.0 & 0.0 & 0.0 & 0.0 & 0.1 & 0.041 & 0.091 & 0.038 \\  \hline 
(50, 1000, 0.0; 10) & 0.0 & 0.0 & 0.0 & 0.0 & 0.0 & 0.0 & 0.121 & 0.047 & 0.104 & 0.042 \\  \hline 
(50, 1000, 0.0; 11) & 0.0 & 0.0 & 0.0 & 0.0 & 0.0 & 0.0 & 0.078 & 0.033 & 0.097 & 0.037 \\  \hline 
(50, 1000, 0.0; 12) & 0.0 & 0.0 & 0.0 & 0.0 & 0.0 & 0.0 & 0.088 & 0.038 & 0.088 & 0.035 \\  \hline 
(50, 1000, 0.0; 13) & 0.0 & 0.0 & 0.0 & 0.0 & 0.0 & 0.0 & 0.073 & 0.033 & 0.073 & 0.03 \\  \hline 
(50, 1000, 0.0; 14) & 0.0 & 0.0 & 0.0 & 0.0 & 0.0 & 0.0 & 0.088 & 0.037 & 0.124 & 0.042 \\  \hline 
(50, 1000, 0.0; 15) & 0.0 & 0.0 & 0.0 & 0.0 & 0.0 & 0.0 & 0.082 & 0.036 & 0.106 & 0.038 \\  \hline 
(50, 1000, 0.0; 16) & 0.0 & 0.0 & 0.0 & 0.0 & 0.0 & 0.0 & 0.067 & 0.03 & 0.144 & 0.044 \\  \hline 
(50, 1000, 0.0; 17) & 0.0 & 0.0 & 0.0 & 0.0 & 0.0 & 0.0 & 0.101 & 0.04 & 0.132 & 0.045 \\  \hline 
(50, 1000, 0.0; 18) & 0.0 & 0.0 & 0.0 & 0.0 & 0.0 & 0.0 & 0.083 & 0.035 & 0.121 & 0.041 \\  \hline 
(50, 1000, 0.0; 19) & 0.0 & 0.0 & 0.0 & 0.0 & 0.0 & 0.0 & 0.093 & 0.039 & 0.107 & 0.041 \\  \hline 
(50, 1000, 0.0; 20) & 0.0 & 0.0 & 0.0 & 0.0 & 0.0 & 0.0 & 0.1 & 0.04 & 0.251 & 0.062 \\  \hline  

\end{tabular}}
\end{center}
\end{table}

\begin{table}[ht]
\caption {\small Realizable Case $p = 50, n = 1000, \text{sparsity} = 0.5$ with $\beta^{\ast} \sim N(0.5 \cdot \mathbf{1}_p, 10 \cdot I_p)$} \label{tab:realizable_50_1000_50}
\begin{center}
\resizebox{0.7 \textwidth}{!}{\begin{tabular}{|l|l|l|l|l|l|l|l|l|l|l|} \hline
	\multirow{2}*{\bf Settings }& \multicolumn{2}{|l|}{\bf Sorting } & \multicolumn{2}{|l|}{\bf Sorting + Iter} & \multicolumn{2}{|l|}{\bf Sorting + GD } & \multicolumn{2}{|l|}{\bf GD } & \multicolumn{2}{|l|}{\bf SGD }  \\ \cline{2 - 11}
	& Prediction & Recovery & Prediction & Recovery & Prediction & Recovery &  Prediction & Recovery  & Prediction & Recovery \\ \specialrule{.15em}{.05em}{.05em} 

(50, 1000, 0.0; 1) & 0.0 & 0.0 & 0.0 & 0.0 & 0.0 & 0.0 & 0.026 & 0.015 & 0.037 & 0.016 \\  \hline 
(50, 1000, 0.0; 2) & 0.0 & 0.0 & 0.0 & 0.0 & 0.0 & 0.0 & 0.026 & 0.015 & 0.034 & 0.016 \\  \hline 
(50, 1000, 0.0; 3) & 0.0 & 0.0 & 0.0 & 0.0 & 0.0 & 0.0 & 0.028 & 0.016 & 0.045 & 0.018 \\  \hline 
(50, 1000, 0.0; 4) & 0.0 & 0.0 & 0.0 & 0.0 & 0.0 & 0.0 & 0.024 & 0.014 & 0.021 & 0.012 \\  \hline 
(50, 1000, 0.0; 5) & 0.0 & 0.0 & 0.0 & 0.0 & 0.0 & 0.0 & 0.02 & 0.012 & 0.028 & 0.014 \\  \hline 
(50, 1000, 0.0; 6) & 0.0 & 0.0 & 0.0 & 0.0 & 0.0 & 0.0 & 0.02 & 0.012 & 0.022 & 0.012 \\  \hline 
(50, 1000, 0.0; 7) & 0.0 & 0.0 & 0.0 & 0.0 & 0.0 & 0.0 & 0.017 & 0.011 & 0.024 & 0.012 \\  \hline 
(50, 1000, 0.0; 8) & 0.0 & 0.0 & 0.0 & 0.0 & 0.0 & 0.0 & 0.022 & 0.013 & 0.039 & 0.016 \\  \hline 
(50, 1000, 0.0; 9) & 0.0 & 0.0 & 0.0 & 0.0 & 0.0 & 0.0 & 0.019 & 0.011 & 0.015 & 0.009 \\  \hline 
(50, 1000, 0.0; 10) & 0.0 & 0.0 & 0.0 & 0.0 & 0.0 & 0.0 & 0.019 & 0.012 & 0.034 & 0.015 \\  \hline 
(50, 1000, 0.0; 11) & 0.0 & 0.0 & 0.0 & 0.0 & 0.0 & 0.0 & 0.018 & 0.012 & 0.014 & 0.009 \\  \hline 
(50, 1000, 0.0; 12) & 0.0 & 0.0 & 0.0 & 0.0 & 0.0 & 0.0 & 0.024 & 0.014 & 0.02 & 0.011 \\  \hline 
(50, 1000, 0.0; 13) & 0.0 & 0.0 & 0.0 & 0.0 & 0.0 & 0.0 & 0.024 & 0.015 & 0.039 & 0.017 \\  \hline 
(50, 1000, 0.0; 14) & 0.0 & 0.0 & 0.0 & 0.0 & 0.0 & 0.0 & 0.037 & 0.019 & 0.04 & 0.017 \\  \hline 
(50, 1000, 0.0; 15) & 0.0 & 0.0 & 0.0 & 0.0 & 0.0 & 0.0 & 0.021 & 0.013 & 0.042 & 0.016 \\  \hline 
(50, 1000, 0.0; 16) & 0.0 & 0.0 & 0.0 & 0.0 & 0.0 & 0.0 & 0.02 & 0.013 & 0.053 & 0.019 \\  \hline 
(50, 1000, 0.0; 17) & 0.0 & 0.0 & 0.0 & 0.0 & 0.0 & 0.0 & 0.019 & 0.012 & 0.017 & 0.011 \\  \hline 
(50, 1000, 0.0; 18) & 0.0 & 0.0 & 0.0 & 0.0 & 0.0 & 0.0 & 0.015 & 0.01 & 0.014 & 0.009 \\  \hline 
(50, 1000, 0.0; 19) & 0.0 & 0.0 & 0.0 & 0.0 & 0.0 & 0.0 & 0.018 & 0.012 & 0.034 & 0.016 \\  \hline 
(50, 1000, 0.0; 20) & 0.0 & 0.0 & 0.0 & 0.0 & 0.0 & 0.0 & 0.019 & 0.012 & 0.015 & 0.01 \\  \hline   

\end{tabular}}
\end{center}
\end{table}

\begin{table}[ht]
\caption {\small Realizable Case $p = 50, n = 1000, \text{sparsity} = 0.75$ with $\beta^{\ast} \sim N(0.5 \cdot \mathbf{1}_p, 10 \cdot I_p)$} \label{tab:realizable_50_1000_75}
\begin{center}
\resizebox{0.7 \textwidth}{!}{\begin{tabular}{|l|l|l|l|l|l|l|l|l|l|l|} \hline
	\multirow{2}*{\bf Settings }& \multicolumn{2}{|l|}{\bf Sorting } & \multicolumn{2}{|l|}{\bf Sorting + Iter} & \multicolumn{2}{|l|}{\bf Sorting + GD } & \multicolumn{2}{|l|}{\bf GD } & \multicolumn{2}{|l|}{\bf SGD }  \\ \cline{2 - 11}
	& Prediction & Recovery & Prediction & Recovery & Prediction & Recovery &  Prediction & Recovery  & Prediction & Recovery \\ \specialrule{.15em}{.05em}{.05em} 

(50, 1000, 0.0; 1) & 0.0 & 0.0 & 0.0 & 0.0 & 0.0 & 0.0 & 0.009 & 0.007 & 0.009 & 0.006 \\  \hline 
(50, 1000, 0.0; 2) & 0.0 & 0.0 & 0.0 & 0.0 & 0.0 & 0.0 & 0.011 & 0.008 & 0.014 & 0.007 \\  \hline 
(50, 1000, 0.0; 3) & 0.0 & 0.0 & 0.0 & 0.0 & 0.0 & 0.0 & 0.011 & 0.008 & 0.027 & 0.01 \\  \hline 
(50, 1000, 0.0; 4) & 0.0 & 0.0 & 0.0 & 0.0 & 0.0 & 0.0 & 0.01 & 0.007 & 0.01 & 0.006 \\  \hline 
(50, 1000, 0.0; 5) & 0.0 & 0.0 & 0.0 & 0.0 & 0.0 & 0.0 & 0.008 & 0.006 & 0.015 & 0.007 \\  \hline 
(50, 1000, 0.0; 6) & 0.0 & 0.0 & 0.0 & 0.0 & 0.0 & 0.0 & 0.005 & 0.005 & 0.008 & 0.005 \\  \hline 
(50, 1000, 0.0; 7) & 0.0 & 0.0 & 0.0 & 0.0 & 0.0 & 0.0 & 0.016 & 0.01 & 0.006 & 0.005 \\  \hline 
(50, 1000, 0.0; 8) & 0.0 & 0.0 & 0.0 & 0.0 & 0.0 & 0.0 & 0.007 & 0.006 & 0.009 & 0.006 \\  \hline 
(50, 1000, 0.0; 9) & 0.0 & 0.0 & 0.0 & 0.0 & 0.0 & 0.0 & 0.012 & 0.008 & 0.011 & 0.006 \\  \hline 
(50, 1000, 0.0; 10) & 0.0 & 0.0 & 0.0 & 0.0 & 0.0 & 0.0 & 0.008 & 0.007 & 0.022 & 0.009 \\  \hline 
(50, 1000, 0.0; 11) & 0.0 & 0.0 & 0.0 & 0.0 & 0.0 & 0.0 & 0.007 & 0.006 & 0.01 & 0.006 \\  \hline 
(50, 1000, 0.0; 12) & 0.0 & 0.0 & 0.0 & 0.0 & 0.0 & 0.0 & 0.008 & 0.006 & 0.014 & 0.007 \\  \hline 
(50, 1000, 0.0; 13) & 0.0 & 0.0 & 0.0 & 0.0 & 0.0 & 0.0 & 0.007 & 0.006 & 0.008 & 0.006 \\  \hline 
(50, 1000, 0.0; 14) & 0.0 & 0.0 & 0.0 & 0.0 & 0.0 & 0.0 & 0.006 & 0.006 & 0.026 & 0.011 \\  \hline 
(50, 1000, 0.0; 15) & 0.0 & 0.0 & 0.0 & 0.0 & 0.0 & 0.0 & 0.007 & 0.006 & 0.006 & 0.005 \\  \hline 
(50, 1000, 0.0; 16) & 0.0 & 0.0 & 0.0 & 0.0 & 0.0 & 0.0 & 0.006 & 0.005 & 0.011 & 0.006 \\  \hline 
(50, 1000, 0.0; 17) & 0.0 & 0.0 & 0.0 & 0.0 & 0.0 & 0.0 & 0.006 & 0.005 & 0.01 & 0.006 \\  \hline 
(50, 1000, 0.0; 18) & 0.0 & 0.0 & 0.0 & 0.0 & 0.0 & 0.0 & 0.009 & 0.007 & 0.007 & 0.005 \\  \hline 
(50, 1000, 0.0; 19) & 0.0 & 0.0 & 0.0 & 0.0 & 0.0 & 0.0 & 0.006 & 0.006 & 0.011 & 0.007 \\  \hline 
(50, 1000, 0.0; 20) & 0.0 & 0.0 & 0.0 & 0.0 & 0.0 & 0.0 & 0.011 & 0.008 & 0.014 & 0.007 \\  \hline 

\end{tabular}}
\end{center}
\end{table}

\begin{table}[ht]
\caption {\small Realizable Case $p = 50, n = 1000, \text{sparsity} = 0.9$ with $\beta^{\ast} \sim N(0.5 \cdot \mathbf{1}_p, 10 \cdot I_p)$} \label{tab:realizable_50_1000_90}
\begin{center}
\resizebox{0.7 \textwidth}{!}{\begin{tabular}{|l|l|l|l|l|l|l|l|l|l|l|} \hline
	\multirow{2}*{\bf Settings }& \multicolumn{2}{|l|}{\bf Sorting } & \multicolumn{2}{|l|}{\bf Sorting + Iter} & \multicolumn{2}{|l|}{\bf Sorting + GD } & \multicolumn{2}{|l|}{\bf GD } & \multicolumn{2}{|l|}{\bf SGD }  \\ \cline{2 - 11}
	& Prediction & Recovery & Prediction & Recovery & Prediction & Recovery &  Prediction & Recovery  & Prediction & Recovery \\ \specialrule{.15em}{.05em}{.05em} 

(50, 1000, 0.0; 1) & 0.0 & 0.0 & 0.0 & 0.0 & 0.0 & 0.0 & 0.005 & 0.005 & 0.006 & 0.004 \\  \hline 
(50, 1000, 0.0; 2) & 0.0 & 0.0 & 0.0 & 0.0 & 0.0 & 0.0 & 0.006 & 0.005 & 0.009 & 0.005 \\  \hline 
(50, 1000, 0.0; 3) & 0.0 & 0.0 & 0.0 & 0.0 & 0.0 & 0.0 & 0.005 & 0.005 & 0.008 & 0.005 \\  \hline 
(50, 1000, 0.0; 4) & 0.0 & 0.0 & 0.0 & 0.0 & 0.0 & 0.0 & 0.002 & 0.003 & 0.03 & 0.009 \\  \hline 
(50, 1000, 0.0; 5) & 0.0 & 0.0 & 0.0 & 0.0 & 0.0 & 0.0 & 0.004 & 0.004 & 0.009 & 0.005 \\  \hline 
(50, 1000, 0.0; 6) & 0.0 & 0.0 & 0.0 & 0.0 & 0.0 & 0.0 & 0.008 & 0.006 & 0.008 & 0.005 \\  \hline 
(50, 1000, 0.0; 7) & 0.0 & 0.0 & 0.0 & 0.0 & 0.0 & 0.0 & 0.008 & 0.006 & 0.01 & 0.005 \\  \hline 
(50, 1000, 0.0; 8) & 0.0 & 0.0 & 0.0 & 0.0 & 0.0 & 0.0 & 0.005 & 0.005 & 0.012 & 0.005 \\  \hline 
(50, 1000, 0.0; 9) & 0.0 & 0.0 & 0.0 & 0.0 & 0.0 & 0.0 & 0.009 & 0.007 & 0.01 & 0.006 \\  \hline 
(50, 1000, 0.0; 10) & 0.0 & 0.0 & 0.0 & 0.0 & 0.0 & 0.0 & 0.005 & 0.004 & 0.01 & 0.005 \\  \hline 
(50, 1000, 0.0; 11) & 0.0 & 0.0 & 0.0 & 0.0 & 0.0 & 0.0 & 0.008 & 0.006 & 0.005 & 0.004 \\  \hline 
(50, 1000, 0.0; 12) & 0.0 & 0.0 & 0.0 & 0.0 & 0.0 & 0.0 & 0.007 & 0.006 & 0.006 & 0.004 \\  \hline 
(50, 1000, 0.0; 13) & 0.0 & 0.0 & 0.0 & 0.0 & 0.0 & 0.0 & 0.005 & 0.005 & 0.007 & 0.004 \\  \hline 
(50, 1000, 0.0; 14) & 0.0 & 0.0 & 0.0 & 0.0 & 0.0 & 0.0 & 0.003 & 0.003 & 0.006 & 0.004 \\  \hline 
(50, 1000, 0.0; 15) & 0.0 & 0.0 & 0.0 & 0.0 & 0.0 & 0.0 & 0.007 & 0.006 & 0.006 & 0.004 \\  \hline 
(50, 1000, 0.0; 16) & 0.0 & 0.0 & 0.0 & 0.0 & 0.0 & 0.0 & 0.006 & 0.005 & 0.004 & 0.003 \\  \hline 
(50, 1000, 0.0; 17) & 0.0 & 0.0 & 0.0 & 0.0 & 0.0 & 0.0 & 0.007 & 0.006 & 0.01 & 0.005 \\  \hline 
(50, 1000, 0.0; 18) & 0.0 & 0.0 & 0.0 & 0.0 & 0.0 & 0.0 & 0.005 & 0.004 & 0.008 & 0.004 \\  \hline 
(50, 1000, 0.0; 19) & 0.0 & 0.0 & 0.0 & 0.0 & 0.0 & 0.0 & 0.005 & 0.005 & 0.013 & 0.006 \\  \hline 
(50, 1000, 0.0; 20) & 0.0 & 0.0 & 0.0 & 0.0 & 0.0 & 0.0 & 0.005 & 0.005 & 0.008 & 0.005 \\  \hline     

\end{tabular}}
\end{center}
\end{table}

\end{document}